\documentclass[10pt]{amsart}

\usepackage[latin2]{inputenc}
\usepackage{amsmath}
\usepackage{graphicx}
\usepackage{amssymb}
\usepackage{esint}
\usepackage{color}
\usepackage{amsthm}
\usepackage{epsfig}
\usepackage{tikz}
\usepackage{enumitem}
\usepackage{mathtools}
\usepackage{esvect}
\usepackage{framed}
\usepackage[english]{babel}
\allowdisplaybreaks[2]
\usetikzlibrary{patterns}

\newtheorem{theorem}{Theorem}
\newtheorem{proposition}[theorem]{Proposition}
\newtheorem{lemma}[theorem]{Lemma}

\newtheorem{definition}[theorem]{Definition}
\newtheorem{remark}[theorem]{Remark}

\newtheorem*{theorem*}{Theorem}

\def\XXint#1#2#3{{\setbox0=\hbox{$#1{#2#3}{\int}$ }
\vcenter{\hbox{$#2#3$ }}\kern-.6\wd0}}

\definecolor{Yellow}{rgb}{0.95,0.9,0.0} 
\definecolor{Red}{rgb}{0.8,0.1,0.1}
\definecolor{Green}{rgb}{0.1,0.65,0.2}
\definecolor{Blue}{rgb}{0.1,0.1,0.8}
\definecolor{Purple}{rgb}{0.7,0.1,0.7}
\definecolor{Grey}{rgb}{0.6,0.6,0.6}

\newcommand{\supp}{\operatorname{supp}}

\newcommand{\dist}{\operatorname{dist}}

\newcommand{\Id}{\operatorname{Id}}
\newcommand{\sigdist}{\operatorname{dist}^{\pm}} 
\newcommand{\BV}{\operatorname{BV}}
\renewcommand{\vec}{\mathrm}
\newcommand{\R}{\mathbb{R}}
\newcommand{\Rd}{{\mathbb{R}^d}}
\newcommand{\N}{\mathbb{N}}
\newcommand{\Z}{\mathbb{Z}}
\newcommand{\Sb}{\mathbb{S}}
\newcommand{\Sd}{{\mathbb{S}^{d-1}}}
\newcommand{\M}{\mathcal{M}}

\newcommand{\mres}{\mathbin{\vrule height 1.6ex depth 0pt width
0.13ex\vrule height 0.13ex depth 0pt width 1.3ex}}
\newcommand{\dt}{{\mathrm{d}}t}

\newcommand{\dx}{{\mathrm{d}}x}
\newcommand{\dy}{{\mathrm{d}}y}

\newcommand{\dS}{{\mathrm{d}}S}
\newcommand{\Dsym}{{D^{\operatorname{sym}}}}

\newcommand{\Tend}{{T_{vari}}}
\newcommand{\Tmax}{{T_{strong}}}

\definecolor{LightRed}{rgb}{0.8,0.5,0.5}
\definecolor{LightBlue}{rgb}{0.3,0.2,0.7}

\begin{document}

\title[Weak-strong uniqueness for two-phase flow with sharp interface]{Weak-strong uniqueness for the Navier-Stokes equation for two fluids with surface tension}

\author{Julian Fischer}
\address{Institute of Science and Technology Austria (IST Austria), Am Campus 1, 3400 Klosterneuburg, Austria}
\author{Sebastian Hensel}
\address{Institute of Science and Technology Austria (IST Austria), Am Campus 1, 3400 Klosterneuburg, Austria}

\thanks{This project has received funding from the European Union's Horizon 2020 research and 
innovation programme under the Marie Sk\l{}odowska-Curie Grant Agreement No.\ 665385 
\begin{tabular}{@{}c@{}}\includegraphics[width=3ex]{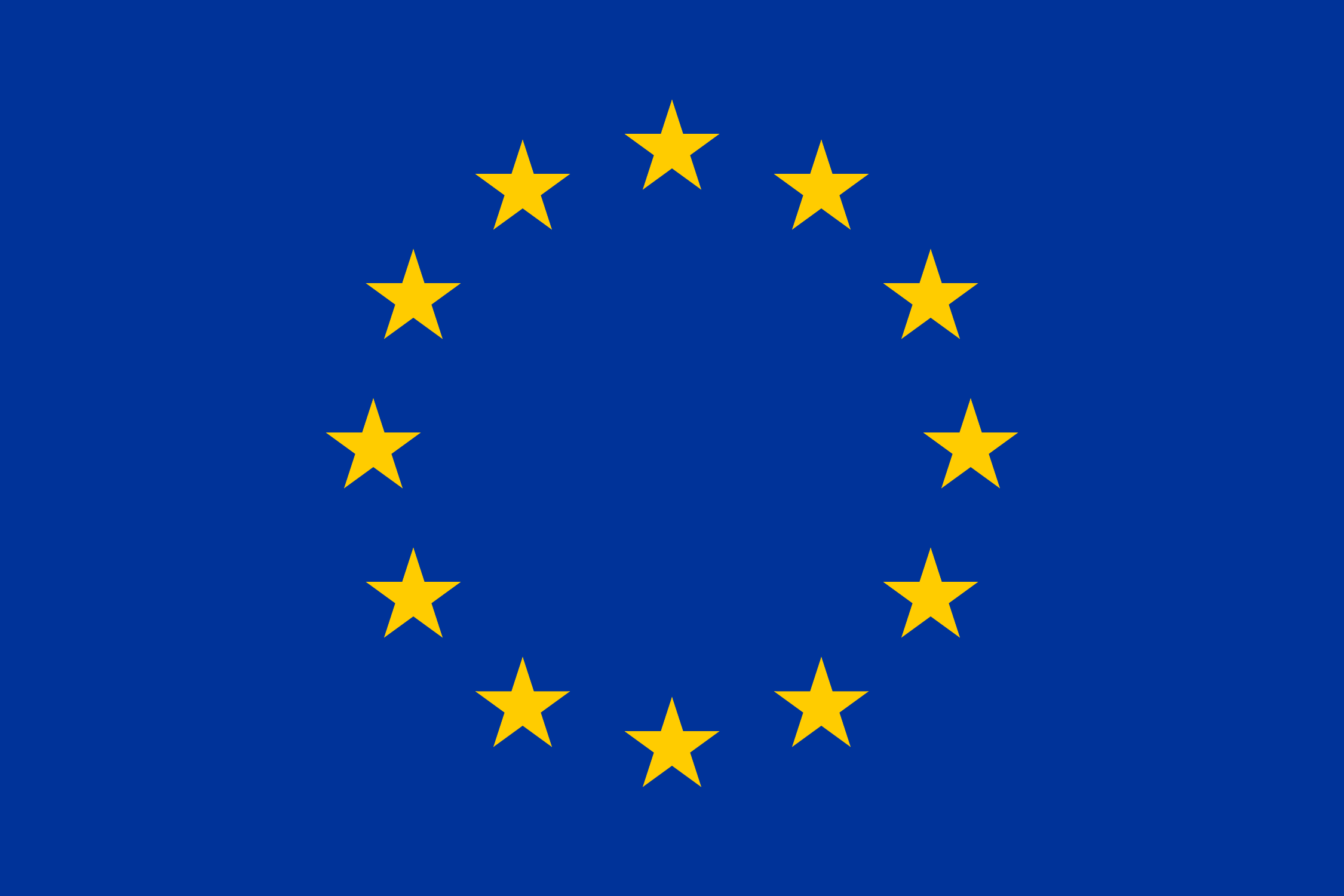}\end{tabular}}
\begin{abstract}
In the present work, we consider the evolution of two fluids separated by a sharp interface in the presence of surface tension -- like, for example, the evolution of oil bubbles in water. Our main result is a weak-strong uniqueness principle for the corresponding free boundary problem for the incompressible Navier-Stokes equation: As long as a strong solution exists, any varifold solution must coincide with it.
In particular, in the absence of physical singularities the concept of varifold solutions -- whose global in time existence has been shown by Abels \cite{Abels} for general initial data -- does not introduce a mechanism for non-uniqueness.
The key ingredient of our approach is the construction of a relative entropy functional capable of controlling the interface error. If the viscosities of the two fluids do not coincide, even for classical (strong) solutions the gradient of the velocity field becomes discontinuous at the interface, introducing the need for a careful additional adaption of the relative entropy.
\end{abstract}

\maketitle

\section{Introduction}

In evolution equations for interfaces, topological changes and geometric singularities often occur naturally, one basic example being the pinchoff of liquid droplets (see Figure~\ref{FigurePinchoff}). As a consequence, strong solution concepts for such PDEs are naturally limited to short-time existence results or particular initial configurations like perturbations of a steady state. At the same time, the transition from strong to weak solution concepts for PDEs is prone to incurring unphysical non-uniqueness of solutions: For example, Brakke's concept of varifold solutions for mean curvature flow admits sudden vanishing of the evolving surface at any time \cite{Brakke}; for the Euler equation, even for vanishing initial data there exist nonvanishing solutions with compact support \cite{Scheffer}, and the notion of mild solutions to the Navier-Stokes equation allows any smooth flow to transition into any other smooth flow \cite{BuckmasterColomboVicol}. In the context of fluid mechanics, the concept of relative entropies has proven successful in ruling out the aforementioned examples of non-uniqueness: Energy-dissipating weak solutions e.\,g.\ to the incompressible Navier-Stokes equation are subject to a \emph{weak-strong uniqueness} principle \cite{Leray,Prodi,Serrin2}, which states that as long as a strong solution exists, any weak solution satisfying the precise form of the energy dissipation inequality must coincide with it. However, in the context of evolution equations for interfaces, to the best of our knowledge the concept of relative entropies has not been applied successfully so far to obtain weak-strong uniqueness results.

In the present work, we are concerned with the most basic model for the evolution of two fluids separated by a sharp interface (like, for instance, the evolution of oil bubbles in water): The flow of each single fluid is described by the incompressible Navier-Stokes equation, while the fluid-fluid interface evolves by pure transport along the fluid flow and a surface tension force acts at the fluid-fluid interface. For this free boundary problem for the flow of two immiscible incompressible fluids with surface tension, Abels \cite{Abels} has established the global existence of varifold solutions for quite general initial data.

The main result of the present work is a weak-strong uniqueness result for this free boundary problem for the Navier-Stokes equation for two fluids with surface tension: In Theorem~\ref{weakStrongUniq} below we prove that as long as a strong solution to this evolution problem exists, any varifold solution in the sense of Abels \cite{Abels2018} must coincide with it.

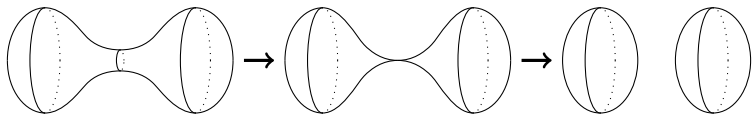
\begin{figure}
\begin{center}
\begin{tikzpicture}[xscale=0.5,yscale=0.7]
\begin{scope}
\clip(-3,-1) rectangle (-2,1);
\draw (-2,0) ellipse (0.4 and 1.0);
\end{scope}
\begin{scope}
\clip(-2,-1) rectangle (-1,1);
\draw[dotted] (-2,0) ellipse (0.4 and 1.0);
\end{scope}
\begin{scope}
\clip(1,-1) rectangle (2,1);
\draw (2,0) ellipse (0.4 and 1.0);
\end{scope}
\begin{scope}
\clip(2,-1) rectangle (3,1);
\draw[dotted] (2,0) ellipse (0.4 and 1.0);
\end{scope}
\begin{scope}
\clip(-1,-1) rectangle (0,1);
\draw (0,0) ellipse (0.1 and 0.2);
\end{scope}
\begin{scope}
\clip(0,-1) rectangle (1,1);
\draw[dotted] (0,0) ellipse (0.1 and 0.2);
\end{scope}
\foreach \x in {-1,1}
\foreach \y in {-1,1}
  \draw (-3*\x,0*\y) .. controls (-3*\x,0.6*\y) and (-2.5*\x,1*\y) .. (-2*\x,1*\y) .. controls (-1.5*\x,1*\y) and (-1.1*\x,0.6*\y) .. (-0.9*\x,0.45*\y) .. controls (-0.7*\x,0.3*\y) and (-0.4*\x,0.2*\y) .. (0*\x,0.2*\y);
\end{tikzpicture}
\begin{tikzpicture}[xscale=0.5,yscale=0.7]
\draw[white] (0,-1) -- (0,1);
\draw[very thick,->] (-0.4,0) -- (0.4,0);
\end{tikzpicture}
\begin{tikzpicture}[xscale=0.5,yscale=0.7]
\begin{scope}
\clip(-3,-1) rectangle (-2,1);
\draw (-2,0) ellipse (0.4 and 1.0);
\end{scope}
\begin{scope}
\clip(-2,-1) rectangle (-1,1);
\draw[dotted] (-2,0) ellipse (0.4 and 1.0);
\end{scope}
\begin{scope}
\clip(1,-1) rectangle (2,1);
\draw (2,0) ellipse (0.4 and 1.0);
\end{scope}
\begin{scope}
\clip(2,-1) rectangle (3,1);
\draw[dotted] (2,0) ellipse (0.4 and 1.0);
\end{scope}
\foreach \x in {-1,1}
\foreach \y in {-1,1}
  \draw (-3*\x,0*\y) .. controls (-3*\x,0.6*\y) and (-2.5*\x,1*\y) .. (-2*\x,1*\y) .. controls (-1.5*\x,1*\y) and (-1.1*\x,0.45*\y) .. (-0.9*\x,0.3*\y) .. controls (-0.7*\x,0.15*\y) and (-0.4*\x,0.0*\y) .. (0*\x,0.0*\y);
\end{tikzpicture}
\begin{tikzpicture}[xscale=0.5,yscale=0.7]
\draw[white] (0,-1) -- (0,1);
\draw[very thick,->] (-0.4,0) -- (0.4,0);
\end{tikzpicture}
\begin{tikzpicture}[xscale=0.5,yscale=0.7]
\draw (-1.5,0) ellipse (1.0 and 1.0);
\draw (1.5,0) ellipse (1.0 and 1.0);
\begin{scope}
\clip(-2.5,-1) rectangle (-1.5,1);
\draw (-1.5,0) ellipse (0.4 and 1.0);
\end{scope}
\begin{scope}
\clip(-1.5,-1) rectangle (-0.5,1);
\draw[dotted] (-1.5,0) ellipse (0.4 and 1.0);
\end{scope}
\begin{scope}
\clip(0.5,-1) rectangle (1.5,1);
\draw (1.5,0) ellipse (0.4 and 1.0);
\end{scope}
\begin{scope}
\clip(1.5,-1) rectangle (2.5,1);
\draw[dotted] (1.5,0) ellipse (0.4 and 1.0);
\end{scope}
\end{tikzpicture}
\caption{Pinchoff of a liquid droplet driven by surface tension.\label{FigurePinchoff}}
\end{center}
\end{figure}

\subsection{Free boundary problems for the Navier-Stokes equation}

The free boundary problem for the Navier-Stokes equation has been studied in mathematical fluid mechanics for several decades. Physically, it describes the evolution of a viscous incompressible fluid surrounded by or bordering on vacuum. The (local-in-time) existence of strong solutions for the free boundary problem for the Navier-Stokes equation has been proven by Solonnikov \cite{Solonnikov0,Solonnikov2,Solonnikov} in the presence of surface tension and by Shibata and Shimizu \cite{ShibataShimizu} in the absence of surface tension; see also Beale \cite{Beale,BealeSurfaceWave}, Abels \cite{AbelsLq}, and Coutand and Shkoller \cite{CoutandShkollerExistence} for related or further results. While the existence theory for global weak solutions for the Navier-Stokes equation in a fixed domain like $\mathbb{R}^d$, $d\leq 3$, has been developed starting with the seminal work of Leray \cite{Leray} in 1934, the question of the global existence of any kind of solution to the free boundary problem for the Navier-Stokes equation has remained an open problem. An important challenge for a global existence theory of weak solutions to the free boundary problem for the Navier-Stokes equation is the possible formation of ``splash singularities'', which are smooth solutions to the Lagrangian formulation of the equations which develop self-interpenetration. Such solutions have been constructed by Castro, Cordoba, Fefferman, Gancedo, and Gomez-Serrano \cite{CastroCordobaFeffermanGancedoGomezSerrano2}, see also \cite{CastroCordobaFeffermanGancedoGomezSerrano,CoutandShkoller,
DiIorioMarcatiSpirito} for splash singularities in related models in fluid mechanics.

In the present work we consider a closely related problem, namely the flow of two incompressible and immiscible fluids with surface tension at the fluid-fluid interface, like for example the flow of oil bubbles immersed in water or vice versa. For this free boundary problem for the Navier-Stokes equation for two fluids -- described by the system of PDEs \eqref{NavierStokesFreeBoundary} below -- , a global existence theory for generalized solutions is in fact available: In a rather recent work, Abels \cite{Abels} has constructed varifold solutions which exist globally in time. In an earlier work, Plotnikov~\cite{Plotnikov} had treated the case of non-Newtonian (shear-thickening) fluids.
The local-in-time existence of strong solutions has been established by Denisova \cite{Denisova}; for an interface close to the half-space, an existence and instant analyticity result has been derived by Pr\"uss and Simonett \cite{PruessSimonett,PruessSimonett3}.
Existence results for the two-phase Stokes and Navier-Stokes equation in the absence of surface tension have been established by Giga and Takahashi \cite{GigaTakahashi} and Nouri and Poupaud \cite{NouriPoupaud}.
Note that in contrast to the case of a single fluid in vacuum, for the flow of two incompressible immiscible inviscid fluids splash singularities cannot occur as shown by Fefferman, Ionescu, and Lie \cite{FeffermanIonescuLie} and Coutand and Shkoller \cite{CoutandShkollerVortexSheet}; one would expect a similar result to hold for viscous fluids. However, solutions may be subject to the Rayleigh-Taylor instability as proven by Pr\"uss and Simonett \cite{PruessSimonett2}.

In terms of a PDE formulation, the flow of two immiscible incompressible fluids with surface tension may be described by the indicator function $\chi=\chi(x,t)$ of the volume occupied by the first fluid, the local fluid velocity $v=v(x,t)$, and the local pressure $p=p(x,t)$. The fluid-fluid interface moves just according to the fluid velocity, the evolution of the velocity of each fluid and the pressure are determined by the Navier-Stokes equation, and the fluid-fluid interface exerts a surface tension force on the fluids proportional to the mean curvature of the interface. Together with the natural no-slip boundary condition and the appropriate boundary conditions for the stress tensor on the fluid-fluid interface, one may assimilate the Navier-Stokes equations for the two fluids into a single one, resulting in the system of equations
\begin{subequations}
\label{NavierStokesFreeBoundary}
\begin{align}
\label{EquationTransport}
\partial_t \chi + (v\cdot \nabla) \chi &=0,
\\
\label{EquationMomentum}
\rho(\chi) \partial_t v
+\rho(\chi) (v \cdot \nabla) v
&=-\nabla p + \nabla \cdot (\mu(\chi)(\nabla v+\nabla v^T)) + \sigma \vec{H} |\nabla \chi|,
\\
\label{EquationIncompressibility}
\nabla \cdot v &=0,
\end{align}
\end{subequations}
where $\vec{H}$ denotes the mean curvature vector of the interface $\partial \{\chi=0\}$ and $|\nabla \chi|$ denotes the surface measure $\mathcal{H}^{d-1}|_{\partial \{\chi=0\}}$. Here, $\mu(0)$ and $\mu(1)$ are the shear viscosities of the two fluids and $\rho(0)$ and $\rho(1)$ are the densities of the two fluids. The constant $\sigma$ is the surface tension coefficient. The total energy of the system is given by the sum of kinetic and surface tension energies
\begin{align*}
E[\chi,v]:=\int_{\Rd} \frac{1}{2} \rho(\chi) |v|^2 \,dx
+\sigma \int_{\Rd} 1 \,d|\nabla \chi|.
\end{align*}
It is at least formally subject to the energy dissipation inequality
\begin{align*}
E[\chi,v](T)
+ \int_0^T \int_\Rd \frac{\mu(\chi)}{2} \big|\nabla v+\nabla v^T\big|^2 \,dx \leq E[\chi,v](0).
\end{align*}
Note that the concept of varifold solutions requires a slight adjustment of the definition of the energy: The surface area $\int_{\Rd} 1 \,d|\nabla \chi|$ is replaced by the corresponding quantity of the varifold, namely its mass.

A widespread numerical approximation method for the free boundary problem \eqref{EquationTransport}-\eqref{EquationIncompressibility} capable of capturing geometric singularities and topological changes in the fluid phases are phase-field models of Navier-Stokes-Cahn-Hilliard type or Navier-Stokes-Allen-Cahn type, see for example the review \cite{AndersonFcFaddenWheeler}, \cite{AbelsGarckeGruen,LiuShen,GurtinPolignoneVinals,HohenbergHalperin,
LowengrubTruskinovsky} for modeling aspects, \cite{AbelsPhaseFieldExistence,AbelsDepnerGarcke} for the existence analysis of the corresponding PDE systems, and \cite{AbelsLiu,AbelsLiuSchoettl,AbelsRoeger} for results on the sharp-interface limit.

\subsection{Weak solution concepts in fluid mechanics and (non-)uniqueness}

In the case of the free boundary problem for the Navier-Stokes equation -- both for a single fluid and for a fluid-fluid interface -- , a concept of weak solutions is expected to play an even more central role in the mathematical theory than in the case of the standard Navier-Stokes equation: In three spatial dimensions $d=3$, even for smooth initial interfaces topological changes may occur naturally in finite time, for example by asymptotically self-similar pinchoff of bubbles \cite{EggersFontelos} (see Figure~\ref{FigurePinchoff}). In contrast, for the incompressible Navier-Stokes equation without free boundary the global existence of strong solutions for any sufficiently regular initial data remains a possibility.

However, in general weakening the solution concept for a PDE may lead to artificial (unphysical) non-uniqueness, even in the absence of physically expected singularities.
A particularly striking instance of this phenomenon is the recent example of non-uniqueness of mild (distributional) solutions to the Navier-Stokes equation by Buckmaster and Vicol \cite{BuckmasterVicol} and Buckmaster, Colombo, and Vicol \cite{BuckmasterColomboVicol}: In the framework of mild solutions to the Navier-Stokes equation, any smooth flow may transition into any other smooth flow \cite{BuckmasterColomboVicol}.
The result of \cite{BuckmasterColomboVicol,BuckmasterVicol} are based on convex integration techniques for the Euler equation, which have been developed starting with the works of De~Lellis and Sz\'ekelyhidi \cite{DeLellisSzekelyhidi,DeLellisSzekelyhidi2} (see also \cite{Buckmaster,BuckmasterDeLellisIsettSzekelyhidi,DaneriSzekelyhidi,Isett}).

In contrast to the case of distributional or mild solutions, for the stronger notion of weak solutions to the Navier-Stokes equation with energy dissipation in the sense of Leray \cite{Leray} a weak-strong uniqueness theorem is available: As long as a strong solution to the Navier-Stokes equation exists, any weak solution with energy dissipation must coincide with it.
Recall that for a weak solution to the Navier-Stokes equation $v$, besides the Ladyzhenskaya-Prodi-Serrin regularity criterion $v\in L^p([0,T];L^q(\R^3;\R^3))$ with $\smash{\frac{2}{p}+\frac{3}{q}}\leq 1$ and $p\geq 2$ \cite{Prodi,Serrin}, both a lower bound on the pressure \cite{SereginSverak} and a geometric assumption on the vorticity \cite{ConstantinFefferman} are known to imply smoothness of $v$. Interestingly, weak-strong uniqueness of energy-dissipating solutions fails if the Laplacian in the Navier-Stokes equation is replaced by a fractional Laplacian $-(-\Delta)^\alpha$ with power $\alpha<\frac{1}{3}$, see Colombo, De~Lellis, and De~Rosa \cite{ColomboDeLellisDeRosa} and De~Rosa \cite{DeRosa}.

Another way of interpreting a weak-strong uniqueness result is that nonuniqueness of weak solutions may only arise as a consequence of physical singularities: Only when the unique strong solution develops a singularity, the continuation of solutions beyond the singularity -- by means of the weak solution concept -- may be nonunique.
The main theorem of our present work provides a corresponding result for the flow of two incompressible immiscible fluids with surface tension: Varifold solutions to the free boundary problem for the Navier-Stokes equation for two fluids are unique until the strong solution for the free boundary problem develops a singularity.

\subsection{(Non)-Uniquenesss in interface evolution problems}
Weak solution concepts for the evolution of interfaces are often subject to nonuniqueness, even in the absence of topology changes. For example, Brakke's concept of varifold solutions for the evolution of surfaces by mean curvature \cite{Brakke} suffers from a particularly drastic failure of uniqueness: The interface is allowed to suddenly vanish at an arbitrary time.
In the context of viscosity solutions to the level-set formulation of two-phase mean curvature flow, the formation of geometric singularities may still cause fattening of level-sets \cite{BarlesSonerSouganidis} and thereby nonuniqueness of the mean-curvature evolution, even for smooth initial surfaces \cite{AngenentIlmanenChopp}.

To the best of our knowledge, the only known uniqueness result for weak or varifold solutions for an evolution problem for interfaces is a consequence of the relation between Brakke solutions and viscosity solutions for two-phase mean-curvature flow, see Ilmanen~\cite{Ilmanen}: As long as a smooth solution to the level-set formulation exists, the support of any Brakke solution must be contained in the corresponding level-set of the viscosity solution. As a consequence, as long as a smooth evolution of the interface by mean curvature exists, the ``maximal'' unit-density Brakke solution corresponds to the smoothly evolving interface. The proof of this inclusion relies on the properties of the distance function to a surface undergoing evolution by mean curvature respectively the comparison principle for mean curvature flow. Both of these properties do not generalize to other interface evolution equations.

Besides Ilmanen's varifold comparison principle, the only uniqueness results in the context of weak solutions to evolution problems for lower-dimensional objects that we are aware of are a weak-strong uniqueness principle for the higher-codimension mean curvature flow by Ambrosio and Soner \cite{AmbrosioSoner} and a weak-strong uniqueness principle for binormal curvature motion of curves in $\mathbb{R}^3$ by Jerrard and Smets \cite{JerrardSmets}. The interface contribution in our relative entropy \eqref{DefinitionRelativeEntropyFunctional} may be regarded as the analogue for surfaces of the relative entropy for curves introduced in \cite{JerrardSmets}.

\subsection{The concept of relative entropies}

The concept of relative entropies in continuum mechanics has been introduced by Dafermos \cite{Dafermos,DafermosBook} and DiPerna \cite{DiPerna} in the study of the uniqueness properties of systems of conservation laws. Proving weak-strong uniqueness results for conservation laws or even the incompressible Navier-Stokes equation typically faces the problem that an error between a weak solution $u$ and a strong solution $v$ must be measured by a quantity $E[u|v]$ which is nonlinear even as a function of $u$ alone, like a norm $||u-v||$. To evaluate the time evolution $\smash{\frac{d}{dt}} E[u|v]$ of such a quantity, one would need to test the evolution equation for $u$ by the nonlinear function $D_u E[u|v]$, which is often not possible due to the limited regularity of $u$. The concept of relative entropies overcomes this issue if the physical system possesses a strictly convex entropy (or energy) $E[u]$ subject to a dissipation estimate $\smash{\frac{d}{dt}} E[u]\leq -\mathcal{D}[u]$: For strictly convex entropies $E[u]$, the relative entropy
\begin{align*}
E[u|v]:=E[u]-DE[v] (u-v)-E[v]
\end{align*}
is a measure for the error between $u$ and $v$, as it is nonzero if and only if $u=v$. At the same time, to evaluate the time evolution $\smash{\frac{d}{dt}} E[u|v]$ of the relative entropy, it is sufficient to exploit the entropy dissipation inequality $\smash{\frac{d}{dt}} E[u]\leq -\mathcal{D}[u]$ for the weak solution $u$ and test the weak formulation of the evolution equation for $u$ by the typically more regular test function $DE[v]$. Having derived an explicit expression for the time derivative $\smash{\frac{d}{dt}} E[u|v]$ of the relative entropy, it is often possible to derive a Gronwall-type estimate like $\smash{\frac{d}{dt}} E[u|v] \leq C E[u|v]$ for the relative entropy and thereby a weak-strong uniqueness result.

Since Dafermos \cite{Dafermos}, the concept of relative entropies has found many applications in the analysis of continuum mechanics, providing weak-strong uniqueness results for the compressible Navier-Stokes equation \cite{FeireislJinNovotny,Germain}, the Navier-Stokes-Fourier system \cite{FeireislNovotny}, fluid-structure interaction problems \cite{ChemetovNecasovaMuha}, renormalized solutions for dissipative reaction-diffusion systems \cite{ChenJuengel,FischerReactionDiffusionUniqueness}, as well as weak-strong uniqueness results for measure-valued solutions for the Euler equation \cite{BrenierDeLellisSzekelyhidi}, compressible fluid models~\cite{GwiazdaWiedemann}, wave equations in nonlinear elastodynamics  \cite{DemouliniStuartTzavaras}, and models for liquid crystals \cite{EmmrichLasarzik}, to name just a few.

The concept of relative entropies has also been employed in the justification of singular limits of PDEs, see for example the work of Yau \cite{Yau} on the hydrodynamic limit of the Ginzburg-Landau lattice model, the works of Bardos, Golse, and Levermore \cite{BardosGolseLevermore}, Saint-Raymond \cite{SaintRaymond,SaintRaymond2}, and Golse and Saint-Raymond \cite{GolseSaintRaymond} on the derivation of the Euler equation and the incompressible Navier-Stokes equation from the Boltzmann equation, the work of Brenier \cite{Brenier} on the Euler limit of the Vlasov-Poissson equation, and the works of Serfaty \cite{Serfaty} and Duerinckx \cite{Duerinckx} on mean-field limits of interacting particles.
In the context of numerical analysis, it may also be used to derive a~posteriori estimates for model simplification errors \cite{FischerModelErrorNSE,GiesselmannPryer}.

Jerrard and Smets \cite{JerrardSmets} have used a relative entropy ansatz to establish a weak-strong uniqueness principle for the evolution of curves in $\mathbb{R}^3$ by binormal curvature flow. Their relative entropy may be regarded as the analogue for curves of the interfacial energy contribution to our relative entropy (i.\,e.\ the terms $\sigma \int_\Rd 1-\xi(\cdot,T) \cdot \smash{\frac{\nabla \chi_u(\cdot,T)}{|\nabla \chi_u(\cdot,T)|}} \,d|\nabla \chi_u(\cdot,T)|+\sigma \int_\Rd 1-\theta_T \,d|V_T|_{\Sb^{d-1}}$ in \eqref{DefinitionRelativeEntropyFunctional} below). It has subsequently been used by Jerrard and Seis \cite{JerrardSeis} to prove that the evolution of solutions to the Euler equation with near-vortex-filament initial data is governed by binormal curvature flow, as long as a strong solution to the latter (without self-intersections) exists and as long as the vorticity remains concentrated along some curve.

One of the key challenges in the derivation of our result is the development of a notion of relative entropy which provides strong enough control of the interfacial error. The key idea to control the error between an interface $\partial \{\chi_u(\cdot,t)=1\}$ and a smoothly evolving interface $I_v(t)=\partial \{\chi_v(\cdot,t)=1\}$ by a relative entropy is to introduce a vector field $\xi$ which is an extension of the unit normal of $I_v(t)$, multiplied with a cutoff. The interfacial contribution $\smash{\sigma \int_\Rd 1-\xi(\cdot,T) \cdot \smash{\frac{\nabla \chi_u(\cdot,T)}{|\nabla \chi_u(\cdot,T)|}} \,d|\nabla \chi_u(\cdot,T)|}$ to the relative entropy then controls the interface error in a sufficiently strong way, see Section~\ref{SectionStrategy} for details.

However, in the case of different viscosities $\mu^+\neq \mu^-$ of the two fluids, the velocity gradient of the strong solution $\nabla v$ at the interface will be discontinuous. This necessitates an additional adaption of our relative entropy: If one were to directly compare the velocity fields $u$ and $v$ of two solutions by the relative entropy, the difference of the viscous stresses $\mu(\chi_u)D^{\operatorname{sym}}u-\mu(\chi_v)D^{\operatorname{sym}}v$ could not be estimated appropriately to derive a Gronwall-type estimate. We rather have to compare the velocity field $u$ to an adapted velocity field $v+w$, where $w$ is constructed in a way that the adapted velocity gradient $\nabla v+\nabla w$ approximately accounts for the shifted location of the interface.

The approximate adaption of the interface of the strong solution to the higher-order approximation for the interface is distantly reminiscent of an ansatz by Leger and Vasseur \cite{LegerVasseur} and Kang, Vasseur, and Wang \cite{KangVasseurWang}, who establish $L^2$ contractions up to a shift for solutions to conservation laws close to a shock profile. However, it differs both in purpose and in the actual construction from \cite{KangVasseurWang,LegerVasseur}: The interfacial shift in \cite{KangVasseurWang,LegerVasseur} essentially serves the purpose of compensating the difference in the propagation speed of the shocks of the two solutions, while we need the higher-order approximation of the interface to compensate for the discontinuity in the velocity gradient at the interface. While the interfacial shift in \cite{KangVasseurWang} is given as the solution to an appropriately defined time-dependent PDE, in our case we obtain the interfacial shift by applying at any fixed time a suitable regularization operator to the interface of the weak solution near the interface of the strong solution.

\subsection{Derivation of the model}
\label{SectionModelDerivation}

Let us briefly comment on the derivation of the system of equations \eqref{NavierStokesFreeBoundary}. We consider the flow of two viscous, immiscible, and incompressible fluids.
Each fluid occupies a domain $\Omega^+_t$ resp.\ $\Omega^-_t$, $t\geq 0$, and the interface separating both phases will be denoted by $I(t)$. In particular, $\R^d=\Omega^+_t\cup\Omega^-_t\cup I(t)$ for every $t\geq 0$.
Within each of these domains $\Omega^\pm_t$, the evolution of the fluid velocity is modeled by means of the incompressible Navier--Stokes equations for a Newtonian fluid 
\begin{subequations}
\begin{align}\label{tpNS}
	\partial_t (\rho_{\pm} v^\pm_t) + \nabla\cdot(\rho_{\pm} v^\pm_t\otimes v^\pm_t)  &= -\nabla p^\pm_t + \mu^\pm \Delta v^\pm_t,
\\
\label{tpNSDiv}
	\nabla\cdot v^\pm_t &= 0,
\end{align}
\end{subequations}
where $v^+_t:\Omega^+_t \rightarrow \Rd$ and $v^-_t:\Omega^-_t \rightarrow \Rd$ denote the velocity fields of the two fluids, $p^+_t:\Omega^+_t\rightarrow \R$ and $p^-_t:\Omega^-_t\rightarrow \R$ the pressure, $\rho_+, \rho_->0$ the densities of the two fluids, and $\mu^+$ and $\mu^-$ the shear viscosities. On the interface of the two fluids $I(t)$ a no-slip boundary condition $v^+_t=v^-_t$ is imposed. As the two velocities $v^+_t$ and $v^-_t$ are defined on complementary domains and coincide on the interface, this enables us to assimilate the two velocity fields into a single velocity field $\smash{v:\Rd\times [0,T) \rightarrow \Rd}$, $\smash{v_t:=v^+_t \chi_{\Omega^+_t} + v^- (1-\chi_{\Omega^+_t})}$. Note that the velocity field $v$ inherits the incompressibility \eqref{EquationIncompressibility} from the incompressibility of $v^+$ and $v^-$ \eqref{tpNSDiv}. We also assimilate the pressures $p^+_t$ and $p^-_t$ into a single pressure $p$, which however may be discontinuous across the interface.

Additionally, we assume that the evolution of the interface $I(t)$ occurs only as a result of the transport of the two fluids along the flow. Denoting by $\vec{n}$ the outward unit normal vector field of the interface $I(t)$ and by $V_{\vec{n}}$ the associated normal speed of the interface, this gives rise to the equation
\begin{align}\label{kinematicInt}
	V_{\vec{n}} = \vec{n}\cdot v
	\quad\quad \text{on }I(t)\text{ for all }t\geq 0.
\end{align}
This condition may equivalently be rewritten as the transport equation for the indicator function $\chi$ of the first fluid phase
\begin{align*}
\partial_t \chi + (v\cdot \nabla)\chi =0,
\end{align*}
see for example Remark~\ref{strongKinematic} below for the (standard) arguments.

In order to assimilate the equations \eqref{tpNS} for the velocities $v^\pm$ of the two fluids into the single equation \eqref{EquationMomentum}, a condition on the jump of the normal component of the stress tensor $T=\mu^\pm \, (\nabla v+\nabla v^T)-\nabla p \Id$ at the interface $I(t)$ is required. In the case of positive surface tension constant $\sigma>0$ at the interface, the balance of forces at the interface reads
\begin{align}\label{balForcesInterface}
	[[T\vec{n}\,]] = \sigma\vec{H},
\end{align}
where the right-hand side $\sigma \vec{H}$ accounts for the surface tension force. Here, $\vec{H}$ denotes the mean curvature vector of the interface and $[[f]]$ denotes the jump in normal direction of a quantity $f$. In combination with \eqref{tpNS} 
and the no-slip boundary condition $v^+=v^-$ on $I(t)$, this yields the equation for 
the momentum balance \eqref{EquationMomentum}.

\section{Main results}

The main result of the present work is the derivation of a weak-strong uniqueness principle 
for varifold solutions to the free boundary problem for the Navier--Stokes equation for two 
immiscible incompressible fluids with surface tension: As long as a strong solution to the 
free boundary problem \eqref{EquationTransport}-\eqref{EquationIncompressibility} exists, 
any varifold solution must coincide with it. In particular, the concept of varifold solutions 
developed by Abels \cite{Abels} (see Definition~\ref{DefinitionVarifoldSolution} below for a 
precise definition) does not introduce an additional mechanism for non-uniqueness, at least as 
long as a classical solution exists.
At the same time, the concept of varifold solutions of Abels 
allows for the construction of globally existing solutions \cite{Abels}, while any concept of 
strong solutions is limited to the absence of geometric singularities and therefore -- at least in three spatial dimensions $d=3$ -- to short-time existence results.

Furthermore, we prove a quantitative stability result \eqref{StabilityEstimate} for varifold solutions with respect to changes in the data: As long as a classical solution exists, any varifold solution with slightly perturbed initial data remains close to it.

\begin{theorem}[Weak-strong uniqueness principle]
\label{weakStrongUniq}
Let $d \in \{2,3\}$. Let $(\chi_u,u,V)$ be a varifold solution to the free boundary problem for 
the incompressible Navier--Stokes equation for two fluids \eqref{EquationTransport}--\eqref{EquationIncompressibility} 
in the sense of Definition~\ref{DefinitionVarifoldSolution} on some time interval $[0,\Tend)$. 
Let $(\chi_v,v)$ be a strong solution to \eqref{EquationTransport}--\eqref{EquationIncompressibility} 
in the sense of Definition~\ref{DefinitionStrongSolution} on some time interval $[0,\Tmax)$ with 
$\Tmax\leq \Tend$. Let the relative entropy $E[\chi_u,u,V|\chi_v,v](t)$ be defined as in Proposition~\ref{PropositionRelativeEntropyInequalityFull}.

Then there exist constants $C,c>0$ such that the stability estimate
\begin{align}
\label{StabilityEstimate}
&E\big[\chi_u,u,V\big|\chi_v,v\big](T) 
\leq C(E[\chi_u,u,V|\chi_v,v](0))^{e^{-CT}}
\end{align}
holds for almost every $T\in [0,\Tmax)$, provided that the initial relative entropy satisfies $E[\chi_u,u,V|\chi_v,v](0) \leq c$. The constants $c>0$ and $C>0$ depend only on the data and the strong solution.

In particular, if the initial data of the varifold solution and the strong solution coincide, the varifold solution must be equal to the strong solution in the sense that
\begin{align*}
\chi_u(\cdot,t)=\chi_v(\cdot,t) &&\text{and}&& u(\cdot,t)=v(\cdot,t)
\end{align*}
hold almost everywhere for almost every $t\in [0,\Tmax)$, and the varifold is given for almost every $t\in [0,\Tmax)$ by
\begin{align*}
\mathrm{d}V_t = \delta_{\frac{\nabla \chi_v}{|\nabla \chi_v|}} \mathrm{d}|\nabla \chi_v|.
\end{align*}
\end{theorem}

The following notion of varifold solutions for the free boundary problem associated 
with the flow of two immiscible incompressible viscous fluids with surface tension has 
been introduced by Abels~\cite{Abels}. For the notion of an oriented varifold, see the section on notation just prior to Section~\ref{SectionStrategy}.

\begin{definition}[Varifold solution for the two-phase Navier--Stokes equation]
\label{DefinitionVarifoldSolution}
Let a surface tension constant $\sigma>0$, the densities and shear viscosities of the 
two fluids $\rho^\pm$, $\mu^\pm>0$, a finite time $\Tend>0$, a solenoidal initial velocity 
profile $v_0\in L^2(\R^d;\R^d)$, and an indicator function of the volume occupied initially by 
the first fluid $\chi_0\in \BV(\R^d)$ be given.

A triple $(\chi,v,V)$ consisting of a velocity field $v$, an indicator function $\chi$ of the volume 
occupied by the first fluid, and an oriented varifold $V$ with
\begin{align*}
v &\in L^2([0,\Tend];H^1(\R^d;\R^d))\cap L^\infty([0,\Tend];L^2(\R^d;\Rd)),
\\
\chi&\in L^\infty([0,\Tend];\BV(\R^d;\{0,1\})),
\\
V &\in L^\infty_w([0,\Tend];\mathcal{M}(\R^d{\times}\Sb^{d-1})),
\end{align*}
is called a \emph{varifold solution} to the free boundary problem for the Navier-Stokes 
equation for two fluids with initial data $(\chi_0,v_0)$ if the following conditions are satisfied:
\begin{subequations}
\begin{itemize}[leftmargin=0.5cm]
\item[i)] The velocity field $v$ has vanishing divergence $\nabla\cdot v=0$ and the equation for the momentum balance
\begin{align}
\nonumber
\int_\Rd& \rho(\chi(\cdot,T)) v(\cdot,T) \cdot \eta(\cdot,T) \,\dx
- \int_\Rd \rho(\chi_0) v_0 \cdot \eta(\cdot,0) \,\dx 
\\
\label{weakNS}
&= \int_0^T \int_\Rd \rho(\chi) v \cdot \partial_t \eta \,\dx\,\dt
+\int_0^T \int_\Rd \rho(\chi) v \otimes v : \nabla \eta \,\dx\,\dt
\\&~~~~
\nonumber
-\int_0^T\int_\Rd \mu(\chi) (\nabla v+\nabla v^T) : \nabla \eta \,\dx\,\dt
\\&~~~~
\nonumber
-\sigma\int_0^T\int_{\R^d\times\Sb^{d-1}}
(\mathrm{Id}-s\otimes s): \nabla \eta \,\mathrm{d}V_t(x,s)\,\dt
\end{align}
is satisfied for almost every $T\in [0,\Tend)$ and every smooth vector field 
$\eta \in C^\infty_{cpt}(\Rd\times [0,\Tend);\Rd)$ with $\nabla \cdot \eta=0$. 
For the sake of brevity, we have used the abbreviations $\rho(\chi):=\rho^+ \chi + \rho^- (1-\chi)$ and 
$\mu(\chi):=\mu^+ \chi + \mu^- (1-\chi)$.
\item[ii)] The indicator function $\chi$ of the volume occupied by the first fluid satisfies 
the weak formulation of the transport equation
\begin{align}\label{weakTransport}
\int_{\R^d}\chi(\cdot,T) \varphi(\cdot,T) \,\dx
-\int_{\R^d}\chi_0\varphi(\cdot,0)\,\dx
= \int_0^T\int_{\R^d}\chi\,(\partial_t\varphi+(v\cdot\nabla)\varphi)\,\dx\,\dt
\end{align}
for almost every $T\in [0,\Tend)$ and all $\varphi\in C^\infty_{cpt}(\Rd\times [0,\Tend))$.
\item[iii)] The energy dissipation inequality
\begin{align}
\nonumber
&\int_\Rd \frac{1}{2} \rho(\chi(\cdot,T)) |v(\cdot,T)|^2 \,\dx + \sigma |V_T|(\Rd\times \Sd)
\\&~~
\label{genEI}
+\int_0^T \int_\Rd \frac{\mu(\chi)}{2}
\big|\nabla v + \nabla v^T\big|^2 \,\dx\,\dt
\\&
\nonumber
\leq \int_\Rd \frac{1}{2} \rho(\chi_0(\cdot)) |v_0(\cdot)|^2 \,\dx + 
\sigma |\nabla \chi_0(\cdot)|(\Rd)
\end{align}
is satisfied for almost every $T\in [0,\Tend)$, and the energy
\begin{align}
\label{EnergyVarifold}
E[\chi,v,V](t):=
\int_\Rd \frac{1}{2} \rho(\chi(\cdot,t)) |v(\cdot,t)|^2 \,\dx + \sigma |V_t|(\Rd\times \Sd)
\end{align}
is a nonincreasing function of time.
\item[iv)] The phase boundary $\partial \{\chi(\cdot,t)=0\}$ and the varifold $V$ satisfy the compatibility condition
\begin{align}
\label{varifoldComp}
\int_{\R^d\times\Sb^{d-1}}\psi(x)s\,\mathrm{d}V_t(x,s)
=
\int_{\R^d}\psi(x) \,\mathrm{d}\nabla\chi(x)
\end{align}
for almost every $T\in [0,\Tend)$ and every smooth function $\psi \in C^\infty_{cpt}(\Rd)$.
\end{itemize}
\end{subequations}
\end{definition}

Let us continue with a few comments on the relation between the varifold $V_t$ and the interface described by the indicator function $\chi(\cdot,t)$.

\begin{remark}
Let $V_t\in\M(\R^d{\times}\Sb^{d-1})$ denote the non-negative measure
representing (at time $t$) the varifold associated to a varifold solution $(\chi,v,V)$
to the free boundary problem for the incompressible Navier-Stokes equation for two fluids.
The compatibility condition \eqref{varifoldComp} entails that $|\nabla\chi_u(t)|$
is absolutely continuous with respect to $|V_t|_{\Sb^{d-1}}$. Hence, we may define the Radon--Nikodym derivative
\begin{align}\label{DefinitionTheta}
\theta_t := \frac{\mathrm{d}|\nabla\chi_u(t)|}{\mathrm{d} |V_t|_{\Sb^{d-1}}},
\end{align}
which is a $|V_t|_{\Sb^{d-1}}$-measurable function with $|\theta_t(x)|\leq 1$ for $|V_t|_{\Sb^{d-1}}$-almost every $x\in\Rd$.
In particular, we have
\begin{align}\label{RadonNikodym}
    \int_{\R^d}f(x)\,\mathrm{d}|\nabla\chi(\cdot,t)|(x) = \int_{\R^d}\theta_t(x)f(x)\,\mathrm{d}|V_t|_{\Sb^{d-1}}(x)
\end{align}
for every $f\in L^1(\R^d,|\nabla\chi(\cdot,t)|)$ and almost every $t\in [0,\Tend)$.
\end{remark}
The compatibility condition between the varifold $V_t$ and the interface described by the indicator function $\chi(\cdot,t)$ has the following consequence.
\begin{remark}
Consider a varifold solution $(\chi,v,V)$
to the free boundary problem for the incompressible Navier-Stokes equation for two fluids. 
Let $E_t$ be the measurable set $\{x\in\R^d\colon \chi(x,t) = 1\}$. Note that for almost every $t\in [0,\Tend)$ 
this set is then a Caccioppoli set in $\R^d$. Let $\vec{n}(\cdot,t)=\frac{\nabla\chi}{|\nabla\chi|}$ 
denote the measure theoretic unit normal vector field on the reduced boundary $\partial^* E_t$.
By means of the compatibility condition \eqref{varifoldComp} and the definition \eqref{DefinitionTheta}
we obtain
\begin{align}\label{varifoldMom}
\frac{\mathrm{d} \int_{\Sb^{d-1}}s\,\mathrm{d}V_t(\cdot,s)}{\mathrm{d} |V_t|_{\Sb^{d-1}}(\cdot)}=
\begin{cases}																		\theta_t(x) \vec{n}(x,t) &\text{for } x\in\partial^*E_t,
\\
0 & \text{else},
\end{cases}
\end{align}
for almost every $t\in [0,\Tend)$ and $|V_t|_{\Sb^{d-1}}$-almost every $x\in\Rd$.
\end{remark}

In order to define a notion of strong solutions to the free boundary problem for the flow 
of two immiscible fluids, let us first define a notion of smoothly evolving domains.

\begin{definition}[Smoothly evolving domains and surfaces]
\label{definition:domains}
Let $\Omega_0^+$ be a bounded domain of class $C^3$ and consider a family 
$(\Omega_t^+)_{t\in [0,\Tmax)}$ of open sets in $\R^d$. Let $I(t)=\partial\Omega_t^+$ and 
$\Omega_t^-=\R^d\setminus(\Omega_t^+\cup I(t))$ for every $t\in [0,\Tmax]$. 

We say that $\Omega_t^+$, $\Omega_t^-$ are \emph{smoothly evolving domains} and that $I(t)$ are
\emph{smoothly evolving surfaces} if we have $\Omega_t^+=\Psi^t(\Omega_0^+)$, $\Omega_t^-=\Psi^t(\Omega_0^-)$, 
and $I(t)=\Psi^t(I(0))$ for a map $\Psi\colon \R^d\times[0,\Tmax)\to\R^d$, $(t,x)\mapsto\Psi(t,x) = \Psi^t(x)$, 
subject to the following conditions:
\begin{itemize}
\item[i)] We have $\Psi^0 = \mathrm{Id}$.
\item[ii)] For any fixed $t\in [0,\Tmax)$, the map $\Psi^t\colon\R^d\to\R^d$ is a 
$C^3$-diffeomorphism. Moreover, we assume $\|\Psi\|_{L^\infty_tW^{3,\infty}_x}<\infty$.
\item[iii)] We have $\partial_t \Psi \in C^0([0,\Tmax);C^2(\Rd;\Rd))$
and $\|\partial_t\Psi\|_{L^\infty_tW^{2,\infty}_x}<\infty$.
\end{itemize}
Moreover, we assume that there exists $r_c\in (0,\frac{1}{2}]$ with the following property: For all $t\in [0,\Tmax)$
and all $x\in I(t)$ there exists a function $g\colon B_1(0)\subset\R^{d-1}\to \R$ 
with $\nabla g(0)=0$ such that after a rotation and a translation, $I(t)\cap B_{2r_c}(x)$ is given by the graph $\{(x,g(x)):x\in \R^{d-1}\}$.
Furthermore, for any of these functions $g$ the pointwise bounds $|\nabla^m g|\leq  \smash{r_c^{-{(m-1)}}}$ hold for all $1\leq m\leq 3$.
\end{definition}

We have everything in place to give the definition of a strong solution to the free boundary
problem for the Navier--Stokes equation for two fluids.

\begin{definition}[Strong solution for the two-phase Navier--Stokes equation]
\label{DefinitionStrongSolution}
Let a surface tension constant $\sigma>0$, the densities and shear viscosities of the 
two fluids $\rho^\pm,\mu^\pm>0$, a finite time $\Tmax>0$, a solenoidal initial velocity profile 
$v_0\in L^2(\R^d;\Rd)$, and a domain $\Omega_0^+$ occupied initially by the first fluid be 
given. Let the initial interface between the fluids $\partial \Omega_0^+$ be a compact $C^3$-manifold.

A pair $(\chi,v)$ consisting of a velocity field $v$ and an indicator function $\chi$ of the 
volume occupied by the first fluid with
\begin{align*}
&v \in W^{1,\infty}([0,\Tmax];H^1(\R^d;\R^d)) \cap W^{1,\infty}([0,\Tmax];W^{1,\infty}(\Rd;\Rd)),
\\
&\nabla v \in L^1([0,\Tmax];\BV(\Rd;\R^{d\times d})),
\\
&\chi \in L^\infty([0,\Tmax];\BV(\R^d;\{0,1\})),
\end{align*}
is called a \emph{strong solution} to the free boundary problem for the Navier--Stokes equation 
for two fluids with initial data $(\chi_0,v_0)$ if the volume occupied by the first fluid 
$\Omega^+_t:=\{x\in \Rd:\chi(x,t)=1\}$ is a smoothly evolving domain and the interface 
$I_v(t) := \partial\Omega^+_t$ is a smoothly evolving surface in the sense of 
Definition~\ref{definition:domains} and if additionally the following conditions are satisfied:
\begin{subequations}
\begin{itemize}[leftmargin=0.5cm]
\item[i)]  The velocity field $v$ has vanishing divergence $\nabla\cdot v=0$ and the equation for the momentum balance
\begin{align}
\nonumber
\int_\Rd& \rho(\chi(\cdot,T)) v(\cdot,T) \cdot \eta(\cdot,T) \,\dx
- \int_\Rd \rho(\chi_0) v_0 \cdot \eta(\cdot,0) \,\dx 
\\
\label{weakNSb}
&= \int_0^T \int_\Rd \rho(\chi) v \cdot \partial_t \eta \,\dx\,\dt
+ \int_0^T \int_\Rd \rho(\chi) v \otimes v : \nabla \eta \,\dx\,\dt
\\&~~~~
\nonumber
-\int_0^T\int_\Rd \mu(\chi) (\nabla v+\nabla v^T) : \nabla \eta \,\dx\,\dt
\\&~~~~
\nonumber
+\sigma\int_0^T\int_{I_v(t)} \vec{H}\cdot \eta \,dS \,\dt
\end{align}
is satisfied for almost every $T\in [0,\Tmax)$ and every smooth vector field 
$\eta \in C^\infty_{cpt}(\Rd\times [0,\Tmax);\Rd)$ with $\nabla \cdot \eta=0$. Here, $\vec{H}$ denotes the mean curvature vector of the interface $I_v(t)$. 
For the sake of brevity, we have used the abbreviations $\rho(\chi):=\rho^+ \chi + \rho^- (1-\chi)$ and $\mu(\chi):=\mu^+ \chi + \mu^- (1-\chi)$.
\item[ii)] The indicator function $\chi$ of the volume occupied by the first fluid satisfies 
the weak formulation of the transport equation
\begin{align}\label{weakTransportB}
\int_{\R^d}\chi(\cdot,T) \varphi(\cdot,T) \,\dx
-\int_{\R^d}\chi_0\varphi(\cdot,0)\,\dx
= \int_0^T\int_{\R^d}\chi\,(\partial_t\varphi+(v\cdot\nabla)\varphi)\,\dx\,\dt
\end{align}
for almost every $T\in [0,\Tmax)$ and all $\varphi\in C^\infty_{cpt}(\Rd\times [0,\Tmax))$.
\item[iii)] In the domain $\bigcup_{t\in [0,\Tmax)} (\Omega_t^+\cup \Omega_t^-) \times \{t\}$ the third 
spatial derivatives of the velocity field exist and satisfy
\begin{align*}
\sup_{t\in [0,\Tmax)} \sup_{x\in \Omega_t^+\cup \Omega_t^-} |\nabla^3 v(x,t)| <\infty.
\end{align*}
\end{itemize}
\end{subequations}
\end{definition}

Before we state the main ingredient for the proof of Theorem~\ref{weakStrongUniq},
we proceed with two remarks on the notion of strong solutions. The first concerns
the consistency with the notion of varifold solutions due to Abels~\cite{Abels}.

\begin{remark}
Every strong solution $(\chi,v)$ to the free boundary problem for the incompressible Navier--Stokes 
equation for two fluids \eqref{EquationTransport}--\eqref{EquationIncompressibility} in the sense of 
Definition~\ref{DefinitionStrongSolution} canonically defines a varifold solution in the sense of
Definition~\ref{DefinitionVarifoldSolution}. Indeed, we can define the varifold $V$ by means of $\mathrm{d}V_t = \delta_{\frac{\nabla \chi}{|\nabla \chi|}} \mathrm{d}|\nabla \chi|$.
Due to the regularity requirements on the family of smoothly evolving surfaces $I(t)$, see Definition~\ref{definition:domains}, it then follows
\begin{align*}
\int_0^T\int_{I(t)}\vec{H}\cdot\varphi\,\mathrm{d}S\,\mathrm{d}t 
&=-\int_0^T\int_{\R^d}(\mathrm{Id}-n\otimes n):\nabla\varphi\,\mathrm{d}|\nabla\chi(\cdot,t)|\,\mathrm{d}t \\
&=-\int_0^T\int_{\R^d\times\Sb^{d-1}}(\mathrm{Id}-s\otimes s):\nabla\varphi\,\mathrm{d}V_t(x,s)\,\mathrm{d}t,
\end{align*}
for almost every $T\in [0,\Tmax)$ and all $\varphi\in C^\infty_{cpt}(\Rd\times [0,\Tend);\Rd)$, 
see for instance \cite[Lemma 3.4]{Abels2018}. Moreover, it follows from the regularity requirements of a strong solution
that the velocity field $v$ also satisfies the energy dissipation inequality \eqref{genEI}. This proves the claim.
\end{remark}

The second remark concerns the validity of \eqref{kinematicInt} in a strong sense for a strong solution, i.e.,
that the evolution of the interface $I(t)$ occurs only as a result of the transport of the 
two fluids along the flow.

\begin{remark}\label{strongKinematic}
Let $(\chi,v)$ be a strong solution to the free boundary problem for the incompressible Navier--Stokes 
equation for two fluids \eqref{EquationTransport}--\eqref{EquationIncompressibility} in the sense of 
Definition~\ref{DefinitionStrongSolution} on some time interval $[0,\Tmax)$. Let $V_{\vec{n}}(x,t)$ denote 
the normal speed of the interface at $x\in I_v(t)$, i.e., the normal component of $\partial_t\Psi(x,t)$
where $\Psi\colon \R^d\times[0,\Tmax)\to\R^d$ is the family of diffeomorphisms from the definition of a family of smoothly evolving domains (Definition~\ref{definition:domains}). Furthermore, let $\varphi\in C^\infty_{cpt}(\Rd\times (0,\Tmax))$.
Due to the regularity requirements on a family of smoothly evolving domains, 
see Definition~\ref{definition:domains}, we obtain (see for instance \cite[Theorem 2.6]{Abels2018})
\begin{align*}
\int_0^\Tmax\int_{\Rd}\chi\partial_t\varphi\,\mathrm{d}x\,\mathrm{d}t
=-\int_0^\Tmax\int_{I_v(t)}V_{\vec{n}}\varphi\,\mathrm{d}S\,\mathrm{d}t.
\end{align*}
On the other side, subtracting from the former identity the equation \eqref{weakTransportB} satisfied by 
the indicator function $\chi$ and making use of the incompressibility of the velocity 
field $v$ we deduce 
\begin{align*}
\int_0^\Tmax\int_{I_v(t)}(V_{\vec{n}}-\vec{n}\cdot v)\varphi\,\mathrm{d}S\,\mathrm{d}t = 0.
\end{align*}
Since $\varphi\in C^\infty_{cpt}(\Rd\times (0,\Tmax))$ was arbitrary we recover the identity
\begin{align*}
V_{\vec{n}} = \vec{n}\cdot v\quad\text{on}\quad\bigcup_{t\in (0,\Tmax)}\{t\}\times I_v(t),
\end{align*}
that is to say, the kinematic condition \eqref{kinematicInt} of the interface being transported with the flow
is satisfied in its strong formulation.
\end{remark}

Our weak-strong uniqueness result in Theorem~\ref{weakStrongUniq} relies on the following relative 
entropy inequality. The regime of equal shear viscosities $\mu_+=\mu_-$ corresponds to 
the choice of $w=0$ in the statement below. Note also that in this case the viscous stress term $R_{visc}$ disappears due to $\mu(\chi_u)-\mu(\chi_v)=0$. 

\begin{proposition}[Relative entropy inequality]
\label{PropositionRelativeEntropyInequalityFull}
Let $d\leq 3$. Let $(\chi_u,u,V)$ be a varifold solution to the free boundary problem 
for the incompressible Navier--Stokes equation for two fluids 
\eqref{EquationTransport}--\eqref{EquationIncompressibility} in the sense of 
Definition~\ref{DefinitionVarifoldSolution} on some time interval $[0,\Tend)$. 
Let $(\chi_v,v)$ be a strong solution to \eqref{EquationTransport}--\eqref{EquationIncompressibility} 
in the sense of Definition~\ref{DefinitionStrongSolution} on some time interval $[0,\Tmax)$ with 
$\Tmax\leq \Tend$ and let 
\begin{align*}
w\in L^2([0,\Tmax);H^1(\Rd;\Rd))\cap H^1([0,\Tmax);L^{4/3}(\Rd;\Rd)+L^2(\Rd;\Rd))
\end{align*} 
be a solenoidal vector field with bounded spatial derivative $\|\nabla w\|_{L^\infty}<\infty$.
Suppose furthermore that for almost every $t\geq 0$, for every $x\in \Rd$ either $x$ is a Lebesgue point of $\nabla w(\cdot,t)$ or there exists a half-space $H_x$ such that $x$ is a Lebesgue point for both $\nabla w(\cdot,t)|_{H_x}$ and $\nabla w(\cdot,t)|_{\Rd\setminus H_x}$.

For a point $(x,t)$ such that $\dist(x,I_v(t))<r_c$, denote by $P_{I_v(t)}x$ the projection of $x$ onto the interface $I_v(t)$ of the strong solution. Introduce the extension $\xi$ of the unit normal $\vec{n}_v$ of the interface of the strong solution defined by
\begin{align*}
\xi(x,t):=\vec{n}_v(P_{I_v(t)} x) (1-\dist(x,I_v(t))^2) \eta(\dist(x,I_v(t)))
\end{align*}
for some cutoff $\eta$ with $\eta(s)= 1$ for $s\leq \frac{1}{2}r_c$ and $\eta\equiv 0$ for $s\geq r_c$.
Let $\bar{V}_{\vec{n}}(x,t):=(\vec{n}(P_{I_v(t)}x,t)\cdot v(P_{I_v(t)}x,t)) \vec{n}(P_{I_v(t)}x,t)$ be an extension of the normal velocity of the interface of the strong solution $I_v(t)$ to an $r_c$-neighborhood of $I_v(t)$. Let $\theta$ be the density $\theta_t=\frac{\mathrm{d}|\nabla\chi_u(\cdot,t)|}{\mathrm{d}|V_t|_{\Sb^{d-1}}}$ as defined in \eqref{DefinitionTheta} and let $\beta:\R\rightarrow \R$ be a truncation of the identity with $\beta(r)=r$ for $|r|\leq \frac{1}{2}$, $|\beta'|\leq 1$,  $|\beta''|\leq C$, and $\beta'(r)=0$ for $|r|\geq 1$.

Then the relative entropy
\begin{align}\label{DefinitionRelativeEntropyFunctional}
E\big[\chi_u,u,V\big|\chi_v,v\big](T) &:=
\sigma\int_\Rd 1-\xi(\cdot,T)\cdot
\frac{\nabla\chi_u(\cdot,T)}{|\nabla\chi_u(\cdot,T)|} 
\,\mathrm{d}|\nabla \chi_u(\cdot,T)| 
\\&~~~~\nonumber
+ \int_{\Rd} \frac{1}{2} \rho\big(\chi_u(\cdot,T)\big)
\big|u-v-w\big|^2(\cdot,T) \,\dx 
\\&~~~~\nonumber
+\int_\Rd \big|\chi_u(\cdot,T)-\chi_v(\cdot,T)\big|\,\Big|\beta\Big(\frac{\sigdist(\cdot,I_v(T))}{r_c}\Big)\Big|\,\dx
\\&~~~~\nonumber
+\sigma\int_\Rd 1-\theta_T\,\mathrm{d}|V_T|_{\Sb^{d-1}}
\end{align}
is subject to the relative entropy inequality
\begin{align*}
&E\big[\chi_u,u,V\big|\chi_v,v\big](T) 
+ \int_0^T\int_{\mathbb{R}^d} 2\mu(\chi_u)\big|\Dsym (u - v - w)\big|^2 \,\dx\,\dt
\\&
\leq E\big[\chi_u,u,V\big|\chi_v,v\big](0)
+ R_{surTen} + R_{dt} + R_{visc} + R_{adv} + R_{weightVol} 
\\&~~~~
+ A_{visc} + A_{dt} + A_{adv} + A_{surTen} + A_{weightVol}
\end{align*}
for almost every $T\in (0,\Tmax)$, where we made use of the abbreviations
\begin{align*}
&R_{surTen} :=
\\&~~~
-\sigma\int_0^T\int_{\Rd\times\Sd} 
(s-\xi)\cdot\big((s-\xi)\cdot\nabla\big)v\,\mathrm{d}V_t(x,s)\,\dt
\\&~~~
+\sigma\int_0^T\int_\Rd (1-\theta_t)\,\xi\cdot(\xi\cdot\nabla) v
\,\mathrm{d}|V_t|_{\Sb^{d-1}}\,\dt
\\&~~~
+\sigma\int_0^T\int_\Rd(\chi_u-\chi_v)\big((u-v-w)\cdot\nabla\big)(\nabla\cdot\xi)\,\dx\,\dt
\\&~~~
-\sigma\int_0^T\int_{\Rd}
\Big(\xi\cdot\frac{\nabla\chi_u}{|\nabla\chi_u|}\Big)\,
\vec{n}_v(P_{I_v(t)}x)\cdot\big(\vec{n}_v(P_{I_v(t)}x)\cdot\nabla\big)v
-\xi\cdot(\xi\cdot\nabla) v
\,\mathrm{d}|\nabla\chi_u|\,\dt
\\&~~~
+\sigma\int_0^T\int_{\Rd}\frac{\nabla\chi_u}{|\nabla\chi_u|}
\cdot\big((\mathrm{Id}{-}\vec{n}_v(P_{I_v(t)}x)\otimes\vec{n}_v(P_{I_v(t)}x))
(\nabla\bar{V}_{\vec{n}}{-}\nabla v)^T\cdot\xi\big)
\,\mathrm{d}|\nabla\chi_u|\,\dt
\\&~~~
+\sigma\int_0^T\int_{\Rd}\frac{\nabla\chi_u}{|\nabla\chi_u|}
\cdot\big((\bar{V}_{\vec{n}}-v)\cdot\nabla\big)\xi\,\mathrm{d}|\nabla\chi_u|\,\dt
\end{align*}
and
\begin{align*}
R_{dt} :=& -\int_0^T\int_\Rd \big(\rho(\chi_u)-\rho(\chi_v)\big)
(u-v-w)\cdot\partial_t v\,\dx\,\dt,
\\
R_{visc} :=& -\int_0^T \int_{\mathbb{R}^d} 2\big(\mu(\chi_u)-\mu(\chi_v)\big)
\Dsym v:\Dsym (u-v) \,\dx\,\dt,
\\
R_{adv} := &-\int_0^T\int_\Rd\big(\rho(\chi_u)-\rho(\chi_v)\big)
(u-v-w)\cdot(v\cdot\nabla)v\,\dx\,\dt 
\\&
-\int_0^T\int_\Rd \rho(\chi_u)(u-v-w)\cdot
\big((u-v-w)\cdot\nabla\big)v\,\dx\,\dt,
\end{align*}
as well as
\begin{align*}
&R_{weightVol} :=
\\&~~~~
\int_0^T\int_\Rd(\chi_u{-}\chi_v)
\big(\big(\bar{V}_{\vec{n}}{-}(v\cdot\vec{n}_v(P_{I_v(t)}x))\vec{n}_v(P_{I_v(t)}x)\big)\cdot\nabla\big)
\beta\Big(\frac{\sigdist(\cdot,I_v)}{r_c}\Big)\,\dx\,\dt
\\&~~~~
+\int_0^T\int_{\R^d}(\chi_u{-}\chi_v)\big((u{-}v{-}w)\cdot\nabla\big)
\beta\Big(\frac{\sigdist(\cdot,I_v)}{r_c}\Big)\,\dx\,\dt.
\end{align*}
Moreover, we have abbreviated
\begin{align*}
A_{visc}:=&\int_0^T \int_{\mathbb{R}^d} 2\big(\mu(\chi_u)-\mu(\chi_v)\big)\Dsym v:
\Dsym w \,\dx\,\dt
\\&
-\int_0^T \int_{\mathbb{R}^d} 2\mu(\chi_u) \Dsym w: \Dsym(u-v-w) \,\dx\,\dt,
\end{align*}
and
\begin{align*}
A_{dt}:=&-\int_0^T \int_{\mathbb{R}^d} \rho(\chi_u) (u-v-w)\cdot 
\partial_t w \,\dx\,\dt
\\
&-\int_0^T\int_\Rd \rho(\chi_u) (u-v-w) \cdot (v\cdot\nabla) w\,\dx\,\dt,
\\
A_{adv}:=&-\int_0^T\int_\Rd \rho(\chi_u) (u-v-w)\cdot (w\cdot \nabla)(v+w) \,\dx\,\dt
\\
&-\int_0^T\int_\Rd \rho(\chi_u) (u-v-w)\cdot \big((u-v-w)\cdot \nabla\big)w \,\dx\,\dt,
\\
A_{weightVol} :=&
\int_0^T\int_{\R^d}(\chi_u{-}\chi_v)(w\cdot\nabla)
\beta\Big(\frac{\sigdist(\cdot,I_v)}{r_c}\Big)\,\dx\,\dt,
\end{align*}
as well as
\begin{align*}
A_{surTen}:= 
&-\sigma\int_0^T\int_{\R^d\times\Sb^{d-1}}(s{-}\xi)
\cdot\big((s{-}\xi)\cdot\nabla\big) w \,\mathrm{d}V_t(x,s)\,\dt 
\\&
+\sigma\int_0^T\int_{\R^d}(1-\theta_t)\,
\xi\cdot(\xi\cdot\nabla)w\,\mathrm{d}|V_t|_{\Sb^{d-1}}(x)\,\dt
\\&
+\sigma\int_0^T\int_\Rd(\chi_u-\chi_v)(w\cdot\nabla)(\nabla\cdot\xi)\,\dx\,\dt
\\&
+\sigma\int_0^T\int_\Rd (\chi_u-\chi_v) \nabla w :\nabla \xi^T \,\dx\,\dt
\\&
-\sigma\int_0^T\int_\Rd \xi \cdot \big((\vec{n}_u-\xi) \cdot \nabla\big) w \,\mathrm{d}|\nabla \chi_u|\,\dt.
\end{align*}
\end{proposition}

\subsection*{Notation}
We use $a\wedge b$ (respectively $a\vee b$) as a shorthand notation for the minimum (respectively maximum) of two numbers $a,b\in R$.

Let $\Omega\subset\R^d$ be open.
For a function $u:\Omega\times [0,T]\rightarrow \R$, we denote by $\nabla u$ its distributional derivative with respect to space and by $\partial_t u$ its derivative with respect to time.
For $p\in [1,\infty]$ and an integer $k\in\N_0$, we denote by $L^p(\Omega)$
and $W^{k,p}(\Omega)$ the usual Lebesgue and Sobolev spaces. 
In the special case $p=2$ we use as usual $H^k(\Omega):=W^{k,2}(\Omega)$ to denote the Sobolev space.
For integration of a function $f$ with respect to the $d$-dimensional Lebesgue measure respectively the $d-1$-dimensional surface measure, we use the usual notation $\int_\Omega f \mathrm{d}x$ respectively $\int_I f \mathrm{d}S$.
For measures other than the natural measure (the Lebesgue measure in case of domains $\Omega$ and the surface measure in case of surfaces $I$), we denote the corresponding Lebesgue spaces by $L^p(\Omega,\mu)$.
The space of all compactly supported and infinitely differentiable functions on $\Omega$
is denoted by $C_{cpt}^\infty(\Omega)$. The closure of $C_{cpt}^\infty(\Omega)$ with respect to
the Sobolev norm $\|\cdot\|_{W^{k,p}(\Omega)}$ is $W_0^{k,p}(\Omega)$,
and its dual will be denoted by $W^{-1,p'}(\Omega)$ where $p'\in [0,\infty]$
is the conjugated H\"{o}lder exponent of $p$, i.e.\ $1/p+1/p'=1$.
For vector-valued fields, say with range in $\R^d$, we use the notation $L^p(\Omega;\R^d)$, and so on. 
For a Banach space $X$, a finite time $T>0$ and a number $p\in [1,\infty]$ 
we denote by $L^p([0,T];X)$ the usual Bochner--Lebesgue space. If $X$ itself is
a Sobolev space $W^{k,q}$, we denote the norm of $L^p([0,T];X)$  as $\|\cdot\|_{L^p_tW^{k,q}_x}$.
When writing $L^\infty_w([0,T];X')$ we refer to the space of bounded and weak-$^\ast$ measurable
maps $f\colon [0,T]\to X'$, where $X'$ is the dual space of $X$.
By $L^p(\Omega)+L^q(\Omega)$ we denote the normed space of all functions $u:\Omega\rightarrow \R$ which may be written as the sum of two functions $v\in L^p(\Omega)$ and $w\in L^q(\Omega)$.
The space $C^k([0,T];X)$ contains all $k$-times continuously differentiable 
and $X$-valued functions on $[0,T]$.

In order to give a suitable weak description of the evolution of the sharp interface,
we have to recall the concepts of Caccioppoli sets as well as varifolds. To this
end, let $\Omega\subset\R^d$ be open. We denote by $\BV(\Omega)$ the space of
functions with bounded variation in $\Omega$. A measurable subset $E\subset\Omega$
is called a set of finite perimeter in $\Omega$ (or a Caccioppoli subset of $\Omega$)
if its characteristic function $\chi_E$ is of bounded variation in $\Omega$.
We will write $\partial^*E$ when referring to the reduced boundary of a
Caccioppoli subset $E$ of $\Omega$; whereas $\vec{n}$ 
denotes the associated measure theoretic (inward pointing) unit normal
vector field of $\partial^*E$. For detailed definitions of all these concepts
from geometric measure theory, we refer to \cite{Federer1978,Chlebik2005}.
In case $\Omega$ has a $C^2$ boundary, we denote by $\vec{H}(x)$ the mean curvature
vector at $x\in\partial\Omega$.
Recall that for a convex function $g:\Rd\rightarrow \mathbb{R}$ the \emph{recession function} $g^{rec}:\Rd\rightarrow \R$ is defined as $g^{rec}(x):=\lim_{\tau\rightarrow \infty} \tau^{-1} g(\tau x)$.

An oriented varifold is simply a non-negative measure $V\in\mathcal{M}(\Omega{\times}\Sb^{d-1})$,
where $\Omega\subset\R^d$ is open and $\Sb^{d-1}$ denotes the $(d{-}1)$-dimensional sphere. For a varifold $V$, we denote by $|V|_{\Sb^{d-1}}\in \mathcal{M}(\Omega)$ its local mass density given by $|V|_{\Sb^{d-1}}(A):=V(A\times \Sb^{d-1})$ for any Borel set $A\subset \Omega$.
For a locally compact separable metric space $X$ we write $\mathcal{M}(X)$
to refer to the space of (signed) finite Radon-measures on $X$.
If $A\subset X$ is a measurable set
and $\mu\in\mathcal{M}(X)$, we let $\mu\mres A$ be the restriction of $\mu$ on $A$.
The $k$-dimensional Hausdorff measure
on $\R^d$ will be denoted by $\mathcal{H}^k$, whereas we write $\mathcal{L}^d(A)$
for the $d$-dimensional Lebesgue measure of a Lebesgue measurable set $A\subset\R^d$.

Finally, let us fix some tensor notation. First of all, we use $(\nabla v)_{ij} = \partial_j v_i$ 
as well as $\nabla\cdot v = \sum_i\partial_i v_i$
for a Sobolev vector field $v\colon\R^d\to\R^d$. 
The symmetric gradient is denoted by $\Dsym v := \frac{1}{2}(\nabla v + \nabla v^T)$. For time-dependent fields
$v\colon\Rd\times[0,T)\to\R^n$ we denote by $\partial_tv$ the partial derivative with respect to time.
Tensor products of vectors $u,v\in\R^d$ will be given by 
$(u\otimes v)_{ij} = u_iv_j$. For tensors $A=(A_{ij})$ and $B=(B_{ij})$ 
we write $A:B = \sum_{ij} A_{ij}B_{ij}$.    

\section{Outline of the strategy}
\label{SectionStrategy}

\subsection{The relative entropy}

The basic idea of the present work is to measure the ``distance'' between a varifold solution to the two-phase Navier-Stokes equation $(\chi_u,u,V)$ and a strong solution to the two-phase Navier-Stokes equation $(\chi_v,v)$ by means of the relative entropy functional
\begin{align}
\nonumber
E\big[\chi_u,u,V\big|\chi_v,v\big](t) :=&
\sigma \int_\Rd 1-\xi \cdot \frac{\nabla \chi_u}{|\nabla \chi_u|} \,\mathrm{d}|\nabla \chi_u| + \int_\Rd \frac{\rho(\chi_u)}{2}|u-v-w|^2 \,\dx
\\&
\label{DefinitionRelativeEntropy}
+\sigma \int_\Rd 1-\theta_t \,\mathrm{d}|V_t|_{\Sb^{d-1}}
\\&\nonumber
+\int_\Rd |\chi_u-\chi_v| \left|\beta\Big(\frac{\sigdist(x,I_v(t))}{r_c}\Big)\right| \,\dx
\end{align}
where $\xi:\Rd\times [0,\Tmax)\rightarrow \Rd$ is a suitable extension of the unit normal vector field of the interface of the strong solution and where $w$ is a vector field that will be constructed below and that vanishes in case of equal viscosities $\mu^+=\mu^-$. More precisely, we choose $\xi$ as
\begin{align*}
\xi(x,t):=\vec{n}_v(P_{I_v(t)} x) (1-\dist(x,I_v(t))^2) \eta(\dist(x,I_v(t)))
\end{align*}
for some cutoff $\eta$ with $\eta(s)= 1$ for $s\leq \frac{1}{2}r_c$ and $\eta\equiv 0$ for $s\geq r_c$, 
where $P_{I_v(t)}x$ denotes for each $t\geq 0$ the projection of $x$ onto the interface $I_v(t)$ of the 
strong solution and where the unit normal vector field $\vec{n}_v$ of the interface of the strong solution 
is oriented to point towards $\{\chi_v(\cdot,t)=1\}$.

Rewriting the relative entropy functional in the form
\begin{align*}
&E\big[\chi_u,u,V\big|\chi_v,v\big](t)
\\&
=E[\chi_u,u,V](t) + \int_\Rd \chi_u \nabla\cdot \xi \,\dx - \int_\Rd \rho(\chi_u)u\cdot (v+w) \,\dx
\\&~~~~~
+ \int_\Rd \frac{1}{2}\rho(\chi_u)|v+w|^2 \,\dx
+\int_\Rd (\chi_u-\chi_v) \beta\Big(\frac{\sigdist(x,I_v(t))}{r_c}\Big) \,\dx
\end{align*}
with the energy \eqref{EnergyVarifold}, we see that we may estimate the time evolution of the relative entropy $E\big[\chi_u,u,V\big|\chi_v,v\big](t)$ by exploiting the energy dissipation property \eqref{genEI} of the varifold solution and by testing the weak formulation of the two-phase Navier-Stokes equation \eqref{weakNS} and \eqref{weakTransport} against the (sufficiently regular) test functions $v+w$ respectively $\frac{1}{2}|v+w|^2$, $\nabla \cdot \xi$, and $\beta(\frac{\sigdist(x,I_v(t))}{r_c})$.

As usual in the derivation of weak-strong uniqueness results by the relative entropy method of Dafermos \cite{Dafermos} and Di~Perna \cite{DiPerna}, in the case of equal viscosities $\mu^+=\mu^-$ the goal is the derivation of an estimate of the form
\begin{align}
\label{EstimateGronwallType}
&E\big[\chi_u,u,V\big|\chi_v,v\big](T)+ c\int_0^T \int_\Rd |\nabla u-\nabla v|^2 \,\dx\,\dt
\\&\nonumber
\leq C(v,I_v,\text{data}) \int_0^T E\big[\chi_u,u,V\big|\chi_v,v\big](t) \,\dt
\end{align}
which implies uniqueness and stability by means of the Gronwall lemma and by the coercivity properties of the relative entropy functional discussed in the next section.

In the case of different viscosities $\mu^+\neq \mu^-$, we will derive a slightly weaker (but still sufficient) result of roughly speaking the form
\begin{align}
\label{WeakerThanGronwall}
&E\big[\chi_u,u,V\big|\chi_v,v\big](T) + c\int_0^T \int_\Rd |\nabla u-\nabla v-\nabla w|^2 \,\dx\,\dt
\\&
\nonumber
\leq C(v,I_v,\text{data}) \int_0^T E\big[\chi_u,u,V\big|\chi_v,v\big](t) ~ 
\left|\log E\big[\chi_u,u,V\big|\chi_v,v\big](t)\right| \,\dt,
\end{align}
along with estimates on $w$ which include in particular the bound
\begin{align*}
\int_\Rd |w(\cdot,T)|^2 \,dx
\leq C(v,I_v,\text{data}) E\big[\chi_u,u,V\big|\chi_v,v\big](T).
\end{align*}

\subsection{The error control provided by the relative entropy functional}

The relative entropy functional \eqref{DefinitionRelativeEntropy} provides control of the following quantities (up to bounded prefactors):
~\\
\noindent
{\bf Velocity error control.}
The relative entropy $E[\chi_u,u,V|\chi_v,v](t)$ controls the square of the velocity error in the $L^2$ norm
\begin{align*}
\int_\Rd |u(\cdot,t)-v(\cdot,t)|^2 \,\dx
\end{align*}
at any given time $t$. In the case of equal viscosities, this is immediate from \eqref{DefinitionRelativeEntropy} by $w\equiv 0$, while in the case of different viscosities this follows by the estimate $\int_\Rd |w|^2 \,\dx\leq C \|\nabla v\|_{L^\infty} \int_\Rd 1-\xi \cdot \frac{\nabla \chi_u}{|\nabla \chi_u|} \,\mathrm{d}|\nabla \chi_u|$ which is a consequence of the construction of $w$ and the choice of $\xi$, see below. 
\begin{figure}
\begin{center}
\begin{tikzpicture}[scale=2.0]
\draw[white,fill=LightRed!50!white] (0,0) .. controls (0.5,0.15) and (0.8,0.3) .. (1.5,0.4) .. controls (2.2,0.5) and (2.5,0.5) .. (3.0,0.4) .. controls (3.5,0.3) and (4.0,0.1) .. (5.0,0.2) -- (5,1.6) -- (0,1.6) -- cycle;
\draw[white,fill=LightBlue!50!white] (0,0) .. controls (0.5,0.15) and (0.8,0.3) .. (1.5,0.4) .. controls (2.2,0.5) and (2.5,0.5) .. (3.0,0.4) .. controls (3.5,0.3) and (4.0,0.1) .. (5.0,0.2) -- (5,-0.6) -- (0,-0.6) -- cycle;
\draw[thick] (0,0) .. controls (0.5,0.15) and (0.8,0.3) .. (1.5,0.4) .. controls (2.2,0.5) and (2.5,0.5) .. (3.0,0.4) .. controls (3.5,0.3) and (4.0,0.1) .. (5.0,0.2);
\draw[draw=none,pattern=north east lines, pattern color=LightBlue, draw opacity=0.5] (0,0.0+0.1) .. controls (0.5,0.15+0.1) and (0.5,0.25) .. (0.7,0.3) .. controls (0.8,0.35) and (1.6,0.4) .. (1.6,0.6) .. controls (1.6,0.8) and (0.3,0.2) .. (0.3,0.5) .. controls (0.3,0.8) and (1.2,0.65) .. (1.2,0.85) .. controls (1.2,1.05) and (0.5,0.75) .. (0.5,1.0) .. controls (0.5,1.25) and (1.4,0.3+1.0) .. (1.7,0.4+0.5) .. controls (2.0,0.5) and (2.5,0.5+0.3) .. (3.0,0.4+0.3) .. controls (3.5,0.3+0.3) and (4.0,0.1-0.1) .. (5.0,0.2-0.1) -- (5,-0.6) -- (0,-0.6) -- cycle;
\draw[thick,dotted] (0,0.0+0.1) .. controls (0.5,0.15+0.1) and (0.5,0.25) .. (0.7,0.3) .. controls (0.8,0.35) and (1.6,0.4) .. (1.6,0.6) .. controls (1.6,0.8) and (0.3,0.2) .. (0.3,0.5) .. controls (0.3,0.8) and (1.2,0.65) .. (1.2,0.85) .. controls (1.2,1.05) and (0.5,0.75) .. (0.5,1.0) .. controls (0.5,1.25) and (1.4,0.3+1.0) .. (1.7,0.4+0.5) .. controls (2.0,0.5) and (2.5,0.5+0.3) .. (3.0,0.4+0.3) .. controls (3.5,0.3+0.3) and (4.0,0.1-0.1) .. (5.0,0.2-0.1);
\draw[draw=none,pattern=north east lines, pattern color=LightBlue, draw opacity=0.5] (0.28,0.93) circle (0.135);
\draw[thick,dotted] (0.28,0.93) circle (0.135);
\draw[densely dotted,draw=Purple] (3,0.4) -- (3+4.1*0.065,0.4+4.1*0.29);
\draw[thick,draw=Purple] (3,0.4) -- (3.065,0.69);
\draw[color=Purple] (3.4,0.8) node {h$^+$(x)};
\draw[color=LightBlue!45!black] (2.3,0.83) node {$\mathbf{\{\chi_u=1\}}$};
\draw[color=LightBlue!45!black] (2.15,0.57) node {$\mathbf{\{\chi_u=0\}}$};
\draw[color=LightBlue!20!black] (4.56,0.25) node {$\mathbf{\{\chi_v=1\}}$};
\draw[color=LightBlue!20!black] (4.6,-0.03) node {$\mathbf{\{\chi_v=0\}}$};
\draw[densely dotted,draw=Purple] (0.53,0.17) -- (0.53-1.6*0.27,0.17+1.6*0.89);
\draw[thick,draw=Purple] (0.53,0.17) -- (0.53-0.105*0.27,0.17+0.105*0.89);
\draw[thick,draw=Purple] (0.53-0.27*0.27,0.17+0.27*0.89) -- (0.53-0.52*0.27,0.17+0.52*0.89);
\draw[thick,draw=Purple] (0.53-0.71*0.27,0.17+0.71*0.89) -- (0.53-1.01*0.27,0.17+1.01*0.89);
\end{tikzpicture}
\caption{An illustration of the interface error. The red and the blue region (separated by the black solid curve) correspond to the regions occupied by the two fluids in the strong solution. The shaded area corresponds to the region occupied by the blue fluid in the varifold solution, the interface in the varifold solution corresponds to the dotted curve. \label{FigureIntefaceError}}
\end{center}
\end{figure}
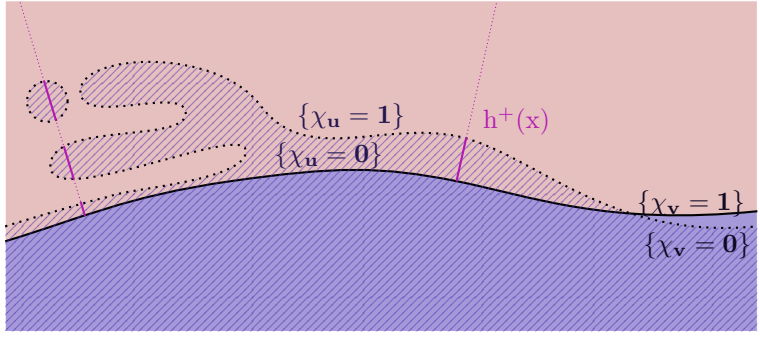

\noindent
{\bf Interface error control.}
The relative entropy provides a tilt-excess type control of the error in the interface normal
\begin{align*}
\int_\Rd 1-\xi \cdot \vec{n}_u \,\mathrm{d}|\nabla \chi_u|.
\end{align*}
In particular, it controls the squared error in the interface normal
\begin{align*}
\int_\Rd |\vec{n}_u-\xi|^2 \,\mathrm{d}|\nabla \chi_u|.
\end{align*}
The term also controls the total length respectively area (for $d=2$ respectively $d=3$) of the part of the interface $I_u$ which is not locally a graph over $I_v$, see Figure~\ref{FigureIntefaceError}. For example, in the region around the left purple half-ray the interface of the weak solution is not a graph over the interface of the weak solution. Furthermore, the term controls the length respectively area (for $d=2$ respectively $d=3$) of the part of the interface with distance to $I_v(t)$ greater than the cutoff length $r_c$, as there we have $\xi\equiv 0$.

Denote the local height of the one-sided interface error by $h^+:I_v(t)\rightarrow \mathbb{R}_0^+$ as measured along orthogonal rays originating on $I_v(t)$ (with some cutoff applied away from the interface $I_v(t)$ of the strong solution); denote by $h^-$ the corresponding height of the interface error as measured in the other direction. For example, in Figure~\ref{FigureIntefaceError} the quantity $h^+(x)$ for some base point $x\in I_v(t)$ would correspond to the accumulated length of the solid segments in each of the purple rays, the dotted segments not being counted. Note that the rays are orthogonal on $I_v(t)$. Then the tilt-excess type term in the relative entropy also controls the gradient of the one-sided interface error heights 
\begin{align*}
\int_{I_v(t)} \min\{|\nabla h^\pm|^2, |\nabla h^\pm|\} \,\dS.
\end{align*}
Note that wherever $I_u(t)$ is locally a graph over $I_v(t)$ and is not too far away from $I_v(t)$, it must be the graph of the function $h^+-h^-$. Here, the graph of a function $g$ over the curved interface $I_v(t)$ is defined by the set of points obtained by shifting the points of $I_v(t)$ by the corresponding multiple of the surface normal, i.\,e.\ $\{x+g(x) \vec{n}_v(x):x\in I_v(t)\}$.

\noindent
{\bf Varifold multiplicity error control.}
For varifold solutions, the relative entropy controls the multiplicity error of the varifold
\begin{align*}
\int_\Rd 1-\theta_t(x) \,\mathrm{d}|V_t|_{\Sb^{d-1}}
\end{align*}
(note that $\smash{\frac{1}{\theta_t(x)}}$ corresponds to the multiplicity of the varifold), which in turn by the compatibility condition \eqref{varifoldComp} and the definition of $\theta_t$ (see \eqref{DefinitionTheta}) controls the squared error in the normal of the varifold
\begin{align*}
\int_\Rd \int_{\mathbb{S}^{d-1}} |s-\vec{n}_u|^2 \,\mathrm{d}V_t(s,x).
\end{align*}

\noindent
{\bf Weighted volume error control.}
Furthermore, the error in the volume occupied by the two fluids weighted with the distance to the interface of the strong solution
\begin{align*}
\int_\Rd |\chi_u-\chi_v| \min\{\dist(x,I_v),1\} \,\dx
\end{align*}
is controlled. Note that this term is the only term in the relative entropy which is not obtained by the usual relative entropy ansatz $E[x|y]=E[x]-DE[y](x-y)-E[y]$. We have added this lower-order term -- as compared to the term $\int_\Rd 1-\xi\cdot \frac{\nabla \chi_u}{|\nabla \chi_u|} \,d|\nabla \chi_u|$ which provides tilt-excess-type control -- to the relative entropy in order to remove the lack of coercivity of the term $\int_\Rd 1-\xi\cdot \frac{\nabla \chi_u}{|\nabla \chi_u|} \,d|\nabla \chi_u|$ in the limit of vanishing interface length of the varifold solution.

\noindent
{\bf Control of velocity gradient error by dissipation.}
By means of Korn's inequality, the dissipation term controls the $L^2$-error in the gradient
\begin{align*}
\int_0^T \int_\Rd |\nabla u-\nabla v-\nabla w|^2 \,\dx\,\dt.
\end{align*}

\subsection{The case of equal viscosities}

For equal viscosities $\mu^+=\mu^-$, one may choose $w\equiv 0$. As a consequence, the right-hand side in the relative entropy inequality -- see Proposition~\ref{PropositionRelativeEntropyInequalityFull} above -- may be post-processed to yield the Gronwall-type estimate \eqref{EstimateGronwallType}. The details are provided in Section~\ref{MainResultEqualViscosities}.

\subsection{Additional challenges in the case of different viscosities}

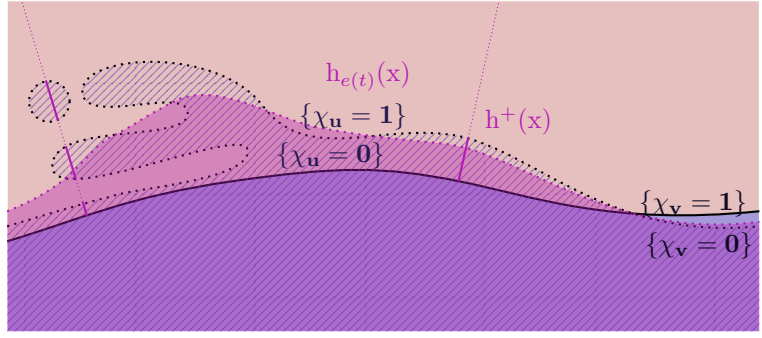
\begin{figure}
\begin{center}
\begin{tikzpicture}[scale=2.0]
\draw[white,fill=LightRed!50!white] (0,0) .. controls (0.5,0.15) and (0.8,0.3) .. (1.5,0.4) .. controls (2.2,0.5) and (2.5,0.5) .. (3.0,0.4) .. controls (3.5,0.3) and (4.0,0.1) .. (5.0,0.2) -- (5,1.6) -- (0,1.6) -- cycle;
\draw[white,fill=LightBlue!50!white] (0,0) .. controls (0.5,0.15) and (0.8,0.3) .. (1.5,0.4) .. controls (2.2,0.5) and (2.5,0.5) .. (3.0,0.4) .. controls (3.5,0.3) and (4.0,0.1) .. (5.0,0.2) -- (5,-0.6) -- (0,-0.6) -- cycle;
\draw[thick] (0,0) .. controls (0.5,0.15) and (0.8,0.3) .. (1.5,0.4) .. controls (2.2,0.5) and (2.5,0.5) .. (3.0,0.4) .. controls (3.5,0.3) and (4.0,0.1) .. (5.0,0.2);
\draw[draw=none,pattern=north east lines, pattern color=LightBlue, draw opacity=0.5] (0,0.0+0.1) .. controls (0.5,0.15+0.1) and (0.5,0.25) .. (0.7,0.3) .. controls (0.8,0.35) and (1.6,0.4) .. (1.6,0.6) .. controls (1.6,0.8) and (0.3,0.2) .. (0.3,0.5) .. controls (0.3,0.8) and (1.2,0.65) .. (1.2,0.85) .. controls (1.2,1.05) and (0.5,0.75) .. (0.5,1.0) .. controls (0.5,1.25) and (1.4,0.3+1.0) .. (1.7,0.4+0.5) .. controls (2.0,0.5) and (2.5,0.5+0.3) .. (3.0,0.4+0.3) .. controls (3.5,0.3+0.3) and (4.0,0.1-0.1) .. (5.0,0.2-0.1) -- (5,-0.6) -- (0,-0.6) -- cycle;
\draw[thick,dotted] (0,0.0+0.1) .. controls (0.5,0.15+0.1) and (0.5,0.25) .. (0.7,0.3) .. controls (0.8,0.35) and (1.6,0.4) .. (1.6,0.6) .. controls (1.6,0.8) and (0.3,0.2) .. (0.3,0.5) .. controls (0.3,0.8) and (1.2,0.65) .. (1.2,0.85) .. controls (1.2,1.05) and (0.5,0.75) .. (0.5,1.0) .. controls (0.5,1.25) and (1.4,0.3+1.0) .. (1.7,0.4+0.5) .. controls (2.0,0.5) and (2.5,0.5+0.3) .. (3.0,0.4+0.3) .. controls (3.5,0.3+0.3) and (4.0,0.1-0.1) .. (5.0,0.2-0.1);
\draw[draw=none,pattern=north east lines, pattern color=LightBlue, draw opacity=0.5] (0.28,0.93) circle (0.135);
\draw[thick,dotted] (0.28,0.93) circle (0.135);
\draw[color=Purple] (2.4,1.1) node {h$_{e(t)}$(x)};
\draw[color=Purple] (3.4,0.8) node {h$^+$(x)};
\draw[densely dotted,draw=Purple] (0.53,0.17) -- (0.53-1.6*0.27,0.17+1.6*0.89);
\draw[thick,draw=Purple] (0.53,0.17) -- (0.53-0.105*0.27,0.17+0.105*0.89);
\draw[thick,draw=Purple] (0.53-0.27*0.27,0.17+0.27*0.89) -- (0.53-0.52*0.27,0.17+0.52*0.89);
\draw[thick,draw=Purple] (0.53-0.71*0.27,0.17+0.71*0.89) -- (0.53-1.01*0.27,0.17+1.01*0.89);
\draw[densely dotted,draw=Purple] (3,0.4) -- (3+4.1*0.065,0.4+4.1*0.29);
\draw[thick,draw=Purple] (3,0.4) -- (3.065,0.69);
\draw[thick,dotted,draw=Purple] (0,0.0+0.2) .. controls (0.5,0.4) and (0.5,0.6) .. (0.7,0.7) .. controls (0.9,0.8) and (0.85,0.80) .. (1.15,0.95) .. controls (1.45,1.05) and (1.8,0.8) .. (2.1,0.75) .. controls (2.4,0.7) and (2.4,0.4+0.3) .. (2.9,0.35+0.3) .. controls (3.4,0.3+0.3) and (4.0,0.1-0.1) .. (5.0,0.23-0.1);
\draw[draw=none,fill=Purple,opacity=0.35] (0,0.0+0.2) .. controls (0.5,0.4) and (0.5,0.6) .. (0.7,0.7) .. controls (0.9,0.8) and (0.85,0.80) .. (1.15,0.95) .. controls (1.45,1.05) and (1.8,0.8) .. (2.1,0.75) .. controls (2.4,0.7) and (2.4,0.4+0.3) .. (2.9,0.35+0.3) .. controls (3.4,0.3+0.3) and (4.0,0.1-0.1) .. (5.0,0.23-0.1) -- (5,-0.6) -- (0,-0.6) -- cycle;
\draw[color=LightBlue!45!black] (2.3,0.83) node {$\mathbf{\{\chi_u=1\}}$};
\draw[color=LightBlue!45!black] (2.15,0.57) node {$\mathbf{\{\chi_u=0\}}$};
\draw[color=LightBlue!20!black] (4.56,0.25) node {$\mathbf{\{\chi_v=1\}}$};
\draw[color=LightBlue!20!black] (4.6,-0.03) node {$\mathbf{\{\chi_v=0\}}$};
\end{tikzpicture}
\caption{An illustration of the approximation of the interface error by the mollified height function $h_{e(t)}^+$. \label{FigureInterfaceMollified}}
\end{center}
\end{figure}

In the case of different viscosities $\mu_1\neq \mu_2$ of the two fluids, even for strong solutions the normal derivative of the tangential velocity features a discontinuity at the interface: By the no-slip boundary condition, the velocity is continuous across the interface $[v]=0$ and the same is true for its tangential derivatives $[(\vec{t}\cdot \nabla) v]=0$. As a consequence of this, the discontinuity of $\mu(\chi_v)$ across the interface and the equilibrium condition for the stresses at the interface
\begin{align*}
[[\mu(\chi) \vec{t}\cdot (\vec{n}\cdot \nabla) v+\mu(\chi) \vec{n}\cdot (\vec{t}\cdot \nabla)v]]=0
\end{align*}
entail for generic data a discontinuity of the normal derivative of the tangential velocity $\vec{t}\cdot (\vec{n}\cdot \nabla) v$ across the interface.

As a consequence, it becomes impossible to establish a Gronwall estimate for the standard relative entropy \eqref{DefinitionRelativeEntropy} with $w\equiv 0$.
To see this, consider in the two-dimensional case $d=2$ two strong solutions $u$ and $v$ with coinciding initial velocities $u(\cdot,0)=v(\cdot,0)=u_0(\cdot)$, but slightly different initial interfaces $\chi_v(\cdot,0)=\chi_{\{|x|\leq 1\}}$ and $\chi_u(\cdot,0)=\chi_{\{|x|\leq 1-\varepsilon\}}$ for some $\varepsilon>0$. The initial relative entropy is then of the order $\sim \varepsilon^2$. Suppose that (in polar coordinates) the initial velocity $u_0$ has a profile near the interface like
\begin{align*}
u_0(x,y)
=
\begin{cases}
\mu^- (r-1) e_\phi &\text{for }r=\sqrt{x^2+y^2}<1,
\\
\mu^+ (r-1) e_\phi &\text{for }r>1.
\end{cases}
\end{align*}
Note that this velocity profile features a kink at the interface. As one verifies readily, as far as the viscosity term is concerned this corresponds to a near-equilibrium profile for the solution $(\chi_v,v)$ (in the sense that the viscosity term is bounded). However, in the solution $(\chi_u,u)$ the interface is shifted by $\varepsilon$ and the profile is no longer an equilibrium profile.
By the scaling of the viscosity term, the timescale within which the profile $u_0$ equilibrates in the annulus of width $\varepsilon$ towards a near-affine profile is of the order of $\varepsilon^2$. After this timescale, the velocity $u$ will have changed by about $\varepsilon$ in a layer of width $\sim \varepsilon$ around the interface; at the same time, due to the mostly parallel transport at the interface the solution will not have changed much otherwise.
As a consequence, the term $\int \frac{1}{2} \rho(\chi_u) |u-v|^2 \,dx$ will be of the order of at least $c\varepsilon^3$ after a time $T\sim\varepsilon^2$, while the other terms in the relative entropy are essentially the same. Thus, the relative entropy has grown by a factor of $1+c\varepsilon$ within a timescale $\varepsilon^2$, which prevents any Gronwall-type estimate.

At the level of the relative entropy inequality (see Proposition~\ref{PropositionRelativeEntropyInequalityFull}), the derivation of the Gronwall inequality is prevented by the viscosity terms, which read for $w\equiv 0$
\begin{align*}
&-\int \frac{\mu(\chi_u)}{2} \big|\nabla u+\nabla u^T-(\nabla v+\nabla v^T)\big|^2\,\dx
\\&
+\int (\mu(\chi_v)-\mu(\chi_u))\nabla v :\big(\nabla u+\nabla u^T-(\nabla v+\nabla v^T)\big) \,\dx.
\end{align*}
The latter term prevents the derivation of a dissipation estimate: While it is formally quadratic in the difference of the two solutions $(\chi_u,u)$ and $(\chi_v,v)$, due to the (expected) jump of the velocity gradients $\nabla v$ and $\nabla u$ at the respective interfaces it is in fact only linear in the interface error.

The key idea for our weak-strong uniqueness result in the case of different viscosities is to construct a vector field $w$ which is small in the $L^2$ norm but whose gradient compensates for most of the problematic term $(\mu(\chi_v)-\mu(\chi_u))(\nabla v+\nabla v^T)$. To be precise, it is only the normal derivative of the tangential component of $v$ which may be discontinuous at the interface; the tangential derivatives are continuous by the no-slip boundary condition, while the normal derivative of the normal component is continuous by the condition $\nabla \cdot v=0$.

Let us explain our construction of the vector field $w$ at the simple two-dimensional example of a planar interface of the strong solution $I_v=\{(x,0):x\in \mathbb{R}\}$. In this setting, we would like to set for $y>0$
\begin{align*}
w^+(x,y,t):=&c(\mu^+,\mu^-) \int_0^{y \wedge h^+(x)} (e_x \cdot \partial_y v) (x,\tilde y) e_x \,\mathrm{d}\tilde y
\end{align*}
(where $e_x$ just denotes the first vector of the standard basis).
Note that due to the bounded integrand, this vector field $w^+(x,y)$ is bounded by $C h^+(x)$, i.\,e.\ it is bounded by the interface error. As we shall see in the proof, the time derivative of $w^+$ is also bounded in terms of other error terms. The tangential spatial derivative of this vector field $\partial_x w^+(x,y,t)$ is given (up to a constant factor) by $\smash{\int_0^{y \wedge h^+(x)}} (e_x \cdot \partial_x \partial_y v) (x,\tilde y) e_x \,\mathrm{d}\tilde y
+\chi_{y\geq h^+(x)} (e_x \cdot \partial_y v) (x,h^+(x)) \partial_x h^+(x) e_x$ which is also a term controlled by $C h^+(x) + C |\partial_x h^+(x)|$. The normal derivative, on the other hand, is given by $\partial_y w^+(x,y,t)=c(\mu^+,\mu^-) \chi_{\{0\leq y\leq h^+(x)\}} (e_x \cdot \partial_y v) (x,y) e_x$ which (upon choosing $c(\mu^+,\mu^-)$) would precisely compensate the discontinuity of $\partial_y (e_x\cdot v)$ in the region in which the interface of the weak solution is a graph of a function over $I_v$. Note that our relative entropy functional provides a higher-order control of the size of the region in which the interface of the weak solution is not a graph over the interface of the strong solution.

However, with this choice of vector field $w^+(x,y,t)$, two problems occur: First, the vector field is not solenoidal. For this reason, we introduce an additional Helmholtz projection. Second -- and constituting a more severe problem -- , the vector field would not necessarily be (spatially) Lipschitz continuous (as the derivative contains a term with $\partial_x h^+(x)$ which is not necessarily bounded), which due to the surface tension terms would be required for the derivation of a Gronwall-type estimate. For this reason, we first regularize the height function $h^+$ by mollification on a scale of the order of the error. See Proposition~\ref{PropositionInterfaceErrorHeight} and Proposition \ref{PropositionInterfaceErrorHeightRegularized} for details of our construction of the regularized height function. The actual construction of our compensation function $w$ is performed in Proposition~\ref{PropositionCompensationFunction}. We then derive an estimate in the spirit of \eqref{WeakerThanGronwall} in Proposition~\ref{PropositionRelativeEntropyInequalityPostProcessed}.

\section{Time evolution of geometric quantities and further coercivity properties
of the relative entropy functional}

\subsection{Time evolution of the signed distance function}
In order to describe the time evolution of various constructions, we need to recall
some well-known properties of the signed distance function. Let us start by introducing
notation. For a family $(\Omega_t^+)_{t\in [0,\Tmax)}$ of smoothly evolving domains
with smoothly evolving interfaces $I(t)$ in the sense of Definition~\ref{definition:domains}, 
the associated signed distance function is given by
\begin{align}\label{defSignedDistance}
\sigdist(x,I(t)) := \begin{cases}
															\mathrm{dist}(x,I(t)), & x\in\Omega^+_t, \\
															-\mathrm{dist}(x,I(t)), & x\notin\Omega^+_t.
													 \end{cases}
\end{align} 
From Definition \ref{definition:domains} of a family of smoothly evolving domains it
follows that the family of maps $\Phi_t\colon I(t)\times (-r_c,r_c) \to \Rd$
given by $\Phi_t(x,y) := x+y\vec{n}(x,t)$ are $C^2$-diffeomorphisms onto their
image $\{x\in\R^d\colon \dist(x,I(t))<r_c\}$ subject to the bounds
\begin{align}
\label{BoundNablaPhiNablaPhi-1}
|\nabla \Phi_t|\leq C,
\quad\quad\quad
|\nabla \Phi_t^{-1}|\leq C.
\end{align}
The signed distance
function (resp.\ its time derivative) to the interface of the strong solution 
is then of class $C^0_tC^3_x$ (resp.\ $C^0_tC^2_x$) in the space-time tubular neighborhood 
$\bigcup_{t\in [0,\Tmax)}\mathrm{im}(\Phi_t)\times \{t\}$ due to the regularity
assumptions in Definition \ref{definition:domains}. We also have the bounds
\begin{align}\label{Bound2ndDerivativeSignedDistance}
|\nabla^{k+1}\sigdist(x,I(t))|\leq Cr_c^{-k},\quad, k=1,2,
\end{align}
and in particular for the mean curvature vector
\begin{align}\label{BoundMeanCurvature}
|\vec{H}| \leq Cr_c^{-1}.
\end{align}
Moreover, the projection $P_{I(t)}x$ of a point $x$ onto the nearest point on the manifold $I(t)$ is well-defined
and of class $C^0_tC^2_x$ in the same tubular neighborhood. 

After having introduced the necessary notation we study the time
evolution of the signed distance function to the interface of the strong solution. Because of the kinematic
condition that the interface is transported with the flow, we obtain the following statement.

\begin{lemma}\label{lem:evolSignedDist}
Let $\chi_v\in L^\infty([0,\Tmax);\BV(\mathbb{R}^d;\{0,1\}))$
be an indicator function such that $\Omega^+_t :=\{x\in\R^d\colon\chi_v(x,t)=1\}$ is a family 
of smoothly evolving domains and $I_v(t) := \partial\Omega^+_t$ is a family of
smoothly evolving surfaces in the sense of Definition~\ref{definition:domains}.
Let $v\in L^2_{loc}([0,\Tmax];H^1_{loc}(\Rd;\Rd))$ be a continuous solenoidal vector field 
such that $\chi_v$ solves the equation $\partial_t \chi_v = -\nabla \cdot(\chi_v v)$.
The time evolution of the signed distance function to the interface $I_v(t)$ is then given by
\begin{equation}\label{transportProblem}
\begin{aligned}
    \partial_t\sigdist(x,I_v(t)) = -\big(\bar{V}_{\vec{n}}(x,t)\cdot\nabla\big)\sigdist(x,I_v(t))
\end{aligned}
\end{equation}
for any $t\in [0,\Tmax]$ and any $x\in \Rd$ with $\dist(x,I_v(t))\leq r_c$, where $\bar{V}_{\vec{n}}$ is the extended normal velocity of the interface given by
\begin{align}\label{projNormalVel}
\bar{V}_{\vec{n}}(x,t) = \big(v(P_{I_v(t)}x,t)\cdot\vec{n}_v(P_{I_v(t)}x,t)\big)
\vec{n}_v(P_{I_v(t)}x,t).
\end{align}
Moreover, the following formulas hold true
\begin{align}
\nabla\sigdist(x,I_v(t)) &= \vec{n}_v(P_{I_v(t)}x,t), \label{sdistAux4}
\\
\nabla\sigdist(x,I_v(t))\cdot\partial_t\nabla\sigdist(x,I_v(t)) &= 0, \label{sdistAux1}
\\
\nabla\sigdist(x,I_v(t))\cdot\partial_j\nabla\sigdist(x,I_v(t)) &= 0, \quad j=1,\ldots,d, \label{sdistAux2}
\\
\partial_t\sigdist(x,I_v(t)) &= \big.\partial_t\sigdist(y,I_v(t))\big|_{y=P_{I_v(t)}x}, \label{sdistAux3}
\end{align}
for all $(x,t)$ such that $\dist(x,I_v(t))\leq r_c$. The gradient of the projection onto the
nearest point on the interface $I_v(t)$ is given by
\begin{align}\label{gradientProjection}
\nabla P_{I_v(t)}x = \mathrm{Id}-\vec{n}_v(P_{I_v(t)}x)\otimes\vec{n}_v(P_{I_v(t)}x)
-\sigdist(x,I_v(t))\nabla^2\sigdist(x,I_v(t)).
\end{align}
In particular, we have the bound
\begin{align}\label{BoundGradientProjection}
|\nabla P_{I_v(t)}x| \leq C
\end{align}
for all $(x,t)$ such that $\dist(x,I_v(t))\leq r_c$.
\end{lemma}
\begin{proof}
Recall that $\nabla\sigdist(x,I_v(t))$ for a point $x\in I_v(t)$ on the interface equals the inward pointing 
normal vector $\vec{n}_v(x,t)$ of the interface $I_v(t)$. This also extends away from the interface in 
the sense that
\begin{align}\label{gradSigDist}
\big.\nabla\sigdist(y,I_v(t))\big|_{y=P_{I_v(t)}x} = \vec{n}_v(P_{I_v(t)}x,t)
= \big.\nabla\sigdist(y,I_v(t))\big|_{y=x} 
\end{align}
for all $(x,t)$ such that $\dist(x,I_v(t))<r_c$, i.\,e.\ \eqref{sdistAux4} holds. Hence, we also have the formula
$P_{I_v(t)}x = x - \sigdist(x,I_v(t))\nabla\sigdist(x,I_v(t))$. Differentiating
this representation of the projection onto the interface 
and using the fact that $\vec{n}_v$ is a unit vector we obtain using also \eqref{SomeTimeDerivativeFormula}
\begin{align*}
&\big.\nabla\sigdist(y,I_v(t))\big|_{y=P_{I_v(t)}x}\cdot\partial_tP_{I_v(t)}x
\\&
=-\partial_t\sigdist(x,I_v(t))-\sigdist(x,I_v(t))\nabla\sigdist(P_{I_v(t)}x,I_v(t))
\partial_t\sigdist(x,I_v(t))
\\&
=-\partial_t\sigdist(x,I_v(t))-\sigdist(x,I_v(t))\partial_t\Big(\frac{1}{2}|\nabla\sigdist(x,I_v(t))|\Big)
\\&
=-\partial_t\sigdist(x,I_v(t)).
\end{align*}
Hence, we obtain in addition to \eqref{gradSigDist} the formula
\begin{align*}
\partial_t\sigdist(x,I_v(t)) = \big.\partial_t\sigdist(y,I_v(t))\big|_{y=P_{I_v(t)}x}.
\end{align*}
On the other side, on the interface the time derivative of the signed distance function equals 
up to a sign the normal speed. In our case, the latter is given by the normal component of the 
given velocity field $v$ evaluated on the interface, see Remark \ref{strongKinematic}. 
This concludes the proof of \eqref{transportProblem}.
Moreover, \eqref{sdistAux1} as well as \eqref{sdistAux2} follow immediately from differentiating 
$|\nabla\sigdist(x,I_v(t))|^2 = 1$. Finally, \eqref{gradientProjection} and \eqref{BoundGradientProjection}
follow immediately from \eqref{Bound2ndDerivativeSignedDistance} and 
$P_{I_v(t)}x=x-\sigdist(x,I_v(t))\vec{n}_v(P_{I_v(t)}x)$.

In the above considerations, we have made use of the following result:
Consider the auxiliary function $g(x,t)=\sigdist(P_{I_v(t)}x,I_v(t))$ for $(x,t)$ such that $\dist(x,I_v(t))<r_c$. Since this function vanishes on the space-time tubular neighborhood of the interface 
$\bigcup_{t\in (0,\Tmax)}\{x\in\Rd\colon\dist(x,I_v(t))<r_c\}\times\{t\}$ we compute
\begin{align}
\label{SomeTimeDerivativeFormula}
0 = \frac{\mathrm{d}}{\dt}g(x,t)= \big.\partial_t\sigdist(y,I_v(t))\big|_{y=P_{I_v(t)}x}
+ \big.\nabla\sigdist(y,I_v(t))\big|_{y=P_{I_v(t)}x}\cdot\partial_tP_{I_v(t)}x.
\end{align}
\end{proof}

\begin{remark}\label{regVectorField}
Consider the situation of Lemma \ref{lem:evolSignedDist}. We proved that
\begin{align*}
\partial_t\sigdist(x,I_v(t)) = -v(P_{I_v(t)}x,t)\cdot\vec{n}_v(P_{I_v(t)}x,t).
\end{align*}
The right hand side of this identity is of class $L^\infty_t W^{2,\infty}_x$, as the normal component $\vec{n}_v(P_{I_v(t)}) \cdot \nabla v$ of the velocity gradient $\nabla v$ of a strong solution is continuous across the interface $I_v(t)$. To see this, one first observes that the tangential derivatives $((\Id-\vec{n}_v(P_{I_v(t)}\otimes \vec{n}_v(P_{I_v(t)})\nabla) v$ are naturally continuous across the interface; one then uses the incompressibility constraint $\nabla \cdot v=0$ to deduce that $\vec{n}_v(P_{I_v(t)}\cdot (\vec{n}_v(P_{I_v(t)} \cdot \nabla)v$ is also continuous across the interface.
\end{remark}

\subsection{Properties of the vector field $\xi$}
The vector field $\xi$ -- as defined in Proposition~\ref{PropositionRelativeEntropyInequalityFull} -- is an extension of the unit normal
vector field $\vec{n}_v$ associated to the family of smoothly evolving domains 
occupying the first fluid of the strong solution. We now provide a more detailed account of its definition. The construction in fact
consists of two steps. First, we extend the normal vector field $\vec{n}_v$
to a (space-time) tubular neighborhood of the evolving interfaces $I_v(t)$
by projecting onto the interface. Second, we multiply this construction with a cutoff which decreases quadratically in the distance to the interface of the strong solution (see \eqref{controlByTiltExcess2}). 
\begin{definition}\label{def:ExtNormal}
Let $\chi_v\in L^\infty([0,\Tmax);\BV(\mathbb{R}^d;\{0,1\}))$
be an indicator function such that $\Omega^+_t :=\{x\in\R^d\colon\chi_v(x,t)=1\}$ is a family 
of smoothly evolving domains and $I_v(t) := \partial\Omega^+_t$ is a family of
smoothly evolving surfaces in the sense of Definition~\ref{definition:domains}.
Let $\eta$ be a smooth cutoff function with $\eta(s)= 1$ for $s\leq \frac{1}{2}$ and $\eta\equiv 0$ for $s\geq 1$.
Define another smooth cutoff function $\zeta\colon\R\to[0,\infty)$ as follows:
\begin{align}\label{quadCutOff}
			\zeta(r) = (1-r^2)\eta(r),\quad r\in [-1,1],
\end{align}
and $\zeta\equiv 0$ for $|r|>1$. Then, we define a vector field 
$\xi\colon\R^d\times [0,\Tmax)\to\R^d$ by
\begin{align}\label{CutOffNormal}
		\xi(x,t) := \begin{cases}
								\zeta\Big(\frac{\sigdist(x,I_v(t))}{r_c}\Big)\vec{n}_v(P_{I_v(t)}x,t) &
								\text{for }(x,t)\text{ with } \dist(x,I_v(t)) < r_c, \\
								0 & \text{else}.
								\end{cases}
\end{align}
\end{definition} 
The definition of $\xi$ has the following consequences.
\begin{remark}\label{regExtNormal}
Observe that the vector field $\xi$ is indeed well-defined in the space-time domain $\Rd\times [0,\Tmax)$
due to the action of the cut-off function $\zeta$; it also satisfies $|\xi|\leq 1$ or, more precisely, the sharper inequality $|\xi|\leq (1-\dist(x,I_v(t))^2)_+$. Furthermore, the extension $\vec{\xi}$ inherits its regularity
from the regularity of the signed distance function to the interface $I_v(t)$.
More precisely, it follows that the vector field $\xi$ (resp.\ its time derivative) is of class
$L^\infty_t W^{2,\infty}_x$ (resp.\ $W^{1,\infty}_t W^{1,\infty}_x$) globally in $\Rd\times [0,\Tmax)$, and the restrictions to the domains $\{\chi_v=0\}$ and $\{\chi_v=1\}$ are of class $L^\infty_t C^2_x$.
This turns out to be sufficient for our purposes.
\end{remark}
The time derivative of our vector field $\xi$ is given as follows.
\begin{lemma}\label{lem:evolExtNormal}
Let $\chi_v\in L^\infty([0,\Tmax);\BV(\mathbb{R}^d;\{0,1\}))$
be an indicator function such that $\Omega^+_t :=\{x\in\R^d\colon\chi_v(x,t)=1\}$ is a family 
of smoothly evolving domains and $I_v(t) := \partial\Omega^+_t$ is a family of
smoothly evolving surfaces in the sense of Definition~\ref{definition:domains}.
Let $v\in L^2_{loc}([0,\Tmax];H^1_{loc}(\Rd;\Rd))$ be a continuous solenoidal vector field 
such that $\chi_v$ solves the equation $\partial_t \chi_v = -\nabla \cdot(\chi_v v)$.
Let $\bar{V}_{\vec{n}}$ be the extended normal velocity of the interface \eqref{projNormalVel}.
Then the time evolution of the vector field $\xi$ from Definition~\ref{def:ExtNormal}
is given by
\begin{align}\label{evolExtNormal}
\partial_t\xi = -(\bar{V}_{\vec{n}}\cdot\nabla)\xi - 
\big(\mathrm{Id}{-}\vec{n}_v(P_{I_v(t)}x)\otimes\vec{n}_v(P_{I_v(t)}x)\big)(\nabla\bar{V}_{\vec{n}})^T \xi
\end{align}
in the space-time domain $\dist(x,I_v(t))<r_c$. Here, we made use of the abbreviation
$\vec{n}_v(P_{I_v(t)}x)=\vec{n}_v(P_{I_v(t)}x,t)$. 
\end{lemma}

\begin{proof}
We start by deriving a formula for the time evolution of the normal
vector field $\vec{n}_v(P_{I_v(t)}x,t)$ in the space-time tubular neighborhood
$\dist(x,I_v(t))<r_c$. By \eqref{sdistAux4}, we may use
the formula for the time evolution of the signed distance function
from Lemma~\ref{lem:evolSignedDist}. More precisely, due to the
regularity of the signed distance function to the interface of the strong
solution and the regularity of the vector field $\bar{V}$ (Remark~\ref{regVectorField}),
we can interchange the differentiation in time and space to obtain
\begin{align*}
\partial_t\nabla\sigdist(x,I_v(t)) &= \nabla\partial_t\sigdist(x,I_v(t)) \\
&\stackrel{\eqref{transportProblem}}{=} -\nabla\big((\bar{V}_{\vec{n}}\cdot\nabla)\sigdist(x,I_v(t))\big) \\
&~= -(\bar{V}_{\vec{n}}\cdot\nabla)\vec{n}_v(P_{I_v(t)}x) - (\nabla\bar{V}_{\vec{n}})^T\cdot\vec{n}_v(P_{I_v(t)}x).
\end{align*}
Next, we show that the normal-normal component of $\nabla\bar{V}_{\vec{n}}$ vanishes.
Observe that by Remark~\ref{regVectorField} and \eqref{sdistAux4} it holds
\begin{align*}
\bar{V}_{\vec{n}}(x,t) = -\partial_t\sigdist(x,I_v(t))\nabla\sigdist(x,I_v(t)).
\end{align*}
Hence, by \eqref{sdistAux4}--\eqref{sdistAux3} and this formula we obtain
\begin{align*}
&(\nabla\bar{V}_{\vec{n}})^T(x,t):\vec{n}_v(P_{I_v(t)}x)\otimes\vec{n}_v(P_{I_v(t)}x) 
\\&~~~
=\nabla\bar{V}_{\vec{n}}(x,t)\nabla\sigdist(x,I_v(t))\cdot\nabla\sigdist(x,I_v(t)) 
\\&~~~
= -\nabla\sigdist(x,I_v(t))\cdot\partial_t\nabla\sigdist(x,I_v(t)) 
\\&~~~~~~
+\bar{V}_{\vec{n}}(x,t)\otimes\nabla\sigdist(x,I_v(t)):\nabla^2\sigdist(x,I_v(t)) 
\\&~~~
= 0
\end{align*}
as desired. In summary, we have proved so far that
\begin{align}\label{TimeEvolutionNormalNoCutoff}
\partial_t\vec{n}_v(P_{I_v(t)}x) &= -(\bar{V}_{\vec{n}}\cdot\nabla)\vec{n}_v(P_{I_v(t)}x) 
\\&~~~\nonumber
- \big(\mathrm{Id}{-}\vec{n}_v(P_{I_v(t)}x)\otimes\vec{n}_v(P_{I_v(t)}x)\big)
(\nabla\bar{V}_{\vec{n}})^T\cdot\vec{n}_v(P_{I_v(t)}x),
\end{align}
which holds in the space-time domain $\dist(x,I_v(t))<r_c$. However, applying the 
chain rule to the cut-off function $r\mapsto\zeta(r)$ from \eqref{quadCutOff}
together with the evolution equation \eqref{transportProblem} for the signed distance 
to the interface shows that the cut-off away from the interface is also
subject to a transport equation:
\begin{align*}
\partial_t\zeta\Big(\frac{\sigdist(x,I_v(t))}{r_c}\Big) = -(\bar{V}_{\vec{n}}(x,t)\cdot\nabla)
\zeta\Big(\frac{\sigdist(x,I_v(t))}{r_c}\Big).
\end{align*}
By the definition of the vector field $\xi$, see \eqref{CutOffNormal}, and the product rule, this
concludes the proof.
\end{proof}

\subsection{Properties of the weighted volume term}

We next discuss the weighted volume contribution $\int_\Rd |\chi_u-\chi_v| \dist(x,I_v(t)) \,dx$ to the relative entropy in more detail.

\begin{remark}\label{regWeigthVol}
Let $\beta$ be a truncation of the identity as in Proposition~\ref{PropositionRelativeEntropyInequalityFull}.
Let $\chi_v\in L^\infty([0,\Tmax);\BV(\mathbb{R}^d;\{0,1\}))$
be an indicator function such that $\Omega^+_t :=\{x\in\R^d\colon\chi_v(x,t)=1\}$ is a family 
of smoothly evolving domains, and $I_v(t) := \partial\Omega^+_t$ is a family of
smoothly evolving surfaces, in the sense of Definition~\ref{definition:domains}.
The map $$\Rd\times[0,\Tmax)\ni(x,t)\mapsto\beta\big(\sigdist(x,I_v(t))/r_c\big)$$
inherits the regularity of the signed distance function to the interface 
$I_v(t)$. More precisely, this map (resp.\ its time derivative) is
of class $C^0_tC^3_x$ (resp.\ $C^0_tC^2_x$). 
\end{remark}

\begin{lemma}\label{lem:evolWeightVol}
Let $\chi_v\in L^\infty([0,\Tmax);\BV(\mathbb{R}^d;\{0,1\}))$
be an indicator function such that $\Omega^+_t :=\{x\in\R^d\colon\chi_v(x,t)=1\}$ is a family 
of smoothly evolving domains and $I_v(t) := \partial\Omega^+_t$ is a family of
smoothly evolving surfaces in the sense of Definition~\ref{definition:domains}.
Let $v\in L^2_{loc}([0,\Tmax];H^1_{loc}(\Rd;\Rd))$ be a continuous solenoidal vector field 
such that $\chi_v$ solves the equation $\partial_t \chi_v = -\nabla \cdot(\chi_v v)$.
Let $\bar{V}_{\vec{n}}$ be the extended normal velocity of the interface \eqref{projNormalVel}.
Then the time evolution of the weight function $\beta$ composed with the signed distance
function to the interface $I_v(t)$ is given by the transport equation
\begin{equation}\label{transportWeight}
\begin{aligned}
    \partial_t\beta\Big(\frac{\sigdist(\cdot,I_v)}{r_c}\Big) 
		= -\big(\bar{V}_{\vec{n}}\cdot\nabla\big)\beta\Big(\frac{\sigdist(\cdot,I_v)}{r_c}\Big)
\end{aligned}
\end{equation}
for space-time points $(x,t)$ such that $\dist(x,I_v(t))<r_c$.
\end{lemma}

\begin{proof}
This is immediate from the chain rule and the time evolution of
the signed distance function to the interface of the strong solution, 
see Lemma~\ref{lem:evolSignedDist}.
\end{proof}

\subsection{Further coercivity properties of the relative entropy}
We collect some further coercivity properties of the relative entropy functional
$E\big[\chi_u,u,V\big|\chi_v,v\big]$ as defined in \eqref{DefinitionRelativeEntropyFunctional}.
These will be of frequent use in the post-processing of the terms occurring 
on the right hand side of the relative entropy inequality from 
Proposition~\ref{PropositionRelativeEntropyInequalityFull}. We start 
for reference purposes with trivial consequences
of our choices of the vector field $\xi$ and the weight function $\beta$.

\begin{lemma}
Consider the situation of Proposition~\ref{PropositionRelativeEntropyInequalityFull}.
In particular, let $\beta$ be the truncation of the identity from Proposition~\ref{PropositionRelativeEntropyInequalityFull}.
By definition, it holds
\begin{align}\label{coercivityWeight}
\min\Big\{\frac{\dist(x,I_v(t))}{r_c},1\Big\} \leq 
\Big|\beta\Big(\frac{\sigdist(x,I_v(t))}{r_c}\Big)\Big|.
\end{align}
Let $\xi$ be the vector field from Definition~\ref{def:ExtNormal} with
cutoff multiplier $\zeta$ as given in \eqref{quadCutOff}. By the choice of
the cutoff $\zeta$, it holds
\begin{align}\label{controlByTiltExcess2}
\frac{|\sigdist(x,I_v(t))|^2}{r_c^2} \leq 1-\zeta\Big(\frac{\sigdist(x,I_v(t))}{r_c}\Big).
\end{align}
We will also make frequent use of the fact that for any unit vector $\vec{b}\in\Rd$ we have
\begin{align}\label{controlByTiltExcess}
1-\zeta\Big(\frac{\sigdist(x,I_v(t))}{r_c}\Big) \leq 1-\vec{b}\cdot\xi\quad\text{ and } \quad
|\vec{b}-\xi|^2 \leq 2(1-\vec{b}\cdot\xi).
\end{align}
\end{lemma}

We also want to emphasize that the relative entropy functional
controls the squared error in the normal of the varifold. 

\begin{lemma}
Consider the situation of Proposition~\ref{PropositionRelativeEntropyInequalityFull}.
We then have
\begin{align}\label{controlSquaredErrorNormalVarifold}
\int_{\Rd\times\Sb^{d-1}}\frac{1}{2}|s-\xi|^2|\,\mathrm{d}V_t(x,s)
\leq E\big[\chi_u,u,V\big|\chi_v,v\big](t)
\end{align}
for almost every $t\in [0,\Tmax)$.
\end{lemma}
\begin{proof}
Observe first that by means of the compatibility condition \eqref{varifoldComp} we have
\begin{align*}
\int_{\Rd\times\Sb^{d-1}}(1-s\cdot\vec{\xi}\,)\,\mathrm{d}V_t(x,s) &=
\int_{\Rd\times\Sb^{d-1}}1\,\mathrm{d}V_t(x,s) - \int_{\Rd}\vec{n}_u\cdot\vec{\xi}\,\mathrm{d}|\nabla\chi_u(\cdot,t)| \\
&=\int_{\Rd}1\,\mathrm{d}|V_t|_{\Sb^{d-1}} - \int_{\Rd}\vec{n}_u\cdot\vec{\xi}\,\mathrm{d}|\nabla\chi_u(\cdot,t)|,
\end{align*}
which holds for almost every $t\in [0,\Tmax)$. In addition, due to \eqref{RadonNikodym} one obtains
\begin{align*}
\int_{\Rd}1-\theta_t\,\mathrm{d}|V_t|_{\Sb^{d-1}}
= \int_{\Rd}1\,\mathrm{d}|V_t|_{\Sb^{d-1}} - \int_{\Rd}1\,\mathrm{d}|\nabla\chi_u(\cdot,t)|
\end{align*}
for almost every $t\in [0,\Tmax)$. This in turn entails the following identity
\begin{align*}
&\int_{\Rd}\big(1-\vec{n}_u\cdot\vec{\xi}\big)\,\mathrm{d}|\nabla\chi^u|  +
\int_{\Rd}1-\theta_t\,\mathrm{d}|V_t|_{\Sb^{d-1}} \nonumber\\
&=\int_{\Rd\times\Sb^{d-1}}(1-s\cdot\vec{\xi}\,)\,\mathrm{d}V_t(x,s),
\end{align*}
which holds true for almost every $t\in [0,\Tmax)$. However, the functional on
the right hand side controls the squared error in the normal of the varifold:
$|s-\xi|^2\leq 2(1-s\cdot\xi)$. This proves the claim.
\end{proof}

We will also refer multiple times to the following bound. 
In the regime of equal shear viscosities $\mu_+=\mu_-$ we may
apply this result with the choice $w=0$. In the general case,
we have to include the compensation function $w$ for the
velocity gradient discontinuity at the interface.

\begin{lemma}\label{lemmaBoundRMixed}
Let $(\chi_u,u,V)$ be a varifold solution to \eqref{EquationTransport}--\eqref{EquationIncompressibility} 
in the sense of Definition~\ref{DefinitionVarifoldSolution} on a time interval $[0,\Tend)$
with initial data $(\chi_u^0,u_0)$. Let $(\chi_v,v)$ be a strong solution to 
\eqref{EquationTransport}--\eqref{EquationIncompressibility} in the sense of 
Definition~\ref{DefinitionStrongSolution} on a time interval $[0,\Tmax)$ with $\Tmax\leq \Tend$
and initial data $(\chi_v^0,v_0)$. Let $w\in L^2([0,\Tmax);H^1(\Rd;\Rd))$ be an arbitrary
vector field, and let $F\in L^\infty(\Rd\times[0,\Tmax);\Rd)$ be a bounded vector field. 
Then
\begin{align*}
&\bigg|\int_0^T\int_{\Rd}(\chi_u{-}\chi_v)(u-v-w)\cdot F\,\dx\,\dt\bigg|
\\&
\leq
\delta\int_0^T\int_{\Rd}|\nabla(u-v-w)|^2\,\dx\,\dt
+C \frac{1+\|F\|_{L^\infty}^2}{\delta}\int_0^T\int_{\Rd}\rho(\chi_u)|u{-}v{-}w|^2\,\dx\,\dt
\\&~~~
+ \frac{C\|F\|_{L^\infty}}{\delta}\int_0^T\int_{\Rd}|\chi_u{-}\chi_v|
\Big|\beta\Big(\frac{\sigdist(\cdot,I_v)}{r_c}\Big)\Big|\,\dx\,\dt
\end{align*}
for almost every $T\in [0,\Tmax)$ and all $0<\delta\leq 1$. The absolute constant $C>0$ only depends on the densities $\rho_{\pm}$.
\end{lemma}

\begin{proof}
We first argue how to control the part away from the interface of the strong solution, i.e.,
outside of $\{(x,t)\colon\dist(x,I_v(t))\geq r_c\}$.
A straightforward estimate using H\"{o}lder's and Young's inequality yields
\begin{align*}
\bigg|\int_0^T&\int_{\{\dist(x,I_v(t))\geq r_c\}}
(\chi_u{-}\chi_v)(u{-}v{-}w)\cdot F\,\dx\,\dt\bigg| \\
&\leq\frac{\|F\|_{L^\infty}}{2}\int_0^T\int_{\{\dist(x,I_v(t))\geq r_c\}}
|\chi_u-\chi_v|\,\dx\,\dt \\ &~~~~+
\frac{\|F\|_{L^\infty}}{2}\int_0^T\int_{\{\dist(x,I_v(t))\geq r_c\}}
|u-v-w|^2\,\dx\,\dt.
\end{align*} 
Note that by the properties of the truncation of the identity $\beta$, see Proposition~\ref{PropositionRelativeEntropyInequalityFull}, 
it follows that $|\beta(\sigdist(x,I_v(t))/r_c)|\equiv 1$ on 
$\{(x,t)\colon\dist(x,I_v(t))\geq r_c\}$. Hence, we obtain
\begin{equation}
\begin{aligned}\label{eq:boundAwayInterface}
\bigg|\int_0^T&\int_{\{\dist(x,I_v(t))\geq r_c\}}
(\chi_u{-}\chi_v)(u{-}v{-}w)\cdot F\,\dx\,\dt\bigg| \\
&\leq\frac{\|F\|_{L^\infty}}{2}\int_0^T\int_{\R^d}
|\chi_u-\chi_v|\cdot\Big|\beta\Big(\frac{\sigdist(\cdot,I_v)}{r_c}\Big)\Big|
\,\dx\,\dt \\ &~~~~+ \frac{\|F\|_{L^\infty}}{2(\rho_+\wedge\rho_-)}
\int_0^T\int_{\R^d}\rho(\chi_u)|u-v-w|^2\,\dx\,\dt,
\end{aligned}
\end{equation}
which is indeed a bound of required order.

We proceed with the bound for the contribution in the vicinity of the interface of
the strong solution. To this end, recall that we are equipped with a family of maps 
$\Phi_t\colon I_v(t)\times (-r_c,r_c) \to \Rd$ given by $\Phi_t(x,y) := x+y\vec{n}_v(x,t)$, 
which are $C^2$-diffeomorphisms onto their image $\{x\in\R^d\colon \dist(x,I_v(t))<r_c\}$. Recall the estimates \eqref{BoundNablaPhiNablaPhi-1}.
We then move on with a change of variables, the one-dimensional Gagliardo-Nirenberg-Sobolev interpolation inequality
\begin{align*}
\|g\|_{L^\infty(-r_c,r_c)}\leq C\|g\|_{L^2(-r_c,r_c)}^\frac{1}{2}\|\nabla g\|_{L^2(-r_c,r_c)}^\frac{1}{2}
+C\|g\|_{L^2(-r_c,r_c)}
\end{align*}
as well as H\"{o}lder's and Young's inequality to obtain the bound
\begin{align*}
&\bigg|\int_0^T\int_{\{\dist(x,I_v(t))<r_c\}}
(\chi_u-\chi_v)(u-v-w)\cdot F\,\dx\,\dt\bigg| 
\\&
\leq C\|F\|_{L^\infty}\int_0^T\int_{I_v(t)}\int_{-r_c}^{r_c}
|(\chi_u{-}\chi_v)|(\Phi_t(x,y))\,|(u{-}v{-}w)|(\Phi_t(x,y))\,\mathrm{d}y\,\mathrm{d}S(x)\,\mathrm{d}t
\\&
\leq C\|F\|_{L^\infty}\int_0^T\int_{I_v(t)}\sup_{y\in [-r_c,r_c]}
|u-v-w|(x+y\vec{n}_v(x,t))
\\&\hspace{3.5cm}\times\bigg(\int_{-r_c}^{r_c}|(\chi_u{-}\chi_v)|(x+y\vec{n}_v(x,t))
\,\mathrm{d}y\bigg)\,\mathrm{d}S(x)\,\mathrm{d}t
\\&
\leq C\frac{\|F\|_{L^\infty}+\|F\|_{L^\infty}^2}{\delta}\int_0^T\int_{\Rd}|u-v-w|^2\,\dx\,\dt
+\delta\int_0^T\int_\Rd |\nabla(u-v-w)|^2\,\dx\,\dt
\\&~~~
+C\|F\|_{L^\infty}\int_0^T\int_{I_v(t)}\bigg(\int_{-r_c}^{r_c}
|(\chi_u{-}\chi_v)|(x+y\vec{n}_v(x,t))\,\mathrm{d}y\bigg)^2\,\mathrm{d}S(x)\,\mathrm{d}t.
\end{align*}
It thus suffices to derive an estimate for the $L^2$-norm of
the local interface error height in normal direction
\begin{align*}
		h(x) = \int_{-r_c}^{r_c}|(\chi_u{-}\chi_v)|(x+y\vec{n}_v(x,t))\,\mathrm{d}y.
\end{align*}
The proof of Proposition~\ref{PropositionInterfaceErrorHeight} below, 
where we establish next to the required $L^2$-bound also several other properties 
of the local interface error height, shows that (see \eqref{L2InterfaceErrorHeightNoCutoff})
\begin{align}\label{L2hNoCutoff}
\int_{I_v(t)} |h(x)|^2\,\dS \leq
C\int_\Rd |\chi_u{-}\chi_v| \min\Big\{\frac{\dist(x,I_v(t))}{r_c},1\Big\} \,\dx.
\end{align}
This then concludes the proof.
\end{proof}

We conclude this section with an $L^2_\mathrm{tan}L^\infty_\mathrm{nor}$-bound for $H^1$-functions
on the tubular neighborhood around the evolving interfaces
as well as a bound for the derivatives of
the normal velocity of the interface of a strong solution in terms of the associated
velocity field $v$, both of which will be used several times in the post-processing
of the terms on the right hand side of the relative entropy inequality of
Proposition~\ref{PropositionRelativeEntropyInequalityFull}.

\begin{lemma}
\label{LemmaEstimateL2SupByH1}
Consider the situation of Proposition~\ref{PropositionRelativeEntropyInequalityFull}.
We have the estimate
\begin{align}\label{BoundL2Linfty}
&\int_{I_v(t)} \sup_{y\in [-r_c,r_c]} |g(x+y\vec{n}_v(x,t))|^2 \,\mathrm{d}S 
\leq C(\|g\|_{L^2}\|\nabla g\|_{L^2}+\|g\|_{L^2}^2)
\end{align}
valid for any $g\in H^1(\Rd)$.
\end{lemma}
\begin{proof}
Let $f\in H^1(-r_c,r_c)$. The one-dimensional Gagliardo-Nirenberg-Sobolev inerpolation inequality then implies
\begin{align*}
\|f\|_{L^\infty(-r_c,r_c)}\leq C\|f\|_{L^2(-r_c,r_c)}^\frac{1}{2}\| f'\|_{L^2(-r_c,r_c)}^\frac{1}{2}
+C\|f\|_{L^2(-r_c,r_c)}.
\end{align*}
From this we obtain together with H\"older's inequality
\begin{align*}
&\int_{I_v(t)} \sup_{y\in [-r_c,r_c]} |g(x+y\vec{n}_v(x,t))|^2 \,\mathrm{d}S
\\&
\leq C\int_{I_v(t)} \int_{-r_c}^{r_c}|g(x{+}y\vec{n}_v(x,t))|^2\,\dy \,\mathrm{d}S 
\\&~~~
+C\bigg(\int_{I_v(t)} \int_{-r_c}^{r_c}| g(x{+}y\vec{n}_v(x,t))|^2\,\dy \,\mathrm{d}S\bigg)^\frac{1}{2}
\\&~~~~~~~~
\times
\bigg(\int_{I_v(t)} \int_{-r_c}^{r_c}|\nabla g(x{+}y\vec{n}_v(x,t))|^2\,\dy \,\mathrm{d}S\bigg)^\frac{1}{2}.
\end{align*}
This implies \eqref{BoundL2Linfty} by making use of the 
$C^2$-diffeomorphisms $\Phi_t\colon I_v(t)\times (-r_c,r_c) \to \Rd$ 
given by $\Phi_t(x,y) = x+y\vec{n}_v(x,t)$ and the associated change of variables, using also the bound \eqref{BoundNablaPhiNablaPhi-1}.
\end{proof}

\begin{lemma}
Consider the situation of Proposition~\ref{PropositionRelativeEntropyInequalityFull} and
define the vector field
\begin{align*}
V_{\vec{n}}(x,t) &:= \big(v(x,t)\cdot\vec{n}_v(P_{I_v(t)}x,t)\big)\vec{n}_v(P_{I_v(t)}x,t),
\end{align*}
for $(x,t)\in\Rd\times[0,\Tmax)$ such that $\dist(x,I_v(t))<r_c$. Then
\begin{align}\label{Bound1stDerivativeNormalVelocity}
\|\nabla V_{\vec{n}}\|_{L^\infty(\mathcal{O})} &\leq Cr_c^{-1}\|v\|_{L^\infty}
+C\|\nabla v\|_{L^\infty}, \\
\label{BoundDerivativesNormalVelocity}
\|\nabla^2V_{\vec{n}}\|_{L^\infty(\mathcal{O})} &\leq Cr_c^{-2}\|v\|_{L^\infty}
+Cr_c^{-1}\|\nabla v\|_{L^\infty}+C\|\nabla^2 v\|_{L^\infty_tL^\infty_x(\Rd\setminus I_v(t))},
\end{align}
where $\mathcal{O}=\bigcup_{t\in (0,\Tmax)}\{x\in\Rd\colon\dist(x,I_v(t))<r_c\}\times\{t\}$
denotes the space-time tubular neighborhood of width $r_c$ of the evolving interface of 
the strong solution.

In particular, we have for $\bar V_{\vec{n}}(x,t):=V_{\vec{n}}(P_{I_v(t)}x,t)$ the estimate
\begin{align}
\label{EstimateBarVV}
|\bar V_{\vec{n}}(x,t)-V_{\vec{n}}(x,t)| \leq C r_c^{-1} ||v||_{W^{1,\infty}} \dist(x,I_v(t)).
\end{align}
\end{lemma}

\begin{proof}
The estimates \eqref{Bound1stDerivativeNormalVelocity} and \eqref{BoundDerivativesNormalVelocity} 
are a direct consequence of the regularity requirements on the velocity field $v$ of a
strong solution, see Definition~\ref{DefinitionStrongSolution}, the pointwise bounds 
\eqref{Bound2ndDerivativeSignedDistance} and the representation
of the normal vector field on the interface in terms of the signed distance function
\eqref{sdistAux4}.
\end{proof}

\section{Weak-strong uniqueness of varifold solutions to two-fluid Navier-Stokes flow: The case of equal viscosities}
\label{MainResultEqualViscosities}

In this section we provide a proof of the weak-strong uniqueness principle
to the free boundary problem for the incompressible Navier-Stokes equation for two fluids 
\eqref{EquationTransport}--\eqref{EquationIncompressibility} in the case of equal
shear viscosities $\mu_+=\mu_-$. Note that in this case the problematic viscous stress term $R_{visc}$ in the relative entropy inequality (see Proposition~\ref{PropositionRelativeEntropyInequalityFull}) vanishes because
of $\mu(\chi_u)-\mu(\chi_v) = 0$. In this setting, it is possible to choose $w\equiv 0$ which directly implies $A_{visc}=0$, $A_{adv}=0$, $A_{dt}=0$, $A_{weightVol}=0$, and $A_{surTen}=0$. What remains to be done is a post-processing of the terms $R_{surTen}$, $R_{adv}$, $R_{dt}$, and $R_{weightVol}$ which remain on the right-hand side of the relative entropy inequality. 

\subsection{Estimate for the surface tension terms}
We start by post-processing the terms related to surface tension $R_{surTen}$.
\begin{lemma}
Consider the situation of Proposition~\ref{PropositionRelativeEntropyInequalityFull}.
The terms related to surface tension $R_{surTen}$ are estimated by
\begin{align}
\nonumber
R_{surTen}
&\leq
\delta \int_0^T \int_\Rd |\nabla (u-v-w)|^2 \,dx\,dt
\\&~~~\label{RsurTenEqualVisc}
+C(\delta) r_c^{-4}\big(1+\|v\|_{L^\infty_tW^{2,\infty}_x}^2\big)
\int_0^TE[\chi_u,u,V|\chi_v,v](t)\,\dt
\end{align}
for any $\delta>0$.
\end{lemma}

\begin{proof}
We start by using \eqref{controlByTiltExcess} and \eqref{CutOffNormal} to estimate
\begin{align}\label{auxSurTenGronwall1}
\nonumber
&-\sigma\int_0^T\int_{\Rd}
\Big(\xi\cdot\frac{\nabla\chi_u}{|\nabla\chi_u|}\Big)\,
\vec{n}_v(P_{I_v(t)}x)\cdot\big(\vec{n}_v(P_{I_v(t)}x)\cdot\nabla\big)v-\xi\cdot(\xi\cdot\nabla) v
\,\mathrm{d}|\nabla\chi_u|\,\dt
\\&\nonumber
= \sigma\int_0^T\int_{\Rd}
\Big(1-\xi\cdot\frac{\nabla\chi_u}{|\nabla\chi_u|}\Big)\,
\vec{n}_v(P_{I_v(t)}x)\cdot\big(\vec{n}_v(P_{I_v(t)}x)\cdot\nabla\big)v
\,\mathrm{d}|\nabla\chi_u|\,\dt 
\\&~~~\nonumber
+\sigma\int_0^T\int_{\Rd}
\xi\cdot(\xi\cdot\nabla) v - \vec{n}_v(P_{I_v(t)}x)
\cdot\big(\vec{n}_v(P_{I_v(t)}x)\cdot\nabla\big)v
\,\mathrm{d}|\nabla\chi_u|\,\dt
\\&\nonumber
\leq C\|\nabla v \|_{L^\infty}\int_0^T\int_{\Rd}
1-\xi\cdot\frac{\nabla\chi_u}{|\nabla\chi_u|}
\,\mathrm{d}|\nabla\chi_u|\,\dt 
\\&~~~\nonumber
+C\|\nabla v \|_{L^\infty}\int_0^T\int_\Rd
1-\zeta\Big(\frac{\sigdist(x,I_v(t))}{r_c}\Big)
\,\mathrm{d}|\nabla\chi_u|\,\dt 
\\&
\leq C\|v\|_{L^\infty_tW^{1,\infty}_x}\int_0^TE[\chi_u,u,V|\chi_v,v](t)\,\dt.
\end{align}
Recall from \eqref{controlSquaredErrorNormalVarifold} that the squared
error in the varifold normal is controlled by the relative entropy functional.
Together with the bound from Lemma~\ref{lemmaBoundRMixed}, \eqref{Bound2ndDerivativeSignedDistance}
as well as \eqref{auxSurTenGronwall1} we get an estimate for the first four terms of $R_{surTen}$
\begin{align}
\label{FirstEstimateRsurf}
&R_{surTen}
\\&
\nonumber
\leq C(\delta)r_c^{-4}(1+\|v\|_{L^\infty_tW^{1,\infty}_x})
\int_0^TE[\chi_u,u,V|\chi_v,v](t)\,\dt
\\&~~~
\nonumber
+\frac{\delta}{2}\int_0^T\int_\Rd |\nabla(u-v-w)|^2\,\dx\,\dt
\\&~~~
\nonumber
+\sigma\int_0^T\int_{\Rd}\frac{\nabla\chi_u}{|\nabla\chi_u|}
\cdot\big((\mathrm{Id}{-}\vec{n}_v(P_{I_v(t)}x)\otimes\vec{n}_v(P_{I_v(t)}x))
(\nabla\bar{V}_{\vec{n}}{-}\nabla v)^T \xi\big)
\,\mathrm{d}|\nabla\chi_u|\,\dt
\\&~~~
\nonumber
+\sigma\int_0^T\int_{\Rd}\frac{\nabla\chi_u}{|\nabla\chi_u|}
\cdot\big((\bar{V}_{\vec{n}}-v)\cdot\nabla\big)\xi\,\mathrm{d}|\nabla\chi_u|\,\dt
\end{align}
for almost every $T\in [0,\Tmax)$ and all $\delta\in (0,1]$. 
To estimate the remaining two terms we decompose $\bar{V}_{\vec{n}}-v$ as 
\begin{align}\label{decompX}
\bar{V}_{\vec{n}} - v = (\bar{V}_{\vec{n}} - V_{\vec{n}}) + (V_{\vec{n}} - v),
\end{align}
where the vector field $V_{\vec{n}}$ is given by
\begin{align}
V_{\vec{n}}(x,t) &:= \big(v(x,t)\cdot\vec{n}_v(P_{I_v(t)}x,t)\big)\vec{n}_v(P_{I_v(t)}x,t) \label{normalCompVel} 
\end{align}
in the space-time domain $\{\dist(x,I_v(t))<r_c\}$ (i.\,e.\ in contrast to $\vec V_{vec{n}}$, for $V_{\vec{n}}$ the velocity $v$ is evaluated not at the projection of $x$ onto the interface, but at $x$ itself). Note that it will not matter as to how $V_{\vec{n}}$ and similar quantities are defined outside of the area $\{\dist(x,I_v(t))<r_c\}$, as the terms will always be multiplied by suitable cutoffs which vanish outside of $\{\dist(x,I_v(t))<r_c\}$. In the next two steps, we compute and bound
the contributions from the two different parts in the decomposition \eqref{decompX} of
the error $\bar{V}_{\vec{n}} - v$. 

\subsubsection*{First step: Controlling the error $V_{\vec{n}} - v$}
By definition of the vector field $V_{\vec{n}}$ in \eqref{normalCompVel},
we may write $V_{\vec{n}} - v = -\big(\mathrm{Id}-\vec{n}_v(P_{I_v(t)}x)\otimes\vec{n}_v(P_{I_v(t)}x)\big)v$.
It is then not clear why the term
\begin{align*}
&\sigma\int_0^T\int_{\Rd}\frac{\nabla\chi_u}{|\nabla\chi_u|}
\cdot\big((\mathrm{Id}{-}\vec{n}_v(P_{I_v(t)}x)\otimes\vec{n}_v(P_{I_v(t)}x))
(\nabla V_{\vec{n}}{-}\nabla v)^T \xi\big)
\,\mathrm{d}|\nabla\chi_u|\,\dt
\\&
+\sigma\int_0^T\int_{\Rd}\frac{\nabla\chi_u}{|\nabla\chi_u|}
\cdot\big((V_{\vec{n}}-v)\cdot\nabla\big)\xi\,\mathrm{d}|\nabla\chi_u|\,\dt
\end{align*}
should be controlled by our relative entropy functional. However, the integrands enjoy 
a crucial cancellation
\begin{align}\label{tangentialCancel}
(\mathrm{Id}{-}\vec{n}_v(P_{I_v(t)}x)\otimes\vec{n}_v(P_{I_v(t)}x))
(\nabla V_{\vec{n}}{-}\nabla v)^T \xi + \big((V_{\vec{n}}-v)\cdot\nabla\big)\xi = 0
\end{align}
in the space-time domain $\{(x,t)\in\R^d\times [0,\Tmax):\dist(x,I_v(t))<r_c\}$.
To verify this cancellation, we first recall from \eqref{sdistAux4} that
$\nabla\sigdist(x,I_v(t)) = \vec{n}_v(P_{I_v(t)}x,t)$. We then start by rewriting
\begin{align*}
\big((V_{\vec{n}}-v)\cdot\nabla\big)\xi = 
-\nabla\xi~(\mathrm{Id}-\nabla\sigdist(\cdot,I_v)\otimes\nabla\sigdist(\cdot,I_v))v.
\end{align*}  
Note that when the derivative hits the cutoff multiplier in the definition of $\xi$ (see \eqref{CutOffNormal}),
the resulting term on the right hand side of the last identity vanishes. Hence, we obtain
together with \eqref{sdistAux2}
\begin{align*}
&\big((V_{\vec{n}}-v)\cdot\nabla\big)\xi 
\\&
= -\zeta\big(r_c^{-1}\sigdist(\cdot,I_v)\big)\big(\nabla^2\sigdist(\cdot,I_v)\big)
(\mathrm{Id}-\nabla\sigdist(\cdot,I_v)\otimes\nabla\sigdist(\cdot,I_v))v
\\&
= -\zeta\big(r_c^{-1}\sigdist(\cdot,I_v)\big)\big(\nabla^2\sigdist(\cdot,I_v)\big)v.
\end{align*}
On the other side, another application of \eqref{sdistAux2} yields
\begin{align*}
&(\nabla V_{\vec{n}}{-}\nabla v)^T\xi 
\\&
=-(\nabla v)^T\big(\mathrm{Id}{-}\vec{n}_v(P_{I_v(t)}x)\otimes\vec{n}_v(P_{I_v(t)}x)\big)\xi
+\zeta\big(r_c^{-1}\sigdist(\cdot,I_v)\big)\big(\nabla^2\sigdist(\cdot,I_v)\big)v
\\&
=\zeta\big(r_c^{-1}\sigdist(\cdot,I_v)\big)\big(\nabla^2\sigdist(\cdot,I_v)\big)v.
\end{align*}
Therefore, the desired cancellation \eqref{tangentialCancel} indeed holds true since by \eqref{sdistAux2}
the right-hand side of the last computation remains unchanged after projecting via 
$\mathrm{Id}-\vec{n}_v\otimes\vec{n}_v$.

\subsubsection*{Second step: Controlling the error $\bar{V}_{\vec{n}}-V_{\vec{n}}$} It remains to control
the contributions from the following two quantities:
\begin{align*}
I &:= \int_0^T\int_{\Rd}\vec{n}_u\cdot
\big(\big(\mathrm{Id}{-}\vec{n}_v(P_{I_v(t)}x)\otimes\vec{n}_v(P_{I_v(t)}x)\big)
(\nabla\bar{V}_{\vec{n}}{-}\nabla V_{\vec{n}})^T \xi\big)
\,\mathrm{d}|\nabla\chi_u|\,\dt,
\\
II &:= \int_0^T\int_{\Rd}\vec{n}_u\cdot
\big((\bar{V}_{\vec{n}}-V_{\vec{n}})\cdot\nabla\big)\xi\,\mathrm{d}|\nabla\chi_u|\,\dt.
\end{align*}
Note first that we can write
\begin{align*}
I = \int_0^T\int_{\Rd}(\vec{n}_u-\xi)\cdot
\big(\big(\mathrm{Id}{-}\vec{n}_v(P_{I_v(t)}x)\otimes\vec{n}_v(P_{I_v(t)}x)\big)
(\nabla\bar{V}_{\vec{n}}{-}\nabla V_{\vec{n}})^T\xi\big)
\,\mathrm{d}|\nabla\chi_u|\,\dt.
\end{align*}
Moreover, recall from \eqref{gradientProjection} the formula for the gradient of the 
projection onto the nearest point on the interface $I_v(t)$.
The definition of $V_{\vec n}$ (see \eqref{normalCompVel}) and $\bar V_{\vec n}(x)=V_{vec n}(P_{I_v(t)}x)$, the product rule, \eqref{sdistAux4}, \eqref{Bound2ndDerivativeSignedDistance}, and \eqref{sdistAux2} imply using the definition of $\xi$ and the property $|\xi|\leq 1$
\begin{align*}
&\Big|\big(\mathrm{Id}{-}\vec{n}_v(P_{I_v(t)}x)\otimes\vec{n}_v(P_{I_v(t)}x)\big)
(\nabla\bar{V}_{\vec{n}}{-}\nabla V_{\vec{n}})^T \xi\Big|
\\&
\leq \Big|\big(\mathrm{Id}{-}\vec{n}_v(P_{I_v(t)}x)\otimes\vec{n}_v(P_{I_v(t)}x)\big)(\nabla (v(P_{I_v(t)}x)) - \nabla v(x))^T \Big|
\\&~~~
+\Big|\big(\mathrm{Id}{-}\vec{n}_v(P_{I_v(t)}x)\otimes\vec{n}_v(P_{I_v(t)}x)\big)(\nabla (\vec{n}_v(P_{I_v(t)}x))^T (v(P_{I_v(t)}x)-v(x)) \Big|
\\&~~~
+||v||_{L^\infty} \Big|(\nabla (\vec{n}_v(P_{I_v(t)}x)))^T \xi\Big|
\\&
\leq C r_c^{-1} ||v||_{W^{2,\infty}(\Rd\setminus I_v(t))} \dist(x,I_v(t))
\end{align*}
where in the last step we have used also \eqref{gradientProjection}.
Together
Young's inequality
and the coercivity properties of the relative entropy \eqref{controlByTiltExcess2} and \eqref{controlByTiltExcess}
we then immediately get the estimate
\begin{align}\label{proError1}
\nonumber
I &\leq 
 C\int_0^T\int_\Rd |\vec{n}_u - \xi|^2\,\mathrm{d}|\nabla\chi_u|\,\dt
\\&~~~\nonumber
+Cr_c^{-4}\|v\|^2_{L^\infty_tW^{2,\infty}_x(\Rd\setminus I_v(t))}
\int_0^T\int_\Rd |\dist(x,I_v(t))|^2\,\mathrm{d}|\nabla\chi_u|\,\dt
\\&
\leq C(1+r_c^{-4}\|v\|^2_{L^\infty_tW^{2,\infty}_x(\Rd\setminus I_v(t))})
\int_0^TE[\chi_u,u,V|\chi_v,v](t)\,\dt.
\end{align}
To estimate the second term $II$, we start by adding zero and then use again
$\bar{V}_{\vec{n}}(x,t) = V_{\vec{n}}(P_{I_v(t)}x,t)$, \eqref{Bound1stDerivativeNormalVelocity}, 
\eqref{Bound2ndDerivativeSignedDistance} as well as \eqref{controlByTiltExcess2} and \eqref{controlByTiltExcess}
\begin{align*}
II &= \int_0^T\int_{\Rd}(\vec{n}_u-\xi)\cdot
\big((\bar{V}_{\vec{n}}-V_{\vec{n}})\cdot\nabla\big)\xi\,\mathrm{d}|\nabla\chi_u|\,\dt
\\&~~~
+\int_0^T\int_{\Rd}\xi\cdot
\big((\bar{V}_{\vec{n}}-V_{\vec{n}})\cdot\nabla\big)\xi\,\mathrm{d}|\nabla\chi_u|\,\dt
\\&
\leq C(1+r_c^{-2}\|v\|^2_{L^\infty_tW^{1,\infty}_x})
\int_0^TE[\chi_u,u,V|\chi_v,v](t)\,\dt
\\&~~~
+\int_0^T\int_{\Rd}\xi\cdot
\big((\bar{V}_{\vec{n}}-V_{\vec{n}})\cdot\nabla\big)\xi\,\mathrm{d}|\nabla\chi_u|\,\dt.
\end{align*}
Using \eqref{sdistAux2}, we continue by computing
\begin{align*}
&\int_0^T\int_{\Rd}\xi\cdot\big((\bar{V}_{\vec{n}}-V_{\vec{n}})
\cdot\nabla\big)\xi\,\mathrm{d}|\nabla\chi_u|\,\dt
\\&
=r_c^{-1}\int_0^T\int_{\Rd}\zeta'\Big(\frac{\sigdist(x,I_v(t))}{r_c}\Big)
\xi\otimes(\bar{V}_{\vec{n}}-V_{\vec{n}}):\vec{n}_v(P_{I_v(t)})\otimes\vec{n}_v(P_{I_v(t)})
\,\mathrm{d}|\nabla\chi_u|\,\dt
\end{align*}
Hence, it follows from $\zeta'(0)=0$ and $|\zeta''|\leq C$ as well as \eqref{EstimateBarVV}
that
\begin{align}\label{proError2}
\nonumber
II &\leq C(1+r_c^{-2}\|v\|^2_{L^\infty_tW^{1,\infty}_x})
\int_0^TE[\chi_u,u,V|\chi_v,v](t)\,\dt
\\&~~~\nonumber
+Cr_c^{-3}\|v\|_{L^\infty_tW^{1,\infty}_x}\int_0^T\int_\Rd
|\dist(x,I_v(t))|^2\,\mathrm{d}|\nabla\chi_u|\,\dt
\\&
\leq C(1 + r_c^{-3}\|v\|^2_{L^\infty_tW^{1,\infty}_x})
\int_0^TE[\chi_u,u,V|\chi_v,v](t)\,\dt.
\end{align}

\subsubsection*{Third step: Summary} Inserting \eqref{tangentialCancel},
\eqref{proError1}, and \eqref{proError2} into \eqref{FirstEstimateRsurf} entails the bound
\begin{align*}
&R_{surTen} 
\\&
\leq \frac{C(\delta)}{r_c^4}
(1{+}\|v\|_{L^\infty_tW^{2,\infty}_x(\Rd\setminus I_v(t))}\vee
\|v\|_{L^\infty_tW^{2,\infty}_x(\Rd\setminus I_v(t))}^2)
\int_0^TE[\chi_u,u,V|\chi_v,v](t)\,\dt
\\&~~~
+\delta\int_0^T\int_\Rd |\nabla(u-v-w)|^2\,\dx\,\dt.
\end{align*}
This yields the desired estimate.
\end{proof}

\subsection{Estimate for the remaining terms $R_{adv}$, $R_{dt}$, and $R_{weightVol}$} To bound the advection-related terms
\begin{align*}
R_{adv} = &-\int_0^T\int_\Rd\big(\rho(\chi_u)-\rho(\chi_v)\big)
(u-v-w)\cdot(v\cdot\nabla)v\,\dx\,\dt 
\\&
-\int_0^T\int_\Rd \rho(\chi_u)(u-v-w)\cdot
\big((u-v-w)\cdot\nabla\big)v\,\dx\,\dt
\end{align*}
from the relative entropy inequality, the time-derivative related terms $R_{dt}$, and the terms resulting from the weighted volume control term in the relative entropy
\begin{align*}
R_{weightVol} := &\int_0^T\int_\Rd(\chi_u{-}\chi_v)
\big(\big(\bar{V}_{\vec{n}}{-}V_{\vec{n}}\big)\cdot\nabla\big)
\beta\Big(\frac{\sigdist(\cdot,I_v)}{r_c}\Big)\,\dx\,\dt
\\&
+\int_0^T\int_{\R^d}(\chi_u{-}\chi_v)\big((u{-}v{-}w)\cdot\nabla\big)
\beta\Big(\frac{\sigdist(\cdot,I_v)}{r_c}\Big)\,\dx\,\dt
\end{align*}
(with $V_{\vec{n}}(x,t):=(\vec{n}_v(P_{I_v(t)}x,t) \otimes \vec{n}_v(P_{I_v(t)}x,t)) v(x,t)$),
we use mostly straightforward estimates.
\begin{lemma}
Consider the situation of Proposition~\ref{PropositionRelativeEntropyInequalityFull}. The terms $R_{adv}$, $R_{dt}$, and $R_{weightVol}$ are subject to the bounds
\begin{align}\label{RadvEqualVisc}
R_{adv} &\leq C(\delta)(1+\|v\|_{L^\infty_tW^{1,\infty}_x}^4)
\int_0^TE[\chi_u,u,V|\chi_v,v](t)\,\dt
\\&~~~~\nonumber
+\delta \int_0^T\int_\Rd |\nabla (u-v-w)|^2\,\dx\,\dt,
\end{align}
\begin{align}\label{RdtEqualVisc}
R_{dt}
\leq \delta \int_0^T\int_\Rd |\nabla (u-v-w)|^2\,\dx\,\dt
+C(\delta)\|\partial_t v\|_{L^\infty_{x,t}}\int_0^T E[\chi_u,u,V|\chi_v,v](t)\,\dt,
\end{align}
and
\begin{align}
\label{RweightVolEqualVisc}
R_{weightVol} &\leq
\delta \int_0^T\int_\Rd |\nabla (u-v-w)|^2\,\dx\,\dt
\\&~~~~\nonumber
+C(\delta)r_c^{-2}(1+\|v\|_{L^\infty_tW^{1,\infty}_x})\int_0^T E[\chi_u,u,V|\chi_v,v](t)\,\dt
\end{align}
for any $\delta>0$.
\end{lemma}
\begin{proof}
To derive \eqref{RadvEqualVisc}, we use a direct estimate for the second term in $R_{adv}$ as well as Lemma~\ref{lemmaBoundRMixed} for the first term.

The bound \eqref{RdtEqualVisc} is derived similarly.

Finally, we show estimate \eqref{RweightVolEqualVisc}.  Note that by definition we have $\bar{V}_{\vec{n}}(x,t)=V_{\vec{n}}(P_{I_v(t)}x,t)$. Hence, we obtain using the bound \eqref{EstimateBarVV} as well as \eqref{coercivityWeight} and $|\beta'|\leq C$
\begin{align*}
R_{weightVol} &\leq C\|v\|_{L^\infty_tW^{1,\infty}_x}\int_0^T\int_\Rd|\chi_u{-}\chi_v|\,
\big|\beta\Big(\frac{\sigdist(x,I_v(t))}{r_c}\Big)\Big|\,\dx\,\dt
\\&~~~
+C r_c^{-1} \int_0^T\int_{\{\dist(x,I_v(t))\leq r_c\}}|\chi_u -\chi_v| |u-v-w| \,\dx\,\dt.
\end{align*}
An application of Lemma~\ref{lemmaBoundRMixed} yields \eqref{RweightVolEqualVisc}.
\end{proof}

\subsection{The weak-strong uniqueness principle in the case of equal viscosities}
We conclude our discussion of the case of equal shear viscosities $\mu_+=\mu_-$ for
the free boundary problem for the incompressible Navier-Stokes equation 
for two fluids \eqref{EquationTransport}--\eqref{EquationIncompressibility}
with the proof of the weak-strong uniqueness principle. 

\begin{proposition}
\label{PropositionWeakStrongUniquenessEqualViscosities}
Let $d\leq 3$. Let $(\chi_u,u,V)$ be a varifold solution to the free boundary problem for 
the incompressible Navier-Stokes equation for two fluids \eqref{EquationTransport}--\eqref{EquationIncompressibility} 
in the sense of Definition~\ref{DefinitionVarifoldSolution} on some time interval $[0,\Tend)$
with initial data $(\chi_u^0,u_0)$. Let $(\chi_v,v)$ be a strong solution to 
\eqref{EquationTransport}--\eqref{EquationIncompressibility} in the sense of 
Definition~\ref{DefinitionStrongSolution} on some time interval $[0,\Tmax)$ with $\Tmax\leq \Tend$
and initial data $(\chi_v^0,v_0)$. We assume that the shear viscosities of the two fluids coincide, i.e., $\mu^+=\mu^-$.

Then, there exists a constant $C>0$ which only depends on the data of the strong solution such that
the stability estimate
\begin{align*}
E[\chi_u,u,V|\chi_v,v](T) \leq E[\chi_u,u,V|\chi_v,v](0)e^{CT}
\end{align*}
holds. In particular, if the initial data of the varifold solution and the strong solution coincide, 
the varifold solution must be equal to the strong solution in the sense
\begin{align*}
\chi_u(\cdot,t)=\chi_v(\cdot,t) &&\text{and}&& u(\cdot,t)=v(\cdot,t)
\end{align*}
almost everywhere for almost every $t\in [0,\Tmax)$. Furthermore, in this case the varifold is given by
\begin{align*}
\mathrm{d}V_t = \delta_{\frac{\nabla \chi_v}{|\nabla \chi_v|}} \mathrm{d}|\nabla \chi_v|
\end{align*}
for almost every $t\in [0,\Tmax)$.
\end{proposition}

\begin{proof}
Applying the relative entropy inequality from 
Proposition~\ref{PropositionRelativeEntropyInequalityFull} with $w=0$,
using the fact that the problematic term $R_{visc}$ vanishes in the
case of equal shear viscosities $\mu_+=\mu_-$, as well as using the bounds 
from \eqref{RsurTenEqualVisc}, \eqref{RadvEqualVisc},
\eqref{RdtEqualVisc} and \eqref{RweightVolEqualVisc},
we observe that we established the following bound
\begin{align}\label{GronwallEqualVisc}
&E[\chi_u,u,V|\chi_v,v](T) + c \int_0^T\int_\Rd |\nabla u- \nabla v|^2\,\dx\,\dt
\\&
\nonumber
\leq E[\chi_u,u,V|\chi_v,v](0)
+\delta \int_0^T \int_\Rd |\nabla u-\nabla v|^2 \dx\,\dt
\\&~~~
\nonumber
+ \frac{C(\delta)}{r_c^4}(1{+}\|\partial_tv\|_{L^\infty_{x,t}}
{+}\|v\|_{L^\infty_tW^{2,\infty}_x(\Rd\setminus I_v(t))}
\vee\|v\|_{L^\infty_tW^{2,\infty}_x(\Rd\setminus I_v(t))}^2)
\\&~~~~\times
\nonumber
\int_0^TE[\chi_u,u,V|\chi_v,v](t)\,\dt
\end{align}
for almost every $T\in [0,\Tmax)$. An absorption argument along with a subsequent application of Gronwall's lemma then 
immediately yields the asserted stability estimate.

Consider the case of coinciding initial conditions, i.e., $E[\chi_u,u,V|\chi_v,v](0)=0$.
In this case, we deduce from the stability estimate that the relative entropy
vanishes for almost every $t\in [0,\Tmax)$.
From this it immediately follows that $u(\cdot,t)=v(\cdot,t)$ as well as 
$\chi_u(\cdot,t)=\chi_v(\cdot,t)$ almost everywhere for almost every $t\in [0,\Tmax)$. 

The asserted representation of the varifold $V$ of the varifold solution follows
from the following considerations. First, we deduce $|\nabla\chi_u(\cdot,t)|=|V_t|_{\Sb^{d-1}}$ 
for almost every $t\in [0,\Tmax)$ as a consequence of the fact that 
the density of the varifold satisfies $\theta_t=\smash{\frac{\mathrm{d}|\nabla\chi_u(\cdot,t)|}{\mathrm{d}|V_t|_{\Sb^{d-1}}}}\equiv 1$ 
almost everywhere for almost every $t\in [0,\Tmax)$.
The remaining fact that the measure on $\Sd$ is given by $\smash{\delta_{\vec{n}_u(x,t)}}$
for $\smash{|V_t|_{\Sb^{d-1}}}$-almost every $x\in\Rd$ for almost every $t\in [0,\Tmax)$ then follows 
from the control of the squared error in the normal of the varifold by the relative
entropy functional, see \eqref{controlSquaredErrorNormalVarifold}. This concludes the proof.
\end{proof}

\section{Weak-strong uniqueness of varifold solutions to two-fluid Navier-Stokes flow: The case of different viscosities}

We turn to the derivation of the weak-strong uniqueness principle in the case of different shear viscosities of the two fluids. In this regime, we cannot anymore ignore the viscous stress term $(\mu(\chi_v)-\mu(\chi_u))(\nabla v+\nabla v^T)$.
The key idea is to construct a solenoidal vector field $w$ which is small in the $L^2$-norm but whose gradient compensates for most of this problematic term, and then use the relative entropy inequality from Proposition~\ref{PropositionRelativeEntropyInequalityFull} with this function. The precise definition as well as a list of all the relevant properties of this vector field are the content of Proposition~\ref{PropositionCompensationFunction}. 

A main ingredient for the construction of $w$ are the local interface error heights as measured in orthogonal direction from the interface of the strong solution (see Figure~\ref{FigureIntefaceError}). For this reason, we first prove the relevant properties of the local heights of the interface error in Proposition~\ref{PropositionInterfaceErrorHeight}. However, in order to control certain surface-tension terms in the relative entropy inequality, we actually need the vector field $w$ to have bounded spatial derivatives. To this aim, we perform an additional regularization of the height functions. This will be carried out in detail in Proposition~\ref{PropositionInterfaceErrorHeightRegularized} by a (time-dependent) mollification. 
After all these preparations, in Section~\ref{SectionEstimateAsurTen}--\ref{SectionEstimateAweightVol} we then perform the post-processing of the additional terms $A_{visc}$, $A_{dt}$, $A_{adv}$, and $A_{surTen}$ in the relative entropy inequality from Proposition~\ref{PropositionRelativeEntropyInequalityFull}. Based on these bounds, in Section~\ref{SectionProofResult} we finally provide the proof of the stability estimate and the weak-strong uniqueness principle for varifold solutions to the free boundary problem for  the incompressible Navier-Stokes equation for two fluids \eqref{EquationTransport}--\eqref{EquationIncompressibility} from Theorem~\ref{weakStrongUniq}.

\subsection{The evolution of the local height of the interface error}
Consider a strong solution $(\chi_v,v)$ to the free boundary problem for 
the incompressible Navier-Stokes equation for two fluids 
\eqref{EquationTransport}--\eqref{EquationIncompressibility} in the sense of 
Definition~\ref{DefinitionStrongSolution} on some time interval $[0,\Tmax)$.
For the sake of better readability, let us recall some definitions and constructions
related to the associated family of evolving interfaces $I_v(t)$ of the strong solution.

For the family $(\Omega_t^+)_{t\in [0,\Tmax)}$ of smoothly evolving domains
of the strong solution, the associated signed distance function is given by
\begin{align*}
\sigdist(x,I_v(t)) = \begin{cases}
															\mathrm{dist}(x,I_v(t)), & x\in\Omega^+_t, \\
															-\mathrm{dist}(x,I_v(t)), & x\notin\Omega^+_t.
													 \end{cases}
\end{align*} 
From Definition \ref{definition:domains} of a family of smoothly evolving domains it
follows that the family of maps $\Phi_t\colon I_v(t)\times (-r_c,r_c) \to \Rd$
given by $\Phi_t(x,y) := x+y\vec{n}_v(x,t)$ are $C^2$-diffeomorphisms onto their
image $\{x\in\R^d\colon \dist(x,I_v(t))<r_c\}$. Here, $\vec{n}_v(\cdot,t)$ denotes the
normal vector field of the interface $I_v(t)$ pointing inwards $\{x\in\Rd\colon \chi_v(x,t)=1\}$.
The signed distance function (resp.\ its time derivative) to the interface $I_v(t)$ of the strong solution 
is then of class $C^0_tC^3_x$ (resp.\ $C^0_tC^2_x$) in the space-time tubular neighborhood 
$\bigcup_{t\in [0,\Tmax)}\mathrm{im}(\Phi_t)\times \{t\}$ due to the regularity
assumptions in Definition \ref{definition:domains}. Moreover, the projection $P_{I_v(t)}x$
of a point $x$ onto the nearest point on the manifold $I_v(t)$ is well-defined
and of class $C^0_tC^2_x$ in the same tubular neighborhood. Observe that the
inverse of $\Phi_t$ is given by is given by $\Phi_t^{-1}(x) = (P_{I_v(t)}x,\sigdist(x,I_v(t)))$
for all $x\in\Rd$ such that $\dist(x,I_v(t))<r_c$.

In Lemma \ref{lem:evolSignedDist}, we computed the time evolution of the signed distance function
to the interface $I_v(t)$ of a strong solution. Recall also the various relations 
for the projected inner unit normal vector field $\vec{n}_v(P_{I_v(t)}x,t)$ from Lemma \ref{lem:evolSignedDist}, 
which will be of frequent use in subsequent computations. Finally, we remind the
reader of the definition of the vector field $\xi$ from Definition \ref{def:ExtNormal}, 
which is a global extension of the inner unit normal vector field of the interface $I_v(t)$.

\begin{proposition}
\label{PropositionInterfaceErrorHeight}
Let $\chi_v\in L^\infty([0,\Tmax);\BV(\mathbb{R}^d;\{0,1\}))$
be an indicator function such that $\Omega^+_t :=\{x\in\R^d\colon\chi_v(x,t)=1\}$ is a family 
of smoothly evolving domains and $I_v(t) := \partial\Omega^+_t$ is a family of
smoothly evolving surfaces in the sense of Definition~\ref{definition:domains}.
Let $\xi$ be the extension of the unit normal vector field $\vec{n}_v$
from Definition~\ref{def:ExtNormal}.

Let $\theta:[0,\infty)\rightarrow [0,1]$ be a smooth cutoff 
with $\theta\equiv 0$ outside of $[0,\frac{1}{2}]$ and $\theta\equiv 1$ in $[0,\frac{1}{4}]$. 
For an indicator function $\chi_u\in L^\infty([0,\Tmax];\BV(\mathbb{R}^d;\{0,1\}))$ and $t\geq 0$, 
we define the local height of the one-sided interface error $h^+(\cdot,t):I_v(t)\rightarrow \mathbb{R}_0^+$ as
\begin{align}
\label{DefinitionHeightFunction}
h^+(x,t):=
\int_0^\infty (1-\chi_u)(x+y\vec{n}_v(x,t),t) \, \theta\Big(\frac{y}{r_c}\Big) \,\dy.
\end{align}
Similarly, we introduce the local height of the interface error in the other direction
\begin{align*}
h^-(x,t):=
\int_0^\infty \chi_u(x-y\vec{n}_v(x,t),t) \theta\Big(\frac{y}{r_c}\Big) \,\dy.
\end{align*}
Then $h^+$ and $h^-$ have the following properties:

\begin{subequations}
\noindent{\bf a) ($L^2$-bound)} We have the estimates $|h^\pm(x,t)|\leq \frac{r_c}{2}$ and
\begin{align}\label{HeightFunctionEstimate}
\int_{I_v(t)} |h^\pm(x,t)|^2\,\dS(x)
\leq C\int_\Rd |\chi_u{-}\chi_v| \min\Big\{\frac{\dist(x,I_v(t))}{r_c},1\Big\} \,\dx.
\end{align}

\noindent{\bf b) ($H^1$-bound)} Moreover, the estimate holds
\begin{align}
\label{HeightFunctionGradientEstimate}
&\int_{I_v(t)} \min\{|\nabla^{\tan} h^\pm (x,t)|^2,|\nabla^{\tan} h^\pm(x,t)|\} \,\dS  + |D^s h^\pm|(I_v(t))
\\&~
\nonumber
\leq
C \int_\Rd 1-\xi\cdot \frac{\nabla\chi_u}{|\nabla\chi_u|} 
\,\mathrm{d}|\nabla\chi_u|+\frac{C}{r_c^2} \int_\Rd |\chi_u{-}\chi_v| \min\Big\{\frac{\dist(x,I_v(t))}{r_c},1\Big\} \,\dx.
\end{align}

\noindent{\bf c) (Approximation property)} The functions $h^+$ and $h^-$ provide an approximation of the set 
$\{\chi_u=1\}$ in terms of a subgraph over the set $I_v(t)$ by setting
\begin{align*}
\chi_{v,h^+,h^-}:=\chi_v
- \chi_{0\leq \sigdist(x,I_v(t))\leq h^+(P_{I_v(t)}x,t)}
+ \chi_{-h^-(P_{I_v(t)}x,t)\leq \sigdist(x,I_v(t))\leq 0},
\end{align*}
up to an error of
\begin{align}
\label{ErrorImprovedInterfaceApproximationNonsmooth}
\nonumber
&\int_\Rd \big|\chi_u-\chi_{v,h^+,h^-}\big| \,\dx
\\&
\leq C \int_\Rd 1-\xi\cdot \frac{\nabla\chi_u}{|\nabla\chi_u|} \,\mathrm{d}|\nabla\chi_u|
+ C \int_\Rd |\chi_u-\chi_v| \min\Big\{\frac{\dist(x,I_v(t))}{r_c},1\Big\} \,\dx.
\end{align}

\noindent{\bf d) (Time evolution)} Let $v$ be a solenoidal vector field
\begin{align*}
v\in L^2([0,\Tmax];H^1(\Rd;\Rd))\cap L^\infty([0,\Tmax];W^{1,\infty}(\Rd;\Rd))
\end{align*}
such that in the domain $\bigcup_{t\in [0,\Tmax)} (\Omega_t^+\cup \Omega_t^-) \times \{t\}$ the second
spatial derivatives of the vector field $v$ exist and satisfy
$\sup_{t\in [0,\Tmax)} \sup_{x\in \Omega_t^+\cup \Omega_t^-} |\nabla^2 v(x,t)| <\infty$.
Assume that $\chi_v$ solves the equation 
$\partial_t \chi_v = -\nabla \cdot(\chi_v v)$. If $\chi_u$ 
solves the equation $\partial_t \chi_u = -\nabla \cdot(\chi_u u)$ 
for another solenoidal vector field $u\in L^2([0,\Tmax];H^1(\mathbb{R}^d;\mathbb{R}^d))$, 
we have the following estimate on the time derivative of the local interface
error heights $h^\pm$:
\begin{align}
\label{HeightFunctiondtEstimate}
&\bigg|\frac{\mathrm{d}}{dt} \int_{I_v(t)} \eta(x) h^\pm(x,t) \,\dS(x) -
\int_{I_v(t)} h^\pm(x,t) (\mathrm{Id}{-}\vec{n}_v\otimes\vec{n}_v)v(x,t) \cdot \nabla \eta(x) \,\dS(x)\bigg|
\\&\nonumber
\leq 
\frac{C}{r_c^2}\|\eta\|_{W^{1,4}(I_v(t))}\Bigg(\int_{I_v(t)} |\bar{h}^\pm|^4 \,\dS \Bigg)^{1/4}
\\&~~~~~~~~~~~~~
\nonumber
\times
\Bigg(\int_{I_v(t)}  \sup_{y\in [-r_c,r_c]} |u-v|^2(x+y\vec{n}_v(x,t),t) \,\dS(x)\Bigg)^{1/2} 
\\&~~~~
\nonumber
+C\frac{1+\|v\|_{W^{2,\infty}(\Rd\setminus I_v(t))}}{r_c^3}\|\eta\|_{L^2(I_v(t))} 
\\&~~~~~~~~~~~~~\nonumber
\times
\Bigg(\int_{\Rd}|\chi_u(x,t)-\chi_v(x,t)| 
\,\min\Big\{\frac{\dist(x,I_v(t))}{r_c},1\Big\} \,\dx\Bigg)^\frac{1}{2}
\\&~~~~\nonumber
+\frac{C(1+\|v\|_{W^{1,\infty}})}{r_c^2}\max_{p\in\{2,4\}}\|\eta\|_{W^{1,p}(I_v(t))}
\int_{\Rd} 1-\xi\cdot \frac{\nabla\chi_u}{|\nabla\chi_u|} \,\mathrm{d}|\nabla\chi_u|
\\&~~~~
\nonumber
+C\|\eta\|_{L^2(I_v(t))}\bigg(\int_{I_v(t)} |u-v|^2 \,\dS \bigg)^{1/2}
\end{align}
for any test function $\eta\in C_{cpt}^\infty(\Rd)$ with $\vec{n}_v\cdot \nabla \eta=0$ on the interface $I_v(t)$, 
and where $\bar{h}^\pm$ is defined as $h^\pm$
but now with respect to the modified cut-off $\bar{\theta}(\cdot)=\theta\big(\frac{\cdot}{2}\big)$.
\end{subequations}
\end{proposition}

\begin{proof}
{\bf Step 1: Proof of the estimate on the $L^2$-norm.}
The trivial estimate $|h^\pm(x,t)|\leq\frac{r_c}{2}$ follows directly from the definition of $h^\pm$.
To establish the $L^2$-estimate, let
$\ell^+(x):=\int_0^{r_c}(1-\chi_u)(x+y\vec{n}_v(x,t),t)\,\dy$. A straighforward estimate
then gives 
\begin{align}\label{L2InterfaceErrorHeightNoCutoff}
|\ell^+(x)|^2 = 2\int_0^{\ell^+(x)}y\,\dy 
\leq C\int_0^{r_c}|\chi_u(\Phi_t(x,y),t)-\chi_v(\Phi_t(x,y),t)|\frac{y}{r_c}\,\dy.
\end{align}
Note that the term on the left hand side dominates $|h^+|^2$ since we dropped the cutoff function. 
Hence, the desired estimate on the $L^2$-norm of $h^+$ follows at once by a change of variables
and recalling the fact that $\dist(\Phi_t(x,y),I_v(t))=y$. The corresponding bound
for $h^-$ then follows along the same lines.

{\bf Step 2: Proof of the estimate on the spatial derivative \eqref{HeightFunctionGradientEstimate}.}
The definition \eqref{DefinitionHeightFunction} is equivalent to
\begin{align*}
h^+(\Phi_t(x,0),t)=\int_0^\infty (1-\chi_u)(\Phi_t(x,y)) \, \theta\Big(\frac{y}{r_c}\Big) \,\dy.
\end{align*}
We compute for any smooth vector field $\eta\in C^\infty_{cpt}(\Rd;\Rd)$ 
(recall that $\Phi_t(x,0)=x$ and $\dist(\Phi_t(x,y),I_v(t))=y$ 
for any $x\in I_v(t)$ and any $y$ with $|y|\leq r_c$)
\begin{align*}
&\int_{I_v(t)} \eta(x) \cdot \mathrm{d} (D_x^{\tan} h^+(\cdot,t))(x)
\\&
=-\int_{I_v(t)} h^+(x,t) \nabla^{\tan} \cdot \eta(x) \,\dS(x)
- \int_{I_v(t)} h^+(x,t) \eta(x) \cdot \vec{H}(x,t) \,\dS(x)
\\&
=-\int_0^{r_c} \int_{I_v(t)} (1-\chi_u)(\Phi_t(x,y),t) 
\theta\Big(\frac{y}{r_c}\Big) \nabla^{\tan} \cdot \eta(x)  \,\dS(x) \,\dy
\\&~~~~
-\int_0^{r_c} \int_{I_v(t)}(1-\chi_u)(\Phi_t(x,y),t) 
\theta\Big(\frac{y}{r_c}\Big)  \eta(x) \cdot \vec{H}(\Phi_t(x,0),t) \,\dS(x) \,\dy
\\&
=-\int_{\Rd} (1-\chi_u)(x,t) \theta\Big(\frac{\dist(x,I_v(t))}{r_c}\Big) |\det \nabla\Phi_t^{-1}(x)| 
\\&~~~~~~~~~~~~~~~~~~~
\times
(\Id{-}\vec{n}_v(P_{I_v(t)}x)\otimes\vec{n}_v(P_{I_v(t)}x)):\nabla \eta(P_{I_v(t)}x) \,\dx
\\&~~~~
-\int_{\Rd} (1-\chi_u)(x,t) \theta\Big(\frac{\dist(x,I_v(t))}{r_c}\Big)  
\eta(P_{I_v(t)}x) \cdot \vec{H}(P_{I_v(t)}x) |\det \nabla\Phi_t^{-1}(x)| \,\dx
\\&
=-\int_{\Rd} \theta\Big(\frac{\dist(x,I_v(t))}{r_c}\Big)  |\det \nabla\Phi_t^{-1}(x)| 
\eta(P_{I_v(t)}x)(\Id{-}\vec{n}_v(P_{I_v(t)}x)\otimes\vec{n}_v(P_{I_v(t)}x)) \cdot \,\mathrm{d} \nabla\chi_u
\\&~~~~
+\int_{\Rd} (1-\chi_u)(x,t) \theta\bigg(\frac{\dist(x,I_v(t))}{r_c}\bigg) \eta(P_{I_v(t)}x)
\\&~~~~~~~~~~
\cdot \Big(\nabla \cdot \big( (\Id{-}\vec{n}_v(P_{I_v(t)}x)\otimes\vec{n}_v(P_{I_v(t)}x))|\det \nabla\Phi_t^{-1}|\big)
-\vec{H}(P_{I_v(t)}x) |\det \nabla\Phi_t^{-1}|\Big) \,\dx,
\end{align*}
where in the last step we have used $\nabla\sigdist(x,I_v(t))=\vec{n}_v(P_{I_v(t)}x)$.
This yields by another change of variables in the second integral, the fact that 
$\chi_v(\Phi_t(x,y),t)=1$ for any $y>0$, \eqref{Bound2ndDerivativeSignedDistance}, 
\eqref{BoundMeanCurvature}, $|\det \nabla\Phi_t^{-1}|\leq C$ 
as well as by abbreviating $\vec{n}_u = \frac{\nabla\chi_u}{|\nabla\chi_u|}$
\begin{align*}
&\int_{U\cap I_v(t)} 1 \,\mathrm{d}|D_x^{\tan} h^+(\cdot,t)|
\\&\leq
C\int_{\{x+y\vec{n}_v(x,t):\,x\in U\cap I_v(t),y\in (-r_c,r_c)\}} 
\big|\vec{n}_v(P_{I_v(t)}x) - \vec{n}_u\big| \,\mathrm{d}|\nabla\chi_u(\cdot,t)|
\\&~~~
+\frac{C}{r_c} \int_{U\cap I_v(t)} \int_0^{r_c} |\chi_u(\Phi_t(x,y),t)-\chi_v(\Phi_t(x,y),t)| \,\dy \,\dS(x)
\end{align*}
for any Borel set $U\subset \Rd$. 
Recall that the indicator function $\chi_u(\cdot,t)$ of the varifold solution 
is of bounded variation in $I:=\{x\in\Rd\colon\sigdist(x,I_v(t))\in (-r_c,r_c)\}$.
In particular, $E^+:=\{x\in\Rd\colon \chi_u>0\}\cap I$ is a set
of finite perimeter in $I$. Applying Theorem~\ref{TheoG} in local coordinates the sections
$$E^+_x = \{y\in (-r_c,r_c)\colon \chi_u(x+y\vec{n}_v(x,t))>0\}$$
are guaranteed to be one-dimensional Caccioppoli sets in $(-r_c,r_c)$ for $\mathcal{H}^{d-1}$-almost
every $x\in I_v(t)$.
Note that whenever $|\vec{n}_v\cdot\vec{n}_u|\leq\frac{1}{2}$
then $1-\vec{n}_v\cdot\vec{n}_u\geq\frac{1}{2}$, and therefore using also the co-area formula for rectifiable
sets (see \cite[(2.72)]{Ambrosio2000a})
\begin{align}\label{auxSpatialDerivativeHplus}
&\int_{U\cap I_v(t)} 1 \,\mathrm{d}|D_x^{\tan} h^+(\cdot,t)| 
\\&\nonumber
\leq
\frac{C}{r_c} \int_{U\cap I_v(t)} \int_0^{r_c} |\chi_u(\Phi_t(x,y),t)-\chi_v(\Phi_t(x,y),t)| \,\dy \,\dS(x)
\\&~~~\nonumber
+ C\int_{U\cap I_v(t)}\int_{\partial^*E^+_x\cap
\{\vec{n}_v(x)\cdot\vec{n}_u(x{+}y\vec{n}_v(x,t))\geq\frac{1}{2}\}\cap(-r_c,r_c)} 
\frac{|\vec{n}_v(x)-\vec{n}_u|}{|\vec{n}_v(x)\cdot\vec{n}_u|} \,\mathrm{d}\mathcal{H}^0(y)\,\mathrm{d}S(x)
\\&~~~\nonumber
+ C\int_{\{x{+}y\vec{n}_v(x,t):\,x\in U\cap I_v(t),y\in (-r_c,r_c),
\vec{n}_v(x)\cdot\vec{n}_u(x{+}y\vec{n}_v(x,t))\leq\frac{1}{2}\}} 
\big(1-\vec{n}_v(P_{I_v(t)}x)\cdot\vec{n}_u\big) \,\mathrm{d}|\nabla\chi_u(\cdot,t)|.
\end{align}

We now distinguish between different cases depending on $x\in I_v(t)$ up to $\mathcal{H}^{d-1}$-measure zero.
We start with the set of points $x\in A_1\subset I_v(t)$ such that
\begin{align}\label{DefinitionA1}
&\int_0^{r_c} |\chi_u(\Phi_t(x,y),t)-\chi_v(\Phi_t(x,y),t)| \,\dy
\\&\nonumber
+\int_{\partial^*E^+_x\cap
\{\vec{n}_v(x)\cdot\vec{n}_u(x{+}y\vec{n}_v(x,t))\geq\frac{1}{2}\}\cap(-r_c,r_c)} 
\frac{|\vec{n}_v(x)-\vec{n}_u|}{|\vec{n}_v(x)\cdot\vec{n}_u|} \,\mathrm{d}\mathcal{H}^0(y)
\\&\nonumber
+\sup_{y\in\{\tilde y\in (-r_c,r_c)\cap\partial^*E^+_x\colon
\vec{n}_v(x)\cdot\vec{n}_u(x{+}\tilde y\vec{n}_v(x,t))\leq\frac{1}{2}\}}
1-\vec{n}_v(P_{I_v(t)}x)\cdot\vec{n}_u(x{+}y\vec{n}_v(x,t))
\\&\nonumber
\leq \frac{1}{4}.
\end{align}
By splitting the measure $D_x^\mathrm{tan}h^+$ into a part which is absolutely continuous with respect
to the surface measure on $I_v(t)$, for which we denote the density by $\nabla^\mathrm{tan} h^+$, 
as well as a singular part $D^sh^+$, we obtain from \eqref{auxSpatialDerivativeHplus}
(note that the third integral in \eqref{auxSpatialDerivativeHplus}
does not contribute to this estimate by the definition of the set $A_1\subset I_v(t)$)
\begin{align*}
&\int_{U\cap I_v(t)\cap A_1} |\nabla^\mathrm{tan}h^+|(x)\,\dS(x)
\\&
\leq \int_{U\cap I_v(t)\cap A_1} \frac{C}{r_c}\int_0^{r_c} |\chi_u(\Phi_t(x,y),t)-\chi_v(\Phi_t(x,y),t)| \,\dy \,\dS(x)
\\&~\,
+ \int_{U\cap I_v(t)\cap A_1}C\int_{\partial^*E^+_x\cap
\{\vec{n}_v(x)\cdot\vec{n}_u(x{+}y\vec{n}_v(x,t))\geq\frac{1}{2}\}\cap(-r_c,r_c)}  
\frac{|\vec{n}_v(x)-\vec{n}_u|}{|\vec{n}_v(x)\cdot\vec{n}_u|} \,\mathrm{d}\mathcal{H}^0(y)\,\mathrm{d}S(x)
\end{align*}
for every Borel set $U\subset\Rd$. Since $U$ was arbitrary, we deduce that $|\nabla^\mathrm{tan}h^+|$
is bounded on $A_1$ by the two integrands on the right hand side of the last inequality. 
Hence, we obtain
\begin{align*}
&\int_{A_1} |\nabla^\mathrm{tan}h^+|^2(x)\,\dS(x) + |D^sh^+|(A_1)
\\&
\leq Cr_c^{-2}\int_{I_v(t)}\bigg|\int_0^{r_c} |\chi_u(\Phi_t(x,y),t)-\chi_v(\Phi_t(x,y),t)| \,\dy\bigg|^2 \,\dS(x)
\\&~\,
+ C\int_{I_v(t)\cap A_1}\bigg|\int_{\partial^*E^+_x\cap
\{\vec{n}_v(x)\cdot\vec{n}_u(x{+}y\vec{n}_v(x,t))\geq\frac{1}{2}\}\cap(-r_c,r_c)}  
|\vec{n}_v-\vec{n}_u|\,\mathrm{d}\mathcal{H}^0(y) \bigg|^2\,\dS(x).
\end{align*}
The first term on the right hand side can be estimated as in the proof of the $L^2$-bound for $h^\pm$.
To bound the second term, we make the following observation. First, we may represent the one-dimensional 
Caccioppoli sets $E^+_x$ as a finite union of disjoint intervals (see \cite[Proposition~3.52]{Ambrosio2000a}).
It then follows from property iv) in Theorem~\ref{TheoG} that
$\partial^*E^+_x\cap(-r_c,r_c)$ can only contain at most one point. Indeed, otherwise
we would find at least one point $y\in \partial^*E^+_x\cap(-r_c,r_c)$ 
such that $\vec{n}_v(x)\cdot\vec{n}_u(x{+}y\vec{n}_v(x,t))<0$ which is a contradiction to 
the definition of $A_1$. By another application of the co-area formula for rectifiable
sets (see \cite[(2.72)]{Ambrosio2000a}) we therefore get
\begin{align}\label{square}\nonumber
&\int_{A_1} |\nabla^\mathrm{tan}h^+|^2(x)\,\dS(x) + |D^sh^+|(A_1)
\\&
\leq \frac{C}{r_c^2} \int_\Rd |\chi_u-\chi_v| \min\Big\{\frac{\dist(x,I_v(t))}{r_c},1\Big\} \,\dx
\\&\nonumber~~~
+C\int_{\{\dist(x,I_v(t))<r_c\}} 1-\vec{n}_v(P_{I_v(t)}x)\cdot \frac{\nabla\chi_u}{|\nabla\chi_u|} 
\,\mathrm{d}|\nabla\chi_u|(x).
\end{align}

We now turn to the second case, namely the set of points $A_2:=I_v(t)\setminus A_1$.
We begin with a preliminary computation.
When splitting $E^+_x$ into a finite family of disjoint open intervals
as before, it again follows from property iv) in Theorem~\ref{TheoG} that every second point
$y\in\partial^*E^+_x\cap(-r_c,r_c)$ has to have the property that $\vec{n}_v(x)\cdot\vec{n}_u(x{+}y\vec{n}_v(x,t))<0$, i.e.,
$|\vec{n}_v(x)-\vec{n}_u|\leq 2\leq 2(1-\vec{n}_v(x)\cdot\vec{n}_u)$. In particular, by another application of 
the co-area formula for rectifiable sets (see \cite[(2.72)]{Ambrosio2000a}) we obtain the bound
\begin{align}\nonumber\label{auxLinear}
&\int_{A_2}\int_{\partial^*E^+_x\cap
\{\vec{n}_v(x)\cdot\vec{n}_u(x{+}y\vec{n}_v(x,t))\geq\frac{1}{2}\}\cap(-r_c,r_c)}  
\frac{|\vec{n}_v(x)-\vec{n}_u|}{|\vec{n}_v(x)\cdot\vec{n}_u|} \,\mathrm{d}\mathcal{H}^0(y)\,\mathrm{d}S(x)
\\&
\leq 8\int_{\{\dist(x,I_v(t))<r_c\}} 1-\vec{n}_v(P_{I_v(t)}x)\cdot \frac{\nabla\chi_u}{|\nabla\chi_u|} 
\,\mathrm{d}|\nabla\chi_u|(x).
\end{align}

Now, we proceed as follows. By definition of $A_2$, either one of the three summands 
in \eqref{DefinitionA1} has to be $\geq \frac{1}{12}$. We distinguish between two cases.
If the third one is not, then this actually means that the set
$\{\tilde y\in (-r_c,r_c)\cap\partial^*E^+_x\colon
\vec{n}_v(x)\cdot\vec{n}_u(x{+}\tilde y\vec{n}_v(x,t))\leq\frac{1}{2}\}$
is empty, i.e., the third summand has to vanish. Hence, either one of the first
two summands in \eqref{DefinitionA1} has to be $\geq \frac{1}{8}$. If the first one
is not, we use that $\int_0^{r_c} |\chi_u(\Phi_t(x,y),t)-\chi_v(\Phi_t(x,y),t)| \,\dy \leq r_c$
and bound this by the second term and then \eqref{auxLinear}. If the second one is not, then
\begin{align}\label{auxLinear2}\nonumber
\ell^+(x)&:=\int_0^{r_c} |\chi_u(\Phi_t(x,y),t)-\chi_v(\Phi_t(x,y),t)| \,\dy \leq r_c 
\\&
\leq \frac{C}{r_c}\int_0^{\ell^+(x)}y\,\dy 
\leq C\int_0^{r_c}|\chi_u(\Phi_t(x,y),t)-\chi_v(\Phi_t(x,y),t)|\frac{y}{r_c}\,\dy.
\end{align}
Now, we move on with the remaining case, i.e., that the third summand in \eqref{DefinitionA1}
does not vanish. In other words, $\{\tilde y\in (-r_c,r_c)\cap\partial^*E^+_x\colon
\vec{n}_v(x)\cdot\vec{n}_u(x{+}\tilde y\vec{n}_v(x,t))\leq\frac{1}{2}\}$ is non-empty.
We then estimate
\begin{align}\label{auxLinear3}\nonumber
&\int_0^{r_c} |\chi_u(\Phi_t(x,y),t)-\chi_v(\Phi_t(x,y),t)| \,\dy
\\&
\leq r_c
\leq 2r_c\int_{\partial^*E^+_x\cap (-r_c,r_c)}
1-\vec{n}_v(x)\cdot\vec{n}_u(x{+}\tilde y\vec{n}_v(x,t))
\,\mathrm{d}\mathcal{H}^0(y).
\end{align}

Taking finally $U=A_2$ in \eqref{auxSpatialDerivativeHplus}, the conclusions of the above case study together with
the three estimates \eqref{auxLinear}, \eqref{auxLinear2} and \eqref{auxLinear3} followed
by another application of the co-area formula for rectifiable
sets (see \cite[(2.72)]{Ambrosio2000a}) to further estimate the latter,
then imply that
\begin{align}\label{linear}\nonumber
&\int_{A_2} |\nabla^\mathrm{tan}h^+|(x)\,\dS(x) + |D^sh^+|(A_2)
\\&
\leq \frac{C}{r_c} \int_\Rd |\chi_u-\chi_v| \min\Big\{\frac{\dist(x,I_v(t))}{r_c},1\Big\} \,\dx
\\&\nonumber~~~
+C\int_{\{\dist(x,I_v(t))<r_c\}} 1-\vec{n}_v(P_{I_v(t)}x)\cdot \frac{\nabla\chi_u}{|\nabla\chi_u|} 
\,\mathrm{d}|\nabla\chi_u|(x).
\end{align}

The two estimates \eqref{square} and \eqref{linear} thus entail the desired upper 
bound \eqref{HeightFunctionGradientEstimate} for the (tangential) 
gradient of $h^\pm$ with $\xi$ replaced by $\vec{n}_v(P_{I_v(t)}x)$. 
However, one may replace $\vec{n}_v(P_{I_v(t)}x)$ by $\xi$ because of
\eqref{controlByTiltExcess}.

{\bf Step 3: Proof of the approximation property for the interface \eqref{ErrorImprovedInterfaceApproximationNonsmooth}.}
In order to establish \eqref{ErrorImprovedInterfaceApproximationNonsmooth}, 
we rewrite using the coordinate transform $\Phi_t$ (recall that $\sigdist(\Phi_t(x,y),I_v(t))=y$ and 
that $|h^\pm|\leq r_c$)
\begin{align}
\nonumber
\int_\Rd& |\chi_u-\chi_{v,h^+,h^-}| \,\dx
\\&
\label{BetterApproximationFirst}
= \int_{I_v(t)} \int_{0}^{r_c} \det\nabla \Phi_t(x,y) |\chi_u(\Phi_t(x,y))-1+\chi_{\{y\leq h^+(x)\}}| \,\dy \,\dS(x)
\\&~~~~
\nonumber
+ \int_{I_v(t)} \int_{-r_c}^0 \det\nabla \Phi_t(x,y) |\chi_u(\Phi_t(x,y))-\chi_{\{y\geq -h^-(x)\}}| \,\dy \,\dS(x)
\\&~~~~
+ \int_{\{\dist(x,I_v(t))\geq r_c\}} |\chi_u-\chi_v| \,\dx.
\nonumber
\end{align}
In order to derive a bound for the first term on the right-hand side of 
\eqref{BetterApproximationFirst}, we distinguish between different cases
depending on $x\in I_v(t)$ up to $\mathcal{H}^{d-1}$-measure zero. We
first distinguish between $h^+(x)\geq\frac{r_c}{4}$ and $h^+(x)<\frac{r_c}{4}$.  
In the former case, a straightforward estimate yields (recall \eqref{BoundNablaPhiNablaPhi-1})
\begin{align}\label{auxCase1}
\nonumber
&\bigg|\int_{0}^{r_c} \det\nabla \Phi_t(x,y) |\chi_u(\Phi_t(x,y))-1+\chi_{\{y\leq h^+(x)\}}| \,\dy\bigg|
\\&~~~~\leq Cr_c \leq\frac{C}{r_c}\int_0^{h^+(x)}y\,\dy 
\leq C\int_0^{r_c}|\chi_u(\Phi_t(x,y))-\chi_v(\Phi_t(x,y))|\frac{y}{r_c}\,\dy,
\end{align}
which is indeed of required order after a change of variables. 
We now consider the other case, i.e., $h^+(x)<\frac{r_c}{4}$.
Recall that the indicator function $\chi_u(\cdot,t)$ of the varifold solution 
is of bounded variation in $I^+:=\{x\in\Rd\colon\sigdist(x,I_v(t))\in (0,r_c)\}$.
In particular, $E^+:=\{x\in\Rd\colon 1-\chi_u>0\}\cap I_+$ is a set
of finite perimeter in $I^+$. Recall also that $E^+=I^+\cap\{x\in\Rd\colon(\chi_v-\chi_u)_+>0\}$ 
since $\chi_v\equiv 1$ in $I^+$. Applying Theorem~\ref{TheoG} in local coordinates, the sections
\begin{align*}
E^+_x = \{y\in (0,r_c)\colon 1-\chi_u(x+y\vec{n}_v(x,t))>0\}
\end{align*}
are guaranteed to be one-dimensional Caccioppoli sets in $(0,r_c)$ for $\mathcal{H}^{d-1}$-almost
every $x\in I_v(t)$. Hence, we may represent the one-dimensional section $E^+_x$ 
for such $x\in I_v(t)$ as a finite union of disjoint intervals (see \cite[Proposition~3.52]{Ambrosio2000a})
\begin{align*}
E^+_x\cap (0,r_c) = \bigcup_{m=1}^{K(x)}(a_m,b_m).
\end{align*}
If $K(x) = 0$ then $h^+(x)=0$, and the inner integral in the first term on the right hand side of
\eqref{BetterApproximationFirst} vanishes for this $x$. If $K(x) = 1$
and $a_1=0$, then by definition of $h^+(x)$ we have $(a_1,b_1)=(0,h^+(x))$ (recall that we now consider the case $h^+(x)\leq \frac{r_c}{4}$). Thus, again the inner integral in the first term on the right hand side of
\eqref{BetterApproximationFirst} vanishes for this $x$.
Hence, it remains
to discuss the case that there is at least one non-empty interval in the 
decomposition of $E^+_x$, say $(a,b)$, such that $a\in (0,r_c)$. From property iv) in
Theorem~\ref{TheoG} it then follows that 
\begin{align*}
\vec{n}_v(x,t)\cdot \frac{-\nabla\chi_{E^+}}{|\nabla\chi_{E^+}|}(x+a\vec{n}_v(x,t)) \leq 0.
\end{align*}
Hence, we may bound
\begin{align*}
&\bigg|\int_{0}^{r_c} \det\nabla \Phi_t(x,y) |\chi_u(\Phi_t(x,y))-1+\chi_{\{y\leq h^+(x)\}}| \,\dy\bigg|
\\&~~~~\leq Cr_c \leq C\int_{(0,r_c)\cap(\partial^*E^+)_x}
1-\vec{n}_v(x,t)\cdot \frac{-\nabla\chi_{E^+}}{|\nabla\chi_{E^+}|}(x+y\vec{n}_v(x,t))\,\mathrm{d}\mathcal{H}^0(y)
\end{align*}
Gathering the bounds from the different cases together with the estimate in \eqref{auxCase1}, we therefore obtain 
by the co-area formula for rectifiable
sets (see \cite[(2.72)]{Ambrosio2000a}) together with the change 
of variables $\Phi_t(x,y)$
\begin{align*}
&\bigg|\int_{I_v(t)}\int_{0}^{r_c} \det\nabla \Phi_t(x,y) 
|\chi_u(\Phi_t(x,y))-1+\chi_{\{y\leq h^+(x)\}}| \,\dy\,\dS(x)\bigg| \\
&~~~~\leq C \int_{I_v(t)}\int_{(0,r_c)\cap(\partial^*E^+)_x} 1-\vec{n}_v(x,t)\cdot 
\frac{-\nabla\chi_{E^+}}{|\nabla\chi_{E^+}|}(x+y\vec{n}_v(x,t))\,\mathrm{d}\mathcal{H}^0(y)\,\dS(x) \\
&~~~~~~~~
+ C \int_\Rd \int_{-r_c}^{r_c} |\chi_u(\Phi_t(x,y))-\chi_v(\Phi_t(x,y))| \frac{y}{r_c} \,dy \,\dx
\\
&~~~~\leq C \int_{\{\dist(x,I_v(t))<r_c\}} 1-\vec{n}_v(P_{I_v(t)}x)\cdot \frac{\nabla\chi_u}{|\nabla\chi_u|} 
\,\mathrm{d}|\nabla\chi_u|(x) \\
&~~~~~~~~
+ C \int_\Rd |\chi_u-\chi_v| \min\Big\{\frac{\dist(x,I_v(t))}{r_c},1\Big\} \,\dx,
\end{align*}
which is by \eqref{controlByTiltExcess} as well as \eqref{controlByTiltExcess2} indeed a bound of desired order. Moreover, performing analogous estimates for the second term 
on the right-hand side of \eqref{BetterApproximationFirst} and estimating the third term on the right-hand side of \eqref{BetterApproximationFirst} trivially, we then get
\begin{align*}
\int_\Rd& |\chi_u-\chi_{v,h^+,h^-}| \,\dx \\
&\leq C \int_\Rd 1-\xi\cdot \frac{\nabla\chi_u}{|\nabla\chi_u|} \,\mathrm{d}|\nabla\chi_u|
+ C \int_\Rd |\chi_u-\chi_v| \min\Big\{\frac{\dist(x,I_v(t))}{r_c},1\Big\} \,\dx
\end{align*}
which is precisely the desired estimate \eqref{ErrorImprovedInterfaceApproximationNonsmooth}.

{\bf Step 4: Proof of estimate on the time derivative \eqref{HeightFunctiondtEstimate}.}
To bound the time derivative, we compute using the weak formulation 
of the continuity equation $\partial_t \chi_u = -\nabla \cdot(\chi_u u)$ and 
abbreviating $I^+(t):=\{x\in\Rd:\sigdist(x,I_v(t))\in [0,r_c)\}$ (recall that the boundary $\partial I^+(t) = I_v(t)$ moves with normal speed $\vec{n}_v \cdot v$)
\begin{align*}
&\frac{\mathrm{d}}{\dt} \int_{I_v(t)} \eta(x) h^+(x,t) \,\dS(x)
\\&
=\frac{\mathrm{d}}{\dt} \int_{I_v(t)} \int_0^\infty \eta(x) (1-\chi_u)(x+y\vec{n}_v(x,t),t) 
\, \theta\Big(\frac{y}{r_c}\Big) \,\dy\,\dS(x)
\\&
=\frac{\mathrm{d}}{\dt} \int_{I^+(t)}  \eta(P_{I_v(t)}x) |\det \nabla \Phi_t^{-1}|(x) 
(1-\chi_u)(x,t) \, \theta\Big(\frac{\dist(x,I_v(t))}{r_c}\Big) \,\dx
\\&
=\int_{I^+(t)} (1-\chi_u)(x,t) u \cdot \nabla \Big(\eta(P_{I_v(t)}x) 
|\det \nabla \Phi_t^{-1}|(x) \, \theta\Big(\frac{\dist(x,I_v(t))}{r_c}\Big)\Big) \,\dx
\\&~~\,
+\int_{I_v(t)} (\vec{n}_v\cdot u)(x,t) (1{-}\chi_u)(x,t) \eta(P_{I_v(t)}x) |\det \nabla \Phi_t^{-1}|(x)
\theta \Big(\frac{\dist(x,I_v(t))}{r_c}\Big) \,\dS(x)
\\&~~\,
+\int_{I^+(t)} (1-\chi_u)(x,t) \, \frac{\mathrm{d}}{\dt}\Big(\eta(P_{I_v(t)}x) |\det \nabla \Phi_t^{-1}|(x) 
\theta \Big(\frac{\dist(x,I_v(t))}{r_c}\Big)\Big) \,\dx
\\&~~\,
-\int_{I_v(t)} (\vec{n}_v\cdot v)(x,t) (1{-}\chi_u)(x,t) \eta(P_{I_v(t)}x) |\det \nabla \Phi_t^{-1}|(x)
\theta \Big(\frac{\dist(x,I_v(t))}{r_c}\Big) \,\dS(x).
\end{align*}
Recall from \eqref{gradientProjection} the formula for the gradient of the projection
onto the nearest point on the interface $I_v(t)$. Recalling also the definitions of the 
extended normal velocity $V_\vec{n}(x,t):=
\big(v(x,t)\cdot\vec{n}_v(P_{I_v(t)}x,t)\big)\,\vec{n}_v(P_{I_v(t)}x,t)$
and its projection $\bar{V}_{\vec{n}}(x,t):=V_{\vec{n}}(P_{I_v(t)}x,t)$ from \eqref{normalCompVel}
respectively \eqref{projNormalVel}, we also have
\begin{align*}
&-\int_{I^+} (1-\chi_u(x,t)) |\det \nabla \Phi_t^{-1}|(x) 
\theta \Big(\frac{\dist(x,I_v(t))}{r_c}\Big) (\nabla\eta)(P_{I_v(t)}x)
\\&~~~~~~~~~~~~~~~~~~
\cdot \big((v(P_{I_v(t)}x,t)-\bar{V}_{\vec{n}}(x,t)) 
\cdot \nabla\big) P_{I_v(t)}x \,\dx
\\&
=-\int_{I_v(t)} \int_0^{r_c} (1-\chi_u(\Phi_t(x,y),t)) \theta \Big(\frac{y}{r_c}\Big) \nabla \eta(x)
\\&~~~~~~~~~~~~~~~~~~
\cdot \big((v(x,t)-V_{\vec{n}}(x,t)) \cdot \nabla\big) P_{I_v(t)}(\Phi_t(x,y)) \,\dy \,\dS(x)
\\&
=-\int_{I_v(t)} h^+(x,t) (\Id{-}\vec{n}_v(x)\otimes\vec{n}_v(x))v(x,t) \cdot \nabla \eta(x) \,\dS(x) \\
&~~~~+\int_{I^+(t)} (1-\chi_u(x,t)) |\det \nabla \Phi_t^{-1}|(x) \theta 
\Big(\frac{\dist(x,I_v(t))}{r_c}\Big)\dist(x,I_v(t))(\nabla \eta)(P_{I_v(t)}x) 
\\&~~~~~~~~~~~~~~~~~~
\cdot \big((v(P_{I_v(t)}x,t)-\bar{V}_{\vec{n}}(x,t)) \cdot \nabla\big)\vec{n}_v(P_{I_v(t)}x) \,\dx.
\end{align*}
Adding this formula to the above formula for $\frac{\mathrm{d}}{\dt} \int_{I_v(t)} \eta(x) h^+(x,t) \,\dS(x)$, 
introducing the abbreviation $f:=|\det \nabla \Phi_t^{-1}|(x) \, \theta(\frac{\dist(x,I_v(t))}{r_c})$, 
and using the fact that $\chi_v=1$ in $I^+(t)$, we obtain
\begin{align}
\nonumber
&\frac{\mathrm{d}}{\dt} \int_{I_v(t)} \eta(x) h^+(x,t) \,\dx
-\int_{I_v(t)} h^+(x,t) (\Id{-}\vec{n}_v\otimes\vec{n}_v)v(x,t) \cdot \nabla \eta(x) \,\dS(x)
\\&
\nonumber
=\int_{I^+(t)} (\chi_u(x,t)-\chi_v(x,t))f(x)\dist(x,I_v(t)) 
(\nabla \eta)(P_{I_v(t)}x)
\\&~~~~~~~~~~~~
\nonumber
\cdot\big((v(P_{I_v(t)}x,t)-\bar{V}_{\vec{n}}(x,t))
\cdot \nabla\big)\vec{n}_v(P_{I_v(t)}x) \,\dx 
\\&~~~
\label{TimeDerivativehFirst}
-\int_{I^+(t)} (\chi_u(x,t)-\chi_v(x,t)) \eta(P_{I_v(t)}x)  (u-v) \cdot \nabla f \,\dx
\\&~~~
\nonumber
-\int_{I^+(t)} (\chi_u(x,t)-\chi_v(x,t)) f(x) (\nabla \eta)(P_{I_v(t)}x) \cdot ((u-v) \cdot \nabla)P_{I_v(t)}x  \,\dx
\\&~~~
\nonumber
-\int_{I^+(t)} (\chi_u(x,t)-\chi_v(x,t)) f(x) (\nabla \eta)(P_{I_v(t)}x) 
\\&~~~~~~~~~~~~
\nonumber
\cdot 
\big((v(x,t)-(v(P_{I_v(t)}x,t)-\bar{V}_{\vec{n}}(x,t))) 
\cdot \nabla\big) P_{I_v(t)}x \,\dx
\\&~~~
\nonumber
-\int_{I^+(t)} (\chi_u(x,t)-\chi_v(x,t)) f(x) (\nabla \eta)(P_{I_v(t)}x) \cdot \frac{\mathrm{d}}{\dt} P_{I_v(t)}x \,\dx
\\&~~~
\nonumber
-\int_{I^+(t)} (\chi_u(x,t)-\chi_v(x,t)) \, \eta(P_{I_v(t)}x) \Big(\frac{\mathrm{d}}{\dt}f+v \cdot \nabla f\Big) \,\dx
\\&~~~
\nonumber
+\int_{I_v(t)} \vec{n}_v \cdot (u-v) (1-\chi_u) \eta \,\dS.
\end{align}
Note that $f(x)=|\det \nabla \Phi_t^{-1}|(x) \, \theta(\frac{\dist(x,I_v(t))}{r_c})=1$ for any 
$t$ and any $x\in I_v(t)$. Thus, we have $\frac{\mathrm{d}}{\dt} f + v\cdot \nabla f=0$ on $I_v(t)$.
Furthermore, we have $|\nabla\bar{V}_{\vec{n}}| \leq \frac{C}{r_c^2}\|v\|_{W^{1,\infty}}$
and $|\nabla^2\bar{V}_{\vec{n}}| \leq \frac{C}{r_c^3}\|v\|_{W^{2,\infty}(\Rd\setminus I_v(t))}$ 
because of $\bar{V}_{\vec{n}}(x)=V_{\vec{n}}(P_{I_v(t)}x)$, \eqref{Bound2ndDerivativeSignedDistance},
the corresponding estimate \eqref{Bound1stDerivativeNormalVelocity} for the gradient of $V_{\vec{n}}$ 
as well as the formula \eqref{gradientProjection} for the gradient of $P_{I_v(t)}$.
Because of \eqref{sdistAux4} and the equation \eqref{TimeEvolutionNormalNoCutoff} for the time evolution of the 
normal vector, we thus get the bounds 
$|\frac{\mathrm{d}}{\dt}\nabla\sigdist(\cdot,I_v(t))|\leq \frac{C}{r_c^2}\|v\|_{W^{1,\infty}}$
and $|\nabla\frac{\mathrm{d}}{\dt}\nabla\sigdist(\cdot,I_v(t))|\leq 
\frac{C}{r_c^3}\|v\|_{W^{2,\infty}(\Rd\setminus I_v(t))}$. 
Taking all of these bounds together, we obtain $|f|\leq \frac{C}{r_c}$, $|\nabla f|\leq \frac{C}{r_c^2}$ and 
$|\nabla^2 f|+|\nabla \frac{\mathrm{d}}{\dt} f|\leq \frac{C}{r_c^3}(1+\|v\|_{W^{2,\infty}(\Rd\setminus I_v(t))})$. 
As a consequence, we get
\begin{align}\label{fTransport}
\Big|\frac{\mathrm{d}}{\dt}f+v \cdot \nabla f\Big|
\leq \frac{C}{r_c^3}(1+\|v\|_{W^{2,\infty}(\Rd\setminus I_v(t))})\dist(\cdot,I_v(t)).
\end{align}
Moreover, we may compute
\begin{align}\label{TimeDerivativeProjection}
\frac{\mathrm{d}}{\dt}  P_{I_v(t)}x = -\vec{n}_v(P_{I_v(t)}x)
\frac{\mathrm{d}}{\dt}\sigdist(x,I_v(t)) - \sigdist(x,I_v(t))
\frac{\mathrm{d}}{\dt}(\vec{n}_v(P_{I_v(t)}x)).
\end{align}
Since $\vec{n}_v\cdot\nabla\eta = 0$ holds on the interface $I_v(t)$ by assumption,
we obtain from \eqref{TimeDerivativeProjection}
\begin{align*}
&-\int_{I^+(t)} (\chi_u(x,t)-\chi_v(x,t)) f(x) (\nabla \eta)(P_{I_v(t)}x) 
\cdot \frac{\mathrm{d}}{\dt}  P_{I_v(t)}x \,\dx
\\&
=\int_{I^+(t)} (\chi_u(x,t)-\chi_v(x,t)) \sigdist(x,I_v(t))
 f(x) (\nabla\eta)(P_{I_v(t)}x)\cdot\frac{\mathrm{d}}{dt}(\vec{n}_v(P_{I_v(t)}x)) \,\dx.
\end{align*}
In what follows, we will by slight abuse of notation use $\nabla^{\mathrm{tan}}g(x)$
as a shorthand for $(\mathrm{Id}-\vec{n}_v(P_{I_v(t)}x)\otimes\vec{n}_v(P_{I_v(t)}x))\nabla g(x)$ 
for scalar fields as well as $(\nabla^{\mathrm{tan}}\cdot g)(x)$ instead of 
$(\mathrm{Id}-\vec{n}_v(P_{I_v(t)}x)\otimes\vec{n}_v(P_{I_v(t)}x)):\nabla g(x)$ for vector fields.
Let us also abbreviate $P^\mathrm{tan}x:=(\mathrm{Id}-\vec{n}_v(P_{I_v(t)}x)\otimes\vec{n}_v(P_{I_v(t)}x))$.
Note that by assumption $(\nabla\eta)(P_{I_v(t)}x)=(\nabla^{\mathrm{tan}}\eta)(P_{I_v(t)}x)$.
Moreover, it follows from \eqref{sdistAux1}, \eqref{sdistAux2} and \eqref{sdistAux4}
that $\vec{n}_v(P_{I_v(t)}x)\cdot\frac{\mathrm{d}}{\dt}(\vec{n}_v(P_{I_v(t)}x))=0$. 
Hence, we may rewrite with an integration by parts  (recall the notation $P^\mathrm{tan}(x)=(\Id-\vec{n}_v\otimes \vec{n}_v)(P_{I_v(t)}x,t)$)
\begin{align}\label{boundDtXi}
&\int_{I^+(t)} (\chi_u(x,t)-\chi_v(x,t)) \sigdist(x,I_v(t)) f(x) 
(\nabla^{\mathrm{tan}}\eta)(P_{I_v(t)}x)\cdot\frac{\mathrm{d}}{\dt}(\vec{n}_v(P_{I_v(t)}x)) \,\dx 
\\&\nonumber
=-\int_{I^+(t)} (\chi_u-\chi_v)(x,t) \sigdist(x,I_v(t))\eta(P_{I_v(t)}x)
\\&~~~~~~~~~~~~\nonumber
\times\Big(\frac{\mathrm{d}}{\dt}(\vec{n}_v(P_{I_v(t)}x))\otimes\nabla\Big):f(x)P^\mathrm{tan}(x) \,\dx 
\\&\nonumber
-\int_{I^+(t)} (\chi_u-\chi_v)(x,t) \sigdist(x,I_v(t)) 
f(x)\eta(P_{I_v(t)}x)\nabla^{\mathrm{tan}}\cdot\frac{\mathrm{d}}{\dt}(\vec{n}_v(P_{I_v(t)}x)) \,\dx 
\\&\nonumber
-\int_\Rd \sigdist(x,I_v(t)) f(x)\eta(P_{I_v(t)}x)\Big(\frac{\nabla\chi_u}{|\nabla\chi_u|}-\vec{n}_v(P_{I_v(t)}x)\Big)
\cdot\frac{\mathrm{d}}{\dt}(\vec{n}_v(P_{I_v(t)}x)) \,\mathrm{d}|\nabla\chi_u|.
\end{align}
Using from \eqref{sdistAux2} and \eqref{sdistAux4} that the spatial partial derivatives of the 
extended normal vector field are orthogonal to the gradient of the signed distance function, 
the same argument also shows that
\begin{align}\label{boundTangVel}&\int_{I^+(t)} (\chi_u(x,t)-\chi_v(x,t))f(x)\dist(x,I_v(t)) 
(\nabla^{\mathrm{tan}} \eta)(P_{I_v(t)}x)
\\&~~~~~~~~~~~~\nonumber
\cdot\big((v(P_{I_v(t)}x,t)-\bar{V}_{\vec{n}}(x,t))
\cdot \nabla\big)\vec{n}_v(P_{I_v(t)}x) \,\dx 
\\&\nonumber
=-\int_{I^+(t)} (\chi_u(x,t)-\chi_v(x,t)) \dist(x,I_v(t)) \eta(P_{I_v(t)}x)
\\&~~~~~~~~~~~~~~\nonumber
\times\big(((v(P_{I_v(t)}x,t)-\bar{V}_{\vec{n}}(x,t))\cdot \nabla)
\vec{n}_v(P_{I_v(t)}x)\otimes\nabla\big):f(x)P^\mathrm{tan}(x) \,\dx
\\&~~~\nonumber
-\int_{I^+(t)} (\chi_u(x,t)-\chi_v(x,t)) \dist(x,I_v(t)) f(x)\eta(P_{I_v(t)}x)
\\&~~~~~~~~~~~~~~\nonumber
\times\nabla^\mathrm{tan} 
\cdot\big(\big((v(P_{I_v(t)}x,t)-\bar{V}_{\vec{n}}(x,t))
\cdot \nabla\big)\vec{n}_v(P_{I_v(t)}x) \big)\,\dx
\\&~~~\nonumber
-\int_\Rd \sigdist(x,I_v(t)) f(x)\eta(P_{I_v(t)}x)\Big(\frac{\nabla\chi_u}{|\nabla\chi_u|}-\vec{n}_v(P_{I_v(t)}x)\Big)
\\&~~~~~~~~~~~~~~\nonumber
\cdot\big((v(P_{I_v(t)}x,t)-\bar{V}_{\vec{n}}(x,t)) 
\cdot \nabla\big)\vec{n}_v(P_{I_v(t)}x)\,\mathrm{d}|\nabla\chi_u|.
\end{align} 
It follows from \eqref{gradientProjection} as well as \eqref{sdistAux2} and \eqref{sdistAux4}
that $(\vec{n}_v(P_{I_v(t)}x)\cdot\nabla)P_{I_v(t)}x = 0$. Hence, we obtain
\begin{align}
\label{TimeDerivativehAux1}
&\int_{I^+(t)} (\chi_u(x,t)-\chi_v(x,t)) f(x) (\nabla \eta)(P_{I_v(t)}x) 
\\&~~~~~~~~
\nonumber
\cdot \big((v(x,t)-(v(P_{I_v(t)}x,t)-\bar{V}_{\vec{n}}(x,t))) 
\cdot \nabla\big) P_{I_v(t)}(x) \,\dx
\\&
\nonumber
=\int_{I^+(t)} (\chi_u-\chi_v)(x,t) f(x) (\nabla \eta)(P_{I_v(t)}x) \cdot 
\big((v(x,t)-v(P_{I_v(t)}x,t)) \cdot \nabla\big) P_{I_v(t)}x \,\dx.
\end{align}
Since the domain of integration is $I^+(t)$, we may write
\begin{align*}
&v(x,t)-v(P_{I_v(t)}x,t) \\&= \sigdist(x,I_v(t))\int_{(0,1]}
\nabla v\big(P_{I_v(t)}x+\lambda\sigdist(x,I_v(t))\vec{n}_v(P_{I_v(t)}x)\big)\,d\lambda
\cdot\vec{n}_v(P_{I_v(t)}x).
\end{align*}
From this and the fact $\vec{n}_v(P_{I_v(t)})\cdot \nabla P_{I_v(t)}(x) = 0$,
we deduce by another integration by parts that (where $|F|\leq r_c^{-1}\|v\|_{W^{2,\infty}(\Rd\setminus I_v(t))}$) 
\begin{align}\label{boundVelProj}
&\int_{I^+(t)} (\chi_u(x,t)-\chi_v(x,t)) f(x) (\nabla^\mathrm{tan} \eta)(P_{I_v(t)}x) \cdot 
((v(x,t)-v(P_{I_v(t)} x,t)) \cdot \nabla) P_{I_v(t)}x \,\dx
\\&\nonumber
=-\int_{I^+(t)} (\chi_u(x,t)-\chi_v(x,t)) \eta(P_{I_v(t)}x)
\\&~~~~~~~~~~~~~~\nonumber
\times\big(((v(x,t)-v(P_{I_v(t)} x,t)) \cdot \nabla) P_{I_v(t)}x\otimes\nabla\big):
f(x)P^\mathrm{tan}x \,\dx 
\\&~~~\nonumber
-\int_{I^+(t)} (\chi_u(x,t)-\chi_v(x,t)) f(x) \eta(P_{I_v(t)}x) 
((v(x,t)-v(P_{I_v(t)} x,t)) \cdot \nabla(\nabla^\mathrm{tan}\cdot P_{I_v(t)}x) \,\dx 
\\&~~~\nonumber
-\int_{I^+(t)} (\chi_u(x,t)-\chi_v(x,t)) \dist(x,I_v(t)) f(x) \eta(P_{I_v(t)}x) 
F(x,t):\nabla P_{I_v(t)}x \,\dx 
\\&~~~\nonumber
-\int_\Rd  f(x)\eta(P_{I_v(t)}x)\Big(\frac{\nabla\chi_u}{|\nabla\chi_u|}{-}\vec{n}_v(P_{I_v(t)}x)\Big)\cdot
\big((v(x,t){-}v(P_{I_v(t)} x,t)) \cdot \nabla\big)P_{I_v(t)}x\,\mathrm{d}|\nabla\chi_u|.
\end{align}
Hence, plugging in \eqref{TimeDerivativehAux1}, \eqref{boundTangVel} and \eqref{boundVelProj}, \eqref{boundDtXi} into \eqref{TimeDerivativehFirst} and using the estimates $\smash{|\nabla\bar{V}_{\vec{n}}| \leq \frac{C}{r_c^2}\|v\|_{W^{1,\infty}}}$,
$\smash{|\frac{\mathrm{d}}{\dt}\vec{n}_v(P_{I_v(t)}x)|\leq \frac{C}{r_c^2}\|v\|_{W^{1,\infty}}}$,
$\smash{|\nabla\frac{\mathrm{d}}{\dt}\vec{n}_v(P_{I_v(t)}x)|\leq \frac{C}{r_c^3}\|v\|_{W^{2,\infty}(\Rd\setminus I_v(t))}}$, and
$\smash{|\nabla f|\leq \frac{C}{r_c^2}}$,
we obtain
\begin{align*}
&\bigg|\frac{\mathrm{d}}{\dt} \int_{I_v(t)} \eta(x) h^+(x,t) \,\dx 
- \int_{I_v(t)} h^+(x,t) (\mathrm{Id}{-}\vec{n}_v\otimes\vec{n}_v)v(x,t) \cdot \nabla \eta(x) \,dS(x)\bigg|
\\&
\leq \frac{C}{r_c^2} \int_{\{\dist(x,I_v(t))\leq r_c\}} 
|\chi_u(x,t)-\chi_v(x,t)| |u(x,t)-v(x,t)| |\eta(P_{I_v(t)}x)| \,\dx
\\&~~~
+\frac{C}{r_c}\int_{\{\dist(x,I_v(t))\leq r_c\}}
 |\chi_u(x,t)-\chi_v(x,t)| |u(x,t)-v(x,t)| |\nabla \eta(P_{I_v(t)}x)| \,\dx
\\&~~~
+\frac{C(1{+}\|v\|_{W^{1,\infty}})}{r_c}\int_{\{\dist(x,I_v(t))\leq r_c\}} 
\bigg|\frac{\nabla\chi_u}{|\nabla\chi_u|}-\vec{n}_v(P_{I_v(t)}x)\bigg| 
\frac{|\sigdist(x,I_v(t))|}{r_c} |\eta(P_{I_v(t)}x)| \,\mathrm{d}|\nabla\chi_u|(x)
\\&~~~
+\frac{C(1{+}\|v\|_{W^{2,\infty}(\Rd\setminus I_v(t))})}{r_c^3}\int_{\{\dist(x,I_v(t))\leq r_c\}} 
|\chi_u(x,t){-}\chi_v(x,t)|\frac{|\sigdist(x,I_v(t))|}{r_c} |\eta(P_{I_v(t)}x)| \,\dx
\\&~~~
+C\int_{I_v(t)} |u-v| |\eta| \,\dS.
\end{align*}
This yields by the change of variables $\Phi_t(x,y)$ and a straightforward estimate
\begin{align*}
&\bigg|\frac{\mathrm{d}}{\dt} \int_{I_v(t)} \eta(x) h^+(x,t) \,dx 
- \int_{I_v(t)} h^+(x,t) (\mathrm{Id}{-}\vec{n}_v\otimes\vec{n}_v)v(x,t) \cdot \nabla \eta(x) \,\dS(x)\bigg|
\\&
\leq 
\frac{C}{r_c^2}\|\eta\|_{W^{1,4}(I_v(t))}\bigg(\int_{I_v(t)} \bigg(\int_{0}^{\frac{r_c}{2}} 
|\chi_u-\chi_v|(x+y\vec{n}_v(x,t),t) \,\dy\bigg)^4 \,\dS\bigg)^{1/4}
\\&~~~~~~~~~~
\times
\bigg(\int_{I_v(t)}  \sup_{y\in [-r_c,r_c]} |u-v|^2(x+y\vec{n}_v(x,t),t) \,\dS(x)\bigg)^{1/2} 
\\&~
+\frac{C(1{+}\|v\|_{W^{2,\infty}(\Rd\setminus I_v(t))})}{r_c^3}\|\eta\|_{L^2(I_v(t))} 
\\&~~~~~~~~~~~~
\times
\bigg(\int_{\Rd}|\chi_u(x,t)-\chi_v(x,t)| \, 
\min\Big\{\frac{\dist(x,I_v(t))}{r_c},1\Big\} \,\dx\bigg)^\frac{1}{2}
\\&~
+\frac{C(1{+}\|v\|_{W^{1,\infty}})}{r_c}\|\eta\|_{L^\infty(I_v(t))}
\bigg(\int_{\{\dist(x,I_v(t))\leq r_c\}}
\bigg|\frac{\nabla\chi_u}{|\nabla\chi_u|}-\vec{n}_v(P_{I_v(t)}x)\bigg|^2
\,\mathrm{d}|\nabla\chi_u|\bigg)^\frac{1}{2}
\\&~~~~~~~~~~~~
\times
\bigg(\int_{\{\dist(x,I_v(t))\leq r_c\}}\frac{|\sigdist(x,I_v(t))|^2}{r_c^2}
\,\mathrm{d}|\nabla\chi_u|\bigg)^\frac{1}{2}
\\&~
+C\bigg(\int_{I_v(t)} |u-v|^2 \,dS \bigg)^{1/2} \, \|\eta\|_{L^2(I_v(t))}.
\end{align*}
Using finally the Sobolev embedding to bound the $L^\infty$-norm of $\eta$ on the interface
(which is either one- or two-dimensional; note that the constant in the Sobolev embedding may be bounded by $Cr_c^{-1}$ for our geometry), we infer from this estimate the desired bound 
\eqref{HeightFunctiondtEstimate}, using also
\eqref{controlByTiltExcess} and \eqref{controlByTiltExcess2}. This concludes the proof.
\end{proof}

\subsection{A regularization of the local height of the interface error}

In order to modify our relative entropy to compensate for the velocity gradient discontinuity at 
the interface, we need regularized versions of the local heights of the interface error $h^+$ and 
$h^-$ which in particular have Lipschitz regularity. To this aim, we fix some function $e(t)>0$ and basically 
apply a mollifier on scale $e(t)$ to the local interface error heights $h^+$ and $h^-$ at each time. These regularized 
versions $\smash{h^+_{e(t)}}$ and $\smash{h^-_{e(t)}}$ of the local interface error heights then have the following properties:

\begin{proposition}
\label{PropositionInterfaceErrorHeightRegularized}	
Let $\chi_v\in L^\infty([0,\Tmax);\BV(\mathbb{R}^d;\{0,1\}))$
be an indicator function such that $\Omega^+_t :=\{x\in\R^d\colon\chi_v(x,t)=1\}$ is a family 
of smoothly evolving domains and $I_v(t) := \partial\Omega^+_t$ is a family of
smoothly evolving surfaces in the sense of Definition~\ref{definition:domains}.
Let $\xi$ be the extension of the unit normal vector field $\vec{n}_v$
from Definition~\ref{def:ExtNormal}. 

Let $\chi_u\in L^\infty([0,\Tmax);\BV(\mathbb{R}^d;\{0,1\}))$ be another indicator function
and let then $h^+$ resp.\ $h^-$ be as defined in Proposition~\ref{PropositionInterfaceErrorHeight}. 
Let $\theta\colon\mathbb{R}^+\rightarrow [0,1]$ be a smooth cutoff with $\theta(s)=1$ for $s\in [0,\frac{1}{4}]$ and 
$\theta(s)=0$ for $s\geq \frac{1}{2}$. Let $e\colon [0,\Tmax)\to(0,r_c]$ be a $\smash{C^1}$-function and define the regularized height 
of the local interface error
\begin{align}
\label{Definitionhregularized}
h^\pm_{e(t)}(x,t):= \frac{\int_{I_v(t)} \theta\big(\frac{|\tilde x-x|}{e(t)}\big) h^\pm(\tilde x,t) 
\,\dS(\tilde x)}{\int_{I_v(t)} \theta\big(\frac{|\tilde x-x|}{e(t)}\big) \,\dS(\tilde x)}.
\end{align}
Then $h^+_{e(t)}$ and $h^-_{e(t)}$ have the following properties:

\begin{subequations}
\noindent{\bf a) ($H^1$-bound)} If the interface error terms from the relative entropy are bounded by
\begin{align*}
&\int_\Rd 1-\xi(\cdot,t)\cdot
\frac{\nabla\chi_u(\cdot,t)}{|\nabla\chi_u(\cdot,t)|} 
\,\mathrm{d}|\nabla \chi_u(\cdot,t)| 
\\&
+\int_\Rd \big|\chi_u(\cdot,t)-\chi_v(\cdot,t)\big|\,\Big|\beta\Big(\frac{\sigdist(\cdot,I_v(t))}{r_c}\Big)\Big|\,\dx
\leq e(t)^2,
\end{align*}
we have the Lipschitz estimate $\smash{|\nabla h^\pm_{e(t)}(\cdot,t)|} \leq C r_c^{-2}$, 
the global bound $|\nabla^2 \smash{h^\pm_{e(t)}}(\cdot,t)| \leq C e(t)^{-1} r_c^{-4}$, 
and the bound
\begin{align}\label{EstimateErrorHeightSmooth}
\int_{I_v(t)} |\nabla h^\pm_{e(t)}|^2 + |h^\pm_{e(t)}|^2 \,\dS
&\leq \frac{C}{r_c^2} \int_\Rd 1-\xi\cdot \frac{\nabla\chi_u}{|\nabla\chi_u|} \,\mathrm{d}|\nabla\chi_u|
\\&~~~
\nonumber
+ \frac{C}{r_c^4} \int_\Rd |\chi_u-\chi_v| \min\Big\{\frac{\dist(x,I_v(t))}{r_c},1\Big\} \,\dx.
\end{align}

\noindent{\bf b) (Improved approximation property)} The functions $h^+_{e(t)}$ and $h^-_{e(t)}$ provide an approximation 
for the interface of the weak solution
\begin{align}\label{def:improvedApprox}
\chi_{v,h^+_{e(t)},h^-_{e(t)}}:=&\chi_v
- \chi_{0\leq \sigdist(x,I_v(t))\leq h^+_{e(t)}(P_{I_v(t)}x,t)}
\\&~~~\nonumber
+ \chi_{-h^-_{e(t)}(P_{I_v(t)}x,t)\leq \sigdist(x,I_v(t))\leq 0},
\end{align}
up to an error of
\begin{align}
\label{ErrorImprovedInterfaceApproximationSmooth}\nonumber
&\int_\Rd \big|\chi_u-\chi_{v,h^+_{e(t)},h^-_{e(t)}}\big| \,\dx
\\&
\leq C \int_\Rd 1-\xi\cdot \frac{\nabla\chi_u}{|\nabla\chi_u|} \,\mathrm{d}|\nabla\chi_u|
+ C \int_\Rd |\chi_u-\chi_v| \min\Big\{\frac{\dist(x,I_v(t))}{r_c},1\Big\} \,\dx
\\&~~~
\nonumber
+C e(t) \bigg(\int_\Rd 1-\xi\cdot \frac{\nabla\chi_u}{|\nabla\chi_u|} \,\mathrm{d}|\nabla \chi_u|\bigg)^{1/2} \mathcal{H}^{d-1}(I_v(t))^{1/2}
\\&~~~
\nonumber
+C \frac{e(t)}{r_c} \bigg(\int_\Rd |\chi_u-\chi_v| 
\min\Big\{\frac{\dist(x,I_v(t))}{r_c},1\Big\} \,\dx\bigg)^{1/2} \mathcal{H}^{d-1}(I_v(t))^{1/2}.
\end{align}

\noindent{\bf c) (Time evolution)} Let $v$ be a solenoidal vector field
\begin{align*}
v\in L^2([0,\Tmax];H^1(\Rd;\Rd))\cap L^\infty([0,\Tmax];W^{1,\infty}(\Rd;\Rd))
\end{align*}
such that in the domain $\bigcup_{t\in [0,\Tmax)} (\Omega_t^+\cup \Omega_t^-) \times \{t\}$ the second
spatial derivatives of the vector field $v$ exist and satisfy
$\sup_{t\in [0,\Tmax)} \sup_{x\in \Omega_t^+\cup \Omega_t^-} |\nabla^2 v(x,t)| <\infty$.
Assume that $\chi_v$ solves the equation 
$\partial_t \chi_v = -\nabla \cdot(\chi_v v)$. If $\chi_u$ 
solves the equation $\partial_t \chi_u = -\nabla \cdot(\chi_u u)$ 
for another solenoidal vector field $u\in L^2([0,\Tmax];H^1(\mathbb{R}^d;\mathbb{R}^d))$,  
we have the following estimate on the time derivative of $h^\pm_{e(t)}$:
\begin{align}
\label{HeightFunctionRegularizeddtEstimate}
&\bigg|\frac{d}{dt} \int_{I_v(t)} \eta(x) h^\pm_{e(t)}(x,t) \,\dx
- \int_{I_v(t)} h^\pm_{e(t)}(x,t) (\mathrm{Id}{-}\vec{n}_v\otimes\vec{n}_v)v(x,t) \cdot \nabla \eta(x) \,\dS(x) \bigg|
\\&\nonumber
\leq 
\frac{C}{e(t)r_c^2}\|\eta\|_{L^{4}(I_v(t))}\bigg(\int_{I_v(t)} |\bar{h}^\pm|^4 \,\dS \bigg)^{1/4}
\\&~~~~~~
\nonumber
\times
\bigg(\int_{I_v(t)}  \sup_{y\in [-r_c,r_c]} |u-v|^2(x+y\vec{n}_v(x,t),t) \,\dS(x)\bigg)^{1/2} 
\\&~~~~\nonumber
+C\frac{(1+\|v\|_{W^{1,\infty}})}{e(t) r_c}\max_{p\in\{2,4\}}\|\eta\|_{L^p(I_v(t))}
\int_{\Rd} 1-\xi\cdot\frac{\nabla\chi_u}{|\nabla\chi_u|} \,\mathrm{d}|\nabla\chi_u|
\\&~~~~
\nonumber
+Cr_c^{-4}\|v\|_{W^{1,\infty}}(1+e'(t))
\bigg(\int_\Rd 1-\xi\cdot \frac{\nabla\chi_u}{|\nabla\chi_u|} \,\mathrm{d}|\nabla\chi_u|\bigg)^{1/2} ||\eta||_{L^2(I_v(t))}
\\&~~~~
\nonumber
+C\bigg(\frac{1+\|v\|_{W^{2,\infty}(\Rd\setminus I_v(t))}}{r_c}+\frac{\|v\|_{W^{1,\infty}}}{r_c^6}(1+e'(t))\bigg)
\|\eta\|_{L^2(I_v(t))} 
\\*&~~~~~~~~\nonumber
\times
\bigg(\int_{\Rd}|\chi_u(x,t)-\chi_v(x,t)| \, 
\min\Big\{\frac{\dist(x,I_v(t))}{r_c},1\Big\} \,\dx\bigg)^\frac{1}{2}
\\&~~~~\nonumber
+C\|\eta\|_{L^2(I_v(t))}\bigg(\int_{I_v(t)}|u-v|^2\,\dS\bigg)^\frac{1}{2}
\end{align}
for any smooth test function $\eta\in C_{cpt}^\infty(\Rd)$ with $\vec{n}_v\cdot\nabla\eta = 0$ on the interface $I_v(t)$,
and where $\bar{h}^\pm$ is defined as $h^\pm$ but now with respect to the modified cut-off function 
$\bar{\theta}(\cdot)=\theta\big(\frac{\cdot}{2}\big)$.
\end{subequations}
\end{proposition}

\begin{proof}
{\bf Proof of a).}
In order to estimate the spatial derivative $\nabla {h^\pm_{e(t)}}$, we compute using the 
fact that $\smash{\nabla_x \theta\big(\frac{|x-\tilde x|}{e(t)}\big)=-\nabla_{\tilde x} \theta\big(\frac{|x-\tilde x|}{e(t)}\big)}$ 
(note that all of the subsequent gradients are to be understood in the tangential sense on the manifold $I_v(t)$)
\begin{align*}
\nabla h^\pm_{e(t)} (x,t) &= 
-\frac{\int_{I_v(t)} \nabla_{\tilde x} \theta\big(\frac{|\tilde x-x|}{e(t)}\big) 
h^\pm(\tilde x,t) \,\dS(\tilde x)}{\int_{I_v(t)} \theta\big(\frac{|\tilde x-x|}{e(t)}\big) \,\dS(\tilde x)}
\\&~~~~
+\frac{\int_{I_v(t)} \theta\big(\frac{|\tilde x-x|}{e(t)}\big) h^\pm(\tilde x,t) \,\dS(\tilde x) 
\int_{I_v(t)} \nabla_{\tilde x} \theta\big(\frac{|\tilde x-x|}{e(t)}\big)  
\,\dS(\tilde x)}{\big(\int_{I_v(t)} \theta\big(\frac{|\tilde x-x|}{e(t)}\big) \,\dS(\tilde x)\big)^2}
\\
&= 
\frac{\int_{I_v(t)} \theta\big(\frac{|\tilde x-x|}{e(t)}\big) 
\nabla h^\pm(\tilde x,t) \,\dS(\tilde x)}{\int_{I_v(t)} \theta\big(\frac{|\tilde x-x|}{e(t)}\big) \,\dS(\tilde x)}
+\frac{\int_{I_v(t)} \theta\big(\frac{|\tilde x-x|}{e(t)}\big) \,\mathrm{d}D^s h^\pm(\tilde x)}{\int_{I_v(t)} 
\theta\big(\frac{|\tilde x-x|}{e(t)}\big) \,\dS(\tilde x)}
\\&~~~~
+\frac{\int_{I_v(t)} \theta\big(\frac{|\tilde x-x|}{e(t)}\big) h^\pm(\tilde x,t)  
\vec{H}(\tilde x,t) \,\dS(\tilde x)}{\int_{I_v(t)} \theta\big(\frac{|\tilde x-x|}{e(t)}\big) \,\dS(\tilde x)}
\\&~~~~
-\frac{\int_{I_v(t)} \theta\big(\frac{|\tilde x-x|}{e(t)}\big) h^\pm(\tilde x,t) \,\dS(\tilde x) 
\int_{I_v(t)} \theta\big(\frac{|\tilde x-x|}{e(t)}\big) \vec{H}(\tilde x, t) \,\dS(\tilde x)}{\big(\int_{I_v(t)} 
\theta\big(\frac{|\tilde x-x|}{e(t)}\big) \,\dS(\tilde x)\big)^2}.
\end{align*}
Introduce the convex function
\begin{align}\label{auxGConvex}
G(p):=
\begin{cases}
|p|^2&\text{for }|p|\leq 1,
\\
2|p|-1&\text{for }|p|\geq 1.
\end{cases}
\end{align}
Using the estimate \eqref{BoundMeanCurvature}, the obvious bounds $G(p+\tilde p)\leq CG(p)+CG(\tilde p)$ 
and $G(\lambda p)\leq C(\lambda+\lambda^2) G(p)$ for any $p$, $\tilde p$, and $\lambda>0$, and Jensen's 
inequality, we obtain (as the recession function of $G$ is given by $2|p|$)
\begin{align}
\label{GGradEstimatehEplus}
G(|\nabla h^\pm_{e(t)} (x,t)|)
&\leq C \frac{\int_{I_v(t)} \theta\big(\frac{|\tilde x-x|}{e(t)}\big) \big(G(|\nabla h^\pm(\tilde x,t)|) 
+ G(r_c^{-1} |h^\pm(\tilde x,t)|)\big) \,\dS(\tilde x)}{\int_{I_v(t)} \theta\big(\frac{|\tilde x-x|}{e(t)}\big) \,\dS(\tilde x)}
\\&~~~
\nonumber
+C\frac{\int_{I_v(t)} \theta\big(\frac{|\tilde x-x|}{e(t)}\big) \,\mathrm{d}|D^s h^\pm|(\tilde x,t)}{\int_{I_v(t)} 
\theta\big(\frac{|\tilde x-x|}{e(t)}\big) \,\dS(\tilde x)}.
\end{align}
Consider $x\in I_v(t)$. By the assumption from Definition~\ref{definition:domains}, there is a
$C^3$-function $g\colon B_1(0)\subset\R^{d-1}\to\R$ with $\|\nabla g\|_{L^\infty}\leq 1$, $g(0)=0$, and $\nabla g(0)=0$, and such that $I_v(t)\cap B_{2r_c}(x)$
is after rotation and translation given as the graph $\{(x,g(x)):x\in \R^{d-1}\}$. 
Using the fact that $\theta\equiv 0$ on $\R\setminus[0,\frac{1}{2}]$ and $e(t)<r_c\leq 1$, i.e., the map 
$\smash{I_v(t)\ni\tilde x\mapsto \theta(\frac{|\tilde x - x|}{e(t)})}$ is supported in a coordinate patch given by the graph
of $g$, we then may bound
\begin{align*}
\int_{I_v(t)} \theta\Big(\frac{|\tilde x-x|}{e(t)}\Big) \,\dS(\tilde x)
&\leq \int_{I_v(t)\cap B_{\frac{e(t)}{2}}(x)} 1\,\dS(\tilde x)
\leq C \int_{\{\tilde x\in \R^{d-1} \colon |\tilde x|<\frac{e(t)}{2}\}} 1\,\mathrm{d}\tilde x
\\&
\leq Ce(t)^{d-1}.
\end{align*}
We also obtain a lower bound using that $\theta\equiv 1$ on $[0,\frac{1}{4}]$ and again $e(t)<r_c\leq 1$
\begin{align*}
\int_{I_v(t)} \theta\Big(\frac{|\tilde x-x|}{e(t)}\Big) \,\dS(\tilde x) 
&\geq \int_{I_v(t)\cap B_{\frac{e(t)}{4}}(x)} 1\,\dS(\tilde x)
\geq c \int_{\{\tilde x\in \R^{d-1} \colon |\tilde x|<c e(t)\}} 1\,\mathrm{d}\tilde x
\\&
\geq ce(t)^{d-1}.
\end{align*}
In summary, we infer that
\begin{align}
\label{EstimateMollifier}
c e(t)^{d-1} \leq \int_{I_v(t)} \theta\Big(\frac{|\tilde x-x|}{e(t)}\Big) \,\dS(\tilde x)
\leq C e(t)^{d-1}.
\end{align}
Making use of \eqref{EstimateMollifier}, the assumptions $\int_\Rd 1-\xi \cdot \vec{n}_u \,d|\nabla \chi_u|\leq e(t)^2\leq r_c^2 <1$
and $d\leq 3$, the upper bounds $|\theta|\leq 1$ and $G(\lambda p)\leq C(\lambda+\lambda^2) G(p)$, 
as well as the already established $L^2$- resp.\ $H^1$-bound for the local interface error heights $h^\pm$
from \eqref{HeightFunctionEstimate} resp.\ \eqref{HeightFunctionGradientEstimate} we deduce 
\begin{align*}
G(|\nabla h^\pm_{e(t)} (x,t)|) \leq C r_c^{-2},
\end{align*}
which is precisely the first assertion in a). Similarly, one derives 
the other desired estimate $G(e(t) |\nabla^2 h^\pm_{e(t)}(x,t)|) \leq C r_c^{-4}$.

Integrating \eqref{GGradEstimatehEplus} over $I_v(t)$ and employing the global upper bound 
$|\nabla h^\pm_{e(t)}(\cdot,t)|\leq C r_c^{-2}$, which in turn entails 
$\smash{G(|\nabla h^\pm_{e(t)}(\cdot,t)|)\geq c r_c^2 |\nabla h^\pm_{e(t)}(\cdot,t)|^2}$, we get
\begin{align}\label{preliminaryH1BoundInterfaceErrorHeight}
\nonumber
&\int_{I_v(t)} |\nabla h^\pm_{e(t)} (x,t)|^2 \,\dS(x)
\\&
\leq C r_c^{-2} \int_{I_v(t)} \frac{\int_{I_v(t)} \theta\big(\frac{|\tilde x-x|}{e(t)}\big) 
G(|\nabla h^\pm(\tilde x,t)|) + G(r_c^{-1} |h^\pm(\tilde x,t)|) \,\dS(\tilde x)}{\int_{I_v(t)} 
\theta\big(\frac{|\tilde x-x|}{e(t)}\big) \,\dS(\tilde x)} \,\dS(x)
\\&~~~\nonumber
+C r_c^{-2} \int_{I_v(t)} \frac{\int_{I_v(t)} \theta\big(\frac{|\tilde x-x|}{e(t)}\big) 
\,\mathrm{d}|D^s h^\pm|(\tilde x,t)}{\int_{I_v(t)} \theta\big(\frac{|\tilde x-x|}{e(t)}\big) \,\dS(\tilde x)} \,\dS(x).
\end{align}
Applying Fubini's theorem and using the bounds \eqref{EstimateMollifier}, 
$G(\lambda p)\leq C(\lambda+\lambda^2) G(p)$, as well as \eqref{HeightFunctionEstimate} and 
\eqref{HeightFunctionGradientEstimate} 
we deduce the estimate on $\smash{\int_{I_v(t)} |\nabla h^\pm_{e(t)}|^2 \,\dS}$ stated in a). The estimate on 
$\smash{\int_{I_v(t)} |h^\pm_{e(t)}|^2 \,\dS}$ follows by an analogous argument, first squaring \eqref{Definitionhregularized} 
and applying Jensen's inequality, then integrating over $I_v(t)$, and finally using \eqref{EstimateMollifier}, Fubini
as well as \eqref{HeightFunctionEstimate} and \eqref{HeightFunctionGradientEstimate}. 

{\bf Proof of b).}
We start with a change of variables to estimate (recall \eqref{BoundNablaPhiNablaPhi-1})
\begin{align*}
&\int_\Rd |\chi_{v,h^+_{e(t)},h^-_{e(t)}}-\chi_{v,h^+,h^-}| \,\dx 
\\&
\leq C\int_{I_v(t)}\int_{0}^{r_c}|\chi_{0\leq \sigdist(x,I_v(t))\leq h^+_{e(t)}(P_{I_v(t)}x,t)}
-\chi_{0\leq \sigdist(x,I_v(t))\leq h^+(P_{I_v(t)}x,t)}|\,\dy\,\dS
\\&~
+C\int_{I_v(t)}\int_{0}^{r_c}| \chi_{-h^-_{e(t)}(P_{I_v(t)}x,t)\leq \sigdist(x,I_v(t))\leq 0}
- \chi_{-h^-(P_{I_v(t)}x,t)\leq \sigdist(x,I_v(t))\leq 0}|\,\dy\,\dS
\\&
\leq C\int_{I_v(t)} |h^+_{e(t)}(x,t)-h^+(x,t)| + |h^-_{e(t)}(x,t)-h^-(x,t)| \,\dS(x).
\end{align*}
By adding zero and using \eqref{ErrorImprovedInterfaceApproximationNonsmooth}
we therefore obtain
\begin{align*}
\int_\Rd& |\chi_u-\chi_{v,h^+_{e(t)},h^-_{e(t)}}| \,\dx
\\&
\leq
\int_\Rd |\chi_u-\chi_{v,h^+,h^-}| \,\dx
+\int_\Rd |\chi_{v,h^+_{e(t)},h^-_{e(t)}}-\chi_{v,h^+,h^-}| \,\dx
\\&
\leq
C \int_\Rd 1-\xi\cdot \frac{\nabla \chi_u}{|\nabla \chi_u|} \,d|\nabla \chi_u|
\\&~~~
+ C \int_{\Rd} |\chi_u-\chi_v| \min\Big\{\frac{\dist(x,I_v(t))}{r_c},1\Big\}\,\dx
\\&~~~
+C\int_{I_v(t)} |h^+_{e(t)}(x,t)-h^+(x,t)| + |h^-_{e(t)}(x,t)-h^-(x,t)| \,\dS(x).
\end{align*}
Observe that one can decompose
\begin{align*}
h^\pm(x,t)=h^\pm_{e(t)}(x,t) + \sum_{k=0}^\infty \big(h^\pm_{2^{-k-1}e(t)}(x,t)- h^\pm_{2^{-k}e(t)}(x,t)\big).
\end{align*}
A straightforward estimate in local coordinates then yields
\begin{align*}
\int_{I_v(t)}& \big|h_{2^{-k}e(t)}^\pm-h_{2^{-k-1}e(t)}^\pm \big| \,\dS
\\&
\leq C 2^{-k} e(t) \int_{I_v(t)} 1 \,\mathrm{d}|D^{\operatorname{tan}}h^\pm|
\\&
\leq C 2^{-k} e(t) \int_{I_v(t)} 1 \,\mathrm{d}|D^s h^\pm|
+ C 2^{-k} e(t) \int_{I_v(t)} |\nabla h^+| \chi_{\{|\nabla h^+|\geq 1\}} \,\dS
\\&~~~~
+ C 2^{-k} e(t) \bigg(\int_{I_v(t)} |\nabla h^+|^2 \chi_{\{|\nabla h^+|\leq 1\}} \,\dS\bigg)^{1/2} \mathcal{H}^{d-1}(I_v(t))^{1/2}.
\end{align*}
Using \eqref{HeightFunctionGradientEstimate} and summing with respect to 
$k\in \mathbb{N}$, we get the desired estimate \eqref{ErrorImprovedInterfaceApproximationSmooth}.

{\bf Proof of c).}
Note that
\begin{align*}
\int_{I_v(t)} \eta(x) h^\pm_{e(t)}(x,t) \,\dS = \int_{I_v(t)} h^\pm(\tilde x,t) \int_{I_v(t)} 
\frac{\theta\big(\frac{|\tilde x-x|}{e(t)}\big) \eta(x)}{\int_{I_v(t)}\theta\big(\frac{|\hat x-x|}{e(t)}\big)
\,\dS(\hat x)} \,\dS(x) \,\dS(\tilde x).
\end{align*}
Abbreviating
\begin{align*}
\eta_e(\tilde x,t):=\int_{I_v(t)}\frac{\theta\big(\frac{|\tilde x-x|}{e(t)}\big) 
\eta(x)}{\int_{I_v(t)}\theta\big(\frac{|\hat x-x|}{e(t)}\big)\,\dS(\hat x)} \,\dS(x),
\end{align*}
we compute
\begin{align*}
|\nabla_{\tilde x}^\mathrm{tan}\eta_e(\tilde x,t)| 
&=\bigg|\int_{I_v(t)}\frac{\nabla_{\tilde x}^\mathrm{tan}\theta\big(\frac{|\tilde x-x|}{e(t)}\big) 
\eta(x)}{\int_{I_v(t)}\theta\big(\frac{|\hat x-x|}{e(t)}\big)\,\dS(\hat x)} \,\dS(x)\bigg|
\\&
\leq\int_{I_v(t)}\bigg(\frac{\big|\theta'\big|\big(\frac{|\tilde x-x|}{e(t)}\big)}
{\int_{I_v(t)}\theta\big(\frac{|\hat x-x|}{e(t)}\big)\,\dS(\hat x)}\bigg)\frac{\eta(x)}{e(t)} \,\dS(x).
\end{align*}
As in the argument for \eqref{EstimateMollifier}, one checks that 
$\int_{I_v(t)}|\theta'|(\frac{|\tilde x - x|}{e(t)})\,\dS(x)\leq Ce(t)^{d-1}$.
Using the lower bound from \eqref{EstimateMollifier}, the proof for the standard $L^p$-inequality
for convolutions carries over and we obtain $\|\eta_e\|_{L^p(I_v(t))}\leq C\|\eta\|_{L^p(I_v(t))}$
as well as
\begin{align*}
\int_{I_v(t)} |\nabla \eta_e(x,t)|^{p} \,\dS(x)
&\leq \frac{C}{e(t)^p} \int_{I_v(t)} |\eta(x,t)|^p \,\dS(x)
\end{align*}
for any $p \geq 1$.
As a consequence of \eqref{HeightFunctiondtEstimate} and these considerations, we deduce
\begin{align}
\label{HeightFunctionRegularizedAlmostdtEstimate}\nonumber
&\bigg|\frac{\mathrm{d}}{\dt} \int_{I_v(t)} \eta(x) h^\pm_{e(t)}(x,t) \,\dx
-\int_{I_v(t)} h^\pm(\tilde x,t) \frac{\mathrm{d}}{\dt} \eta_e(\tilde x,t) \,\dS(\tilde x)
\\&~~~~~~~~~~~
\nonumber
-\int_{I_v(t)} h^\pm(\tilde x,t) (\mathrm{Id}{-}\vec{n}_v\otimes\vec{n}_v) v(\tilde x,t) 
\cdot \nabla_{\tilde x} \eta_e(\tilde x,t)\,\dS(\tilde x)\bigg|
\\&
\leq 
\frac{C}{e(t)r_c^2}\|\eta\|_{L^{4}(I_v(t))}\Bigg(\int_{I_v(t)} |\bar{h}^\pm|^4 \,\dS \Bigg)^{1/4}
\\&~~~~~~
\nonumber
\times
\Bigg(\int_{I_v(t)}  \sup_{y\in [-r_c,r_c]} |u-v|^2(x+y\vec{n}_v(x,t),t) \,\dS(x)\Bigg)^{1/2} 
\\&~~~~
\nonumber
+C\frac{1+\|v\|_{W^{2,\infty}(\Rd\setminus I_v(t))}}{r_c}\|\eta\|_{L^2(I_v(t))} 
\\&~~~~~~~~\nonumber
\times
\Bigg(\int_{\Rd}|\chi_u(x,t)-\chi_v(x,t)| 
\,\min\Big\{\frac{\dist(x,I_v(t))}{r_c},1\Big\} \,\dx\Bigg)^\frac{1}{2}
\\&~~~~\nonumber
+C\frac{(1+\|v\|_{W^{1,\infty}})}{r_c e(t)}\max_{p\in\{2,4\}}\|\eta\|_{L^p(I_v(t))}
\int_{\Rd} 1-\xi\cdot\frac{\nabla\chi_u}{|\nabla\chi_u|} \,\mathrm{d}|\nabla\chi_u|
\\&~~~~
\nonumber
+C\|\eta\|_{L^2(I_v(t))}\bigg(\int_{I_v(t)} |u-v|^2 \,\dS \bigg)^{1/2}.
\end{align}
Using the estimate $|v(x,t)-v(\tilde x,t)|\leq C |x-\tilde x| \|\nabla v\|_{L^\infty}$, we infer
\begin{align}
\label{AdvectionTerm1}
&\Bigg|\int_{I_v(t)} h^\pm(\tilde x,t) v(\tilde x,t) \cdot \nabla_{\tilde x} \int_{I_v(t)} 
\frac{\theta\big(\frac{|\tilde x-x|}{e(t)}\big) \eta(x)}{\int_{I_v(t)}
\theta\big(\frac{|\hat x-x|}{e(t)}\big)\,\dS(\hat x)} \,\dS(x) \,\dS(\tilde x)
\\&~~~
\nonumber
+\int_{I_v(t)} \eta(x)(v(x,t)\cdot \nabla) h^\pm_{e(t)}(x,t) \,\dS(x)\Bigg|
\\&
\nonumber
=\Bigg|\int_{I_v(t)} \int_{I_v(t)} \eta(x)h^\pm(\tilde x,t) v(\tilde x,t) \cdot \nabla_{\tilde x} 
\frac{\theta\big(\frac{|\tilde x-x|}{e(t)}\big)}{\int_{I_v(t)}
\theta\big(\frac{|\hat x-x|}{e(t)}\big)\,\dS(\hat x)} \,\dS(x) \,\dS(\tilde x)
\\&~~~~~~
\nonumber
+\int_{I_v(t)} \int_{I_v(t)} \eta(x) h^\pm(\tilde x,t) v(x,t) \cdot \nabla_x 
\frac{\theta\big(\frac{|\tilde x-x|}{e(t)}\big)}{\int_{I_v(t)}\theta\big(\frac{|\hat x-x|}{e(t)}\big)
\,\dS(\hat x)} \,\dS(\tilde x) \,\dS(x)\Bigg|
\\&
\nonumber
\leq
\int_{I_v(t)} \int_{I_v(t)} h^\pm(\tilde x,t) \|\nabla v\|_{L^\infty} \frac{|\theta'|\big(\frac{|\tilde x-x|}{e(t)}\big) 
|\tilde x-x| |\eta(x)|}{e(t) \int_{I_v(t)}\theta\big(\frac{|\hat x-x|}{e(t)}\big)\,\dS(\hat x)} \,\dS(x) \,\dS(\tilde x)
\\&~~~
\nonumber
+\int_{I_v(t)} \int_{I_v(t)} h^\pm(\tilde x,t) \|v\|_{L^\infty} \frac{\theta\big(\frac{|\tilde x-x|}{e(t)}\big)
|\eta(x)|\Big|\nabla_x \int_{I_v(t)}\theta\big(\frac{|\hat x-x|}{e(t)}\big)\,\dS(\hat x)\Big|}{\big(\int_{I_v(t)}
\theta\big(\frac{|\hat x-x|}{e(t)}\big)\,\dS(\hat x)\big)^2} \,\dS(x) \,\dS(\tilde x)
\\&
\nonumber
\leq Cr_c^{-1}\|v\|_{W^{1,\infty}}
\bigg(\int_{I_v(t)} |h^\pm(x,t)|^2  \,\dS(x) \bigg)^{1/2} 
\bigg(\int_{I_v(t)} |\eta(x)|^2  \,\dS(x) \bigg)^{1/2}
\end{align}
where in the last step we have used the simple equality
\begin{align}\label{TangentialDerivativeAreaCutoff}
\nabla_x^\mathrm{tan} \int_{I_v(t)}\theta\Big(\frac{|\hat x-x|}{e(t)}\Big)\,\dS(\hat x)
&
=-\int_{I_v(t)}\nabla_{\hat x}^\mathrm{tan} \theta\Big(\frac{|\hat x-x|}{e(t)}\Big)\,\dS(\hat x)
\\&\nonumber
=\int_{I_v(t)} \theta\Big(\frac{|\hat x-x|}{e(t)}\Big) \vec{H}(\hat x) \,\dS(\hat x)
\end{align}
and the bounds \eqref{BoundMeanCurvature} and \eqref{EstimateMollifier}. Recall from the transport theorem for moving hypersurfaces
(see \cite{Pruess2016}) that we have for any $f\in C^1(\Rd\times [0,\Tmax))$
\begin{align}\label{TransportMovingSurface}
\frac{\mathrm{d}}{\dt}\int_{I_v(t)}f(x,t)\,\dS(x) &= \int_{I_v(t)}\partial_t f(x,t)\,\dS(x)
+\int_{I_v(t)} V_{\vec{n}}\cdot\nabla f(x,t)\,\dS(x) 
\\&~~~\nonumber
+\int_{I_v(t)} f(x,t)\,\vec{H}\cdot V_{\vec{n}}\,\dS(x)
\end{align}
with the normal velocity $V_{\vec{n}}(x,t)=(v(x,t)\cdot\vec{n}_v(P_{I_v(t)}x,t))\vec{n}_v(P_{I_v(t)}x,t)$.
Making use of \eqref{TransportMovingSurface} and $\frac{\mathrm{d}}{\dt}P_{I_v(t)}\tilde x = -V_{\vec{n}}(\tilde x,t)$
for $\tilde x\in I_v(t)$ (see \eqref{TimeDerivativeProjection}), we then compute for every $\tilde x \in I_v(t)$
\begin{align*}
&\frac{\mathrm{d}}{\dt} \int_{I_v(t)} \theta\Big(\frac{|\hat x-x|}{e(t)}\Big)  \,\dS(\hat x)
=\frac{\mathrm{d}}{\dt} \int_{I_v(t)} \theta\Big(\frac{|P_{I_v(t)}\hat x-P_{I_v(t)}x|}{e(t)}\Big)  \,\dS(\hat x)
\\&
=-\frac{e'(t)}{e(t)^2}\int_{I_v(t)} \theta'\Big(\frac{| \hat x- x|}{e(t)}\Big) 
|\hat x- x| 
\,\dS(\hat x)
\\&~~~
+\frac{1}{e(t)}\int_{I_v(t)} \theta'\Big(\frac{| \hat x- x|}{e(t)}\Big) 
\frac{( \hat x- x)\cdot(V_{\vec{n}}(\hat x,t)-V_{\vec{n}}(x,t))}{e(t) |\hat x-x|} 
\,\dS(\hat x)
\\&~~~
+\int_{I_v(t)} \theta\Big(\frac{|\hat x-x|}{e(t)}\Big) V_{\vec{n}}(\hat x) \cdot \vec{H}(\hat x) \,\dS(\hat x).
\end{align*}
This together with another application of \eqref{TransportMovingSurface}
and the fact that $\vec{n}_v\cdot\nabla \eta=0$ on the interface $I_v(t)$ implies for $\tilde x\in I_v(t)$
\begin{align}\label{TimeEvolutionMollifiedTestFunction}
&\frac{\mathrm{d}}{\dt}\eta_e(\tilde x,t) 
=\frac{\mathrm{d}}{\dt}\int_{I_v(t)}
\frac{\theta\big(\frac{|P_{I_v(t)}\tilde x-P_{I_v(t)}x|}{e(t)}\big)\eta(x)}
{\int_{I_v(t)}\theta\big(\frac{|P_{I_v(t)}\hat x-P_{I_v(t)}x|}{e(t)}\big)\,\dS(\hat x)}
\,\dS(x)
\\&\nonumber
=\int_{I_v(t)}\bigg(\frac{\theta\big(\frac{|\tilde x - x|}{e(t)}\big)\eta(x)}
{\int_{I_v(t)}\theta\big(\frac{|\hat x -x|}{e(t)}\big)\,\dS(\hat x)}\bigg)
V_{\vec{n}}(x)\cdot\vec{H}(x)\,\dS(x) 
\\\nonumber&~~~
-\int_{I_v(t)}\frac{\theta\big(\frac{|\tilde x - x|}{e(t)}\big)\eta(x)
\big(\int_{I_v(t)}\theta\big(\frac{|\hat x -x|}{e(t)}\big)V_{\vec{n}}(\hat x)\cdot\vec{H}(\hat x)\,\dS(\hat x)\big)}
{\big(\int_{I_v(t)}\theta\big(\frac{|\hat x -x|}{e(t)}\big)\,\dS(\hat x)\big)^2}\,\dS(x) 
\\\nonumber&~~~
+\int_{I_v(t)}\frac{\eta(x)\theta'\big(\frac{|\tilde x - x|}{e(t)}\big)
\frac{(\tilde x-x)\cdot (V_{\vec{n}}(\tilde x)-V_{\vec{n}}( x))}{e(t)|\tilde x-x|}}
{\int_{I_v(t)}\theta\big(\frac{|\hat x -x|}{e(t)}\big)\,\dS(\hat x)}\,\dS(x)
\\\nonumber&~~~
-\int_{I_v(t)}\frac{\theta\big(\frac{|\tilde x - x|}{e(t)}\big)\eta(x)
\int_{I_v(t)}\theta'\big(\frac{|\hat x -x|}{e(t)}\big)\frac{(\hat x-x)\cdot (V_{\vec n}(\hat x,t)-V_{\vec n}(x,t))}{e(t)|\hat x-x|}\,\dS(\hat x)}
{\big(\int_{I_v(t)}\theta\big(\frac{|\hat x -x|}{e(t)}\big)\,\dS(\hat x)\big)^2}\,\dS(x)
\\\nonumber&~~~
-\frac{e'(t)}{e(t)}\int_{I_v(t)}\frac{F'_{e,\theta}(\tilde x,x)\eta(x)}
{\int_{I_v(t)}\theta\big(\frac{|\hat x -x|}{e(t)}\big)\,\dS(\hat x)}\,\dS(x)
\end{align}
where $F'_{e,\theta}(t)\colon I_v(t)\times I_v(t)\to\R$ is the kernel
\begin{align}\label{kernelRightHandSide}
&F'_{e,\theta}(t)(\tilde x,x) := 
\theta'\Big(\frac{|\tilde x - x|}{e(t)}\Big)
\frac{|P_{I_v(t)} \tilde x - P_{I_v(t)} x|}{e(t)}
\\ \nonumber
&~~~~~~~~~~~~~~~~~~~~~-\theta\Big(\frac{|\tilde x {-} x|}{e(t)}\Big)
\frac{\int_{I_v(t)}\theta'\big(\frac{|\hat x - x|}{e(t)}\big)
\frac{|P_{I_v(t)}\hat x-P_{I_v(t)}x|}{e(t)} \,\dS(\hat x)}
{\int_{I_v(t)}\theta\big(\frac{|\hat x - x|}{e(t)}\big)\,\dS(\hat x)}.
\end{align}
Observe that we have
\begin{align}\label{ZeroAvergeRHS}
\int_{I_v(t)}F'_{e,\theta}(t)(\tilde x,x)\,\dS(\tilde x) = 
0.
\end{align}
By the choice of the cutoff $\theta$, we see that for every given $x\in I_v(t)$
the kernel $F'_{e,\theta}(t)$ is supported in $B_{e(t)/2}(x)\cap I_v(t)$. Moreover,
the exact same argumentation which led to the upper bound in \eqref{EstimateMollifier}
(we only used the support and upper bound for $\theta$ as well as $e(t)\leq r_c$)
shows that the kernel $F'_{e,\theta}$
satisfies the upper bound
\begin{align}\label{upperBoundRHS}
\int_{I_v(t)} |F'_{e,\theta}(\tilde x,x)|^p\,\dS(\tilde x) \leq C(p) e(t)^{d-1}
\end{align}
for any $1\leq p<\infty$.
We next intend to rewrite the function $F'_{e,\theta}(\tilde x,x)$ for fixed $x$ as the divergence of a vector field.
By the property \eqref{ZeroAvergeRHS}, we may consider Neumann problem
for the (tangential) Laplacian with right hand side $F'_{e,\theta}(\cdot,x)$ 
in some neighborhood (of scale $e(t)$) of the point $x$. To do this we first rescale
the setup, i.e., we consider the kernel $F'_1(\tilde x,x):=F'_{e,\theta}(e(t)\tilde x,e(t)x)$ 
for $\tilde x,x\in e(t)^{-1}I_v(t)$. By scaling and the fact that $F'_{e,\theta}$ is 
supported on scale $e(t)/2$, it follows that $F'_1(\cdot,x)$ has zero average on 
$e(t)^{-1}I_v(t)\cap B_1(x)$ for every point $x\in e(t)^{-1}I_v(t)$ and that
\begin{align}\label{upperBoundRHSrescaled}
\int_{e(t)^{-1}I_v(t)} |F'_{1}(\tilde x,x)|^p\,\dS(\tilde x) \leq C(p).
\end{align}
We fix $x\in e(t)^{-1}I_v(t)$ and solve on $e(t)^{-1}I_v(t)\cap B_1(x)$ the weak formulation of the
equation $-\Delta^\mathrm{tan}_{\tilde x}\hat F_{1}(\cdot,x) = F'_{1}(\cdot,x)$ with vanishing
Neumann boundary condition. More precisely, we require $\hat F_1(\cdot, x)$ to have vanishing
average on $e(t)^{-1}I_v(t)\cap B_1(x)$ (note that in the weak formulation the curvature term does not appear because it
gets contracted with the tangential derivative of the test function). By elliptic regularity
and \eqref{upperBoundRHSrescaled}, it follows
\begin{align}\label{L2BoundGradientUnitScale}
||\nabla^\mathrm{tan}\hat F_1(\tilde x, x)||_{L^\infty}
\leq C.
\end{align}
We now rescale back to $I_v(t)$ and define $\hat F_{e,\theta}(\tilde x,x):=e(t)^2\hat F_1(e(t)^{-1}\tilde x,e(t)^{-1}x)$
for $x\in I_v(t)$ and $\tilde x\in I_v(t)\cap B_{e(t)}(x)$. For fixed $x\in I_v(t)$, $\hat F_{e,\theta}(\cdot, x)$
has vanishing average on $I_v(t)\cap B_{e(t)}(x)$ and solves $-\Delta^\mathrm{tan}_{\tilde x}\hat F_{e,\theta}(\cdot,x)
=F'_{e,\theta}(\cdot,x)$ on $I_v(t)\cap B_{e(t)}(x)$ with vanishing Neumann boundary condition.
We finally introduce $F_{e,\theta}(\tilde x,x):=\nabla^\mathrm{tan}_{\tilde x}\hat F_{e,\theta}(\tilde x,x)$
for $x\in I_v(t)$ and $\tilde x\in I_v(t)\cap B_{e(t)}(x)$. It then follows from scaling,
\eqref{L2BoundGradientUnitScale} as well as $e(t)<r_c$ that $\nabla_{\tilde x} \cdot F_{e,\theta}(\tilde x,x)=F'_{e,\theta}$ and
\begin{align}\label{ImprovedKernel}
&
||e^{-1}(t)F_{e,\theta}(\tilde x,x)||_{L^\infty}
\leq C.
\end{align}
We now have everything in place to proceed with estimating the term
\begin{align*}
\bigg|\int_{I_v(t)} h^\pm(\tilde x,t) \frac{\mathrm{d}}{\dt} \eta_e(\tilde x,t) \,\dS(\tilde x)\bigg|.
\end{align*}
To this end, we will make use of \eqref{TimeEvolutionMollifiedTestFunction} and estimate term by term.
Because of \eqref{BoundMeanCurvature}, \eqref{EstimateMollifier}, $\|\eta_e\|_{L^p(I_v(t))}\leq C\|\eta\|_{L^p(I_v(t))}$, the estimate
\begin{align*}
\int_{I_v(t)} |\theta'|\Big(\frac{|\tilde x-x|}{e(t)}\Big) \,\dS(\tilde x)
\leq C e(t)^{d-1},
\end{align*}
the Lipschitz property $|V_{\vec n}(x)-V_{\vec n}(\tilde x)|\leq ||\nabla v||_{L^\infty} |x-\tilde x|$, and the fact that $\theta(s)=0$ for $s\geq 1$,
the first four terms on the right-hand side of \eqref{TimeEvolutionMollifiedTestFunction} are straightforward to estimate and result in the bound
\begin{align}\label{auxTimeEta}
Cr_c^{-1}\|v\|_{W^{1,\infty}}\|h^\pm(\cdot,t)\|_{L^2(I_v(t))}\|\eta\|_{L^2(I_v(t))}.
\end{align}
To estimate the fifth term, we first apply Fubini's theorem
and then perform an integration by parts
(recall that we imposed vanishing Neumann boundary conditions)  
which entails because of the above considerations 
\begin{align*}
&\frac{1}{e(t)}\int_{I_v(t)} h^\pm(\tilde x,t) \int_{I_v(t)}\frac{F'_{e,\theta}(\tilde x,x)}
{\int_{I_v(t)}\theta\big(\frac{|\hat x - x|}{e(t)}\big)\,\dS(\hat x)}\eta(x)\,\dS(x) \,\dS(\tilde x) \\
&=\int_{I_v(t)}\bigg(\int_{I_v(t)\cap B_{\frac{3}{4}e(t)}(x)} h^\pm(\tilde x,t) 
\frac{e(t)^{-1}F'_{e,\theta}(\tilde x,x)}
{\int_{I_v(t)}\theta\big(\frac{|\hat x - x|}{e(t)}\big)\,\dS(\hat x)}\,\dS(\tilde x)\bigg)
\eta(x)\,\dS(x)  \\
&=-\int_{I_v(t)}\bigg(\int_{I_v(t)\cap B_{\frac{3}{4}e(t)}}
\nabla_{\tilde x}h^\pm(\tilde x,t)\cdot\frac{e(t)^{-1}F_{e,\theta}(\tilde x,x)}
{\int_{I_v(t)}\theta\big(\frac{|\hat x - x|}{e(t)}\big)\,\dS(\hat x)}\,\dS(\tilde x)\bigg)\eta(x)\,\dS(x) \\
&~~~-\int_{I_v(t)}h^\pm(\tilde x,t)\vec{H}(\tilde x, t)\cdot\bigg(\int_{I_v(t)}\frac{e(t)^{-1}F_{e,\theta}(\tilde x,x)}
{\int_{I_v(t)}\theta\big(\frac{|\hat x - x|}{e(t)}\big)\,\dS(\hat x)}\eta(x)\,\dS(x)\bigg)\,\dS(\tilde x).
\end{align*}
Using \eqref{ImprovedKernel} as well as the lower bound from \eqref{EstimateMollifier} we see
that the second term can be estimated by a term of the form \eqref{auxTimeEta}. 
For the first term, note that by the properties of $F_{e,\theta}$ we may interpret 
the integral in brackets as the mollification of $\nabla h^\pm$ on scale $e(t)$.
Applying the argument which led to \eqref{preliminaryH1BoundInterfaceErrorHeight}
(for this, we only need the upper bound \eqref{ImprovedKernel} for $F_{e,\theta}$,
a lower bound as in \eqref{EstimateMollifier} is only required for $\theta$)
we observe that one can bound this term similar to $\smash{\|\nabla h^\pm_{e(t)}(\cdot,t)\|_{L^2(I_v(t))}}$.
We therefore obtain the bound
\begin{align*}
&\bigg|\int_{I_v(t)} h^\pm(\tilde x,t) \frac{\mathrm{d}}{\dt} \eta_e(\tilde x,t) \,\dS(\tilde x)\bigg| \\
&\leq Cr_c^{-4}\|v\|_{W^{1,\infty}}(1+e'(t))
\bigg(\int_\Rd 1-\xi\cdot \frac{\nabla\chi_u}{|\nabla\chi_u|} \,\mathrm{d}|\nabla\chi_u|\bigg)^{1/2} ||\eta||_{L^2(I_v(t))}
\\&~~
+Cr_c^{-6}\|v\|_{W^{1,\infty}}(1+e'(t))
\bigg(\int_\Rd |\chi_u-\chi_v| \min\Big\{\frac{\dist(x,I_v(t))}{r_c},1\Big\} \,\dx\bigg)^{1/2} ||\eta||_{L^2(I_v(t))}.
\end{align*}
Hence, combining \eqref{HeightFunctionRegularizedAlmostdtEstimate} with these estimates
for the fourth
term from \eqref{TimeEvolutionMollifiedTestFunction}
as well as \eqref{auxTimeEta} and \eqref{AdvectionTerm1}, we obtain the desired estimate on the time derivative.
This concludes the proof.
\end{proof}

\subsection{Construction of the compensation function $w$ for the velocity gradient discontinuity}
We turn to the construction of a compensating vector field, which shall be small in the $L^2$-norm 
but whose associated viscous stress $\mu(\chi_u)\Dsym w$ shall compensate for (most of) the 
problematic viscous term $(\mu(\chi_u)-\mu(\chi_v))\Dsym v$
appearing on the right hand side of the relative entropy inequality from 
Proposition~\ref{PropositionRelativeEntropyInequalityFull} in the case of different shear viscosities.

Before we state the main result of this section, we introduce some further notation.
Let $\smash{h^+_{e(t)}}$ be defined as in Proposition~\ref{PropositionInterfaceErrorHeightRegularized}. 
We then denote by $\smash{P_{h^+_{e(t)}}}$ the downward projection onto the graph of $\smash{h^+_{e(t)}}$, i.e.,
\begin{align*}
P_{h^+_{e(t)}}(x,t) := P_{I_v(t)}x + h^+_{e(t)}(P_{I_v(t)}x,t)\vec{n}_v(P_{I_v(t)}x,t).
\end{align*}
for all $(x,t)$ such that $\dist(x,I_v(t))<r_c$. Note that this map does not define 
an orthogonal projection. Analogously, one introduces the projection $\smash{P_{h^-_{e(t)}}}$ 
onto the graph of $\smash{h^-_{e(t)}}$.

\begin{proposition}\label{PropositionCompensationFunction}
Let $(\chi_u,u,V)$ be a varifold solution to the free boundary problem for 
the incompressible Navier-Stokes equation for two fluids \eqref{EquationTransport}--\eqref{EquationIncompressibility} 
in the sense of Definition~\ref{DefinitionVarifoldSolution} on some time interval $[0,\Tend)$. 
Let $(\chi_v,v)$ be a strong solution to \eqref{EquationTransport}--\eqref{EquationIncompressibility} in the sense of 
Definition~\ref{DefinitionStrongSolution} on some time interval $[0,\Tmax)$ with $\Tmax\leq \Tend$.
Let $\xi$ be the extension of the inner unit normal vector field $\vec{n}_v$ of the interface $I_v(t)$ 
from Definition~\ref{def:ExtNormal}. Let $e\colon [0,\Tmax)\to(0,r_c]$ be a $C^1$-function
and assume that the relative entropy is bounded by $E[\chi_u,u,V|\chi_v,v](t)\leq e(t)^2$.
Let the regularized local interface error heights $h^+_{e(t)}$ and $h^-_{e(t)}$ be defined as in 
Proposition~\ref{PropositionInterfaceErrorHeightRegularized}. 

Then there exists a solenoidal vector field  $w\in L^2([0,\Tmax];H^1(\Rd))$
such that $w$ is subject to the estimates
\begin{align}
\label{Estimatew}
&\int_\Rd |w|^2 \,\dx
\leq C (r_c^{-4}R^2\|v\|^2_{W^{2,\infty}(\Rd\setminus I_v(t))}+1)
\\&\nonumber
~~~~~~~~~~~~~~~~~~~~~~\times
\int_{I_v(t)}
|h^+_{e(t)}|^2 {+} |\nabla h^+_{e(t)}|^2 + |h^-_{e(t)}|^2 {+} |\nabla h^-_{e(t)}|^2 \,\dS,
\end{align}
where $R>0$ is such that $I_v(t)+B_{r_c}\subset B_R(0)$, and 
\begin{align}
\label{EstimateDw}
&\int_{\{\sigdist(x,I_v(t))\geq 0\}} \big|\nabla w - \chi_{0\leq \sigdist(x,I_v(t))\leq h^+_{e(t)}(P_{I_v(t)} x)}
W \otimes \vec{n}_v(P_{I_v(t)}x,t)\big|^2 \,\dx
\\&
\nonumber
+\int_{\{\sigdist(x,I_v(t))\leq 0\}} \big|\nabla w - \chi_{-h^-_{e(t)}(P_{I_v(t)} x)\leq \sigdist(x,I_v(t))\leq 0} 
W \otimes \vec{n}_v(P_{I_v(t)}x,t)\big|^2 \,\dx
\\&
\nonumber
+\int_\Rd \chi_{\sigdist(x,I_v(t))\notin [-h^-_{e(t)}(P_{I_v(t)}x),h^+_{e(t)}(P_{I_v(t)} x)]} |\nabla w|^2 \,\dx
\\&
\nonumber
\leq  C r_c^{-4}\|v\|_{W^{2,\infty}(\Rd\setminus I_v(t))}^2
\int_{I_v(t)} |h^+_{e(t)}|^2 + |\nabla h^+_{e(t)}|^2 + |h^-_{e(t)}|^2 + |\nabla h^-_{e(t)}|^2 \,\dS,
\end{align}
where the vector field $W$ is given by
\begin{align}\label{definitionW}
W(x,t) &:= \frac{2(\mu_+-\mu_-)}{\mu_+\,(1{-}\chi_v)+\mu_-\chi_v}\big(\Id-\vec{n}_v
\otimes\vec{n}_v\big)(P_{I_v(t)}x)\big(\Dsym v\cdot\vec{n}_v(P_{I_v(t)}x)\big),
\end{align}
with the symmetric gradient defined by $\Dsym v := \frac{1}{2}(\nabla v + \nabla v^T)$),
as well as the estimates
\begin{align}
\label{EstimateL2supw}
&\int_{I_v(t)} \sup_{y\in (-r_c,r_c)} |w(x + y\vec{n}_v(x,t))|^2 \,\dS(x)
\\&\nonumber 
\leq C r_c^{-4}\|v\|_{W^{2,\infty}(\Rd\setminus I_v(t))}^2
\int_{I_v(t)} |h^+_{e(t)}|^2 + |\nabla h^+_{e(t)}|^2 + |h^-_{e(t)}|^2 + |\nabla h^-_{e(t)}|^2 \,\dS,
\end{align}
\begin{align}
\label{EstimateLinftyw}
\|\nabla w\|_{L^\infty} &\leq
Cr_c^{-4}|\log e(t)|\|v\|_{W^{2,\infty}(\Rd\setminus I_v(t))} +
Cr_c^{-3}\|\nabla^3 v\|_{L^\infty(\Rd\setminus I_v(t))}
\\\nonumber&~~~
+ C r_c^{-9} \big(1{+}\mathcal{H}^{d-1}(I_v(t))\big)\|v\|_{W^{2,\infty}(\Rd\setminus I_v(t))} ,
\end{align}
\begin{align}
\label{EstimateL2Linftyw}
&\bigg(\int_{I_v(t)}\sup_{y\in[-r_c,r_c]}|(\nabla w)^T(x+y\vec{n}_v(x,t)) \vec{n}_v(x,t)|^2\,\dS(x)\bigg)^\frac{1}{2} 
\\&\nonumber
\leq Cr_c^{-9}(1+\mathcal{H}^{d-1}(I_v(t)))\|v\|_{W^{2,\infty}(\Rd\setminus I_v(t))}e(t)
+Cr_c^{-2}\|v\|_{W^{3,\infty}(\Rd\setminus I_v(t))}e(t)
\\&~~~\nonumber
+Cr_c^{-1}\|v\|_{W^{2,\infty}(\Rd\setminus I_v(t))}|\log e(t)|^\frac{1}{2}e(t)
\end{align}
and
\begin{align}
\label{Estimatedtw}
\partial_t w(\cdot,t) = - \big(v(\cdot,t)\cdot\nabla\big)w(\cdot,t) + g + \hat{g},
\end{align}
where the vector fields $g$ and $\hat{g}$ are subject to the bounds
\begin{align}\label{timeEvolutionWL43}
&\|\hat{g}\|_{L^\frac{4}{3}(\Rd)} \\ 
&\nonumber
\leq C\frac{\|v\|_{W^{1,\infty}}\|v\|_{W^{2,\infty}(\Rd\setminus I_v(t))}}
{e(t)r_c^3}\bigg(\int_{I_v(t)} |\bar{h}^\pm|^4 \,\dS \bigg)^\frac{1}{4}
\\&\nonumber~~~~~~~~~~~~~~
\times\bigg(\int_{I_v(t)} |h^+_{e(t)}|^2 + |\nabla h^+_{e(t)}|^2 
+ |h^-_{e(t)}|^2 + |\nabla h^-_{e(t)}|^2 \,\dS\bigg)^\frac{1}{2} 
\\&~~~\nonumber
+ C\frac{\|v\|_{W^{1,\infty}}}
{e(t)r_c^2}\bigg(\int_{I_v(t)} |\bar{h}^\pm|^4 \,\dS \bigg)^\frac{1}{4}
(\|u{-}v{-}w\|_{L^2}^\frac{1}{2}\|\nabla(u{-}v{-}w)\|_{L^2}^\frac{1}{2}+\|u{-}v{-}w\|_{L^2})
\\&~~~\nonumber
+C\frac{\|v\|_{W^{1,\infty}}(1+\|v\|_{W^{1,\infty}})}{e(t)}
\int_{\Rd} 1-\xi\cdot\frac{\nabla\chi_u}{|\nabla\chi_u|} \,\mathrm{d}|\nabla\chi_u|,
\end{align}
and
\begin{align}\label{timeEvolutionWL2}
&\|g\|_{L^2(\Rd)} \\
&\nonumber
\leq C \frac{1{+}\|v\|_{W^{1,\infty}}}{r_c^2}
(\|\partial_t \nabla v\|_{L^\infty(\Rd\setminus I_v(t))}{+}(R^2{+}1)\|v\|_{W^{2,\infty}(\Rd\setminus I_v(t))})
\\&\nonumber~~~~~~~~~~~~~~
\times\bigg(\int_{I_v(t)} |h^+_{e(t)}|^2 + |\nabla h^+_{e(t)}|^2 + |h^-_{e(t)}|^2 + |\nabla h^-_{e(t)}|^2 \,\dS\bigg)^\frac{1}{2}
\\&~~~\nonumber
+C\frac{\|v\|_{W^{1,\infty}}(1+\|v\|_{W^{1,\infty}})}{e(t)r_c}
\int_{\Rd} 1-\xi\cdot\frac{\nabla\chi_u}{|\nabla\chi_u|} \,\mathrm{d}|\nabla\chi_u|
\\&~~~\nonumber
+Cr_c^{-2}(1+e'(t))\|v\|_{W^{1,\infty}}^2
\bigg(\int_{I_v(t)}|h^\pm|^2\,\dS\bigg)^\frac{1}{2}
\\&~~~\nonumber
+C\frac{\|v\|_{W^{1,\infty}}(1{+}\|v\|_{W^{2,\infty}(\Rd\setminus I_v(t))})}{r_c}
\bigg(\int_{\Rd}|\chi_u{-}\chi_v| \min\Big\{\frac{\dist(x,I_v(t))}{r_c},1\Big\} \,\dx\bigg)^\frac{1}{2}
\\&~~~\nonumber
+C\|v\|_{W^{1,\infty}}(\|u{-}v{-}w\|_{L^2}^\frac{1}{2}\|\nabla(u{-}v{-}w)\|_{L^2}^\frac{1}{2}+\|u{-}v{-}w\|_{L^2}),
\end{align}
where $\bar{h}^\pm$ is defined as $h^\pm$ but now with respect to the modified cut-off function 
$\bar{\theta}(\cdot)=\theta\big(\frac{\cdot}{2}\big)$, see Proposition~\ref{PropositionInterfaceErrorHeight}.
Furthermore, $w$ may be taken to have the regularity $\nabla w(\cdot,t)\in W^{1,\infty}(\Rd\setminus (I_v(t)\cup I_{h^+_e}(t)\cup I_{h^+_e}(t)))$ for almost every $t$, where $I_{h^\pm_e}(t)$ denotes the $C^3$-manifold $\{x\pm h^\pm_{e(t)}(x) \vec{n}_v(x):x\in I_v(t)\}$.
\end{proposition}

\begin{proof}
{\bf Step 1: Definition of $w$.}
Let $\eta$ be a cutoff supported at each $t\in [0,\Tmax)$ in the set $I_v(t)+B_{r_c/2}$ 
with $\eta\equiv 1$ in $I_v(t)+B_{r_c/4}$ and $|\nabla \eta|\leq C r_c^{-1}$, $|\nabla^2 \eta|\leq C r_c^{-2}$ 
as well as $|\partial_t \eta|\leq C r_c^{-1} \|v\|_{L^\infty}$ and 
$|\partial_t \nabla \eta|\leq C r_c^{-2} \|v\|_{W^{1,\infty}}$. For example, one may choose
$\eta(x,t):=\theta(\frac{\dist(x,I_v(t))}{r_c})$ where 
$\theta\colon\mathbb{R}^+\rightarrow [0,1]$ is the smooth cutoff already
used in the definition of the regularized local interface error heights
in Proposition~\ref{PropositionInterfaceErrorHeightRegularized}.

Define the vector field $W$ as given in \eqref{definitionW} and set
(making use of the notation $a\wedge b = \min\{a,b\}$ and $a\vee b = \max\{a,b\}$)
\begin{align}
\label{Definitionwplus}
w^+(x,t):= \eta \int_0^{(\sigdist(x,I_v(t)) \vee 0) \wedge h^+_{e(t)}(P_{I_v(t)} x)} 
W (P_{I_v(t)} x + y\vec{n}_v(P_{I_v(t)}x,t)) \,\dy
\end{align}
as well as
\begin{align}
\label{Definitionwminus}
w^-(x,t):= \eta \int_{(\sigdist(x,I_v(t)) \wedge 0) \vee -h^-_{e(t)}(P_{I_v(t)} x)}^0 
W (P_{I_v(t)} x + y\vec{n}_v(P_{I_v(t)}x,t)) \,\dy.
\end{align}
For this choice, we have
\begin{align}\label{Dwplus}
&\nabla w^+ (x,t) 
\\&\nonumber
= \chi_{0\leq \sigdist(x,I_v(t))\leq h^+_{e(t)}(P_{I_v(t)} x)} W(x) \otimes \vec{n}_v(P_{I_v(t)}x)
\\&~~~~
\nonumber
+ \eta\, \chi_{\sigdist(x,I_v(t))>h^+_{e(t)}(P_{I_v(t)} x)} W(P_{h^+_{e(t)}} x) \otimes \nabla h^+_{e(t)}(P_{I_v(t)}x) \nabla P_{I_v(t)}(x)
\\&~~~~
\nonumber
+\eta \int_0^{(\sigdist(x,I_v(t)) \vee 0) \wedge h^+_{e(t)}(P_{I_v(t)} x)} \nabla 
W (P_{I_v(t)} x {+} y \vec{n}_v(P_{I_v(t)}x)) (\nabla P_{I_v(t)}x{+}y\nabla \vec{n}_v(P_{I_v(t)}x)) \,\dy
\\&~~~~
\nonumber
+\nabla \eta \int_0^{(\sigdist(x,I_v(t)) \vee 0) \wedge h^+_{e(t)}(P_{I_v(t)} x)} W (P_{I_v(t)} x {+} y \vec{n}_v(P_{I_v(t)}x)) \,\dy
\end{align}
(note that this directly implies the last claim about the regularity of $w$, namely $\nabla w(\cdot,t)\in W^{1,\infty}(\Rd\setminus (I_v(t)\cup I_{h^+_e}(t)\cup I_{h^+_e}(t)))$ for almost every $t$)
as well as
\begin{align}\label{dtwplus}
&\partial_t w^+(x,t)
\\&\nonumber
=\chi_{0\leq \sigdist(x,I_v(t))\leq h^+_{e(t)}(P_{I_v(t)} x)} W (x) \partial_t \sigdist(x,I_v(t))
\\&~~~~
\nonumber
+\eta \, \chi_{\sigdist(x,I_v(t))> h^+_{e(t)}(P_{I_v(t)} x)} W (P_{h^+_{e(t)}} x) 
\big(\partial_t h^+_{e(t)}(P_{I_v(t)}x)+\partial_t P_{I_v(t)}x \cdot \nabla h^+_{e(t)}(P_{I_v(t)}x)\big)
\\&~~~~
\nonumber
+\eta \int_0^{(\sigdist(x,I_v(t)) \vee 0) \wedge h^+_{e(t)}(P_{I_v(t)} x)} 
\partial_t W (P_{I_v(t)} x {+} y \vec{n}_v(P_{I_v(t)}x)) \,\dy
\\&~~~~
\nonumber
+\eta \int_0^{(\sigdist(x,I_v(t)) \vee 0) \wedge h^+_{e(t)}(P_{I_v(t)} x)} 
\nabla W (P_{I_v(t)} x {+} y \vec{n}_v(P_{I_v(t)}x)) (\partial_t P_{I_v(t)}x{+}y\partial_t \vec{n}_v(P_{I_v(t)}x)) \,\dy
\\&~~~~
\nonumber
+\partial_t \eta \int_0^{(\sigdist(x,I_v(t)) \vee 0) 
\wedge h^+_{e(t)}(P_{I_v(t)} x)} W (P_{I_v(t)} x + y \vec{n}_v(P_{I_v(t)}x)) \,\dy.
\end{align}
Moreover, note that \eqref{Dwplus} entails by the definition of the vector field $W$
\begin{align}\label{divwplus}
&\nabla \cdot w^+ (x,t)
\\&\nonumber
=\eta\, \chi_{\sigdist(x,I_v(t))>h^+_{e(t)}(P_{I_v(t)} x)} W(P_{h^+_{e(t)}} x) \cdot \nabla h^+_{e(t)}(P_{I_v(t)}x) \nabla P_{I_v(t)}(x)
\\&~~~
\nonumber
+\eta \int_0^{(\sigdist(x,I_v(t)) \vee 0) \wedge h^+_{e(t)}(P_{I_v(t)} x)} 
\operatorname{tr} \nabla W (P_{I_v(t)} x {+} y \vec{n}_v(P_{I_v(t)}x)) 
(\nabla P_{I_v(t)}x{+}y\nabla \vec{n}_v(P_{I_v(t)}x)) \,\dy
\\&~~~
\nonumber
+\nabla \eta \cdot \int_0^{(\sigdist(x,I_v(t)) \vee 0) \wedge h^+_{e(t)}(P_{I_v(t)} x)}
 W (P_{I_v(t)} x {+} y \vec{n}_v(P_{I_v(t)}x)) \,\dy.
\end{align}
Analogous formulas and properties can be derived for $w^-$.
The function $w^+ + w^-$ would then satisfy our conditions, with the exception 
of the solenoidality $\nabla \cdot w=0$. For this reason, we introduce the (usual) kernel
\begin{align*}
\theta(x):=\frac{1}{\mathcal{H}^{d-1}(\mathbb{S}^{d-1})}\frac{x}{|x|^d}
\end{align*}
and set
\begin{align}\label{DefinitionCompensationFunction}
w(x,t) &:= w^+(x,t)-(\theta \ast \nabla \cdot w^+)(x,t)
+w^-(x,t)-(\theta \ast \nabla \cdot w^-)(x,t).
\end{align}
It is immediate that $\nabla \cdot w=0$. 

{\bf Step 2: Estimates on $w$ and $\nabla w$.}
From \eqref{Dwplus}, $|\nabla\eta|\leq Cr_c^{-1}$ as well as the bounds \eqref{Bound2ndDerivativeSignedDistance}
and \eqref{BoundGradientProjection} we deduce the pointwise bound
\begin{align}
\nonumber
&\big|\nabla w^+ - \chi_{0\leq \sigdist(x,I_v(t))\leq h^+_{e(t)}(P_{I_v(t)} x)} W(x) \otimes\vec{n}_v(P_{I_v(t)}x)\big|
\\&
\label{BoundPointwiseDwplus}
\leq C \chi_{\supp \eta} r_c^{-1} \|\nabla v\|_{L^\infty} |\nabla h^+_{e(t)}(P_{I_v(t)} x)| 
\\&~~~~
\nonumber
+C \chi_{\supp \eta}  \big(r_c^{-2} \|\nabla v\|_{L^\infty} + r_c^{-1}
\|\nabla^2 v\|_{L^\infty(\Rd\setminus I_v(t))}\big)  |h^+_{e(t)}(P_{I_v(t)} x)|
\\&~~~~
\nonumber
+C r_c^{-1} \chi_{\supp \eta} \|\nabla v\|_{L^\infty} |h^+_{e(t)}(P_{I_v(t)} x)|
\end{align}
and therefore by integration and a change of variables $\Phi_t$
\begin{align}\label{auxGradientCompensationWplus}
&\int_\Rd \Big|\nabla w^+ - \chi_{0\leq \sigdist(x,I_v(t))\leq h^+_{e(t)}(P_{I_v(t)} x)} W(x)
\otimes\vec{n}_v(P_{I_v(t)}x)\Big|^2 \,\dx
\\&\nonumber
\leq C (r_c^{-4}\|\nabla v\|_{L^\infty}^2+r_c^{-2}\|\nabla^2 v\|_{L^\infty(\Rd\setminus I_v(t))}^2)
\int_\Rd \chi_{\mathrm{supp}\,\eta}(|h^+_{e(t)}|^2 + |\nabla h^+_{e(t)}|^2)(P_{I_v(t)}x)\,\dx
\\&\nonumber
\leq C r_c^{-4}\|v\|^2_{W^{2,\infty}(\Rd\setminus I_v(t))}
\int_{I_v(t)} |h^+_{e(t)}|^2 + |\nabla h^+_{e(t)}|^2 \,\dS.
\end{align}
Observe that this also implies by \eqref{definitionW}
\begin{align}\label{auxL2divWplus}
&\int_\Rd |\nabla\cdot w^+|^2\,\dx 
\leq C r_c^{-4}\|v\|^2_{W^{2,\infty}(\Rd\setminus I_v(t))}
\int_{I_v(t)} |h^+_{e(t)}|^2 + |\nabla h^+_{e(t)}|^2 \,\dS.
\end{align}
From this, Theorem~\ref{SingularIntegralOperator}, and the fact that 
$\nabla \theta$ is a singular integral kernel subject to the assumptions of 
Theorem~\ref{SingularIntegralOperator}, we deduce
\begin{align}
\int_\Rd& \big|\nabla (\theta \ast (\nabla \cdot w^+)) \big|^2 \,\dx
\label{EstimateDwdiv}
\leq C r_c^{-4}\|v\|_{W^{2,\infty}(\Rd\setminus I_v(t))}^2
\int_{I_v(t)} |h^+_{e(t)}|^2 + |\nabla h^+_{e(t)}|^2 \,\dS.
\end{align}
Combining the estimates \eqref{auxGradientCompensationWplus} and \eqref{EstimateDwdiv} 
with the corresponding inequalities for $w^-$ 
and $\theta \ast \nabla\cdot w^-$, we deduce our estimate \eqref{EstimateDw}.

The trivial estimate $|w^+(x,t)|\leq \chi_{\supp \eta}(x,t) \|\nabla v\|_{L^\infty} h^+_{e(t)}(P_{I_v(t)}x)$ gives
by the change of variables $\Phi_t 	$
\begin{align}\label{L2Wplus}
&\int_\Rd |w^+|^2 \,\dx
\leq C r_c \int_{I_v(t)} |h^+_{e(t)}|^2 \,\dS.
\end{align}
Now, let $R>1$ be big enough such that $I_v(t)+B_{r_c}\subset B_R(0)$ for all $t\in [0,\Tmax)$.
We then estimate with an integration by parts and Theorem~\ref{SingularIntegralOperator}
applied to the singular integral operator $\nabla\theta$ 
\begin{align}\label{auxL2ConvolutionWplus}\nonumber
\int_{\Rd\setminus B_{3R}(0)} \big| \theta \ast (\nabla \cdot w^+)\big|^2 \,\dx
&= \int_{\Rd\setminus B_{3R}(0)}\bigg|\int_{B_R(0)}\theta(x-\tilde x)
(\nabla\cdot w^+(\tilde x))\,\mathrm{d}\tilde x\bigg|^2\,\dx
\\&\nonumber
\leq \int_{\Rd}\bigg|\int_{B_R(0)}
\nabla\theta(x-\tilde x)w^+(\tilde x)
\,\mathrm{d}\tilde x\bigg|^2\,\dx
\\&
\leq C\int_{B_R(0)}|w^+|^2\,\dx.
\end{align}
By Young's inequality for convolutions, \eqref{auxL2divWplus}, \eqref{L2Wplus} and \eqref{auxL2ConvolutionWplus}
we then obtain
\begin{align}\nonumber
&\int_\Rd \big| \theta \ast (\nabla \cdot w^+)\big|^2 \,\dx
\\&
\nonumber
=\int_{B_{3R}(0)} \big| \theta \ast (\nabla \cdot w^+)\big|^2 \,\dx
+\int_{\Rd\setminus B_{3R}(0)} \big| \theta \ast (\nabla \cdot w^+)\big|^2 \,\dx
\\&\label{Estimatewdiv}
\leq C\bigg(\int_{B_{3R}(0)}\frac{1}{|x|^{d-1}}\,\dx\bigg)^2 \int_\Rd |\nabla \cdot w^+ |^2 \,\dx
+C\int_{\Rd}|w^+|^2\,\dx
\\&\nonumber
\leq C (r_c^{-4}R^2\|v\|^2_{W^{2,\infty}(\Rd\setminus I_v(t))}+1)
\int_{I_v(t)} |h^+_{e(t)}|^2 + |\nabla h^+_{e(t)}|^2 \,\dS.
\end{align}
Together with the respective estimates for $w^-$ and $\theta \ast (\nabla \cdot w^-)$, this implies \eqref{Estimatew}.
The estimate \eqref{EstimateL2supw} follows directly from \eqref{Definitionwplus} and 
the estimates \eqref{EstimateDwdiv} and \eqref{Estimatewdiv} on the $H^1$-norm of $\theta \ast (\nabla \cdot w^+)$ as well as 
the definition of $w^-$ and the analogous estimates for $\theta \ast (\nabla \cdot w^-)$. 

{\bf Step 3: $L^\infty$-estimates for $\nabla w$.}
Regarding the estimate \eqref{EstimateLinftyw} on $\|\nabla w\|_{L^\infty}$ 
we have by \eqref{BoundPointwiseDwplus} and the estimates $|\nabla h^+_{e(t)}|\leq Cr_c^{-2}$ 
and $|h^+_{e(t)}|\leq r_c\leq 1$ from Proposition~\ref{PropositionInterfaceErrorHeightRegularized}
\begin{align}
\label{EstimateDwplus}
\|\nabla w^+\|_{L^\infty}\leq Cr_c^{-4}\|v\|_{W^{2,\infty}(\Rd\setminus I_v(t))}.
\end{align}
To estimate $|\nabla (\theta \ast (\nabla \cdot w^+))|$, we first compute starting with \eqref{divwplus}
\begin{align}
\label{BoundD2wplus}
&\nabla (\nabla \cdot w^+) (x,t)
\\&
\nonumber
=\eta\, \chi_{\sigdist(x,I_v)>h^+_{e(t)}(P_{I_v(t)} x)} W(P_{h^+_{e(t)}} x) \cdot \nabla^2 h^+_{e(t)}(P_{I_v(t)}x) 
\nabla P_{I_v(t)}(x)\nabla P_{I_v(t)}(x)
\\&~~~~
\nonumber
+\big(W(P_{h^+_{e(t)}} x) \cdot \nabla h^+_{e(t)}(P_{I_v(t)}x) \nabla P_{I_v(t)}(x)\big) 
\, \nabla \chi_{\sigdist(x,I_v)>h^+_{e(t)}(P_{I_v(t)} x)}
\\&~~~~
\nonumber
+F(x,t),
\end{align}
where $F(x,t)$ is subject to a bound of the form 
$|F(x,t)|\leq C r_c^{-5}\|v\|_{W^{3,\infty}(\Rd\setminus I_v(t))})$ 
and supported in $I_v(t)+B_{r_c}$. Next, we decompose the kernel $\theta$ as $\theta = \sum_{k=-\infty}^\infty \theta_k$ 
with smooth functions $\theta_k$ with $\supp \theta_k \subset B_{2^{k+1}}\setminus B_{2^{k-1}}$. More precisely,
we first choose a smooth function $\varphi\colon\R_+\to [0,1]$ such that $\varphi(s)=0$
whenever $s\notin [-1/2,2]$ and such that $\sum_{k\in\Z}\varphi(2^ks)=1$ for
all $s>0$. Such a function indeed exists, see for instance \cite{Bahouri2011}. 
We then let $\theta_k(x):=\varphi(2^k|x|)\theta(x)$.
Note that $\|\theta_k\|_{L^1(\Rd)}\leq C 2^k$, $\|\nabla \theta_k\|_{L^1(\Rd)}\leq C$
as well as $|\nabla \theta_k|\leq C (2^k)^{-d}$.
We estimate
\begin{align}\label{dyadicDecomposition}
|\nabla (\theta \ast (\nabla \cdot w^+))|
&\leq \sum_{k=\lfloor \log e^2(t)\rfloor}^0 
|\nabla (\theta_k \ast (\nabla \cdot w^+))|
+\sum_{k=1}^\infty
|\nabla (\theta_k \ast (\nabla \cdot w^+))|
\\&~~~\nonumber
+\sum_{k=-\infty}^{\lfloor \log e^2(t)\rfloor-1}
|\theta_k \ast \nabla(\nabla \cdot w^+)|.
\end{align}
Using Young's inequality for convolutions as well as the estimate $\|\nabla \theta_k\|_{L^1(\Rd)}\leq C$
we obtain
\begin{align}\label{middleFrequencies}
\sum_{k=\lfloor \log e^2(t)\rfloor}^0 
|\nabla (\theta_k \ast (\nabla \cdot w^+))|
\leq 2C|\log e(t)|\|\nabla\cdot w^+\|_{L^\infty}.
\end{align}
Moreover, it follows from $|\nabla \theta_k|\leq C (2^k)^{-d}$, the precise formula for $\nabla\cdot w^+$
in \eqref{divwplus}, \eqref{Bound2ndDerivativeSignedDistance}, \eqref{BoundGradientProjection}, a change of variables 
and H\"older's inequality that
\begin{align}\label{longFrequencies}
&\sum_{k=1}^\infty
|\nabla (\theta_k \ast (\nabla \cdot w^+))|
\\&\nonumber
\leq Cr_c^{-2}\|v\|_{W^{2,\infty}(\Rd\setminus I_v(t))}\sum_{k=1}^\infty (2^k)^{-d}
\int_{I_v(t)+B_{r_c/2}} |\nabla h^+_{e(t)}(P_{I_v(t)}x)| + |h^+_{e(t)}(P_{I_v(t)}x)| \,\dx
\\&\nonumber
\leq Cr_c^{-2}\|v\|_{W^{2,\infty}(\Rd\setminus I_v(t))}\sqrt{\mathcal{H}^{d-1}(I_v(t))}
\bigg(\int_{I_v(t)} |\nabla h^+_{e(t)}|^2+|h^+_{e(t)}|^2\,\dS\bigg)^\frac{1}{2}.
\end{align}
Using \eqref{BoundD2wplus}, the estimate 
$|\nabla^2 h^\pm_{e(t)}(\cdot,t)| \leq C r_c^{-4} e(t)^{-1}$ from Proposition~\ref{PropositionInterfaceErrorHeightRegularized},
\eqref{Bound2ndDerivativeSignedDistance}, \eqref{BoundGradientProjection} 
and again Young's inequality for convolutions (recall that $\|\theta_k\|_{L^1(\Rd)}\leq C 2^k$), we get
\begin{align}
\label{shortFrequencies0}
&\sum_{k=-\infty}^{\lfloor \log e^2(t)\rfloor-1}
|\theta_k \ast \nabla(\nabla \cdot w^+)|(\tilde x,t) \leq I + II + III
\end{align}
where the three terms on the right hand side are given by
\begin{align}\label{shortFrequencies1}
I := \sum_{k=-\infty}^{\lfloor \log e^2(t)\rfloor-1} 2^{k} 
Cr_c^{-5}\|v\|_{W^{3,\infty}(\Rd\setminus I_v(t))}
\leq Cr_c^{-5}\|v\|_{W^{3,\infty}(\Rd\setminus I_v(t))}e^2(t)
\end{align}
and
\begin{align}\label{shortFrequencies2}
II := Cr_c^{-5}\| v\|_{W^{1,\infty}} e(t)^{-1} \sum_{k=-\infty}^{\lfloor \log e^2(t)\rfloor-1} 2^{k} 
\leq Cr_c^{-5}\| v\|_{W^{1,\infty}} e(t)
\end{align}
as well as
\begin{align}\label{shortFrequencies3}
&III :=
\sum_{k=-\infty}^{\lfloor \log e^2(t) \rfloor -1} \bigg|\int_\Rd \theta_k(x{-}\tilde x) 
\otimes \big(W(P_{h^+_{e(t)}} x) \cdot  (\nabla P_{I_v(t)})^T(x) \nabla h^+_{e(t)}(P_{I_v(t)} x)\big)
\\&~~~~~~~~~~~~~~~~~~~~~~~~~~~~~~~~~~~~~~~~~~~~~~~~~~~~~~~~~~~~~~\nonumber
\,\mathrm{d}\nabla \chi_{\sigdist(x,I_v(t))>h^+_{e(t)}(P_{I_v(t)}x)}(x)\bigg|.
\end{align}

To estimate the latter term, we proceed as follows. First of all, note that
by the definition of $\smash{h^+_{e(t)}}$ in \eqref{Definitionhregularized}
as well as the trivial bound $|h^+|\leq r_c$ it holds $\smash{|h^+_{e(t)}|}\leq r_c$.
Then for all $\tilde x\in I_v(t)+\{|x|> r_c+2^{\lfloor \log e^2(t) \rfloor}\}$
and all $k\leq \lfloor \log e^2(t) \rfloor - 1$ we observe that 
$\chi_{\{\sigdist(x,I_v(t))>h^+_{e(t)}(P_{I_v(t)}x)\}}(x)=1$ for all $x\in\Rd$
such that $|x-\tilde x|\leq 2^{k+1}$. In particular, for such $\tilde x$
the third term on the right hand side of \eqref{shortFrequencies0} vanishes
since the corresponding second term in the formula for $\nabla(\nabla\cdot w^+)$
(see \eqref{BoundD2wplus}) does not appear anymore.

Hence, let $\tilde x\in I_v(t)+\{|x|\leq r_c+2^{\lfloor \log e^2(t) \rfloor}\}$ and denote 
by $F$ the tangent plane to the manifold $\{\sigdist(x,I_v(t))=h^+_{e(t)}(P_{I_v(t)}x)\}$ 
at the nearest point to $\tilde x$. We then have for any $\psi\in C^\infty_{cpt}(\Rd)$
\begin{align*}
&\int_\Rd \psi(x) \,\mathrm{d}\nabla \chi_{\{\sigdist(x,I_v(t))>h^+_{e(t)}(P_{I_v(t)}x)\}}(x)
-\int_\Rd \psi(x) \,\mathrm{d}\nabla \chi_{\{\sigdist(x,F)>0\}}(x)
\\&
=\int_{\{\sigdist(x,I_v(t))>h^+_{e(t)}(P_{I_v(t)}x)\}} \nabla \psi(x) \,\mathrm{d}x
-\int_{\{\sigdist(x,F)>0\}} \nabla \psi(x) \,\mathrm{d}x
\end{align*}
and as a consequence
\begin{align*}
&\int_\Rd \theta_k(x-\tilde x) \otimes \big(W(P_{h^+_{e(t)}} x) \cdot  (\nabla P_{I_v(t)})^T(x) 
\nabla h^+_{e(t)}(P_{I_v(t)} x)\big) \,\mathrm{d}\nabla \chi_{\{\sigdist(x,I_v(t))>h^+_{e(t)}(P_{I_v(t)}x)\}}(x)
\\&
=\int_F \theta_k(x-\tilde x) \otimes \big(W(P_{h^+_{e(t)}} x) \cdot  
(\nabla P_{I_v(t)})^T(x) \nabla h^+_{e(t)}(P_{I_v(t)} x)\big) \vec{n}_F \,\dS(x)
\\&~~~~
+\int_\Rd (\chi_{\{\sigdist(x,I_v(t))>h^+_{e(t)}(P_{I_v(t)}x)\}}-\chi_{\{\sigdist(x,F)>0\}})
\\&~~~~~~~~~~~~~~
\nabla \big(\theta_k(x-\tilde x) \otimes \big(W(P_{h^+_{e(t)}} x) \cdot  
(\nabla P_{I_v(t)})^T(x) \nabla h^+_{e(t)}(P_{I_v(t)} x)\big)\big) \,\dx.
\end{align*}
Recall that we defined $\theta_k(x):=\varphi(2^k|x|)\theta(x)$ where
$\varphi\colon\R_+\to [0,1]$ is a smooth function such that $\varphi(s)=0$
whenever $s\notin [-1/2,2]$ and such that $\sum_{k\in\Z}\varphi(2^ks)=1$ for
all $s>0$. Hence,
$|\vec{n}_F \cdot \theta_k(x-\tilde x)|\leq 
C\frac{|\vec{n}_F\cdot(x-\cdot\tilde x)|}{|x-\tilde x|^d}
\leq C\frac{\dist(\tilde x,F)}{|x-\tilde x|^d}$ 
for all $x\in F$. It also follows from the definition
of $\theta$ that $\int_F (\Id-\vec{n}_F\otimes \vec{n}_F) \theta_k(x-\tilde x)\,\dS(x)=0$.
Hence we may solve $(\Id-\vec{n}_F\otimes \vec{n}_F) \theta_k(\cdot\,-\tilde x)
=\Delta_x^{\mathrm{tan}}\tilde\theta_k(\cdot,\tilde x)$ on $B_{2^{k+2}}(\tilde x)\cap F$ 
with vanishing Neumann boundary conditions. In particular, for $\hat\theta_k(x,\tilde x):=
\nabla^{\mathrm{tan}}_x\tilde\theta_k(x,\tilde x)$ we obtain
$(\Id-\vec{n}_F\otimes \vec{n}_F) \theta_k(x-\tilde x)=\nabla^{\mathrm{tan}}_x\cdot\nabla_x\hat\theta_k(x,\tilde x)$.
It follows from elliptic regularity that $\hat\theta(\cdot,\tilde x)$ is $C^\infty$. 
Moreover, since we could have rescaled $\theta_k$ first 
to unit scale, then solved the associated problem on that scale, and finally rescaled the solution back 
to the dyadic scale $k$ we see that $|\hat \theta_k(x,\tilde x)|\leq C (2^k)^{2-d}$. 
We then have by an integration by parts
\begin{align*}
\bigg|\int_F(\Id-\vec{n}_F\otimes \vec{n}_F) \theta_k(x-\tilde x)\otimes \psi \,\dS(x)\bigg|
&\leq \int_{F\cap B_{2^{k+1}}(\tilde x)}|\hat \theta_k(x,\tilde x)||\nabla^{\mathrm{tan}}\psi|\,\dS(x) \\
&\leq C(2^k)^{2-d}\int_{F\cap B_{2^{k+1}}(\tilde x)}|\nabla^{\mathrm{tan}}\psi|\,\dS(x)
\end{align*}
for any $\psi\in C^1_{cpt}(\Rd;\Rd)$. Furthermore, it holds
\begin{align*}
\int_{B_{2^k}(\tilde x)} |\chi_{\{\sigdist(x,I_v(t))>h^+_{e(t)}(P_{I_v(t)}x)\}}-\chi_{\{\sigdist(x,F)>0\}}| \,\dx
&\leq C \|\nabla^2 h^+_{e(t)}\|_{L^\infty} (2^k)^{d+1}.
\end{align*}
Using these considerations in the previous formula, we obtain
\begin{align}\label{auxShortFrequencies}
&\bigg|\int_\Rd \theta_k(x-\tilde x) \otimes \big(W(P_{h^+_{e(t)}} x) \cdot  
(\nabla P_{I_v(t)})^T(x) \nabla h^+_{e(t)}(P_{I_v(t)} x)\big) 
\,\mathrm{d}\nabla \chi_{\{\sigdist(x,I_v(t))>h^+_{e(t)}(P_{I_v(t)}x)\}}(x)\bigg|
\\&\nonumber
\leq
\int_{F\cap B_{2^{k+1}}(\tilde x)\setminus B_{2^{k-1}}(\tilde x)} 
\frac{\dist(\tilde x,F)}{|\tilde x-x|^d} |W(P_{h^+_{e(t)}} x) \cdot  (\nabla P_{I_v(t)})^T(x) \nabla h^+_{e(t)}(P_{I_v(t)} x)| \,\dS(x)
\\&~~\nonumber
+\int_{F\cap B_{2^{k+1}}(\tilde x)} C (2^k)^{2-d} |\nabla (W(P_{h^+_{e(t)}} x)
 \cdot (\nabla P_{I_v(t)})^T(x) \nabla h^+_{e(t)}(P_{I_v(t)}x))| \,\dS(x)
\\&~~\nonumber
+C \|\nabla^2 h^+_{e(t)}\|_{L^\infty} (2^k)^{d+1} \big\|\nabla  
\big(\theta_k(x-\tilde x) \otimes \big(W(P_{h^+_{e(t)}} x) \cdot  (\nabla P_{I_v(t)})^T(x) 
\nabla h^+_{e(t)}(P_{I_v(t)} x)\big)\big)\big\|_{L^\infty}.
\end{align}
Making use of the fact that the integral vanishes for $\dist(\tilde x,F)\geq 2^{k+1}$
and the bounds \eqref{Bound2ndDerivativeSignedDistance} and \eqref{BoundGradientProjection}
we obtain
\begin{align}\label{shortFrequencies4}
&\int_{F\cap B_{2^{k+1}}(\tilde x)\setminus B_{2^{k-1}}(\tilde x)} 
\frac{\dist(\tilde x,F)}{|\tilde x-x|^d} |W(P_{h^+_{e(t)}} x) \cdot  (\nabla P_{I_v(t)})^T(x) \nabla h^+_{e(t)}(P_{I_v(t)} x)| \,\dS(x)
\\&\nonumber
\leq \chi_{\{\dist(\tilde x,F)< 2^{k}\}}Cr_c^{-3}\|v\|_{W^{1,\infty}}\frac{\dist(\tilde x, F)}{2^k}
\int_{F\cap B_{2^{k+1}}(\tilde x)\setminus B_{2^{k-1}}(\tilde x)} 
\frac{|\nabla h^+_{e(t)}(P_{I_v(t)} x)|}{|\tilde x -x|^{d-1}} \,\dS(x).
\end{align}
Using also $|\nabla h^+_{e(t)}|\leq Cr_c^{-2}$
and $|\nabla^2 h^+_{e(t)}|\leq Cr_c^{-4}e(t)^{-1}$ 
from Proposition~\ref{PropositionInterfaceErrorHeightRegularized}, we get
\begin{align}\label{shortFrequencies5}
&\int_{F\cap B_{2^{k+1}}(\tilde x)} C (2^k)^{2-d} |\nabla (W(P_{h^+_{e(t)}} x)
\cdot (\nabla P_{I_v(t)})^T(x) \nabla h^+_{e(t)}(P_{I_v(t)}x))| \,\dS(x)
\\&\nonumber
\leq C 2^k \big(e(t)^{-1} r_c^{-5} \|v\|_{W^{1,\infty}}+ 
r_c^{-4}\|v\|_{W^{2,\infty}(\Rd\setminus I_v(t))} \big)
\end{align}
and
\begin{align}\label{shortFrequencies6}
&C \|\nabla^2 h^+_{e(t)}\|_{L^\infty} (2^k)^{d+1} \big\|\nabla  
\big(\theta_k(x-\tilde x) \otimes \big(W(P_{h^+_{e(t)}} x) \cdot  (\nabla P_{I_v(t)})^T(x) 
\nabla h^+_{e(t)}(P_{I_v(t)} x)\big)\big)\big\|_{L^\infty}
\\&\nonumber
\leq C r_c^{-4} e(t)^{-1} 2^k r_c^{-3} \|v\|_{W^{1,\infty}} 
\\&~~~\nonumber
+ C r_c^{-4} e(t)^{-1} (2^k)^2 \big(e(t)^{-1} r_c^{-5} \|v\|_{W^{1,\infty}}+ 
r_c^{-4}\|v\|_{W^{2,\infty}(\Rd\setminus I_v(t))} \big).
\end{align}
Using \eqref{auxShortFrequencies}, \eqref{shortFrequencies4}, \eqref{shortFrequencies5}
and \eqref{shortFrequencies6} to estimate the term in \eqref{shortFrequencies3},
we get
\begin{align}\label{shortFrequencies7}
III &\leq C\frac{\|v\|_{W^{1,\infty}}}{r_c^3}\sum_{k=-\infty}^{\lfloor \log e^2(t) \rfloor -1}
\chi_{\{\dist(\tilde x,F)< 2^{k}\}}\frac{\dist(\tilde x, F)}{2^k}
\int_{F\cap B_{2^{k+1}}(\tilde x)\setminus B_{2^{k-1}}(\tilde x)} 
\frac{|\nabla h^+_{e(t)}(P_{I_v(t)} x)|}{|\tilde x -x|^{d-1}} \,\dS(x)
\\&\nonumber~~~
+Cr_c^{-9}\|v\|_{W^{2,\infty}(\Rd\setminus I_v(t))}e(t).
\end{align}
In turn, combining this with \eqref{shortFrequencies1} and
\eqref{shortFrequencies2} and gathering also \eqref{middleFrequencies}, \eqref{longFrequencies},
\eqref{EstimateDwplus} as well as the corresponding bounds for 
$\nabla w^-$ and $\nabla (\theta \ast \nabla \cdot w^-)$, we then
finally deduce \eqref{EstimateLinftyw}.

{\bf Step 4: $L^2L^\infty$-estimate for $\nabla w$.} By making use of the precise formula \eqref{Dwplus} for $\nabla w^+$
and the definition of the vector field $W$ in \eqref{definitionW}, we immediately get
\begin{align}\label{L2Linftyw1}
&\int_{I_v(t)}\sup_{y\in[-r_c,r_c]}|(\nabla w^+)^T(x+y\vec{n}_v(x,t))\cdot\vec{n}_v(x,t)|^2\,\dS(x)
\\&\nonumber
\leq Cr_c^{-2}\|v\|_{W^{2,\infty}(\Rd\setminus I_v(t))}\int_{I_v(t)}
|h^+_{e(t)}|^2+|\nabla h^+_{e(t)}|^2\,\dS.
\end{align}
To estimate the contribution from $|\nabla (\theta \ast (\nabla \cdot w^+))|$ we use the same 
dyadic decomposition as in \eqref{dyadicDecomposition}. 
We start with the terms in the range $k=\lfloor\log e^2(t)\rfloor,\ldots,0$.

Let $x\in I_v(t)$ and $y\in (-r_c,r_c)$ be fixed. We abbreviate $\bar{x}:=x+y\vec{n}_v(x,t)$.
Denote by $F_x$ the tangent plane of the interface $I_v(t)$ at the point $x$. Let $\Phi_{F_x}\colon F_x\times \R\to \Rd$
be the diffeomorphism given by $\Phi_{F_x}(\hat x,\hat y):=\hat x+\hat y\vec{n}_{F_x}(\hat x)$.
We start estimating using the change of variables $\Phi_{F_x}$, the bound 
$|\nabla\theta_k(x)|\leq C\chi_{2^{k-1}\leq|x|\leq 2^{k+1}}|x|^{-d}$,
as well as the fact that
$\hat x + y\vec{n}_{F_x}(\hat x)=\hat x + y\vec{n}_v(x,t)$ is exactly the point
on the ray originating from $\hat x\in F_x$ in normal direction which is closest to $\bar{x}$
\begin{align*}
&|\big(\nabla (\theta_k \ast (\nabla \cdot w^+))\big)^T(x+y\vec{n}_v(x,t))|
\\&
\leq
\int_{(B_{2^{k+1}}(\bar x)\setminus B_{2^{k-1}}(\bar x))\cap (I_v(t)+B_{r_c/2})}
|\nabla\theta_k(\bar x{-}\tilde x)||(\nabla\cdot w^+)(\tilde x)|
\,\mathrm{d}\tilde x 
\\&
\leq 
C\int_{F_x\cap (B_{2^{k+1}}(x)\setminus B_{2^{k-1}}(x))} 
\sup_{\hat y\in [-r_c,r_c]}
\frac{|(\nabla\cdot w^+)(\hat x {+} \hat y\vec{n}_{F_x}(\hat x))|}
{|x-\hat x|^{d-1}}\,\dS(\hat x).
\end{align*}
Note that the right hand side is independent of $y$. Hence, we may estimate
with Minkowski's inequality 
\begin{align*}
&\bigg(\int_{I_v(t)}\sup_{y\in[-r_c,r_c]}\bigg|\sum_{k=\lfloor\log e^2(t)\rfloor-1}^0
\nabla (\theta_k \ast (\nabla \cdot w^+))(x+y\vec{n}_v(x,t))\bigg|^2\,\dS(x)\bigg)^\frac{1}{2}
\\&
\leq 
C|\log e(t)|
\bigg(\int_{I_v(t)}\bigg|\int_{F_x} 
\sup_{\hat y\in [-r_c,r_c]}
\frac{|(\nabla\cdot w^+)(\hat x {+} \hat y\vec{n}_{F_x}(\hat x))|}
{|x-\hat x|^{d-1}}\,\dS(\hat x)\bigg|^2\,\dS(x)\bigg)^\frac{1}{2}
\end{align*}
The inner integral is to be understood in the Cauchy principal value sense. To proceed
we use the $L^2$-theory for singular operators of convolution type, 
the precise formula \eqref{divwplus} for $\nabla\cdot w^+$ as well as
\eqref{Bound2ndDerivativeSignedDistance} and \eqref{BoundGradientProjection}
which entails
\begin{align*}
&\bigg(\int_{I_v(t)}\bigg|\int_{F_x} 
\sup_{\hat y\in [-r_c,r_c]}
\frac{|(\nabla\cdot w^+)(\hat x {+} \hat y\vec{n}_{F_x}(\hat x))|}
{|x-\hat x|^{d-1}}\,\dS(\hat x)\bigg|^2\,\dS(x)\bigg)^\frac{1}{2}
\\&
\leq C\bigg(\int_{I_v(t)}\sup_{y\in[-r_c,r_c]}|(\nabla\cdot w^+)(x{+}y\vec{n}_v(x,t))|^2
\dS(x)\bigg)^\frac{1}{2}
\\&
\leq Cr_c^{-1}\|v\|_{W^{2,\infty}(\Rd\setminus I_v(t))}^\frac{1}{2}
\bigg(\int_{I_v(t)} |h^+_{e(t)}|^2 + |\nabla h^+_{e(t)}|^2 \,\dS\bigg)^\frac{1}{2}.
\end{align*}
An application of \eqref{EstimateErrorHeightSmooth} and
the assumption $E[\chi_u,u,V|\chi_v,v](t)\leq e^2(t)$ finally yields
\begin{align}\label{L2Linftyw0}
&\bigg(\int_{I_v(t)}\sup_{y\in[-r_c,r_c]}\bigg|\sum_{k=\lfloor\log e^2(t)\rfloor-1}^0
\nabla (\theta_k \ast (\nabla \cdot w^+))(x+y\vec{n}_v(x,t))\bigg|^2\,\dS(x)\bigg)^{1/2}
\\&\nonumber
\leq Cr_c^{-5}\|v\|_{W^{2,\infty}(\Rd\setminus I_v(t))}|\log e(t)|e(t).
\end{align}

We move on with the contributions in the range $k=1,\ldots,\infty$.
Note that by \eqref{longFrequencies} we may directly infer from \eqref{EstimateErrorHeightSmooth}
and the assumption $E[\chi_u,u,V|\chi_v,v](t)\leq e^2(t)$
\begin{align}\label{L2Linftyw2}
&\int_{I_v(t)}\sup_{y\in[-r_c,r_c]}\Big|\sum_{k=1}^\infty
\big(\nabla (\theta_k \ast (\nabla \cdot w^+))\big)^T(x+y\vec{n}_v(x,t))
\cdot\vec{n}_v(x,t)\Big|^2\,\dS(x)
\\&\nonumber
\leq
 Cr_c^{-8}\|v\|^2_{W^{2,\infty}(\Rd\setminus I_v(t))}\mathcal{H}^{d-1}(I_v(t))^2 e^2(t).
\end{align}

Moreover, the contributions estimated in \eqref{shortFrequencies1} and \eqref{shortFrequencies2}
result in a bound of the form (recall that $e(t)<r_c$)
\begin{align}\label{L2Linftyw3}
Cr_c^{-4}\|v\|^2_{W^{3,\infty}(\Rd\setminus I_v(t))}e^2(t) + Cr_c^{-8}\| v\|_{W^{1,\infty}}^2 e^2(t).
\end{align}
Note that when summing the respective bounds from \eqref{shortFrequencies5} and \eqref{shortFrequencies6}
over the relevant range $k=-\infty,\ldots,\lfloor \log e^2(t) \rfloor - 1$, we actually gain
a factor $e(t)$, i.e., the contributions estimated in \eqref{shortFrequencies5} and \eqref{shortFrequencies6}
then directly yield a bound of the form
\begin{align}\label{L2Linftyw4}
Cr_c^{-18}\|v\|_{W^{2,\infty}(\Rd\setminus I_v(t))}^2e^2(t).
\end{align}
Finally, the contribution from \eqref{shortFrequencies4} may be estimated as follows.
Let $x\in I_v(t)$, $y\in [-r_c,r_c]$ and denote by $F_{\bar x}$ the tangent plane
to the manifold $\{\sigdist(x,I_v(t))=h^+_{e(t)}(P_{I_v(t)}x)\}$ at the nearest point
to $\bar x = x+y\vec{n}_v(x,t)$. In light of \eqref{shortFrequencies4}, we start estimating
for $k \leq \lfloor \log e^2(t) \rfloor - 1$ by using Jensen's inequality, the bound
$|\nabla h^+_{e(t)}|\leq Cr_c^{-2}$ from Proposition~\ref{PropositionInterfaceErrorHeightRegularized},
as well as the fact that $|\bar x - \tilde x|\geq |x-\tilde x|$ for all $\tilde x\in I_v(t)$
(since $x=P_{I_v(t)}\bar x$ is the closest point to $\bar x$ on the interface $I_v(t)$)
\begin{align*}
&\bigg|\int_{F_{\bar x}\cap B_{2^{k+1}}(\bar x)\setminus B_{2^{k-1}}(\bar x)} 
\frac{|\nabla h^+_{e(t)}(P_{I_v(t)} \tilde x)|}{|\bar x - \tilde x|^{d-1}} \,\dS(\tilde x)\bigg|^2
\\&
\leq \int_{F_{\bar x}\cap B_{2^{k+1}}(\bar x)\setminus B_{2^{k-1}}(\bar x)} 
\frac{|\nabla h^+_{e(t)}(P_{I_v(t)} \tilde x)|^2}{|\bar x - \tilde x|^{d-1}} \,\dS(\tilde x)
\\&
\leq Cr_c^{-2(d-1)}\int_{I_v(t)\cap B_{Cr_c^{-2}2^{k+1}}(x)} 
\frac{|\nabla h^+_{e(t)}(\tilde x)|^2}{|x - \tilde x|^{d-1}} \,\dS(\tilde x).
\end{align*}
Since this bound does not depend anymore on $y\in [-r_c,r_c]$, we may estimate
the contributions from \eqref{shortFrequencies4} using Minkowski's inequality
as well as once more the $L^2$-theory for singular operators of convolution type
to reduce everything to the $H^1$-bound \eqref{EstimateErrorHeightSmooth}
for the local interface error heights. All in all, the contributions from 
\eqref{shortFrequencies4} are therefore bounded by
\begin{align}\label{L2Linftyw5}
Cr_c^{-14}\|v\|_{W^{1,\infty}}^2e^2(t).
\end{align}
The asserted bound \eqref{EstimateL2Linftyw} then finally follows from
collecting the estimates \eqref{L2Linftyw1}, \eqref{L2Linftyw0}, \eqref{L2Linftyw2}, 
\eqref{L2Linftyw3}, \eqref{L2Linftyw4} and \eqref{L2Linftyw5} together with
the analogous bounds for $\nabla w^-$ and $\nabla(\theta\ast\nabla\cdot w^-)$. 

{\bf Step 5: Estimate on the time derivative $\partial_t w$.}
To estimate $\partial_t w^+$, we first deduce using \eqref{dtwplus}, $|\partial_t \eta|\leq C r_c^{-1} \|v\|_{L^\infty}$,
$|\frac{\mathrm{d}}{\dt}\vec{n}_v(P_{I_v(t)}x)|\leq \frac{C}{r_c^2}\|v\|_{W^{1,\infty}}$
(which follows from \eqref{TimeEvolutionNormalNoCutoff}), \eqref{transportProblem}
and finally \eqref{TimeDerivativeProjection} that 
\begin{align*}
&\partial_t w^+(x,t) \\
&=\chi_{0\leq \sigdist(x,I_v)\leq h^+_{e(t)}(P_{I_v(t)} x)} W (x) \partial_t \sigdist(x,I_v(t))
\\&~~~
+\eta \, \chi_{\sigdist(x,I_v)> h^+_{e(t)}(P_{I_v(t)} x)} W (P_{h^+_{e(t)}} x) 
\big(\partial_t h^+_{e(t)}(P_{I_v(t)}x)+\partial_t P_{I_v(t)}x \cdot \nabla h^+_{e(t)}(P_{I_v(t)}x)\big)
\\&~~~
+\tilde g^+
\end{align*}
for some vector field $\tilde g^+$ subject to 
$\|\tilde g^+(\cdot,t)\|_{L^2}\leq Cr_c^{-2}(1+\|v\|_{W^{1,\infty}}) (\|v\|_{W^{1,\infty}}+\|\partial_t \nabla v\|_{L^\infty}
+\|v\|_{W^{2,\infty}(\Rd\setminus I_v(t))}) (\int_{I_v(t)} |h^+_{e(t)}(\cdot,t)|^2 \,\dS)^{1/2}$.
Using \eqref{Dwplus}, \eqref{transportProblem} as well as \eqref{TimeDerivativeProjection} we may compute
\begin{align*}
&(v(x)\cdot \nabla)w^+(x,t) 
\\&
+ \chi_{0\leq \sigdist(x,I_v(t))\leq h^+_{e(t)}(P_{I_v(t)} x)} W (x) \partial_t \sigdist(x,I_v(t)) 
\\&
+\eta \, \chi_{\sigdist(x,I_v(t))> h^+_{e(t)}(P_{I_v(t)} x)} W (P_{h^+_{e(t)}} x) 
\partial_t P_{I_v(t)}x \cdot \nabla h^+_{e(t)}(P_{I_v(t)}x)
\\&
=\eta\, \chi_{\sigdist(x,I_v(t))>h^+_{e(t)}(P_{I_v(t)} x)} W(P_{h^+_{e(t)}} x) 
(\mathrm{Id}{-}\vec{n}_v\otimes\vec{n}_v)v(P_{I_v(t)}x)
\cdot\nabla h^+_{e(t)}(P_{I_v(t)}x)
\\&~~~
+\chi_{0\leq \sigdist(x,I_v(t))\leq h^+_{e(t)}(P_{I_v(t)} x)} W(x)
\big((v(x)-v(P_{I_v(t)}x)\big)\cdot\vec{n}_v(P_{I_v(t)})
\\&~~~
+ \eta\, \chi_{\sigdist(x,I_v(t))>h^+_{e(t)}(P_{I_v(t)} x)} W(P_{h^+_{e(t)}} x) 
\big(\nabla P_{I_v(t)}(x)v(x)-v(P_{I_v(t)}x)\big)\cdot\nabla h^+_{e(t)}(P_{I_v(t)}x)
\\&~~~
+ \tilde g_1^+,
\end{align*}
for some $\|\tilde g_1^+\|_{L^2}\leq
Cr_c^{-2}\|v\|_{W^{2,\infty}(\Rd\setminus I_v(t))}(\int_{I_v(t)} 
|h^+_{e(t)}(\cdot,t)|^2 + |\nabla h^+_{e(t)}(\cdot,t)|^2 \,\dS)^{\frac{1}{2}}$.
This computation in turn implies
\begin{align}\label{EstimatedtwplusFirst}
&\partial_t w^+(x,t) 
\\&\nonumber
=-(v(x)\cdot \nabla)w^+(x,t)
\\&~~~\nonumber
+\eta \, \chi_{\sigdist(x,I_v(t))> h^+_{e(t)}(P_{I_v(t)} x)} W (P_{h^+_{e(t)}} x) 
\big(\partial_t h^+_{e(t)}(P_{I_v(t)}x)+(\mathrm{Id}{-}\vec{n}_v\otimes\vec{n}_v)v(P_{I_v(t)}x)
\cdot\nabla h^+_{e(t)}(P_{I_v(t)}x)\big)
\\&~~~\nonumber
+g^+
\end{align}
for some $g^+$ with
\begin{align*}
\|g^+\|_{L^2}
&\leq C r_c^{-2} (1{+}\|v\|_{W^{1,\infty}})
(\|\partial_t \nabla v\|_{L^\infty}{+}\|v\|_{W^{2,\infty}(\Rd\setminus I_v(t))})
\bigg(\int_{I_v(t)} |h^+_{e(t)}|^2 {+} |\nabla h^+_{e(t)}|^2 \,\dS\bigg)^{\frac{1}{2}}.
\end{align*}
We now aim to make use of \eqref{HeightFunctionRegularizeddtEstimate} to further
estimate the second term in the right hand side of \eqref{EstimatedtwplusFirst}.
To establish the corresponding $L^2$- resp.\ $L^\frac{4}{3}$-contributions,
we first need to perform an integration by parts in order to use
\eqref{HeightFunctionRegularizeddtEstimate}. The resulting curvature
term as well as all other terms which do not appear in the third term
of \eqref{EstimatedtwplusFirst} can be directly bounded by a term
whose associated $L^2$-norm is controlled by $Cr_c^{-1}\|v\|_{W^{1,\infty}}
\|v\|_{W^{2,\infty}(\Rd\setminus I_v(t))}(\int_{I_v(t)} 
|h^+_{e(t)}(\cdot,t)|^2 {+} |\nabla h^+_{e(t)}(\cdot,t)|^2 \,\dS)^{\frac{1}{2}}$.
Hence, using \eqref{HeightFunctionRegularizeddtEstimate} in \eqref{EstimatedtwplusFirst}
implies
\begin{align}
\label{Estimatedtwplus}
\partial_t w^+(x,t)
&=-(v\cdot \nabla)w^+(x,t)
+\bar g^+
+\hat g^+
\end{align}
with the corresponding $L^2$-bound
\begin{align}\label{aux1TimeDerivativeW}
&\|\bar g^+\|_{L^2(\Rd)}
\\&\nonumber
\leq C \frac{1{+}\|v\|_{W^{1,\infty}}}{r_c^2}
(\|\partial_t \nabla v\|_{L^\infty}{+}\|v\|_{W^{2,\infty}(\Rd\setminus I_v(t)})
\bigg(\int_{I_v(t)} |h^+_{e(t)}|^2 {+} |\nabla h^+_{e(t)}|^2 \,\dS\bigg)^\frac{1}{2}
\\&~~~\nonumber
+C\frac{\|v\|_{W^{1,\infty}}(1+\|v\|_{W^{1,\infty}})}{e(t)r_c}
\int_{\Rd} 1-\xi\cdot\frac{\nabla\chi_u}{|\nabla\chi_u|} \,\mathrm{d}|\nabla\chi_u|
\\&~~~\nonumber
+Cr_c^{-2}\|v\|_{W^{1,\infty}}^2(1+e'(t))
\big(\|h^\pm(\cdot,t)\|_{L^2(I_v(t))}
+\|\nabla h^\pm_{e(t)}(\cdot,t)\|_{L^2(I_v(t))}\big)
\\&~~~\nonumber
+C\frac{\|v\|_{W^{1,\infty}}(1{+}\|v\|_{W^{2,\infty}(\Rd\setminus I_v(t))})}{r_c}
\bigg(\int_{\Rd}|\chi_u{-}\chi_v| \min\Big\{\frac{\dist(x,I_v(t))}{r_c},1\Big\} \,\dx\bigg)^\frac{1}{2}
\\&~~~\nonumber
+C\|v\|_{W^{1,\infty}}\bigg(\int_{I_v(t)}|u-v|^2\,\dS\bigg)^\frac{1}{2}
\end{align}
and $L^\frac{4}{3}$-estimate
\begin{align}\label{aux2TimeDerivativeW}
&\|\hat g^+\|_{L^{\frac{4}{3}}(\Rd)}
\\&\nonumber
\leq C\frac{\|v\|_{W^{1,\infty}}}{e(t)r_c^2}\bigg(\int_{I_v(t)} |\bar{h}^\pm|^4 \,\dS \bigg)^\frac{1}{4}
\bigg(\int_{I_v(t)}  \sup_{y\in [-r_c,r_c]} |u{-}v|^2(x{+}y\vec{n}_v(x,t),t) \,\dS(x)\bigg)^\frac{1}{2} 
\\&~~~\nonumber
+C\frac{\|v\|_{W^{1,\infty}}(1+\|v\|_{W^{1,\infty}})}{e(t)}
\int_{\Rd} 1-\xi\cdot\frac{\nabla\chi_u}{|\nabla\chi_u|} \,\mathrm{d}|\nabla\chi_u|.
\end{align}
In both bounds, we add and subtract the compensation function $w$ and therefore obtain
together with \eqref{EstimateL2supw} and \eqref{BoundL2Linfty}
\begin{align}\label{aux3TimeDerivativeW}\nonumber
\int_{I_v(t)}|u-v|^2\,\dS &\leq \int_{I_v(t)}  \sup_{y\in [-r_c,r_c]} |u{-}v|^2(x{+}y\vec{n}_v(x,t),t) \,\dS(x)
\\&\nonumber
\leq \int_{I_v(t)}  \sup_{y\in [-r_c,r_c]} |u-v-w|^2(x+y\vec{n}_v(x,t),t) \,\dS(x)
\\&~~~\nonumber
+\int_{I_v(t)}  \sup_{y\in [-r_c,r_c]} |w(x+y\vec{n}_v(x,t),t)|^2\,\dS(x)
\\&
\leq C r_c^{-4}\|v\|_{W^{2,\infty}(\Rd\setminus I_v(t))}^2
 \int_{I_v(t)} |h^\pm_{e(t)}|^2 + |\nabla h^\pm_{e(t)}|^2 \,\dS
\\&~~~\nonumber
+C(\|u{-}v{-}w\|_{L^2}\|\nabla(u{-}v{-}w)\|_{L^2}+\|u{-}v{-}w\|_{L^2}^2).
\end{align}
Analogous estimates may be derived for $w^-$.
We therefore proceed with the terms related to $\theta\ast\nabla\cdot w^\pm$.
First of all, note that the singular integral operator $(\theta \ast \nabla \cdot)$ 
satisfies (see Theorem~\ref{SingularIntegralOperator})
\begin{align}\label{aux4TimeDerivativeW}
\|\theta \ast \nabla \cdot \hat g\|_{L^{\frac{4}{3}}(\Rd)} \leq C \|\hat g\|_{L^{\frac{4}{3}}(\Rd)}, \quad
\|\theta \ast \nabla \cdot \bar g\|_{L^{2}(\Rd)} \leq C \|\bar g\|_{L^{2}(\Rd)}.
\end{align}
Furthermore, to estimate 
$\|\theta \ast \nabla \cdot ((v\cdot \nabla)w^+) - (v\cdot \nabla) (\theta \ast \nabla \cdot w^+)\|_{L^{2}(\Rd)}$ 
we first replace $v$ with its normal velocity $V_{\vec{n}}(x):=(v(x)\cdot\vec{n}_v(P_{I_v(t)}x))\vec{n}_v(P_{I_v(t)}x)$.
We want to exploit the fact that the vector field $V_{\vec{n}}$ has bounded derivatives up to second order, see
\eqref{Bound1stDerivativeNormalVelocity} and \eqref{BoundDerivativesNormalVelocity}. 
Moreover, the kernel $\nabla^2 \theta(x-\tilde x) \otimes (\tilde x-x)$ gives rise to a singular 
integral operator of convolution type, as does $\nabla \theta$. To see this, we need
to check whether its average over $\Sd$ vanishes. We write 
$x\otimes\nabla^2\theta(x) = \nabla F(x)-\delta_{ij}e_i\otimes\nabla\theta\otimes e_j$,
where $F(x) = x\otimes\nabla\theta(x)$. Now, since
$\nabla\theta$ is homogeneous of degree $-d$, $F$ itself is homogeneous of degree $-(d-1)$. Hence,
we compute $\int_{B_1\setminus B_r}\nabla F\,\dx=\int_{\Sd}\vec{n}\otimes F\,\dS - \int_{r\Sd}\vec{n}\otimes F\,\dS = 0$
for every $0<r<1$. Passing to the limit $r\to 1$ shows that $\nabla F$, and therefore also $\nabla^2\theta(x)\otimes x$,
have vanishing average on $\Sd$. We may now compute (where the integrals are well defined in the Cauchy
principal value sense due to the above considerations) for almost every $x\in\Rd$
\begin{align*}
&\int_\Rd \nabla \theta(x-\tilde x) \cdot (V_{\vec{n}}(\tilde x,t)\cdot \nabla_{\tilde x}) 
w^+(\tilde x,t)- (V_{\vec{n}}(x,t)\cdot \nabla_x) \nabla \theta(x-\tilde x) \cdot w^+(\tilde x,t) \,\mathrm{d}\tilde x
\\&
=\int_\Rd \nabla \theta(x-\tilde x) ((V_{\vec{n}}(\tilde x,t)-V_{\vec{n}}(x,t))
\cdot \nabla_{\tilde x}) w^+(\tilde x,t) \,\mathrm{d}\tilde x
\\&
=\int_\Rd \nabla^2 \theta(x-\tilde x) : (V_{\vec{n}}(\tilde x,t)-V_{\vec{n}}(x,t)-(\tilde x-x)\cdot 
\nabla V_{\vec{n}}(\tilde x,t)) \otimes w^+(\tilde x,t) \,\mathrm{d}\tilde x
\\&~~~
-\int_\Rd \nabla \theta(x-\tilde x) \cdot (\nabla \cdot V_{\vec{n}})(\tilde x,t) w^+(\tilde x,t) \,\mathrm{d}\tilde x
\\&~~~
+\int_\Rd \nabla^2 \theta(x-\tilde x) : ((\tilde x-x) \cdot \nabla) 
V_{\vec{n}}(\tilde x,t) \otimes \,w^+(\tilde x,t) \,\mathrm{d}\tilde x.
\end{align*}
Note that we have 
$|V_{\vec{n}}(\tilde x,t)-V_{\vec{n}}(x,t)-(\tilde x-x)\cdot 
\nabla V_{\vec{n}}(x,t)|\leq \|\nabla^2 V_{\vec{n}}\|_{L^\infty} |\tilde x-x|^2$ and $|V_{\vec{n}}(\tilde x,t)-V_{\vec{n}}(x,t)-(\tilde x-x)\cdot 
\nabla V_{\vec{n}}(x,t)|\leq ||\nabla V_{\vec n}||_{L^\infty} |\tilde x-x|$.
We then estimate using Young's inequality for convolutions and $|\nabla^2 \theta(x)|\leq |x|^{-d-1}$
\begin{align}\label{aux7TimeDerivativeW}
&\int_{\Rd\setminus B_{3R}(0)}\bigg|\int_{B_R(0)} \nabla^2 \theta(x-\tilde x) : 
(V_{\vec{n}}(\tilde x)-V_{\vec{n}}(x)-(\tilde x-x)\cdot 
\nabla V_{\vec{n}}(\tilde x)) \otimes w^+(\tilde x) \,\mathrm{d}\tilde x\bigg|^2\dx
\\&\nonumber
\leq C\|\nabla V_{\vec{n}}\|_{L^\infty}^2
\int_{\Rd\setminus B_{3R}(0)}\bigg|\int_{B_R(0)} 
\frac{1}{|x-\tilde x|^{d}}
|w^+(\tilde x)| \,\mathrm{d}\tilde x\bigg|^2\dx
\\&\nonumber
\leq C\|\nabla V_{\vec{n}}\|_{L^\infty}^2 ||\,|\cdot|^{-d}\,||_{L^2(\Rd\setminus B_R)}^2 \bigg|\int_{B_R(0)}|w^+| \,\dx\bigg|^2
\\&\nonumber
\leq C R^{-d} R^d \int_{B_R(0)}|w^+|^2 \,\dx.
\end{align}
As a consequence, we obtain from \eqref{aux7TimeDerivativeW}, Young's inequality for convolutions,
\eqref{L2Wplus} as well as \eqref{BoundDerivativesNormalVelocity}
\begin{align}\label{aux8TimeDerivativeW}
&\int_{\Rd}\bigg|\int_{\Rd} \nabla^2 \theta(x-\tilde x) : 
(V_{\vec{n}}(\tilde x)-V_{\vec{n}}(x)-(\tilde x-x)\cdot 
\nabla V_{\vec{n}}(\tilde x)) \otimes w^+(\tilde x) \,\mathrm{d}\tilde x\bigg|^2\dx
\\&\nonumber
\leq C\|\nabla^2 V_{\vec{n}}\|_{L^\infty}^2 
\int_{B_{3R}(0)}\bigg|\int_{\Rd}\frac{|w^+(\tilde x)|}{|x-\tilde x|^{d-1}}\,\mathrm{d}\tilde x\bigg|^2\,\dx
+C\|\nabla V_{\vec{n}}\|_{L^\infty}^2 \int_{B_R(0)}|w^+|^2\,\dx
\\&\nonumber
\leq Cr_c^{-4}\|v\|_{W^{2,\infty}(\Rd\setminus I_v(t))}^2(1{+}R^2)\int_{I_v(t)} |h^+_{e(t)}|^2 \,\dS.
\end{align}
Applying Theorem~\ref{SingularIntegralOperator} to the singular integral operators
$\nabla\theta$ resp.\ $\nabla^2\theta\otimes x$ as well as making use of
\eqref{Bound1stDerivativeNormalVelocity}, \eqref{L2Wplus} and \eqref{aux8TimeDerivativeW}
we then obtain the estimate
\begin{align}\label{aux5TimeDerivativeW}
&\int_\Rd |\theta \ast \nabla \cdot ((V_{\vec{n}}\cdot \nabla)w^+)
-(V_{\vec{n}}\cdot \nabla)(\theta \ast \nabla \cdot w^+)|^2 \,\dx
\\&\nonumber
\leq Cr_c^{-4}\|v\|_{W^{2,\infty}(\Rd\setminus I_v(t))}^2(1{+}R^2)\int_{I_v(t)} |h^+_{e(t)}|^2 \,\dS
\\&\nonumber~~~
+C\|\nabla V_{\vec{n}}\|_{L^\infty}^2 \int_{\Rd}|w^+|^2\,\dx
\\&\nonumber
\leq Cr_c^{-4}\|v\|_{W^{2,\infty}(\Rd\setminus I_v(t))}^2(1{+}R^2)\int_{I_v(t)} |h^+_{e(t)}|^2 \,\dS.
\end{align}
It remains to estimate 
$\|\theta \ast \nabla \cdot ((V_{\mathrm{tan}}\cdot \nabla)w^+) 
- (V_{\mathrm{tan}}\cdot \nabla) (\theta \ast \nabla \cdot w^+)\|_{L^{2}(\Rd)}$ 
with $V_{\mathrm{tan}}(x)=(\mathrm{Id}-\vec{n}_v(P_{I_v(t)}x)\otimes\vec{n}_v(P_{I_v(t)}x))v(x)$ 
denoting the tangential velocity of $v$. To this end, note that we may rewrite
\begin{align*}
&\int_\Rd \nabla \theta(x-\tilde x) \cdot (V_{\mathrm{tan}}(\tilde x,t)\cdot \nabla_{\tilde x}) 
w^+(\tilde x,t) - (\nabla\cdot w^+(\tilde x,t))(V_{\mathrm{tan}}(x,t)\cdot \nabla_x)\theta(x-\tilde x) \,\mathrm{d}\tilde x
\\&
=\int_\Rd \nabla \theta(x{-}\tilde x) \big(\nabla w^+(\tilde x) {-} 
\chi_{0\leq \sigdist(\tilde x,I_v(t))\leq h^+_{e(t)}(P_{I_v(t)} \tilde x)} W(\tilde x)
\otimes\vec{n}_v(P_{I_v(t)}\tilde x)\big)V_{\mathrm{tan}}(\tilde x,t)\,\mathrm{d}\tilde x
\\&~~~
-\int_{\Rd}(\nabla\cdot w^+(\tilde x,t))(V_{\mathrm{tan}}(x,t)\cdot \nabla_x)
\theta(x-\tilde x) \,\mathrm{d}\tilde x.
\end{align*}
Using Theorem~\ref{SingularIntegralOperator}, \eqref{auxGradientCompensationWplus} as well
as \eqref{auxL2divWplus} we then obtain
\begin{align}\label{aux6TimeDerivativeW}\nonumber
&\|\theta \ast \nabla \cdot ((V_{\mathrm{tan}}\cdot \nabla)w^+) 
- (V_{\mathrm{tan}}\cdot \nabla) (\theta \ast \nabla \cdot w^+)\|_{L^{2}(\Rd)}^2
\\&
\leq Cr_c^{-4}\|v\|_{L^\infty}^2\|v\|_{W^{2,\infty}(\Rd\setminus I_v(t))}^2
\int_{I_v(t)} |h^+_{e(t)}|^2 + |\nabla h^+_{e(t)}|^2 \,\dS.
\end{align}
Putting all the estimates \eqref{aux1TimeDerivativeW}, \eqref{aux2TimeDerivativeW},
\eqref{aux3TimeDerivativeW}, \eqref{aux4TimeDerivativeW}, \eqref{aux5TimeDerivativeW}
and \eqref{aux6TimeDerivativeW} together, we get
\begin{align*}
\partial_t w(x,t)+(v\cdot \nabla)w(x,t) = g + \hat g
\end{align*}
with the asserted bounds. This concludes the proof.
\end{proof}

\subsection{Estimate for the additional surface tension terms}
\label{SectionEstimateAsurTen}

Having established all the relevant properties of the compensating vector field
$w$ in Proposition~\ref{PropositionCompensationFunction}, 
we can move on with the post-processing of the additional terms
in the relative entropy inequality from Proposition~\ref{PropositionRelativeEntropyInequalityFull}.
To this end, we start with the additional surface tension terms given by
\begin{align}\label{aux0AsurTen}
A_{surTen} &= 
-\sigma\int_0^T\int_{\R^d\times\Sb^{d-1}}(s{-}\xi)
\cdot\big((s{-}\xi)\cdot\nabla\big) w \,\mathrm{d}V_t(x,s)\,\dt 
\\&~~~\nonumber
+\sigma\int_0^T\int_{\R^d}(1-\theta_t)\,
\xi\cdot(\xi\cdot\nabla)w\,\mathrm{d}|V_t|_{\Sb^{d-1}}(x)\,\dt
\\&~~~\nonumber
+\sigma\int_0^T\int_\Rd(\chi_u-\chi_v)(w\cdot\nabla)(\nabla\cdot\xi)\,\dx\,\dt
\\&~~~\nonumber
+\sigma\int_0^T\int_\Rd (\chi_u-\chi_v) \nabla w :\nabla \xi^T \,\dx\,\dt
\\&~~~\nonumber
-\sigma\int_0^T\int_\Rd \xi \cdot \big((\vec{n}_u-\xi) \cdot \nabla\big) w \,\mathrm{d}|\nabla \chi_u|\,\dt
\\&\nonumber
=: I + II + III + IV + V.
\end{align}
A precise estimate for these terms is the content of the following result.

\begin{lemma}
Let the assumptions and notation of Proposition~\ref{PropositionCompensationFunction} be in place. 
In particular, we assume that there exists a $C^1$-function
$e\colon[0,\Tmax)\to [0,r_c)$ such that the relative entropy is bounded by $E[\chi_u,u,V,|\chi_v,v](t)\leq e^2(t)$.
Then the additional surface tension terms $A_{surTen}$ are bounded by a Gronwall-type term
\begin{align}\label{GronwallAsurTen}
A_{surTen} &\leq
\frac{C}{r_c^{10}}(1+\|v\|_{L^\infty_tW^{2,\infty}_x(\Rd\setminus I_v(t))}^2
+\|v\|_{L^\infty_tW^{3,\infty}_x(\Rd\setminus I_v(t))})
\\&\nonumber~~~~~~~~~~~~~~~~~~
\int_0^T(1+|\log e(t)|)E[\chi_u,u,V|\chi_v,v](t)\,\mathrm{d}t
\\&\nonumber~~~~
+\frac{C}{r_c^{10}}(1+\|v\|_{L^\infty_tW^{2,\infty}_x(\Rd\setminus I_v(t))}^2
+\|v\|_{L^\infty_tW^{3,\infty}_x(\Rd\setminus I_v(t))})
\\&\nonumber~~~~~~~~~~~~~~~~~~
\int_0^T (1+|\log e(t)|)e(t)E[\chi_u,u,V|\chi_v,v]^\frac{1}{2}(t)\,\mathrm{d}t.
\end{align}
\end{lemma}

\begin{proof}
We estimate term by term in \eqref{aux0AsurTen}. A straightforward estimate for the first two terms 
using also the coercivity property \eqref{controlSquaredErrorNormalVarifold} yields
\begin{align}\label{aux1AsurTen}
I + II &\leq C\int_0^T\|\nabla w(t)\|_{L^\infty_{x}}\int_{\R^d\times\Sb^{d-1}}|s-\xi|^2\,\mathrm{d}V_t(x,s)\,\dt 
\\&~~~\nonumber
+C\int_0^T\|\nabla w(t)\|_{L^\infty_{x}}\int_{\R^d}(1-\theta_t)\,\mathrm{d}|V_t|_{\Sb^{d-1}}(x)\,\dt
\\&\nonumber
\leq C\int_0^T\|\nabla w(t)\|_{L^\infty_{x}}E[\chi_u,u,V|\chi_v,v](t)\,\mathrm{d}t.
\end{align}
Making use of \eqref{Bound2ndDerivativeSignedDistance}, a change of variables $\Phi_t$, 
H\"older's and Young's inequality, \eqref{EstimateL2supw}, \eqref{L2hNoCutoff},
\eqref{EstimateErrorHeightSmooth} as well as the coercivity property \eqref{coercivityWeight}
the term $III$ may be bounded by
\begin{align}\label{aux2AsurTen}
III &\leq\frac{C}{r_c^2}\int_0^T\int_{I_v(t)}\sup_{y\in [-r_c,r_c]}|w(x{+}y\vec{n}_v(x,t))|
\int_{-r_c}^{r_c}|\chi_u{-}\chi_v|(x{+}y\vec{n}_v(x,t))\,\dy\,\dS\,\dt
\\&\nonumber
\leq \frac{C}{r_c^2}\int_0^T\int_{I_v(t)}\sup_{y\in [-r_c,r_c]}|w(x{+}y\vec{n}_v(x,t))|^2\,\dS\,\dt
\\&~~~\nonumber
+\frac{C}{r_c^2}\int_0^T\int_{I_v(t)}\bigg|\int_{-r_c}^{r_c}|\chi_u{-}\chi_v|(x{+}y\vec{n}_v(x,t))\,\dy\bigg|^2\,\dS\,\dt
\\&\nonumber
\leq \frac{C}{r_c^6}\|v\|_{L^\infty_tW^{2,\infty}_x(\Rd\setminus I_v(t))}^2
\int_0^T\int_{I_v(t)} |h^\pm_{e(t)}|^2 + |\nabla h^\pm_{e(t)}|^2 \,\dS\,\dt
\\&~~~\nonumber
+\frac{C}{r_c^2}\int_0^T\int_\Rd |\chi_u{-}\chi_v| \min\Big\{\frac{\dist(x,I_v(t))}{r_c},1\Big\} \,\dx\,\dt
\\&\nonumber
\leq \frac{C}{r_c^{10}}\|v\|_{L^\infty_tW^{2,\infty}_x(\Rd\setminus I_v(t))}^2
\int_0^T\int_\Rd 1-\xi\cdot \frac{\nabla\chi_u}{|\nabla\chi_u|} \,\mathrm{d}|\nabla\chi_u|\,\dt
\\&~~~\nonumber
+ \frac{C}{r_c^{10}}(1+\|v\|_{L^\infty_tW^{2,\infty}_x(\Rd\setminus I_v(t))}^2)
\int_0^T\int_\Rd |\chi_u{-}\chi_v| \min\Big\{\frac{\dist(x,I_v(t))}{r_c},1\Big\} \,\dx\,\dt
\\&\nonumber
\leq \frac{C}{r_c^{10}}(1+\|v\|_{L^\infty_tW^{2,\infty}_x(\Rd\setminus I_v(t))}^2)
\int_0^TE[\chi_u,u,V|\chi_v,v](t)\,\mathrm{d}t.
\end{align}
For the term $IV$, we first add zero, then perform an integration by parts
which is followed by an application of H\"older's inequality to obtain
\begin{align}\label{aux3AsurTen}
IV &\leq C\int_0^T\bigg(\int_\Rd|\chi_u-\chi_{v,h^+_{e(t)},h^-_{e(t)}}|\,\dx\bigg)^\frac{1}{2}
\bigg(\int_\Rd|(\nabla w)^T\colon\nabla\xi|^2\,\dx\bigg)^\frac{1}{2}\,\dt
\\&~~~\nonumber
+C\int_0^T\bigg|\int_{\Rd}(\chi_v-\chi_{v,h^+_{e(t)},h^-_{e(t)}})(w\cdot\nabla)(\nabla\cdot\xi)\,\dx\bigg|\,\dt
\\&~~~\nonumber
+C\int_0^T\bigg|\int_{\Rd}((w\cdot\nabla)\xi)\cdot\mathrm{d}\nabla(\chi_v-\chi_{v,h^+_{e(t)},h^-_{e(t)}})\bigg|\,\dt
\\&\nonumber
=: (IV)_a + (IV)_b + (IV)_c.
\end{align}
By definition of $\xi$, see \eqref{CutOffNormal}, recall that 
\begin{align*}
\nabla\xi = \frac{\zeta'\big(\frac{\sigdist(x,I_v(t))}{r_c}\big)}{r_c}
\vec{n}_v(P_{I_v(t)}x)\otimes\vec{n}_v(P_{I_v(t)}x) + \zeta\Big(\frac{\sigdist(x,I_v(t))}{r_c}\Big)
\nabla^2\sigdist(x,I_v(t)).
\end{align*}
Recalling also \eqref{EstimateDw}, \eqref{definitionW} and \eqref{EstimateDwdiv}
as well as making use of \eqref{ErrorImprovedInterfaceApproximationSmooth},
\eqref{Bound2ndDerivativeSignedDistance}, \eqref{BoundGradientProjection},
\eqref{EstimateErrorHeightSmooth} and finally the coercivity property \eqref{coercivityWeight}
the term $(IV)_a$ from \eqref{aux3AsurTen} is estimated by
\begin{align}\label{aux4AsurTen}
(IV)_a
&\leq
\frac{C}{r_c}\int_0^T E[\chi_u,u,V|\chi_v,v](t) + e(t)E[\chi_u,u,V|\chi_v,v]^\frac{1}{2}(t)\,\dt
\\&\nonumber~~~
+\frac{C}{r_c^4}\|v\|_{L^\infty_tW^{2,\infty}_x(\Rd\setminus I_v(t))}^2\int_0^T
\int_{I_v(t)}|h^\pm_{e(t)}|^2+|\nabla h^\pm_{e(t)}|^2\,\dS\,\dt
\\&\nonumber
\leq \frac{C}{r_c^8}(1{+}\|v\|_{L^\infty_tW^{2,\infty}_x(\Rd\setminus I_v(t))}^2)
\int_0^T E[\chi_u,u,V|\chi_v,v](t) {+} e(t)E[\chi_u,u,V|\chi_v,v]^\frac{1}{2}(t)\,\dt.
\end{align}
Recalling from \eqref{def:improvedApprox} the definition of $\chi_{v,h^+_{e(t)},h^-_{e(t)}}$, we may estimate
the term $(IV)_b$ from \eqref{aux3AsurTen} by a change of variables $\Phi_t$, \eqref{Bound2ndDerivativeSignedDistance},
H\"older's and Young's inequality, \eqref{EstimateL2supw} as well as \eqref{EstimateErrorHeightSmooth}
\begin{align}\label{aux5AsurTen}
(IV)_b &\leq
\frac{C}{r_c^2}\int_0^T\int_{I_v(t)}|h^\pm_{e(t)}|^2\,\dS\,\dt 
\\&~~~\nonumber
+\frac{C}{r_c^2}\int_0^T\int_{I_v(t)}\sup_{y\in[-r_c,r_c]}|w(x{+}y\vec{n}_v(x,t))|^2\,\dS\,\dt
\\&\nonumber
\leq \frac{C}{r_c^{10}}\|v\|_{L^\infty_tW^{2,\infty}_x(\Rd\setminus I_v(t))}^2\int_0^T E[\chi_u,u,V|\chi_v,v](t)\,\dt.
\end{align}
To estimate the term $(IV)_c$ from \eqref{aux3AsurTen}, we again make use of the definition of $\chi_{v,h^+_{e(t)},h^-_{e(t)}}$,
\eqref{Bound2ndDerivativeSignedDistance}, H\"older's and Young's inequality, 
\eqref{EstimateL2supw} as well as \eqref{EstimateErrorHeightSmooth} which yields the following bound
\begin{align}\label{aux6AsurTen}
(IV)_c
&\leq\frac{C}{r_c}\int_0^T\int_{I_v(t)}|\nabla h^\pm_{e(t)}|
\sup_{y\in[-r_c,r_c]}|w(x{+}y\vec{n}_v(x,t))|\,\dS\,\dt
\\&\nonumber
\leq \frac{C}{r_c^9}\|v\|_{L^\infty_tW^{2,\infty}_x(\Rd\setminus I_v(t))}
\int_0^T E[\chi_u,u,V|\chi_v,v](t)\,\dt.
\end{align}
Hence, taking together the bounds from \eqref{aux4AsurTen}, \eqref{aux5AsurTen} and \eqref{aux6AsurTen}
we obtain
\begin{align}\label{aux7AsurTen}
IV &\leq \frac{C}{r_c^{10}}(1{+}\|v\|_{L^\infty_tW^{2,\infty}_x(\Rd\setminus I_v(t))}^2)
\int_0^T E[\chi_u,u,V|\chi_v,v](t)\,\dt 
\\&~~~\nonumber
+\frac{C}{r_c^{10}}(1{+}\|v\|_{L^\infty_tW^{2,\infty}_x(\Rd\setminus I_v(t))}^2)
\int_0^Te(t)E^\frac{1}{2}[\chi_u,u,V|\chi_v,v](t)\,\dt.
\end{align}

In order to estimate the term $V$, we argue as follows. In a first step, we split $\Rd$ into the region
$I_v(t)+B_{r_c}$ near to and the region $\Rd\setminus(I_v(t)+B_{r_c})$ away from the interface
of the strong solution. Recall then that the indicator function $\chi_u(\cdot,t)$ of the varifold solution 
is of bounded variation in $I_v(t)+B_{r_c}$.
In particular, $E^+:=\{x\in\Rd\colon \chi_u>0\}\cap (I_v(t)+B_{r_c})$ is a set
of finite perimeter in $I_v(t)+B_{r_c}$. Applying Theorem~\ref{TheoG} in local coordinates, the sections
$$E^+_x = \{y\in (-r_c,r_c)\colon \chi_u(x+y\vec{n}_v(x,t))>0\}$$
are guaranteed to be one-dimensional Caccioppoli sets in $(-r_c,r_c)$,
and such that all of the four properties listed in Theorem~\ref{TheoG} hold true for $\mathcal{H}^{d-1}$-almost
every $x\in I_v(t)$.
Recall from \cite[Proposition~3.52]{Ambrosio2000a} that one-dimensional Caccioppoli sets are in fact finite unions of disjoint intervals.
We then distinguish for $\mathcal{H}^{d-1}$-almost
every $x\in I_v(t)$ between the cases that $\mathcal{H}^0(\partial^*E_x^+)\leq 2$ or
$\mathcal{H}^0(\partial^*E_x^+)> 2$. In other words, we distinguish between those sections which consist
of at most one interval and those which consist of at least two intervals.
It also turns out to be useful to further keep track of whether $\vec{n}_v\cdot\vec{n}_u\leq \frac{1}{2}$
or $\vec{n}_v\cdot\vec{n}_u\geq \frac{1}{2}$ holds.

We then obtain by Young's and H\"older's inequality
as well as the fact that due to Definition~\ref{def:ExtNormal} 
the vector field $\xi$ is supported in $I_v(t)+B_{r_c}$
\begin{align}\label{aux9surTen}
V &\leq 
\int_0^T\bigg(\int_{\{x{+}y\vec{n}_v(x,t)\in\partial^*E^+\colon 
x\in I_v(t),\,|y|<r_c,\,\mathcal{H}^0(\partial^*E_x^+)\leq 2,\,
\vec{n}_v(x)\cdot\vec{n}_u(x{+}y\vec{n}_v(x,t))\geq\frac{1}{2}\}} 
|(\nabla w)^T \xi|^2 \,\mathrm{d}\mathcal{H}^{d-1}\bigg)^{1/2}
\\&~~~~~~~~~~~~\nonumber\times 
\bigg(\int_\Rd |\vec{n}_u-\xi|^2 \,\mathrm{d}|\nabla \chi_u|\bigg)^{1/2}\,\dt 
\\&~~~\nonumber
+C\int_0^T\|\nabla w(t)\|_{L^\infty_x}\bigg(
\int_{\{x{+}y\vec{n}_v(x,t)\in\partial^*E^+\colon 
x\in I_v(t),\,|y|<r_c,\,\mathcal{H}^0(\partial^*E_x^+)> 2,\,
\vec{n}_v(x)\cdot\vec{n}_u(x{+}y\vec{n}_v(x,t))\geq\frac{1}{2}\}}
1\,\mathrm{d}\mathcal{H}^{d-1}\bigg)\,\dt
\\&~~~\nonumber
+C\int_0^T\|\nabla w(t)\|_{L^\infty_x}\bigg(
\int_{\{x{+}y\vec{n}_v(x,t)\in\partial^*E^+\colon 
x\in I_v(t),\,|y|<r_c,\,\vec{n}_v(x)\cdot\vec{n}_u(x{+}y\vec{n}_v(x,t))\leq\frac{1}{2}\}}
1\,\mathrm{d}\mathcal{H}^{d-1}\bigg)\,\dt
\\&~~~\nonumber
+C\int_0^T\|\nabla w(t)\|_{L^\infty_x}\bigg(\int_{\Rd\setminus (I_v(t)+ B_{r_c})}
1\,\mathrm{d}|\nabla\chi_u|\bigg)\,\dt
\\&\nonumber
\leq C\int_0^T\|\nabla w(t)\|_{L^\infty_x}E[\chi_u,u,V|\chi_v,v](t)\,\dt
\\&\nonumber~~~
+C\int_0^T\bigg(\int_{\{x{+}y\vec{n}_v(x,t)\in\partial^*E^+\colon 
x\in I_v(t),\,|y|<r_c,\,\mathcal{H}^0(\partial^*E_x^+)\leq 2,\,
\vec{n}_v(x)\cdot\vec{n}_u(x{+}y\vec{n}_v(x,t))\geq\frac{1}{2}\}} 
|(\nabla w)^T \xi|^2 \,\mathrm{d}\mathcal{H}^{d-1}\bigg)^\frac{1}{2}
\\&~~~~~~~~~~~~\nonumber\times 
\bigg(\int_\Rd |\vec{n}_u-\xi|^2 \,\mathrm{d}|\nabla \chi_u|\bigg)^{1/2}\,\dt 
\\&\nonumber~~~
+C\int_0^T\|\nabla w(t)\|_{L^\infty_x}\bigg(
\int_{\{x{+}y\vec{n}_v(x,t)\in\partial^*E^+\colon 
x\in I_v(t),\,|y|<r_c,\,\mathcal{H}^0(\partial^*E_x^+)> 2,\,
\vec{n}_v(x)\cdot\vec{n}_u(x{+}y\vec{n}_v(x,t))\geq\frac{1}{2}\}}
1\,\mathrm{d}\mathcal{H}^{d-1}\bigg)\,\dt
\\&\nonumber
=: C\int_0^T\|\nabla w(t)\|_{L^\infty_x}E[\chi_u,u,V|\chi_v,v](t)\,\dt
+ V_a + V_b.
\end{align}
To estimate $V_a$ from \eqref{aux9surTen}, we use the co-area formula for rectifiable
sets (see \cite[(2.72)]{Ambrosio2000a}), \eqref{EstimateL2Linftyw}, H\"older's inequality and
the coercivity property \eqref{controlByTiltExcess} which together yield 
(we abbreviate in the first line $F(x,y,t):=(\nabla w)^T(x{+}y\vec{n}_v(x,t))\vec{n}_v(x,t)$)
\begin{align}\label{aux10surTen}
V_a 
&\leq
C\int_0^T\bigg(\int_{\{x\in I_v(t)\colon\mathcal{H}^0(\partial^*E_x^+)\leq 2\}}
\int_{\{y\in\partial^*E^+_x\colon 
\vec{n}_v(x)\cdot\vec{n}_u(x{+}y\vec{n}_v(x,t))\geq\frac{1}{2}\}}|F(x,y,t)|^2
\,\mathrm{d}\mathcal{H}^0(y)\,\dS(x)\bigg)^\frac{1}{2}
\\&~~~~~~~~~~~~\nonumber\times 
\bigg(\int_\Rd |\vec{n}_u-\xi|^2 \,\mathrm{d}|\nabla \chi_u|\bigg)^{1/2}\,\dt
\\&\nonumber
\leq C\int_0^T\bigg(\int_{I_v(t)}\sup_{y\in[-r_c,r_c]}
|(\nabla w)^T(x{+}y\vec{n}_v(x,t))\cdot\vec{n}_v(x,t)|^2\,\dS(x)\bigg)^\frac{1}{2}
\\&~~~~~~~~~~~~\nonumber\times 
\bigg(\int_\Rd |\vec{n}_u-\xi|^2 \,\mathrm{d}|\nabla \chi_u|\bigg)^{1/2}\,\dt
\\&\nonumber
\leq \frac{C}{r_c^{9}}\|v\|_{L^\infty_tW^{3,\infty}_x(\Rd\setminus I_v(t))}
\int_0^T (1+|\log e(t)|)e(t)E[\chi_u,u,V|\chi_v,v]^\frac{1}{2}(t)\,\dt.
\end{align}
It remains to bound the term $V_b$ from \eqref{aux9surTen}. To this end, we make use of the fact that
it follows from property iv) in Theorem~\ref{TheoG} that every second point
$y\in\partial^*E^+_x\cap(-r_c,r_c)$ has to have the property that $\vec{n}_v(x)\cdot\vec{n}_u(x{+}y\vec{n}_v(x,t))<0$, i.e.,
$1\leq 1-\vec{n}_v(x)\cdot\vec{n}_u(x{+}y\vec{n}_v(x,t))$. We may therefore estimate
with the help of the co-area formula for rectifiable
sets (see \cite[(2.72)]{Ambrosio2000a}) and the bound \eqref{EstimateLinftyw}
\begin{align}\label{aux11surTen}
V_b &\leq C\int_0^T\|\nabla w(t)\|_{L^\infty_x}
\int_{\{x\in I_v(t)\colon\mathcal{H}^0(\partial^*E_x^+)> 2\}}
\int_{\{y\in\partial^*E^+_x\colon \vec{n}_v(x)\cdot\vec{n}_u(x{+}y\vec{n}_v(x,t))\geq\frac{1}{2}\}}
1\,\mathrm{d}\mathcal{H}^0(y)\,\dS(x)\,\dt
\\&\nonumber
\leq C
\int_0^T \|\nabla w(t)\|_{L^\infty_x} \int_{I_v(t)}\int_{\partial^*E^+_x}1-\vec{n}_v(x,t)\cdot\vec{n}_u(x{+}y\vec{n}_v(x,t))
\,\mathrm{d}\mathcal{H}^0(y)\,\dS(x)\,\dt
\\&\nonumber
\leq \frac{C}{r_c^9}|\log e(t)| \|v\|_{L^\infty_tW^{3,\infty}_x(\Rd\setminus I_v(t))}
\int_0^TE[\chi_u,u,V|\chi_v,v](t)\,\dt.
\end{align} 
All in all, we obtain from the assumption $E[\chi_u,u,V|\chi_v,v](t)\leq e^2(t)$
as well as \eqref{aux9surTen}, \eqref{aux10surTen}, \eqref{aux11surTen} and \eqref{EstimateLinftyw}
\begin{align}\label{aux8surTen}
V &\leq 
\frac{C}{r_c^{9}}\|v\|_{L^\infty_tW^{3,\infty}_x(\Rd\setminus I_v(t))}
\int_0^T (1+|\log e(t)|)e(t)E[\chi_u,u,V|\chi_v,v]^\frac{1}{2}(t)\,\dt.
\end{align}
Hence, we deduce from the bounds \eqref{aux1AsurTen}, \eqref{aux2AsurTen}, \eqref{aux7AsurTen},
\eqref{aux8surTen} as well as \eqref{EstimateLinftyw} the asserted estimate for the additional surface tension terms.
\end{proof}

\subsection{Estimate for the viscosity terms}
\label{SectionEstimateAvisc}

In contrast to the case of equal shear viscosities $\mu_+=\mu_-$, we have to deal
with the problematic viscous stress term given by $(\mu(\chi_v)-\mu(\chi_u))(\nabla v+\nabla v^T)$. 
We now show that the choice of $w$ indeed compensates for (most of) this term in the sense that
the viscosity terms from Proposition~\ref{PropositionRelativeEntropyInequalityFull}
\begin{align}\label{aux0Avisc}
R_{visc}+A_{visc} &= -\int_0^T \int_{\mathbb{R}^d} 2\big(\mu(\chi_u)-\mu(\chi_v)\big)
\Dsym v:\Dsym (u-v) \,\dx\,\dt
\\&~~~\nonumber
+\int_0^T \int_{\mathbb{R}^d} 2\big(\mu(\chi_u)-\mu(\chi_v)\big)
\Dsym v:\Dsym w \,\dx\,\dt
\\&~~~\nonumber
-\int_0^T \int_{\mathbb{R}^d} 2\mu(\chi_u) \Dsym w:\Dsym(u-v-w) \,\dx\,\dt
\end{align}
may be bounded by a Gronwall-type term.

\begin{lemma}
Let the assumptions and notation of Proposition~\ref{PropositionCompensationFunction} be in place. 
In particular, we assume that there exists a $C^1$-function
$e\colon[0,\Tmax)\to [0,r_c)$ such that the relative entropy is bounded by $E[\chi_u,u,V,|\chi_v,v](t)\leq e^2(t)$.

Then, for any $\delta>0$ there exists a constant $C>0$ such that the viscosity terms $R_{visc}+A_{visc}$ 
may be estimated by
\begin{align}\label{GronwallAvisc}
R_{visc}+A_{visc}
&\leq\frac{C}{r_c^{8}}\|v\|_{L^\infty_tW^{2,\infty}_x(\Rd\setminus I_v(t))}^2
\int_0^TE[\chi_u,u,V|\chi_v,v](t)\,\dt
\\&\nonumber~~~
+\frac{C}{r_c}\|v\|_{L^\infty_tW^{1,\infty}_x}^2\int_0^Te(t)E[\chi_u,u,V|\chi_v,v]^\frac{1}{2}(t)\,\dt
\\&\nonumber~~~
+\delta\int_0^T\int_{\Rd}|\Dsym(u-v-w)|^2\,\dx\,\dt.
\end{align}
\end{lemma}

\begin{proof}
We argue pointwise for the time variable and start by adding zero
\begin{align}\label{aux1Avisc}
&R_{visc}+A_{visc}
\\&\nonumber
=-2\int_\Rd (\mu(\chi_u)-\mu(\chi_v))\Dsym v:\Dsym(u{-}v{-}w) \,\dx
\\&~~\nonumber
-2\int_\Rd \mu(\chi_u)\Dsym w:\Dsym(u-v-w) \,\dx
\\&\nonumber
=-2\int_\Rd \big(\mu(\chi_u)-\mu(\chi_v)-(\mu^--\mu^+)\chi_{0\leq \sigdist(x,I_v(t))\leq h^+_{e(t)}(P_{I_v(t)}x)}
\\&~~~~~~~~~~~\nonumber
-(\mu^+ -\mu^-)\chi_{-h^-_{e(t)}(P_{I_v(t)}x)\leq \sigdist(x,I_v(t))\leq 0}\big)\Dsym v:\Dsym(u{-}v{-}w) \,\dx
\\&~\nonumber
-2\int_\Rd \chi_{\sigdist(x,I_v(t))\notin [-h^-_{e(t)}(P_{I_v(t)}x),h^+_{e(t)}(P_{I_v(t)}x)]}
 \mu(\chi_u) \Dsym w:\Dsym(u{-}v{-}w) \,\dx
\\&~\nonumber
-2\int_\Rd \chi_{0\leq \sigdist(x,I_v(t))\leq h^+_{e(t)}(P_{I_v(t)}x)} 
(\mu(\chi_u)-\mu^-) \Dsym w:\Dsym(u{-}v{-}w) \,\dx
\\&~\nonumber
-2\int_\Rd \chi_{-h^-_{e(t)}(P_{I_v(t)}x)\leq \sigdist(x,I_v(t))\leq 0}
 (\mu(\chi_u)-\mu^+) \Dsym w:\Dsym(u{-}v{-}w) \,\dx
\\&~\nonumber
-2\int_\Rd \chi_{0\leq \sigdist(x,I_v(t))\leq h^+_{e(t)}(P_{I_v(t)}x)} 
((\mu^-{-}\mu^+)\Dsym v+\mu^- \Dsym w):\nabla (u{-}v{-}w) \,\dx
\\&~\nonumber
-2\int_\Rd \chi_{-h^-_{e(t)}(P_{I_v(t)}x)\leq \sigdist(x,I_v(t))\leq 0} 
((\mu^+{-}\mu^-)\Dsym v+\mu^+ \Dsym w):\nabla (u{-}v{-}w) \,\dx
\\&\nonumber
=: I + II + III + IV + V + VI.
\end{align}
We start by estimating the first four terms. Note that $\mu(\chi_u)-\mu^- = (\mu_+-\mu_-)\chi_u$. 
Recalling the definition of $\chi_{v,h^+_{e(t)},h^-_{e(t)}}$ from \eqref{def:improvedApprox}
we see that $$\chi_{0\leq \sigdist(x,I_v(t))\leq h^+_{e(t)}(P_{I_v(t)}x)}\chi_u =
\chi_{0\leq \sigdist(x,I_v(t))\leq h^+_{e(t)}(P_{I_v(t)}x)}(\chi_u-\chi_{v,h^+_{e(t)},h^-_{e(t)}}).$$
Hence, we may rewrite
\begin{align*}
III
&= -2\int_\Rd \chi_{0\leq \sigdist(x,I_v(t))\leq h^+_{e(t)}(P_{I_v(t)}x)} (\mu_+-\mu_-)
(\chi_u-\chi_{v,h^+_{e(t)},h^-_{e(t)}}) 
\\&~~~~~~~~~~~~~~~~
\times (W\otimes\vec{n}_v(P_{I_v(t)}x)):\Dsym(u{-}v{-}w) \,\dx 
\\&~~~
-2\int_\Rd \chi_{0\leq \sigdist(x,I_v(t))\leq h^+_{e(t)}(P_{I_v(t)}x)} (\mu_+-\mu_-)
\\&~~~~~~~~~~~~~~~~
\times (\nabla w-W\otimes\vec{n}_v(P_{I_v(t)}x)):\Dsym(u{-}v{-}w) \,\dx. 
\end{align*}
Carrying out an analogous computation for $IV$,
using again the definition of the smoothed approximation $\chi_{v,h^+_{e(t)},h^-_{e(t)}}$ for $\chi_u$ from \eqref{def:improvedApprox} 
and using \eqref{EstimateDw} as well as \eqref{definitionW}, we then get the bound
\begin{align*}
&I + II + III + IV
\\&
\leq C \|v\|_{W^{1,\infty}} \bigg(\int_\Rd |\chi_u-\chi_{v,h^+_{e(t)},h^-_{e(t)}}| 
\,\dx\bigg)^{1/2} \bigg(\int_\Rd |\Dsym(u{-}v{-}w)|^2 \,\dx\bigg)^{1/2}
\\&~~~
+\frac{C}{r_c^2}\|v\|_{W^{2,\infty}(\Rd\setminus I_v(t))} \bigg(\int_{I_v(t)} |h^\pm_{e(t)}|^2 {+} 
|\nabla h^\pm_{e(t)}|^2 \,\dS\bigg)^{1/2} \bigg(\int_\Rd |\Dsym(u{-}v{-}w)|^2 \,\dx\bigg)^{1/2}.
\end{align*}
Plugging in the estimates \eqref{EstimateErrorHeightSmooth} and 
\eqref{ErrorImprovedInterfaceApproximationSmooth}, we obtain by Young's inequality
\begin{align}\label{aux2Avisc}
I + II + III + IV
&\leq \frac{C\delta^{-1}}{r_c^8 }\|v\|_{W^{2,\infty}(\Rd\setminus I_v(t))} ^2
E[\chi_u,u,V|\chi_v,v](t)
\\&~~~\nonumber
+ \frac{C\delta^{-1}}{r_c}\|v\|_{W^{1,\infty}}^2e(t)E[\chi_u,u,V|\chi_v,v]^\frac{1}{2}(t)
\\&~~~\nonumber
+C\delta^{-1}\|v\|_{W^{1,\infty}}^2E[\chi_u,u,V|\chi_v,v](t)
\\&~~~\nonumber
+\delta\|\Dsym(u-v-w)\|_{L^2}
\end{align}
for every $\delta\in (0,1)$. To estimate the last two terms $V$ and $VI$
in \eqref{aux1Avisc}, we may rewrite making use of the definition 
\eqref{definitionW} of the vector field $W$ and
abbreviating $\vec{n}_v=\vec{n}_v(P_{I_v(t)}x)$, 
$\sigdist=\sigdist(x,I_v(t))$ as well as $h_{e(t)}^+=h_{e(t)}^+(P_{I_v(t)}x)$
\begin{align*}
&-\int_\Rd \chi_{0\leq \sigdist\leq h^+_{e(t)}} 
((\mu^-{-}\mu^+)\Dsym v+\mu^- \Dsym w):\nabla (u{-}v{-}w) \,\dx
\\&
=
-\int_\Rd \chi_{0\leq \sigdist\leq h^+_{e(t)}} ((\mu^-{-}\mu^+)(\Id-\vec{n}_v\otimes \vec{n}_v)
(\Dsym v\cdot \vec{n}_v)\otimes \vec{n}_v+\mu^- \Dsym w)
\\&~~~~~~~~~~~~~~~~~~~~~~~~~~~~~~~~~~~~
~~~~~~~~~~~~~~~~~~~~~~~~~~~~~~~~~~~~~~~~~~~~~~~:\nabla (u{-}v{-}w) \,\dx
\\&~~~
-\int_\Rd \chi_{0\leq \sigdist\leq h^+_{e(t)}} (\mu^-{-}\mu^+)\Dsym v\, 
(\Id-\vec{n}_v\otimes \vec{n}_v):\nabla (u{-}v{-}w) \,\dx
\\&~~~
-\int_\Rd \chi_{0\leq \sigdist\leq h^+_{e(t)}} (\mu^-{-}\mu^+)(\vec{n}_v\cdot \Dsym v\cdot \vec{n}_v) 
(\vec{n}_v\otimes \vec{n}_v): \nabla (u{-}v{-}w) \,\dx
\\&
=
-\int_\Rd \chi_{0\leq \sigdist\leq h^+_{e(t)}} ((\mu^-{-}\mu^+)(\Id-\vec{n}_v\otimes
 \vec{n}_v)(\Dsym v\cdot \vec{n}_v)\otimes \vec{n}_v+\mu^- \Dsym w)
\\&~~~~~~~~~~~~~~~~~~~~~~~~~~~~~~~~~~~~~~~~~~~~~~~
~~~~~~~~~~~~~~~~~~~~~~~~~~~~~~~~~~~~
:\nabla (u{-}v{-}w) \,\dx
\\&~~~
-\int_\Rd \chi_{0\leq \sigdist\leq h^+_{e(t)}} (\mu^-{-}\mu^+)\Dsym v\, 
(\Id-\vec{n}_v\otimes \vec{n}_v):\nabla (u{-}v{-}w) \,\dx
\\&~~~
+\int_\Rd \chi_{0\leq \sigdist\leq h^+_{e(t)}} (\mu^-{-}\mu^+)(\vec{n}_v\cdot \Dsym v\cdot \vec{n}_v)
 (\Id-\vec{n}_v\otimes \vec{n}_v):\nabla (u{-}v{-}w) \,\dx,
\\&
=
\frac{1}{2}\int_\Rd \chi_{0\leq \sigdist\leq h^+_{e(t)}} 
((W\otimes \vec{n}_v-\nabla w)+(W\otimes \vec{n}_v-\nabla w)^T):\nabla (u{-}v{-}w) \,\dx
\\&~~~
+(\mu^--\mu^+)\int_\Rd \chi_{0\leq \sigdist\leq h^+_{e(t)}} 
((\mathrm{Id}{-}\vec{n}_v\otimes\vec{n}_v)(\Dsym v\cdot\vec{n}_v)\otimes\vec{n}_v):\nabla (u{-}v{-}w) \,\dx
\\&~~~~
-\int_\Rd \chi_{0\leq \sigdist\leq h^+_{e(t)}} (\mu^-{-}\mu^+)\Dsym v\, 
(\Id-\vec{n}_v\otimes \vec{n}_v):\nabla (u{-}v{-}w) \,\dx
\\&~~~
+\int_\Rd \chi_{0\leq \sigdist\leq h^+_{e(t)}} (\mu^-{-}\mu^+)(\vec{n}_v\cdot \Dsym v\cdot \vec{n}_v)
 (\Id-\vec{n}_v\otimes \vec{n}_v):\nabla (u{-}v{-}w) \,\dx,
\end{align*}
where in the penultimate step we have used the fact that $\nabla \cdot (u-v-w)=0$,
and in the last step we added zero. This yields after an integration by parts
\begin{align*}
&-\int_\Rd \chi_{0\leq \sigdist\leq h_{e(t)}^+} 
((\mu^-{-}\mu^+)\Dsym v+\mu^- \Dsym w):\nabla (u{-}v{-}w) \,\dx
\\&
=
\frac{1}{2}\int_\Rd \chi_{0\leq \sigdist\leq h_{e(t)}^+} 
((W\otimes \vec{n}_v-\nabla w)+(W\otimes \vec{n}_v-\nabla w)^T):\nabla (u{-}v{-}w) \,\dx
\\&~~~
-(\mu^-{-}\mu^+)\int_\Rd \chi_{0\leq \sigdist\leq h_{e(t)}^+} 
\nabla\cdot(\vec{n}_v\otimes (\mathrm{Id}{-}\vec{n}_v\otimes\vec{n}_v)(\Dsym v\cdot\vec{n}_v)) 
\cdot (u{-}v{-}w) \,\dx
\\&~~~
+(\mu^-{-}\mu^+)\int_\Rd (\vec{n}_v\cdot(u{-}v{-}w))
(\mathrm{Id}{-}\vec{n}_v\otimes\vec{n}_v)(\Dsym v\cdot\vec{n}_v)
\cdot\,\mathrm{d}\nabla \chi_{0\leq \sigdist\leq h_{e(t)}^+}
\\&~~~
+(\mu^-{-}\mu^+)\int_\Rd \chi_{0\leq \sigdist\leq h_{e(t)}^+} 
\nabla \cdot \big((\Dsym v{-}(\vec{n}_v\cdot \Dsym v\cdot \vec{n}_v) \Id)
(\Id{-}\vec{n}_v\otimes \vec{n}_v)\big)
\\&~~~~~~~~~~~~~~~~~~~~~~~~~~~~~~~~~~~~~~~~~~~~~~~~~~~~
~~~~~~~~~~~~~~~~~~~~~~~~~~~~~~~~~~~~
\cdot (u{-}v{-}w) \,\dx
\\&~~~
+(\mu^-{-}\mu^+)\int_\Rd (u{-}v{-}w)
\\&~~~~~~~~~~~~~~~~~~~~~~~~~
\cdot
(\Dsym v{-}(\vec{n}_v\cdot \Dsym v\cdot \vec{n}_v) \Id)
(\Id{-}\vec{n}_v\otimes \vec{n}_v)
\,\mathrm{d}\nabla \chi_{0\leq \sigdist\leq h_{e(t)}^+}.
\end{align*}
As a consequence of \eqref{EstimateDw}, \eqref{EstimateErrorHeightSmooth}, 
\eqref{Bound2ndDerivativeSignedDistance} and the global Lipschitz estimate $|\nabla h_e^\pm(\cdot,t)|\leq Cr_c^{-2}$
from Proposition~\ref{PropositionInterfaceErrorHeightRegularized}, we obtain
\begin{align*}
&\bigg|\int_\Rd \chi_{0\leq \sigdist(x,I_v(t))\leq h_{e(t)}^+(P_{I_v(t)}x)} 
((\mu^--\mu^+)\Dsym v+\mu^- \Dsym w):\nabla (u-v-w) \,\dx\bigg|
\\&
\leq
\frac{C}{r_c^{7/2}} \|v\|_{W^{2,\infty}(\Rd\setminus I_v(t))} 
E\big[\chi_u,u,V\big|\chi_v,v\big]^{1/2} \|\nabla(u-v-v)\|_{L^2}
\\&~~~
+\frac{C}{r_c}\|v\|_{W^{2,\infty}(\Rd\setminus I_v(t))}
\int_\Rd \chi_{0\leq \sigdist(x,I_v(t))\leq h_{e(t)}^+(P_{I_v(t)}x)} |u-v-w| \,\dx
\\&~~~
+\frac{C}{r_c^2} 
\|v\|_{W^{1,\infty}} \int_{I_v(t)} \sup_{y\in (-r_c,r_c)} |u-v-w|(x+y\vec{n}_v(x,t)) |\nabla h_{e(t)}^+(x)| \,\dS(x).
\end{align*}
By a change of variables $\Phi_t$, \eqref{BoundNablaPhiNablaPhi-1}, \eqref{BoundL2Linfty}, \eqref{EstimateErrorHeightSmooth}
and an application of Young's and Korn's inequality, the latter two terms may be further estimated by
\begin{align*}
&\frac{C}{r_c^2}\|v\|_{W^{2,\infty}(\Rd\setminus I_v(t))}
\bigg(\int_{I_v(t)}\sup_{y\in (-r_c,r_c)} |u-v-w|^2(x+y\vec{n}_v(x,t))\,\dS\bigg)^\frac{1}{2}
\\&~~~~~~~~~~~~~\times
\bigg(\int_{I_v(t)}|h_{e(t)}^+|^2+|\nabla h_{e(t)}^+|^2\,\dS\bigg)^\frac{1}{2}
\\&
\leq \frac{C}{r_c^{3}}\|v\|_{W^{2,\infty}(\Rd\setminus I_v(t))}
E[\chi_u,u,V|\chi_v,v]^\frac{1}{2}(t)\|u-v-w\|_{L^2}
\\&~~~
+\frac{C}{r_c^{2}}\|v\|_{W^{2,\infty}(\Rd\setminus I_v(t))}E[\chi_u,u,V|\chi_v,v]^\frac{1}{2}(t)
\|\nabla (u-v-w)\|_{L^2}
\\&
\leq \frac{C\delta^{-1}}{r_c^{4}}\|v\|_{W^{2,\infty}(\Rd\setminus I_v(t))}^2
E[\chi_u,u,V|\chi_v,v](t)
+\delta\|\Dsym(u-v-w)\|_{L^2}
\end{align*}
for every $\delta\in (0,1]$. In total, we obtain the bound
\begin{align}\label{aux3Avisc}
V &\leq \frac{C\delta^{-1}}{r_c^{4}}\|v\|_{W^{2,\infty}(\Rd\setminus I_v(t))}^2
E[\chi_u,u,V|\chi_v,v](t)
+\delta\|\Dsym(u-v-w)\|_{L^2}
\end{align}
where $\delta\in (0,1)$ is again arbitrary. Analogously, one can derive a bound
of the same form for the last term $VI$ in \eqref{aux1Avisc}. Together with the bounds from \eqref{aux2Avisc} as well as \eqref{aux3Avisc} this concludes the proof.
\end{proof}

\subsection{Estimate for terms with the time derivative of the compensation function}

We proceed with the estimate for the terms from the relative entropy inequality of 
Proposition~\ref{PropositionRelativeEntropyInequalityFull}
\begin{align}\label{aux0Adt}
A_{dt}:=&-\int_0^T \int_{\mathbb{R}^d} \rho(\chi_u) (u-v-w)\cdot 
\partial_t w \,\dx\,\dt \\\nonumber
&-\int_0^T\int_\Rd \rho(\chi_u) (u-v-w) \cdot (v\cdot\nabla) w\,\dx\,\dt,
\end{align}
which are related to the time derivative of the compensation function $w$.

\begin{lemma}
Let the assumptions and notation of Proposition~\ref{PropositionCompensationFunction} be in place. 
In particular, we assume that there exists a $C^1$-function
$e\colon[0,\Tmax)\to [0,r_c)$ such that the relative entropy is bounded by $E[\chi_u,u,V,|\chi_v,v](t)\leq e^2(t)$.

Then, for any $\delta>0$ there exists a constant $C>0$ such that $A_{dt}$ 
may be estimated by
\begin{align}\label{GronwallAdt}
A_{dt}
&\leq\frac{C}{r_c^{22}}
\|v\|_{L^\infty_tW^{1,\infty}_x}^2(1{+}\|v\|_{L^\infty_tW^{2,\infty}_x(\Rd\setminus I_v(t))})^2
\int_0^T(1{+}|\log e(t)|)E[\chi_u,u,V|\chi_v,v](t)\,\dt
\\&\nonumber~~~
+ \frac{C}{r_c^{11}}
\|v\|_{L^\infty_tW^{1,\infty}_x}(1{+}\|v\|_{L^\infty_tW^{2,\infty}_x(\Rd\setminus I_v(t))})
\int_0^T(1{+}|\log e(t)|)E[\chi_u,u,V|\chi_v,v](t)\,\dt
\\&\nonumber~~~
+\frac{C}{r_c^{8}}(1{+}\|v\|_{L^\infty_tW^{1,\infty}_x})
(\|\partial_t \nabla v\|_{L^\infty_{x,t}}{+}(R^2{+}1)\|v\|_{L^\infty_t W^{2,\infty}_x(\Rd\setminus I_v(t))})
\\&\nonumber~~~~~~~~~~~~~~~\times
\int_0^TE[\chi_u,u,V|\chi_v,v](t)\,\dt
\\&\nonumber~~~
+\frac{C}{r_c^2}\|v\|_{L^\infty_t W^{1,\infty}_x}^2\int_0^T(1+e'(t))
E[\chi_u,u,V|\chi_v,v](t)\,\dt
\\&\nonumber~~~
+\delta\int_0^T\int_{\Rd}|\Dsym(u-v-w)|^2\,\dx\,\dt.
\end{align}
\end{lemma}

\begin{proof}
To estimate the terms involving the time derivative of $w$ we make use of 
the decomposition of $\partial_t w + (v\cdot\nabla)w$ from \eqref{Estimatedtw}:
\begin{align*}
&\bigg|-\int_0^T\int_\Rd \rho(\chi_u) (u{-}v{-}w) \cdot \partial_t w \,\dx\,\dt
-\int_0^T\int_\Rd \rho(\chi_u) (u{-}v{-}w) \cdot (v\cdot \nabla) w \,\dx\,\dt\bigg| 
\\&\leq \int_0^T\|g\|_{L^2}\|u-v-w\|_{L^2}\,\dt + 
\int_0^T\|\hat g\|_{L^\frac{4}{3}}\|u-v-w\|_{L^4}\,\dt. 
\end{align*}
Employing the bounds \eqref{HeightFunctionEstimate}, \eqref{HeightFunctionGradientEstimate} and the assumption $E[\chi_u,u,V|\chi_v,v](t)\leq e(t)^2$
together with the Orlicz-Sobolev embedding \eqref{boundOrliczHeight} from Proposition~\ref{OrliczHeight}
or \eqref{boundOrliczHeight1D} from Lemma~\ref{OrliczHeight1D} depending
on the dimension, we obtain
\begin{align}\label{aux2Adt}
\bigg(\int_{I_v(t)} |\bar{h}^\pm|^4\,\mathrm{d}S\bigg)^\frac{1}{4} 
\leq \frac{C}{r_c^6}e(t)\Big(1+\log\frac{1}{e(t)}\Big)^\frac{1}{4}.
\end{align}
Making use of \eqref{EstimateErrorHeightSmooth}, the bound for the vector field $\hat g$ from \eqref{timeEvolutionWL43}, the Gagliardo-Nirenberg-Sobolev embedding
$\|u{-}v{-}w\|_{L^4}\leq C\|\nabla(u{-}v{-}w)\|^{1-\alpha}_{L^2}
\|u{-}v{-}w\|^{\alpha}_{L^2}$, with $\alpha=\frac{1}{2}$ for $d=2$ and $\alpha=\frac{1}{4}$ for $d=3$, 
as well as the assumption $E[\chi_u,u,V|\chi_v,v](t)\leq e(t)^2$ we obtain
\begin{align}\label{aux1Adt}
&\|\hat{g}\|_{L^\frac{4}{3}}\|u-v-w\|_{L^4} 
\\&\nonumber
\leq C\frac{\|v\|_{W^{1,\infty}}\|v\|_{W^{2,\infty}(\Rd\setminus I_v(t))}}
{r_c^{11}}\Big(1+\log\frac{1}{e(t)}\Big)^\frac{1}{4}
\\&~~~~~~~~~~~~~~\nonumber\times
(\|\nabla(u{-}v{-}w)\|_{L^2}+
\|u{-}v{-}w\|_{L^2})
E[\chi_u,u,V|\chi_v,v]^\frac{1}{2}(t)
\\&~~~\nonumber
+ C\frac{\|v\|_{W^{1,\infty}}}
{r_c^8}\Big(1+\log\frac{1}{e(t)}\Big)^\frac{1}{4}
(\|\nabla(u{-}v{-}w)\|_{L^2}{+}
\|u{-}v{-}w\|_{L^2})\|u{-}v{-}w\|_{L^2}
\\&~~~\nonumber
+ C\frac{\|v\|_{W^{1,\infty}}}
{r_c^8}\Big(1+\log\frac{1}{e(t)}\Big)^\frac{1}{4}
\|\nabla(u{-}v{-}w)\|^{\frac{3}{2}-\alpha}_{L^2}
\|u{-}v{-}w\|^{\frac{1}{2}+\alpha}_{L^2}
\\&~~~\nonumber
+C\|v\|_{W^{1,\infty}}(1{+}\|v\|_{W^{1,\infty}})
E[\chi_u,u,V|\chi_v,v]^\frac{1}{2}(t)
(\|\nabla(u{-}v{-}w)\|_{L^2}{+}\|u{-}v{-}w\|_{L^2}).
\end{align}
Now, by an application of Young's and Korn's inequality for all the terms
on the right hand side of \eqref{aux1Adt} which include an $L^2$-norm 
of the gradient of $u-v-w$ (in the case $d=3$ we use
$a^\frac{5}{4}b^\frac{3}{4}=(a(8\delta/5)^\frac{1}{2})^\frac{5}{4}
(b(8\delta/5)^{-\frac{5}{6}})^\frac{3}{4}
\leq\delta a^2 + \frac{3}{8}\big(\frac{8}{5}\big)^{-\frac{5}{3}}\delta^{-\frac{5}{3}}b^2$, which follows
from Young's inequality with exponents $p=\frac{8}{5}$ and $q=\frac{8}{3}$) we obtain
\begin{align}\nonumber\label{aux3Adt}
&\|\hat{g}\|_{L^\frac{4}{3}}\|u-v-w\|_{L^4} 
\\&
\leq \frac{C}{\delta^\frac{5}{3} r_c^{22}}
\|v\|_{W^{1,\infty}}^2(1{+}\|v\|_{W^{2,\infty}(\Rd\setminus I_v(t))})^2
(1{+}|\log e(t)|)E[\chi_u,u,V|\chi_v,v](t)
\\&~~~\nonumber
+ \frac{C}{r_c^{11}}
\|v\|_{W^{1,\infty}}(1{+}\|v\|_{W^{2,\infty}(\Rd\setminus I_v(t))})
(1{+}|\log e(t)|)E[\chi_u,u,V|\chi_v,v](t)
\\&~~~\nonumber
+\delta\|\Dsym(u{-}v{-}w)\|^2_{L^2},
\end{align}
where $\delta\in (0,1)$ is arbitrary. This gives the desired bound for the
$L^\frac{4}{3}$-contribution of $\partial_t w + (v\cdot\nabla)w$.
Concerning the $L^2$-contribution, we estimate using \eqref{HeightFunctionEstimate}, \eqref{EstimateErrorHeightSmooth}, the bound for $\|g\|_{L^2}$ from \eqref{timeEvolutionWL2}
as well as the assumption $E[\chi_u,u,V|\chi_v,v](t)\leq e(t)^2$
\begin{align}\label{aux4Adt}
&\|g\|_{L^2}\|u-v-w\|_{L^2} \\
&\nonumber
\leq C \frac{1{+}\|v\|_{W^{1,\infty}}}{r_c^{8}}
(\|\partial_t \nabla v\|_{L^\infty(\Rd\setminus I_v(t))}{+}(R^2{+}1)\|v\|_{W^{2,\infty}(\Rd\setminus I_v(t))})
E[\chi_u,u,V|\chi_v,v]^\frac{1}{2}(t)\|u{-}v{-}w\|_{L^2}
\\&~~~\nonumber
+C\|v\|_{W^{1,\infty}}(1+\|v\|_{W^{1,\infty}})
E[\chi_u,u,V|\chi_v,v]^\frac{1}{2}(t)\|u{-}v{-}w\|_{L^2}
\\&~~~\nonumber
+\frac{C}{r_c^2}(1+e'(t))\|v\|_{W^{1,\infty}}^2
E[\chi_u,u,V|\chi_v,v]^\frac{1}{2}(t)\|u{-}v{-}w\|_{L^2}
\\&~~~\nonumber
+C\frac{\|v\|_{W^{1,\infty}}(1{+}\|v\|_{W^{2,\infty}(\Rd\setminus I_v(t))})}{r_c}
E[\chi_u,u,V|\chi_v,v]^\frac{1}{2}(t)\|u{-}v{-}w\|_{L^2}
\\&~~~\nonumber
+C\|v\|_{W^{1,\infty}}(\|\nabla(u{-}v{-}w)\|_{L^2}+\|u{-}v{-}w\|_{L^2})\|u{-}v{-}w\|_{L^2}.
\end{align}
Hence, by another application of Young's and Korn's inequality, we may bound 
\begin{align}\label{aux5Adt}
&\|g\|_{L^2}\|u-v-w\|_{L^2} \\\nonumber
&\leq\frac{C}{r_c^{8}}(1{+}\|v\|_{W^{1,\infty}})
(\|\partial_t \nabla v\|_{L^\infty(\Rd\setminus I_v(t))}{+}(R^2{+}1)\|v\|_{W^{2,\infty}(\Rd\setminus I_v(t))})
E[\chi_u,u,V|\chi_v,v](t)
\\&\nonumber~~~
+\frac{C}{r_c^2}\|v\|_{W^{1,\infty}}^2(1+e'(t))
E[\chi_u,u,V|\chi_v,v](t)
\\&\nonumber~~~
+C\delta^{-1}\|v\|_{W^{1,\infty}}^2E[\chi_u,u,V|\chi_v,v](t)
\\&\nonumber~~~
+\delta\|\Dsym(u{-}v{-}w)\|^2_{L^2}
\end{align}
where $\delta\in (0,1]$ is again arbitrary. All in all, \eqref{aux3Adt} and \eqref{aux5Adt} therefore imply the desired bound.
\end{proof}

\subsection{Estimate for the additional advection terms}

We move on with the additional advection terms
from the relative entropy inequality of 
Proposition~\ref{PropositionRelativeEntropyInequalityFull}
\begin{align}\label{aux0Aadv}
A_{adv} = &-\int_0^T\int_\Rd \rho(\chi_u) (u-v-w)\cdot (w\cdot \nabla)(v+w) \,\dx\,\dt \\\nonumber
&-\int_0^T\int_\Rd \rho(\chi_u) (u-v-w)\cdot \big((u-v-w)\cdot \nabla\big)w \,\dx\,\dt.
\end{align}
A precise estimate is the content of the following result.

\begin{lemma}
Let the assumptions and notation of Proposition~\ref{PropositionCompensationFunction} be in place. 
In particular, we assume that there exists a $C^1$-function
$e\colon[0,\Tmax)\to [0,r_c)$ such that the relative entropy is bounded by $E[\chi_u,u,V,|\chi_v,v](t)\leq e^2(t)$.
Then the additional advection terms $A_{adv}$ may be bounded by a Gronwall-type term
\begin{align}\label{GronwallAadv}
A_{adv} \leq \frac{C}{r_c^{14}}(1{+}R)\|v\|_{L^\infty_tW^{3,\infty}_x(\Rd\setminus I_v(t))}^2\int_0^T(1{+}|\log e(t)|)
E[\chi_u,u,V|\chi_v,v](t)\,\dt.
\end{align}
\end{lemma}

\begin{proof}
A straightforward estimate yields
\begin{align*}
A_{adv} &\leq C(\|v\|_{L^\infty_tW^{1,\infty}_x}{+}\|\nabla w\|_{L^\infty_{x,t}})
\|u{-}v{-}w\|_{L^2_{x,t}}\bigg(\int_0^T\int_\Rd |w|^2\,\dx\,\dt\bigg)^\frac{1}{2}
\\&~~~
+C\|\nabla w\|_{L^\infty_{x,t}}\|u{-}v{-}w\|_{L^2_{x,t}}^2.
\end{align*}
Making use of \eqref{Estimatew}, \eqref{EstimateLinftyw} as well as \eqref{EstimateErrorHeightSmooth}
immediately shows that the desired bound holds true. 
\end{proof}

\subsection{Estimate for the additional weighted volume term} 
\label{SectionEstimateAweightVol}

It finally remains to state the estimate for the additional weighted volume term
from the relative entropy inequality of 
Proposition~\ref{PropositionRelativeEntropyInequalityFull}
\begin{align}\label{aux0AweightVol}
A_{weightVol} :=
\int_0^T\int_{\R^d}(\chi_u{-}\chi_v)(w\cdot\nabla)
\beta\Big(\frac{\sigdist(\cdot,I_v)}{r_c}\Big)\,\dx\,\dt.
\end{align}
\begin{lemma}

Let the assumptions and notation of Proposition~\ref{PropositionCompensationFunction} be in place.
In particular, we assume that there exists a $C^1$-function
$e\colon[0,\Tmax)\to [0,r_c)$ such that the relative entropy is bounded by $E[\chi_u,u,V,|\chi_v,v](t)\leq e^2(t)$.
Then the additional weighted volume term $A_{weightVol}$ may be bounded by a Gronwall term
\begin{align}\label{GronwallAweightVol}
A_{weightVol}\leq\frac{C}{r_c^{10}}(1+\|v\|_{L^\infty_tW^{2,\infty}_x(\Rd\setminus I_v(t))}^2)
\int_0^TE[\chi_u,u,V|\chi_v,v](t)\,\mathrm{d}t.
\end{align}
\end{lemma}

\begin{proof}
We may use the exact same argument as in the derivation of the estimate for the term $III$ from the additional surface tension terms $A_{surTen}$, see \eqref{aux2AsurTen}.
\end{proof}

\subsection{The weak-strong uniqueness principle with different viscosities}
\label{SectionProofResult}

Before we proceed with the proof of Theorem~\ref{weakStrongUniq}, let us summarize
the estimates from the previous sections in the form of a post-processed relative entropy inequality.
The proof is a direct consequence of the relative entropy inequality from
Proposition~\ref{PropositionRelativeEntropyInequalityFull} and the bounds
\eqref{RsurTenEqualVisc}, \eqref{RadvEqualVisc}, \eqref{RdtEqualVisc}, \eqref{RweightVolEqualVisc},
\eqref{GronwallAsurTen}, \eqref{GronwallAvisc}, \eqref{GronwallAdt}, \eqref{GronwallAadv}
and \eqref{GronwallAweightVol}.

\begin{proposition}[Post-processed relative entropy inequality]
\label{PropositionRelativeEntropyInequalityPostProcessed}
Let $d\leq 3$. Let $(\chi_u,u,V)$ be a varifold solution to the free boundary problem 
for the incompressible Navier--Stokes equation for two fluids 
\eqref{EquationTransport}--\eqref{EquationIncompressibility} in the sense of 
Definition~\ref{DefinitionVarifoldSolution} on some time interval $[0,\Tend)$. 
Let $(\chi_v,v)$ be a strong solution to \eqref{EquationTransport}--\eqref{EquationIncompressibility} 
in the sense of Definition~\ref{DefinitionStrongSolution} on some time interval $[0,\Tmax)$ with $\Tmax\leq \Tend$.

Let $\xi$ be the extension of the inner unit normal vector field $\vec{n}_v$ of the interface $I_v(t)$ 
from Definition~\ref{def:ExtNormal}. Let $w$ be the vector field contructed in Proposition~\ref{PropositionCompensationFunction}.
Let $\beta$ be the truncation of the identity from Proposition~\ref{PropositionRelativeEntropyInequalityFull},
and let $\theta$ be the density $\theta_t=\frac{\mathrm{d}|\nabla\chi_u(\cdot,t)|}{\mathrm{d}|V_t|_{\Sb^{d-1}}}$.
Let $e\colon [0,\Tmax)\to(0,r_c]$ be a $C^1$-function
and assume that the relative entropy
\begin{align*}
E\big[\chi_u,u,V\big|\chi_v,v\big](T) &:=
\sigma\int_\Rd 1-\xi(\cdot,T)\cdot
\frac{\nabla\chi_u(\cdot,T)}{|\nabla\chi_u(\cdot,T)|} 
\,\mathrm{d}|\nabla \chi_u(\cdot,T)| 
\\&~~~~\nonumber
+ \int_{\Rd} \frac{1}{2} \rho\big(\chi_u(\cdot,T)\big)
\big|u-v-w\big|^2(\cdot,T) \,\dx 
\\&~~~~\nonumber
+\int_\Rd \big|\chi_u(\cdot,T)-\chi_v(\cdot,T)\big|\,\Big|\beta\Big(\frac{\sigdist(\cdot,I_v(T))}{r_c}\Big)\Big|\,\dx
\\&~~~~\nonumber
+\sigma\int_\Rd 1-\theta_T\,\mathrm{d}|V_T|_{\Sb^{d-1}}
\end{align*}
is bounded by $E[\chi_u,u,V|\chi_v,v](t)\leq e(t)^2$.

Then the relative entropy is subject to the estimate
\begin{align}\label{eqT1}
&E[\chi_u,u,V|\chi_v,v](T)
+c\int_0^T \int_\Rd |\nabla (u-v-w)|^2 \dx~\dt
\\&\nonumber
\leq E[\chi_u,u,V|\chi_v,v](0)
\\&~~~\nonumber
+ C\int_0^T (1+|\log e(t)|)\,E[\chi_u,u,V|\chi_v,v](t)\,\dt
\\&~~~\nonumber
+ C\int_0^T (1+|\log e(t)|)\,e(t)\sqrt{E[\chi_u,u,V|\chi_v,v](t)}\,\dt
\\&~~~\nonumber
+ C\int_0^T \Big(\frac{\mathrm{d}}{\dt}e(t)\Big)E[\chi_u,u,V|\chi_v,v](t)\,\dt
\end{align}
for almost every $T\in [0,\Tmax)$.
Here, $C>0$ is a constant which is structurally of the form
$C=\widetilde Cr_c^{-22}$ with a constant 
$\tilde C=\tilde C(r_c,\|v\|_{L^\infty_tW^{3,\infty}_x},\|\partial_t v\|_{L^\infty_tW^{1,\infty}_x})$,
depending on the various norms of the velocity field of the strong solution, the regularity parameter $r_c$ of the interface of the strong solution, and the physical parameters $\rho^\pm$, $\mu^\pm$, and $\sigma$.
\end{proposition}

We have everything in place to to prove the main result of this work.

\begin{proof}[Proof of Theorem~\ref{weakStrongUniq}]
The proof of Theorem~\ref{weakStrongUniq} is based on the post-processed relative entropy inequality of Proposition~\ref{PropositionRelativeEntropyInequalityPostProcessed}. It amounts to nothing but a more technical version of the upper bound 
\begin{align*}
E(t)
\leq e^{e^{-C t} \log E(0)}
\end{align*}
valid for all solutions of the differential inequality $\frac{d}{dt} E(t)\leq C E(t) |\log E(t)|$. However, it is made more technical by the more complex right-hand side \eqref{PropositionRelativeEntropyInequalityPostProcessed} in the relative entropy inequality (which involves the anticipated upper bound $e(t)^2$) and the smallness assumption on the relative entropy $E[\chi_u,u,V|\chi_v,v](t)$ needed for the validity of the relative entropy inequality.

We start the proof with the precise choice of the function $e(t)$ as 
well as the necessary smallness assumptions on the initial relative
entropy. We then want to exploit the post-processed form of the
relative entropy inequality from Proposition~\ref{PropositionRelativeEntropyInequalityPostProcessed}
to compare $E[\chi_u,u,V|\chi_v,v](t)$ with $e(t)$. 

Let $C>0$ be the constant from Proposition~\ref{PropositionRelativeEntropyInequalityPostProcessed}
and choose $\delta>0$ such that $\delta<\frac{1}{6(C+1)}$.
Let $\varepsilon>0$ (to be chosen in a moment, but finally we will let $\varepsilon\to 0$) 
and consider the strictly increasing function
\begin{align}\label{definitionUpperBound}
e(t) := e^{\frac{1}{2}e^{-\frac{t}{\delta}}\log (E[\chi_u,u,V|\chi_v,v](0)+\varepsilon)}.
\end{align}
Note that $e^2(0) = E[\chi_u,u,V|\chi_v,v](0)+\varepsilon$
which strictly dominates the relative entropy at the initial time. 
To ensure the smallness of this function, let us choose $c>0$ small
enough such that whenever we have $E[\chi_u,u,V|\chi_v,v](0)<c$ and $\varepsilon<c$,
it holds that
\begin{align}\label{smallness1}
e(t)<\frac{1}{3C}\wedge r_c
\end{align}
for all $t\in [0,\Tmax)$. This is indeed possible since the condition in \eqref{smallness1}
is equivalent to $\frac{1}{2} \log(E[\chi_u,u,V|\chi_v,v](0)+\varepsilon)<e^{\frac{\Tmax}{\delta}}\log(\frac{1}{3C}\wedge r_c)$. 
For technical reasons to be seen later, we will also require $c>0$ be small enough such that
\begin{align}\label{smallness2}
e^{-\frac{\Tmax}{\delta}}\frac{1}{6\delta}
\big|\log (E[\chi_u,u,V|\chi_v,v](0)+\varepsilon)\big|>C
\end{align}
whenever $E[\chi_u,u,V|\chi_v,v](0)<c$ and $\varepsilon<c$. 
We proceed with some further computations. We start with
\begin{align}\label{timeEvolutionUpperBound}
\frac{\mathrm{d}}{\mathrm{d}t} e(t) 
=\frac{1}{2\delta}|\log (E[\chi_u,u,V|\chi_v,v](0)+\varepsilon)|e(t)e^{-\frac{t}{\delta}}
= \frac{1}{\delta}|\log e(t)|e(t).
\end{align}
This in particular entails
\begin{align}\nonumber\label{timeEvolutionSquare}
e^2(T) -e^2(\tau) &= \int_\tau^T \frac{\mathrm{d}}{\mathrm{d}t} e^2(t)\,\dt
\\&
=\frac{1}{\delta}|\log (E[\chi_u,u,V|\chi_v,v](0)+\varepsilon)|
\int_\tau^T  e^2(t)e^{-\frac{t}{\delta}}\,\dt.
\end{align}

After these preliminary considerations, let us consider the relative entropy inequality
from Proposition~\ref{PropositionRelativeEntropyInequalityFull}. 
Arguing similarly to the derivation of the relative entropy inequality in Proposition~\ref{PropositionRelativeEntropyInequalityFull} but using the energy dissipation inequality in its weaker form $E[\chi_u,u,V|\chi_v,v](T)\leq E[\chi_u,u,V|\chi_v,v](\tau)$ for a.\,e.\ $\tau\in [0,T]$, we may deduce (upon modifying the solution on a subset of $[0,\Tmax)$ of vanishing measure)
\begin{align}\label{limsupRelativeEntropy}
\limsup_{T\downarrow\tau}E[\chi_u,u,V|\chi_v,v](T)\leq E[\chi_u,u,V|\chi_v,v](\tau)
\end{align}
for all $\tau\in [0,\Tmax)$. Now, consider the set $\mathcal{T}\subset [0,\Tmax)$
which contains all $\tau\in [0,\Tmax)$ such that 
$\limsup_{T\downarrow\tau}E[\chi_u,u,V|\chi_v,v](T)>e^2(\tau)$. Note that $0\in\mathcal{T}$.
We define
\begin{align*}
T^* := \inf \mathcal{T}.
\end{align*}
Since $E[\chi_u,u,V|\chi_v,v](0)<e^2(0)$ and $e^2$ is strictly increasing, we deduce
by the same argument which established \eqref{limsupRelativeEntropy} that $T^*>0$.
Hence, we can apply Proposition~\ref{PropositionRelativeEntropyInequalityPostProcessed}
at least for times $T<T^*$ (with $\tau = 0$). However, by the same argument as before 
the relative entropy inequality from Proposition~\ref{PropositionRelativeEntropyInequalityFull}
shows that $E[\chi_u,u,V|\chi_v,v](T^*)\leq E[\chi_u,u,V|\chi_v,v](T)+C(T^*-T)$ for all $T<T^*$,
whereas $E[\chi_u,u,V|\chi_v,v](T)$ may be bounded by means of the post-processed relative entropy inequality.
Hence, we obtain using also \eqref{definitionUpperBound} and
\eqref{timeEvolutionUpperBound}
\begin{align}\label{eqT3}
E[\chi_u,u,V|\chi_v,v](T^*) 
&\leq E[\chi_u,u,V|\chi_v,v](0)
\\&~~~\nonumber
+ C\int_0^{T^*} e^2(t)\,\dt 
\\&~~~\nonumber
+ C\frac{1}{2\delta}\big|\log (E[\chi_u,u,V|\chi_v,v](0)+\varepsilon)\big|\int_0^{T^*} e^3(t)e^{-\frac{t}{\delta}}\,\dt
\\&~~~\nonumber
+ C\frac{1}{2}\big|\log (E[\chi_u,u,V|\chi_v,v](0)+\varepsilon)\big|\int_0^{T^*} e^2(t)e^{-\frac{t}{\delta}}\,\dt.
\end{align}
We compare this to the equation \eqref{timeEvolutionSquare} for $e^2(t)$ (with $\tau=0$ and $T=T^*$). 
Recall that $e^2(0)$ strictly dominates the relative entropy at the initial time.
Because of \eqref{smallness2}, the second term on the right hand side of \eqref{eqT3}
is dominated by one third of the right hand side of \eqref{timeEvolutionSquare}. 
Because of \eqref{smallness1} and the choice $\delta<\frac{1}{6(C+1)}$
the same is true for the other two terms on the right hand side of \eqref{eqT3}.
In particular, we obtain using also \eqref{limsupRelativeEntropy} 
\begin{align*}
\limsup_{T\downarrow T^*}E[\chi_u,u,V|\chi_v,v](T) - e^2(T^*) \leq
E[\chi_u,u,V|\chi_v,v](T^*) - e^2(T^*) < 0,
\end{align*}
which contradicts the defining property of $T^*$. This concludes the proof
since the asserted stability estimate as well as the weak-strong uniqueness 
principle is now a consequence of letting $\varepsilon\to 0$.
\end{proof}

\section{Derivation of the relative entropy inequality}
\label{sec:ProofRelativeEntropyInequalityFull}

\begin{proof}[Proof of Proposition~\ref{PropositionRelativeEntropyInequalityFull}] 
We start with the following observation. Since the phase-dependent density $\rho(\chi_v)$ 
depends linearly on the indicator function $\chi_v$ of the volume occupied by the first fluid,
it consequently satisfies
\begin{align}\label{weakDensity}\nonumber
\int_{\R^d}\rho(\chi_v(\cdot,T))&\varphi(\cdot,T)\,\dx -\int_{\R^d}\rho(\chi_v^0)\varphi(\cdot,0)\,\dx 
\\ &= \int_0^T\int_{\R^d} \rho(\chi_v)(\partial_t\varphi+(v\cdot\nabla)\varphi)\,\dx\,\dt
\end{align}
for almost every $T\in [0,\Tmax)$ and all $\varphi\in C_{cpt}^\infty(\R^d\times [0,\Tmax))$. By approximation, the equation holds for all $\varphi\in W^{1,\infty}(\R^d\times [0,\Tmax))$.
Testing this equation with $v\cdot\eta$, where $\eta\in C^\infty_{cpt}(\R^d\times [0,\Tmax);\R^d)$
is a smooth vector field, we then obtain
\begin{equation}
\begin{aligned}\label{weakDensityAux}
\int_{\R^d}&\rho(\chi_v(\cdot,T))v(\cdot,T)\cdot\eta(\cdot,T)\,\dx 
- \int_{\R^d}\rho(\chi_v^0)v_0\cdot\eta(\cdot,0)\,\dx 
\\ &= \int_0^T\int_{\R^d} \rho(\chi_v)(v\cdot\partial_t\eta+\eta\cdot\partial_t v)\,\dx\,\dt
\\ &~~~+\int_0^T\int_{\R^d} \rho(\chi_v)(\eta\cdot(v\cdot\nabla)v + v\cdot(v\cdot\nabla)\eta)\,\dx \,\dt
\end{aligned}
\end{equation}
for almost every $T\in [0,\Tmax)$. Note that the velocity field $v$ of a strong solution has
the required regularity to justify the preceding step. Next, we subtract from \eqref{weakDensityAux}
the equation for the momentum balance \eqref{weakNSb} of the strong solution evaluated
with a test function $\eta\in C^\infty_{cpt}(\R^d\times [0,\Tmax);\R^d)$ such that $\nabla\cdot\eta = 0$.
This shows that the velocity field $v$ of the strong solution satisfies
\begin{equation}
\begin{aligned}\label{equVelField}
0 &= \int_0^T \int_\Rd \rho(\chi_v)\eta\cdot(v\cdot\nabla)v\,\dx\,\dt
+\int_0^T\int_{\R^d} \mu(\chi_v)(\nabla v+\nabla v^T) : \nabla \eta \,\dx\,\dt
\\ &~~~+\int_0^T\int_{\R^d} \rho(\chi_v)\eta\cdot\partial_t v\,\dx\,\dt
-\sigma\int_0^T\int_{I_v(t)} \vec{H}\cdot \eta \,\mathrm{d}S \,\dt
\end{aligned}
\end{equation}
which holds for almost every $T\in [0,\Tmax)$ and all $\eta\in C^\infty_{cpt}(\R^d\times [0,\Tmax);\R^d)$
such that $\nabla\cdot\eta = 0$. The aim is now to test the latter equation with the field $u-v-w$.
To this end, we fix a radial mollifier $\phi\colon\R^d\to[0,\infty)$ such that $\phi$
is smooth, supported in the unit ball and $\int_\Rd \phi\,\dx = 1$. For $n\in\N$ we define 
$\phi_n(\cdot) := n^{d}\phi(n\,\cdot)$ as well as $u_n:=\phi_n\ast u$ and analogously $v_n$ and $w_n$.
We then test \eqref{equVelField} with the test function $u_n-v_n-w_n$ and let $n\to\infty$.
Since the traces of $u_n$, $v_n$ and $w_n$ on $I_v(t)$ converge pointwise almost everywhere to the respective traces of $u$, $v$ and $w$, we indeed
may pass to the limit in the surface tension term of \eqref{equVelField}. Hence, we obtain
the identity
\begin{align}\nonumber
-&\int_0^T\int_{\R^d} \mu(\chi_v)(\nabla v+\nabla v^T) : \nabla (u-v-w) \,\dx\,\dt
\\&\label{equVelFieldTested}
= \int_0^T \int_\Rd \rho(\chi_v)(u-v-w)\cdot(v\cdot\nabla)v\,\dx\,\dt
\\&~~~\nonumber
+\int_0^T\int_{\R^d} \rho(\chi_v)(u-v-w)\cdot\partial_t v\,\dx\,\dt 
\\&~~~\nonumber
-\sigma\int_0^T\int_{I_v(t)} \vec{H}\cdot (u-v-w) \,\mathrm{d}S \,\dt,
\end{align}
which holds true for almost every $T\in [0,\Tmax)$.

In the next step, we test the analogue of \eqref{weakDensity} for the phase-dependent
density $\rho(\chi_u)$ of the varifold solution with the test function $\frac{1}{2}|v+w|^2$
and obtain
\begin{equation}
\begin{aligned}\label{weakDensityAux3}
\int_{\R^d}&\frac{1}{2}\rho(\chi_u(\cdot,T))|v+w|^2(\cdot,T)\,\dx 
-\int_{\R^d}\frac{1}{2}\rho(\chi_u^0)|v_0+w(\cdot,0)|^2\,\dx 
\\&
= \int_0^T\int_{\R^d}\rho(\chi_u)(v+w)\cdot\partial_t(v+w)\,\dx\,\dt
\\&~~~
+\int_0^T\int_{\R^d} \rho(\chi_u)(v+w)\cdot(u\cdot\nabla)(v+w)\,\dx\,\dt
\end{aligned}
\end{equation}
for almost every $T\in [0,\Tmax)$. Recall also from the definition of a varifold
solution that we are equipped with the energy dissipation inequality
\begin{align}\label{EIdens2}
\nonumber
&\int_\Rd\frac{1}{2}\rho(\chi_u(\cdot,T)) |u(\cdot,T)|^2 \,\dx + \sigma |V_T|(\Rd\times\Sd)
\\& +\int_0^T \int_\Rd \frac{\mu(\chi_u)}{2}\big|\nabla u + \nabla u^T\big|^2 \,\dx\,\dt
\\&~~~~
\nonumber
\leq \int_\Rd\frac{1}{2}\rho(\chi_u^0) |u_0|^2 \,\dx + \sigma |\nabla \chi_u^0|(\Rd),
\end{align}
which holds for almost every $T\in [0,\Tmax)$.

Finally, we want to test the equation for the momentum balance \eqref{weakNS}
of the varifold solution with the test function $v+w$. Since the normal derivative
of the tangential velocity of a strong solution may feature a discontinuity at the interface, 
we have to proceed by an approximation argument, i.e., we use the mollified version $v_n+w_n$
as a test function. Note that $v_n$ resp.\ $w_n$ are elements of $L^\infty([0,\Tmax);C^0(\R^d))$.
Hence, we may indeed use $v_n+w_n$ as a test function in the surface tension term
of the equation for the momentum balance \eqref{weakNS} of the varifold solution.
However, it is not clear a priori why one may pass to the limit $n\rightarrow \infty$ in this term. 

To argue that this is actually possible, we choose a precise representative
for $\nabla v$ resp.\ $\nabla w$ on the interface $I_v(t)$. This is indeed
necessary also for the velocity field of the strong solution since
the normal derivative of the tangential component of $v$ may
feature a jump discontinuity at the interface. However, by the regularity
assumptions on $v$, see Definition~\ref{DefinitionStrongSolution} of a strong solution, and the
assumptions on the compensating vector field $w$, for almost every $t\in [0,\Tmax)$ every point $x\in \Rd$ is either a Lebesgue point of $\nabla v$ (respectively $\nabla w$) or there exist two half spaces $H_1$ and $H_2$ passing through $x$ such that $x$ is a Lebesgue point for both $\nabla v|_{H_1}$ and $\nabla v|_{H_2}$ (respectively $\nabla w|_{H_1}$ and $\nabla w|_{H_2}$). In particular, by the $L^\infty$ bounds on $\nabla v$ and $\nabla w$ the limit of the mollifications $\nabla v_n$ respectively $\nabla w_n$ exist at every point $x\in \Rd$ and we may define $\nabla v$ respectively $\nabla w$ at every point $x\in \Rd$ as this limit.

Recall then that we have chosen the mollifiers $\phi_n$ to be radially symmetric.
Hence, the approximating sequences $\nabla v_n$ resp.\ $\nabla w_n$ converge
pointwise everywhere to the precise representation as chosen before. Since
both limits are bounded, we may pass to the limit $n\to\infty$ in every term
appearing from testing the equation for the momentum balance \eqref{weakNS} of the varifold solution
with the test function $v_n+w_n$. This entails
\begin{align}\label{weakAgainstStrong}
&-\int_\Rd \rho(\chi_u(\cdot,T)) u(\cdot,T) \cdot (v+w)(\cdot,T) \,\dx
+ \int_\Rd \rho(\chi_u^0) u_0 \cdot (v+w)(\cdot,0) \,\dx 
\\&
\nonumber
-\int_0^T \int_\Rd \mu(\chi_u) (\nabla u+\nabla u^T) : \nabla (v+w) \,\dx\,\dt
\\\nonumber
&=-\int_0^T\int_{\R^d} \rho(\chi_u)u\cdot\partial_t (v+w)\,\dx\,\dt
-\int_0^T \int_\Rd \rho(\chi_u) u\cdot(u\cdot\nabla)(v+w) \,\dx\,\dt
\\&~~~~
\nonumber
+\sigma\int_0^T\int_{\R^d\times\Sb^{d-1}}
(\mathrm{Id}{-}s\otimes s): \nabla (v+w) \,\mathrm{d}V_t(x,s)\,\dt
\end{align}
for almost every $T\in [0,\Tmax)$.
The next step consists of summing \eqref{equVelFieldTested}, \eqref{weakDensityAux3}, 
\eqref{EIdens2} and \eqref{weakAgainstStrong}. We represent this sum as follows:
\begin{align}\label{velError}
&LHS_{kin}(T) + LHS_{visc} + LHS_{surEn}(T) 
\\&
\nonumber
\leq RHS_{kin}(0) + RHS_{surEn}(0) + RHS_{dt} 
+ RHS_{adv} + RHS_{surTen}, 
\end{align}
where each individual term is obtained in the following way. The terms related to
kinetic energy at time $T$ on the left hand side of \eqref{weakDensityAux3}, 
\eqref{EIdens2} and \eqref{weakAgainstStrong} in total yield the contribution
\begin{align}\label{velError1}
LHS_{kin}(T) = \int_{\Rd}\frac{1}{2}\rho(\chi_u(\cdot,T)) |u-v-w|^2(\cdot,T) \,\dx.
\end{align}
The same computation may be carried out for the initial kinetic energy terms
\begin{align}\label{velError2}
RHS_{kin}(0) = \int_\Rd \frac{1}{2}\rho(\chi_u^0) |u_0-v_0-w(\cdot,0)|^2 \,\dx.
\end{align}
Note that because of \eqref{RadonNikodym} it holds
\begin{align*}
\sigma|V_T|(\Rd\times\Sb^{d-1}) = 
\sigma|\nabla\chi_u(\cdot,T)|(\Rd) +
\sigma\int_\Rd 1-\theta_T\,\mathrm{d}|V_T|_{\Sb^{d-1}}.
\end{align*}
The terms in the energy dissipation inequality related to surface energy are
therefore given by
\begin{align}\label{velError3}
LHS_{surEn}(T) = \sigma |\nabla\chi_u(\cdot,T)|(\Rd) + \sigma\int_\Rd 1-\theta_T\,\mathrm{d}|V_T|_{\Sb^{d-1}}
\end{align}
as well as
\begin{align}\label{velError4}
RHS_{surEn}(0) = \sigma |\nabla \chi_u^0|(\Rd).
\end{align}
Moreover, collecting all advection terms on the right hand side of \eqref{equVelFieldTested}, \eqref{weakDensityAux3}, 
and \eqref{weakAgainstStrong} as well as adding zero gives the contribution
\begin{align}\nonumber
RHS_{adv} &= -\int_0^T\int_\Rd \rho(\chi_u) (u-v-w) \cdot (v\cdot\nabla) w\,\dx\,\dt
\\&~~~\nonumber
-\int_0^T\int_\Rd\big(\rho(\chi_u)-\rho(\chi_v)\big)
(u-v-w)\cdot(v\cdot\nabla)v\,\dx\,\dt 
\\&~~~\nonumber
-\int_0^T\int_\Rd \rho(\chi_u) (u-v-w) \cdot \big((u-v)\cdot\nabla\big)(v+w)\,\dx\,\dt
\\&\label{velError5}
=-\int_0^T\int_\Rd \rho(\chi_u) (u-v-w) \cdot (v\cdot\nabla) w\,\dx\,\dt
\\&~~~\nonumber
-\int_0^T\int_\Rd\big(\rho(\chi_u)-\rho(\chi_v)\big)
(u-v-w)\cdot(v\cdot\nabla)v\,\dx\,\dt 
\\&~~~\nonumber
-\int_0^T\int_\Rd \rho(\chi_u)(u-v-w)\cdot
\big((u-v-w)\cdot\nabla\big)v\,\dx\,\dt
\\&~~~\nonumber
-\int_0^T\int_\Rd \rho(\chi_u) (u-v-w)\cdot (w\cdot \nabla)(v+w) \,\dx\,\dt 
\\&~~~\nonumber
-\int_0^T\int_\Rd \rho(\chi_u) (u-v-w)\cdot \big((u-v-w)\cdot \nabla\big)w \,\dx\,\dt.
\end{align}
Next, we may rewrite those terms on the right hand side of \eqref{equVelFieldTested}, \eqref{weakDensityAux3}, 
and \eqref{weakAgainstStrong} which contain a time derivative as follows
\begin{align}\label{velError6}
RHS_{dt} = &-\int_0^T\int_\Rd \big(\rho(\chi_u)-\rho(\chi_v)\big)
(u-v-w)\cdot\partial_t v\,\dx\,\dt
\\&\nonumber
-\int_0^T\int_\Rd \rho(\chi_u)(u-v-w)\cdot\partial_t w\,\dx\,\dt.
\end{align}
Furthermore, the terms related to surface tension on the right hand side of \eqref{equVelFieldTested}
and \eqref{weakAgainstStrong} are given by
\begin{align}
\label{RHSsurTen}
RHS_{surTen} &= \sigma\int_0^T\int_{\R^d\times\Sb^{d-1}}
(\mathrm{Id}{-}s\otimes s): \nabla v \,\mathrm{d}V_t(x,s)\,\dt
-\sigma\int_0^T\int_{I_v(t)} \vec{H}\cdot (u{-}v) \,\mathrm{d}S \,\dt
\\&~~~
\nonumber
+\sigma\int_0^T\int_{\R^d\times\Sb^{d-1}}
(\mathrm{Id}{-}s\otimes s): \nabla w \,\mathrm{d}V_t(x,s)\,\dt
+\sigma\int_0^T\int_{I_v(t)} \vec{H}\cdot w\,\mathrm{d}S \,\dt.
\end{align}
We proceed by rewriting the surface tension terms. For the sake of brevity, let
us abbreviate from now on $\vec{n}_u=\frac{\nabla\chi_u}{|\nabla\chi_u|}$.
Using the incompressibility of $v$ and adding zero, we start by rewriting
\begin{align*}
\sigma&\int_0^T\int_{\R^d\times\Sb^{d-1}}
(\mathrm{Id}{-}s\otimes s): \nabla v \,\mathrm{d}V_t(x,s)\,\dt \\
&=\sigma\int_0^T\int_{\Rd} \vec{n}_u\cdot(\vec{n}_u\cdot\nabla)v\,\mathrm{d}|\nabla\chi_u|\,\dt
-\sigma\int_0^T\int_{\R^d\times\Sb^{d-1}}s\cdot(s\cdot\nabla) v \,\mathrm{d}V_t(x,s)\,\dt \\
&~~~-\sigma\int_0^T\int_{\Rd} \vec{n}_u\cdot(\vec{n}_u\cdot\nabla)v\,\mathrm{d}|\nabla\chi_u|\,\dt.
\end{align*}
Next, by means of the compatibility condition \eqref{varifoldComp} we can write
\begin{align*}
\sigma&\int_0^T\int_{I_v(t)} \vec{n}_u\cdot(\vec{n}_u\cdot\nabla)v\,\mathrm{d}S\,\dt
-\sigma\int_0^T\int_{\R^d\times\Sb^{d-1}}s\cdot(s\cdot\nabla) v \,\mathrm{d}V_t(x,s)\,\dt \\
&=-\sigma\int_0^T\int_{\R^d\times\Sb^{d-1}}(s{-}\xi)\cdot\big((s{-}\xi)\cdot\nabla\big) 
v \,\mathrm{d}V_t(x,s)\,\dt \\
&~~~-\sigma\int_0^T\int_{\R^d\times\Sb^{d-1}}\xi\cdot\big((s{-}\xi)\cdot\nabla\big) 
v \,\mathrm{d}V_t(x,s)\,\dt \\
&~~~+\sigma\int_0^T\int_{\R^d}\vec{n}_u\cdot\big((\vec{n}_u-\xi)\cdot\nabla\big)
v\,\mathrm{d}|\nabla\chi_u|\,\mathrm{d}t.
\end{align*}
Moreover, the compatibility condition \eqref{varifoldComp} also ensures that
\begin{align*}
-\sigma \int_0^T\int_{\R^d\times\Sb^{d-1}}\xi\cdot(s\cdot\nabla)v\,\mathrm{d}V_t(x,s)\,\mathrm{d}t
=-\sigma \int_0^T\int_{\R^d}\xi\cdot(\vec{n}_u\cdot\nabla)v\,\mathrm{d}|\nabla\chi_u|\,\mathrm{d}t,
\end{align*}
whereas it follows from \eqref{RadonNikodym}
\begin{align*}
&\sigma\int_0^T\int_{\R^d\times\Sb^{d-1}}\xi\cdot(\xi\cdot\nabla)v\,\mathrm{d}V_t(x,s)\,\mathrm{d}t
\\
&=\sigma\int_0^T\int_{\R^d}(1-\theta_t)\,
\xi\cdot(\xi\cdot\nabla)v\,\mathrm{d}|V_t|_{\Sb^{d-1}}(x)\,\dt 
+\sigma\int_0^T\int_{\R^d}\xi\cdot(\xi\cdot\nabla)v\,\mathrm{d}|\nabla\chi_u|\,\mathrm{d}t.
\end{align*}
Using that the divergence of $\xi$ equals the divergence of $\vec{n}_v(P_{I_v(t)}x)$ \textit{on}
the interface of the strong solution (i.\,e.\ $\vec{H}=-(\nabla \cdot \xi) \vec{n}_v$; see Definition~\ref{def:ExtNormal}, i.e., the cutoff
function does not contribute to the divergence \textit{on} the interface), 
that the latter quantity equals the scalar mean curvature (recall
that $\vec{n}_v=\frac{\nabla \chi_v}{|\nabla \chi_v|}$ points inward) as well as once more the incompressibility of the velocity fields 
$v$ resp.\ $u$ we may also rewrite
\begin{align*}
-\sigma\int_0^T\int_{I_v(t)} \vec{H}\cdot (u-v) \,\mathrm{d}S \,\dt
= -\sigma\int_0^T\int_\Rd\chi_v\big((u-v)\cdot\nabla\big)(\nabla\cdot\xi)\,\dx\,\dt.
\end{align*}
The preceding five identities together then imply that
\begin{align}\nonumber
&\sigma\int_0^T\int_{\R^d\times\Sb^{d-1}}
(\mathrm{Id}{-}s\otimes s): \nabla v \,\mathrm{d}V_t(x,s)\,\dt
-\sigma\int_0^T\int_{I_v(t)} \vec{H}\cdot (u{-}v) \,\mathrm{d}S \,\dt
\\&\label{surTenAux1}
=-\sigma\int_0^T\int_{\Rd} \vec{n}_u\cdot(\vec{n}_u\cdot\nabla)v\,\mathrm{d}|\nabla\chi_u|\,\dt 
\\&~~~
-\sigma\int_0^T\int_\Rd\chi_v\big((u-v)\cdot\nabla\big)(\nabla\cdot\xi)\,\dx\,\dt
\nonumber\\&~~~
-\sigma\int_0^T\int_{\R^d\times\Sb^{d-1}}(s{-}\xi)\cdot\big((s{-}\xi)\cdot\nabla\big) v \,\mathrm{d}V_t(x,s)\,\dt 
\nonumber\\&~~~
+\sigma\int_0^T\int_{\R^d}(1-\theta_t)\,
\xi\cdot(\xi\cdot\nabla)v\,\mathrm{d}|V_t|_{\Sb^{d-1}}(x)\,\dt 
\nonumber\\&~~~
+\sigma\int_0^T\int_{\R^d}(\vec{n}_u-\xi)\cdot\big((\vec{n}_u-\xi)\cdot\nabla\big)v
\,\mathrm{d}|\nabla\chi_u|\,\mathrm{d}t. \nonumber
\end{align}
Following the computation which led to \eqref{surTenAux1} we also obtain the identity
\begin{align*}
\sigma&\int_0^T\int_{\Rd\times\Sd}(\mathrm{Id}{-}s\otimes s):\nabla w 
\,\mathrm{d}V_t(x,s)\,\dt
\\&
= -\sigma\int_0^T\int_{\Rd} \vec{n}_u\cdot(\vec{n}_u\cdot\nabla)w\,\mathrm{d}|\nabla\chi_u|\,\dt \nonumber
\\&~~~
-\sigma\int_0^T\int_{\R^d\times\Sb^{d-1}}(s{-}\xi)
\cdot\big((s{-}\xi)\cdot\nabla\big) w \,\mathrm{d}V_t(x,s)\,\dt 
\\&~~~
+\sigma\int_0^T\int_{\R^d}(1-\theta_t)\,
\xi\cdot(\xi\cdot\nabla)w\,\mathrm{d}|V_t|_{\Sb^{d-1}}(x)\,\dt \nonumber
\\&~~~
+\sigma\int_0^T\int_{\R^d}(\vec{n}_u-\xi)\cdot\big((\vec{n}_u-\xi)\cdot\nabla\big)w
\,\mathrm{d}|\nabla\chi_u|\,\mathrm{d}t. \nonumber
\end{align*}
Using the fact that $w$ is divergence-free, we may also rewrite
\begin{align*}
&-\sigma\int_0^T\int_\Rd \vec{n}_u\cdot(\vec{n}_u\cdot\nabla) w \,\mathrm{d}|\nabla \chi_u|\,\dt
\\&
=-\sigma\int_0^T\int_\Rd \vec{n}_u \cdot\big((\vec{n}_u-\xi) \cdot \nabla\big) w \,\mathrm{d}|\nabla \chi_u|\,\dt
+\sigma\int_0^T\int_\Rd \chi_u \nabla \cdot \big((\xi\cdot \nabla) w\big) \,\dx\,\dt
\\&
=-\sigma\int_0^T\int_\Rd \xi \cdot \big((\vec{n}_u-\xi) \cdot \nabla\big) w \,\mathrm{d}|\nabla \chi_u|\,\dt
\\&~~~
-\sigma\int_0^T\int_\Rd (\vec{n}_u-\xi) \cdot \big((\vec{n}_u-\xi)\cdot\nabla\big) w \,\mathrm{d}|\nabla \chi_u|\,\dt
\\&~~~
+\sigma\int_0^T\int_\Rd \chi_u \nabla w :\nabla \xi^T \,\dx\,\dt.
\end{align*}
Appealing once more to the fact that $\xi=\vec{n}_v$ on the interface $I_v$
of the strong solution (see Definition~\ref{def:ExtNormal}) and $\nabla \cdot w=0$, we obtain
\begin{align*}
\sigma&\int_0^T\int_{I_v(t)} \vec{H} \cdot w \,\mathrm{d}S\,\dt \\
&=-\sigma\int_0^T\int_\Rd (\Id-\vec{n}_v\otimes \vec{n}_v): \nabla w \,\mathrm{d}S\,\dt
=\sigma\int_0^T\int_\Rd \vec{n}_v \cdot (\xi\cdot \nabla) w \,\mathrm{d}S\,\dt
\\&
=-\sigma\int_0^T\int_\Rd \chi_v \nabla \cdot\big( (\xi\cdot \nabla)w \big) \,\dx\,\dt
=-\sigma\int_0^T\int_\Rd \chi_v \nabla w : \nabla \xi^T \,\dx\,\dt.
\end{align*}
The last three identities together with \eqref{surTenAux1} and \eqref{RHSsurTen} in total finally yield the following representation
of the surface tension terms on the right hand side of \eqref{equVelFieldTested}
and \eqref{weakAgainstStrong}
\begin{align}\label{velError7}
RHS_{surTen} &=
-\sigma\int_0^T\int_{\Rd} \vec{n}_u\cdot(\vec{n}_u\cdot\nabla)v\,\mathrm{d}|\nabla\chi_u|\,\dt 
\\&~~~
+\sigma\int_0^T\int_{\R^d}(\vec{n}_u-\xi)\cdot\big((\vec{n}_u-\xi)\cdot\nabla\big)v
\,\mathrm{d}|\nabla\chi_u|\,\mathrm{d}t \nonumber
\\&~~~
-\sigma\int_0^T\int_\Rd\chi_v\big((u-v)\cdot\nabla\big)(\nabla\cdot\xi)\,\dx\,\dt
\nonumber\\&~~~
-\sigma\int_0^T\int_{\R^d\times\Sb^{d-1}}(s{-}\xi)\cdot\big((s{-}\xi)\cdot\nabla\big) v \,\mathrm{d}V_t(x,s)\,\dt 
\nonumber\\&~~~
+\sigma\int_0^T\int_{\R^d}(1-\theta_t)\,
\xi\cdot(\xi\cdot\nabla)v\,\mathrm{d}|V_t|_{\Sb^{d-1}}(x)\,\dt \nonumber
\\&~~~
-\sigma\int_0^T\int_{\R^d\times\Sb^{d-1}}(s{-}\xi)
\cdot\big((s{-}\xi)\cdot\nabla\big) w \,\mathrm{d}V_t(x,s)\,\dt \nonumber
\\&~~~
+\sigma\int_0^T\int_{\R^d}(1-\theta_t)\,
\xi\cdot(\xi\cdot\nabla)w\,\mathrm{d}|V_t|_{\Sb^{d-1}}(x)\,\dt \nonumber
\\&~~~
-\sigma\int_0^T\int_\Rd \xi \cdot \big((\vec{n}_u-\xi) \cdot \nabla\big) w \,\mathrm{d}|\nabla \chi_u|\,\dt \nonumber
\\&~~~
+\sigma\int_0^T\int_\Rd (\chi_u-\chi_v) \nabla w :\nabla \xi^T \,\dx\,\dt. \nonumber
\end{align}
It remains to collect the viscosity terms from the left hand side of \eqref{equVelFieldTested}, 
\eqref{EIdens2} and \eqref{weakAgainstStrong}. Adding also zero, we obtain
\begin{align}\nonumber
LHS_{visc} &=
\int_0^T\int_\Rd 2\big(\mu(\chi_u)-\mu(\chi_v)\big)\Dsym v:\Dsym (u-v-w)\,\dx\,\dt 
\\&~~~\nonumber
-\int_0^T\int_\Rd 2\mu(\chi_u)\Dsym v:\Dsym (u-v-w)\,\dx\,\dt 
\\&~~~\nonumber
+\int_0^T\int_\Rd 2\mu(\chi_u)\Dsym u:\Dsym u\,\dx\,\dt 
\\&~~~\nonumber
-\int_0^T\int_\Rd 2\mu(\chi_u)\Dsym u:\Dsym (v+w)\,\dx\,\dt 
\\&\label{velError8}
= \int_0^T\int_\Rd2\mu(\chi_u)|\Dsym (u-v-w)|^2\,\dx\,\dt
\\&~~~\nonumber
+\int_0^T\int_\Rd 2\big(\mu(\chi_u)-\mu(\chi_v)\big)\Dsym v:\Dsym (u-v-w)\,\dx\,\dt 
\\&~~~\nonumber
+\int_0^T\int_\Rd 2\mu(\chi_u)\Dsym w:\Dsym (u-v-w)\,\dx\,\dt .
\end{align}
In particular, as an intermediate summary we obtain the following bound making already use
of the notation of Proposition~\ref{PropositionRelativeEntropyInequalityFull}: Taking the bound
\eqref{velError} together with the identities from \eqref{velError1} to \eqref{velError6} as well
as \eqref{velError7} and \eqref{velError8} yields
\begin{align}\label{velErrorFinal}
\int_\Rd&\frac{1}{2}\rho(\chi_u(\cdot,T)) |u-v-w|^2(\cdot,T) \,\dx +
\int_0^T \int_\Rd 2\mu(\chi_u)|\Dsym (u - v - w)|^2 \,\dx\,\dt \nonumber\\
&+\sigma|\nabla\chi_u(\cdot,T)|(\Rd) +\sigma\int_\Rd 1-\theta_T\,\mathrm{d}|V_T|_{\Sb^{d-1}} \nonumber\\
&\leq \int_\Rd\frac{1}{2}\rho(\chi_u^0) |u_0-v_0-w(\cdot,0)|^2\,\dx
+\sigma |\nabla \chi_u^0|(\Rd) \\
&~~~ + R_{dt} + R_{visc} + R_{adv} + A_{visc} + A_{dt} + A_{adv} + A_{surTen} \nonumber
\\&~~~\nonumber-\sigma\int_0^T\int_{\Rd} \vec{n}_u\cdot(\vec{n}_u\cdot\nabla)v\,\mathrm{d}|\nabla\chi_u|\,\dt 
\\&~~~
+\sigma\int_0^T\int_{\R^d}(\vec{n}_u-\xi)\cdot\big((\vec{n}_u-\xi)\cdot\nabla\big)v
\,\mathrm{d}|\nabla\chi_u|\,\mathrm{d}t \nonumber
\\&~~~
-\sigma\int_0^T\int_\Rd\chi_v\big((u-v-w)\cdot\nabla\big)(\nabla\cdot\xi)\,\dx\,\dt
\nonumber\\&~~~
-\sigma\int_0^T\int_\Rd\chi_u\big(w\cdot\nabla\big)(\nabla\cdot\xi)\,\dx\,\dt
\nonumber\\&~~~
-\sigma\int_0^T\int_{\R^d\times\Sb^{d-1}}(s{-}\xi)\cdot\big((s{-}\xi)\cdot\nabla\big) v \,\mathrm{d}V_t(x,s)\,\dt 
\nonumber\\&~~~
+\sigma\int_0^T\int_{\R^d}(1-\theta_t)\,
\xi\cdot(\xi\cdot\nabla)v\,\mathrm{d}|V_t|_{\Sb^{d-1}}(x)\,\dt \nonumber.
\end{align}
The aim of the next step is to use $\sigma(\nabla\cdot\xi)$ (see Definition~\ref{def:ExtNormal}) 
as a test function in the transport equation \eqref{weakTransport} for the indicator function
$\chi_u$ of the varifold solution. For the sake of brevity, we will write again
$\vec{n}_u=\frac{\nabla\chi_u}{|\nabla\chi_u|}$. Plugging in $\sigma(\nabla\cdot\xi)$ and integrating by parts yields 
\begin{align*}
-\sigma&\int_{\Rd}\vec{n}_u(\cdot,T)\cdot\vec{\xi}(\cdot,T)\,\mathrm{d}|\nabla\chi_u(\cdot,T)|
		+\sigma\int_{\Rd}\vec{n}_u^0\cdot\vec{\xi}(\cdot,0)\,\mathrm{d}|\nabla\chi_u^0| \\
		&=-\sigma\int_0^T\int_{\Rd}\vec{n}_u\cdot\partial_t\vec{\xi}\,\mathrm{d}|\nabla\chi_u|\,\dt
    +\sigma\int_0^T\int_{\Rd}\chi_u(u\cdot\nabla)(\nabla\cdot\vec{\xi}\,)\,\dx\,\dt
\end{align*}
for almost every $T\in[0,\Tmax)$. Making use of the evolution equation 
\eqref{evolExtNormal} for $\xi$ and the fact that $\xi$ is supported
in the space-time domain $\{\dist(x,I_v(t))<r_c\}$, we get by adding zero
\begin{align}\label{chiCross}
&-\sigma\int_{\Rd}\vec{n}_u(\cdot,T)\cdot\vec{\xi}(\cdot,T)\,\mathrm{d}|\nabla\chi_u(\cdot,T)|
		+\sigma\int_{\Rd}\vec{n}_u^0\cdot\vec{\xi}(\cdot,0)\,\mathrm{d}|\nabla\chi_u^0| \\
		&=\sigma\int_0^T\int_\Rd\vec{n}_u\cdot
		\big(\big(\mathrm{Id}{-}\vec{n}_v(P_{I_v(t)}x)\otimes\vec{n}_v(P_{I_v(t)}x)\big)(\nabla v)^T \xi\big)
		\,\mathrm{d}|\nabla\chi_u|\,\dt \nonumber\\
		&~~~+\sigma\int_0^T\int_\Rd\vec{n}_u\cdot
		(v\cdot\nabla)\xi\,\mathrm{d}|\nabla\chi_u|\,\dt \nonumber\\
		&~~~+\sigma\int_0^T\int_\Rd\chi_u(u\cdot\nabla)(\nabla\cdot\vec{\xi}\,)\,\dx\,\dt \nonumber\\
		&~~~+\sigma\int_0^T\int_\Rd\vec{n}_u\cdot
		\big(\big(\mathrm{Id}{-}\vec{n}_v(P_{I_v(t)}x)\otimes\vec{n}_v(P_{I_v(t)}x)\big)
		(\nabla\bar{V}_{\vec{n}}-\nabla v)^T \xi\big)\,\mathrm{d}|\nabla\chi_u|\,\dt \nonumber\\
		&~~~+\sigma\int_0^T\int_\Rd\vec{n}_u\cdot
		\big((\bar{V}_{\vec{n}}-v)\cdot\nabla\big)\xi\,\mathrm{d}|\nabla\chi_u|\,\dt \nonumber
\end{align}
which holds for almost every $T\in[0,\Tmax)$. Next, we study the quantity
\begin{align}\label{chiCrossApprox}
RHS_{tilt} &:= 
		\sigma\int_0^T\int_{\Rd}\vec{n}_u\cdot
		(v\cdot\nabla)\xi\,\mathrm{d}|\nabla\chi_u|\,\dt \nonumber\\&~~~
		+\sigma\int_0^T\int_{\Rd}\chi_u(u\cdot\nabla)(\nabla\cdot\vec{\xi}\,)\,\dx\,\dt 
		\\&~~~
		+\sigma\int_0^T\int_{\Rd}\vec{n}_u\cdot
		\big(\big(\mathrm{Id}{-}\vec{n}_v(P_{I_v(t)}x)\otimes\vec{n}_v(P_{I_v(t)}x)\big)(\nabla v)^T\cdot\xi\big)
		\,\mathrm{d}|\nabla\chi_u|\,\dt. \nonumber
\end{align}
Due to the regularity of $v$ resp.\ $\xi$ as well as the incompressibility of the velocity
field $v$ we get
\begin{align}\label{vn1}
\sigma\int_0^T\int_{\Rd}\vec{n}_u\cdot(v\cdot\nabla)
\vec{\xi}\,\mathrm{d}|\nabla\chi_u|\,\dt 
&=-\sigma\int_0^T\int_{\Rd}\chi_u\nabla\cdot(v\cdot\nabla)
\vec{\xi}\,\mathrm{d}x\,\dt \nonumber\\
&=-\sigma\int_0^T\int_{\Rd}\chi_u
\nabla^2:v\otimes\xi\,\mathrm{d}x\,\dt \nonumber\\
&=-\sigma\int_0^T\int_{\Rd}
\chi_u\nabla\cdot\big((\xi\cdot\nabla)v\big)\,\mathrm{d}x\,\dt \nonumber\\
&~~~-\sigma\int_0^T\int_{\Rd}\chi_u
\nabla\cdot\big(v(\nabla\cdot\xi)\big)\,\mathrm{d}x\,\dt \nonumber\\
&=-\sigma\int_0^T\int_{\Rd}\chi_u
(v\cdot\nabla)(\nabla\cdot\vec{\xi}\,)\,\mathrm{d}x\,\dt \\
&~~~+\sigma\int_0^T\int_{\Rd}\vec{n}_u\cdot(\vec{\xi}\cdot\nabla)
v\,\mathrm{d}|\nabla\chi_u|\,\dt. \nonumber
\end{align}
Exploiting the fact that $\vec{\xi}(x)=\vec{n}_v(P_{I_v(t)}x) \zeta(x)$ and $\vec{n}_v(P_{I_v(t)}x)$
only differ by a scalar prefactor, namely the cut-off multiplier $\zeta(x)$ which one can shift around,
it turns out to be helpful to rewrite
\begin{align}\label{vn2}
&\sigma\int_0^T\int_{\Rd}\vec{n}_u\cdot
		\big(\big(\mathrm{Id}{-}\vec{n}_v(P_{I_v(t)}x)\otimes\vec{n}_v(P_{I_v(t)}x)\big)(\nabla v)^T\cdot\xi\big)
		\,\mathrm{d}|\nabla\chi_u|\,\dt\nonumber\\
&=\sigma\int_0^T\int_{\Rd}\xi\cdot
		\big((\mathrm{Id}{-}\vec{n}_v(P_{I_v(t)}x)\otimes\vec{n}_v(P_{I_v(t)}x))\vec{n}_u\cdot\nabla)v
		\,\mathrm{d}|\nabla\chi_u|\,\dt  \\
		&=\sigma\int_0^T\int_{\Rd}\xi\cdot
		\big(\big(\vec{n}_u-(\vec{n}_v(P_{I_v(t)}x)\cdot\vec{n}_u)\vec{n}_v(P_{I_v(t)}x)\big)\cdot\nabla\big)v
		\,\mathrm{d}|\nabla\chi_u|\,\dt \nonumber
		\\&\nonumber
		=\sigma\int_0^T\int_{\Rd}\xi\cdot\big((\vec{n}_u-\xi)\cdot\nabla\big)
		v\,\mathrm{d}|\nabla\chi_u|\,\dt
		\\&~~~\nonumber
-\sigma\int_0^T\int_{\Rd}
(\xi\cdot\vec{n}_u)\,\vec{n}_v(P_{I_v(t)}x)\cdot\big(\vec{n}_v(P_{I_v(t)}x)\cdot\nabla\big)v
\,\mathrm{d}|\nabla\chi_u|\,\dt
\\&~~~\nonumber
+\sigma\int_0^T\int_{\Rd}
\xi\cdot(\xi\cdot\nabla) v
\,\mathrm{d}|\nabla\chi_u|\,\dt.
\end{align}
Hence, by using \eqref{vn1} and \eqref{vn2} we obtain
\begin{align}\label{tiltError1}
RHS_{tilt} &= \sigma\int_0^T\int_{\Rd} \vec{n}_u\cdot(\vec{n}_u\cdot\nabla)v\,\mathrm{d}|\nabla\chi_u|\,\dt 
\\&~~~\nonumber
-\sigma\int_0^T\int_{\R^d}(\vec{n}_u-\xi)\cdot\big((\vec{n}_u-\xi)\cdot\nabla\big)v
\,\mathrm{d}|\nabla\chi_u|\,\mathrm{d}t 
\\&~~~\nonumber
+\sigma\int_0^T\int_\Rd\chi_u\big((u-v)\cdot\nabla\big)(\nabla\cdot\xi)\,\dx\,\dt 
\\&~~~\nonumber
-\sigma\int_0^T\int_{\Rd}
(\xi\cdot\vec{n}_u)\,\vec{n}_v(P_{I_v(t)}x)\cdot\big(\vec{n}_v(P_{I_v(t)}x)\cdot\nabla\big)v
\,\mathrm{d}|\nabla\chi_u|\,\dt
\\&~~~\nonumber
+\sigma\int_0^T\int_{\Rd}
\xi\cdot(\xi\cdot\nabla) v
\,\mathrm{d}|\nabla\chi_u|\,\dt.
\end{align} 
This in turn finally entails
\begin{align}\label{tiltErrorFinal}
&-\sigma\int_{\Rd}\vec{n}_u(\cdot,T)\cdot\vec{\xi}(\cdot,T)\,\mathrm{d}|\nabla\chi_u(\cdot,T)|
+\sigma\int_{\Rd}\vec{n}_u^0\cdot\vec{\xi}(\cdot,0)\,\mathrm{d}|\nabla\chi_u^0|
\\&\nonumber
=\sigma\int_0^T\int_{\Rd} \vec{n}_u\cdot(\vec{n}_u\cdot\nabla)v\,\mathrm{d}|\nabla\chi_u|\,\dt 
\\&~~~\nonumber
-\sigma\int_0^T\int_{\R^d}(\vec{n}_u-\xi)\cdot\big((\vec{n}_u-\xi)\cdot\nabla\big)v
\,\mathrm{d}|\nabla\chi_u|\,\mathrm{d}t 
\\&~~~\nonumber
+\sigma\int_0^T\int_\Rd\chi_u\big((u-v)\cdot\nabla\big)(\nabla\cdot\xi)\,\dx\,\dt 
\\&~~~\nonumber
-\sigma\int_0^T\int_{\Rd}
(\xi\cdot\vec{n}_u)\,\vec{n}_v(P_{I_v(t)}x)\cdot\big(\vec{n}_v(P_{I_v(t)}x)\cdot\nabla\big)v
-\xi\cdot(\xi\cdot\nabla) v
\,\mathrm{d}|\nabla\chi_u|\,\dt
\\&~~~
+\sigma\int_0^T\int_\Rd\vec{n}_u\cdot
\big(\big(\mathrm{Id}{-}\vec{n}_v(P_{I_v(t)}x)\otimes\vec{n}_v(P_{I_v(t)}x)\big)
(\nabla\bar{V}_{\vec{n}}{-}\nabla v)^T \xi\big)\,\mathrm{d}|\nabla\chi_u|\,\dt \nonumber
\\&~~~
+\sigma\int_0^T\int_\Rd\vec{n}_u\cdot
\big((\bar{V}_{\vec{n}}{-}v)\cdot\nabla\big)\xi\,\mathrm{d}|\nabla\chi_u|\,\dt \nonumber,
\end{align}
which holds for almost every $T\in [0,\Tmax)$. 

In a last step, we use the truncation of the identity $\beta$ from Proposition~\ref{PropositionRelativeEntropyInequalityFull} composed with the signed distance
to the interface of the strong solution as a test function in the transport 
equations \eqref{weakTransport} resp.\ \eqref{weakTransportB} for the indicator functions 
$\chi_v$ resp.\ $\chi_u$ of the two solutions. However, observe first that by
the precise choice of the weight function $\beta$ it holds
\begin{align*}
(\chi_u-\chi_v)\beta\Big(\frac{\sigdist(\cdot,I_v)}{r_c}\Big) =
|\chi_u-\chi_v|\Big|\beta\Big(\frac{\sigdist(\cdot,I_v)}{r_c}\Big)\Big|.
\end{align*}
Hence, when testing the equation \eqref{weakTransport} for the indicator function of the varifold
solution and then subtracting the corresponding result from testing the equation \eqref{weakTransportB}
for the indicator function of the strong solution, we obtain
\begin{align}
\nonumber
\int_{\R^d}&\big|\chi_u(\cdot,T)-\chi_v(\cdot,T)\big|\Big|\beta\Big(\frac{\sigdist(\cdot,I_v(T))}{r_c}\Big)\Big|\,\dx \\
&= \int_{\R^d}\big|\chi_u^0-\chi_v^0\big|\Big|\beta\Big(\frac{\sigdist(\cdot,I_v(0))}{r_c}\Big)\Big|\,\dx 
\label{weightDiff}\\
&~~~~+\int_0^T\int_{\R^d}\chi_u\Big(\partial_t\beta\Big(\frac{\sigdist(\cdot,I_v)}{r_c}\Big)
+(u\cdot\nabla)\beta\Big(\frac{\sigdist(\cdot,I_v)}{r_c}\Big)\Big)\,\dx\,\dt \nonumber\\
&~~~~-\int_0^T\int_{\R^d}\chi_v\Big(\partial_t\beta\Big(\frac{\sigdist(\cdot,I_v)}{r_c}\Big)
+(v\cdot\nabla)\beta\Big(\frac{\sigdist(\cdot,I_v)}{r_c}\Big)\Big)\,\dx\,\dt, \nonumber
\end{align}
which holds for almost every $T\in[0,\Tmax)$. Note that testing with the function $\beta(\frac{\sigdist(x,I_v(t))}{r_c})$
is admissible due to the bound $\chi_u,\chi_v\in L^\infty([0,\Tmax);L^1(\Rd))$ (recall that we assume $\chi_u,\chi_v \in L^\infty([0,\Tmax);\BV(\Rd))$ in our definition of solutions) and due to the fact that $\beta(\frac{\sigdist(x,I_v(t))}{r_c})$ is of class $C^1$. Indeed,
one first multiplies $\beta$ by a cutoff $\theta_{\tilde R}\in C_{cpt}^\infty(\R^d)$ on a scale $\tilde R$, i.e.\ 
$\theta\equiv 1$ on $\{x\in\R^d\colon |x|\leq \tilde R\}$, $\theta\equiv 0$ outside of $\{x\in\R^d\colon |x|\geq 2\tilde R\}$
and $\|\nabla\theta_R\|_{L^\infty(\R^d)}\leq C\tilde R^{-1}$ for some universal constant $C>0$.
Then, one can use $\theta_{\tilde R} \beta$ in the transport equations as test functions and 
pass to the limit $\tilde R\to\infty$ because of the integrability of $\chi_v$ and $\chi_u$.
From this, one obtains the above equation.

Since the weight $\beta$ vanishes at $r=0$, we may infer from the incompressibility of the 
velocity fields that
\begin{align*}
\int_0^T&\int_\Rd\chi_v\big((u{-}v)\cdot\nabla\big)
\beta\Big(\frac{\sigdist(\cdot,I_v)}{r_c}\Big)\,\dx\,\dt \\&
= -\int_0^T\int_{I_v(t)}\big(\vec{n}_v\cdot (u{-}v)\big)\beta(0)\,\mathrm{d}S\,\dt = 0.
\end{align*}
Hence, we can rewrite \eqref{weightDiff} as
\begin{align*}
&\int_{\R^d}\big|\chi_u(\cdot,T)-\chi_v(\cdot,T)\big|\Big|\beta\Big(\frac{\sigdist(\cdot,I_v(T))}{r_c}\Big)\Big|\,\dx \\
&~~=\int_{\R^d}\big|\chi_u^0-\chi_v^0\big|\Big|\beta\Big(\frac{\sigdist(\cdot,I_v(0))}{r_c}\Big)\Big|\,\dx \\
&~~~~+\int_0^T\int_{\R^d}(\chi_u{-}\chi_v)\Big(\partial_t\beta\Big(\frac{\sigdist(\cdot,I_v)}{r_c}\Big)
+\big((u{-}v)\cdot\nabla\big)\beta\Big(\frac{\sigdist(\cdot,I_v)}{r_c}\Big)\Big)\,\dx\,\dt \\
&~~~~+\int_0^T\int_\Rd(\chi_u{-}\chi_v)(v\cdot\nabla)\beta\Big(\frac{\sigdist(\cdot,I_v)}{r_c}\Big)\,\dx\,\dt
\end{align*}
for almost every $T\in [0,\Tmax)$. It remains to make use of the evolution equation
for $\beta$ composed with the signed distance function to the interface of the strong solution. But before
we do so, let us remark that because of \eqref{sdistAux4}
\begin{align*}
(v\cdot\nabla)\beta\Big(\frac{\sigdist(\cdot,I_v)}{r_c}\Big) = 
(V_{\vec{n}}\cdot\nabla)\beta\Big(\frac{\sigdist(\cdot,I_v)}{r_c}\Big),
\end{align*}
where the vector field $V_{\vec{n}}$ is the projection of the velocity field
$v$ of the strong solution onto the subspace spanned by the unit normal $\vec{n}_v(P_{I_v(t)}x)$: 
\begin{align*}
V_{\vec{n}}(x,t) := \big(v(x,t)\cdot\vec{n}_v(P_{I_v(t)}x,t)\big)\vec{n}_v(P_{I_v(t)}x,t)
\end{align*}
for all $(x,t)$ such that $\dist(x,I_v(t))<r_c$. Thus, using the evolution equation \eqref{transportWeight}
we finally obtain the identity
\begin{align}\label{weightVolErrorFinal}
&\int_{\R^d}\big|\chi_u(\cdot,T)-\chi_v(\cdot,T)\big|\Big|
\beta\Big(\frac{\sigdist(\cdot,I_v(T))}{r_c}\Big)\Big|\,\dx \nonumber\\
&~~=\int_{\R^d}\big|\chi_u^0-\chi_v^0\big|\Big|\beta\Big(\frac{\sigdist(\cdot,I_v(0))}{r_c}\Big)\Big|\,\dx \\
&~~~~+\int_0^T\int_{\R^d}(\chi_u{-}\chi_v)\big((u{-}v)\cdot\nabla\big)
\beta\Big(\frac{\sigdist(\cdot,I_v)}{r_c}\Big)\,\dx\,\dt \nonumber\\
&~~~~+\int_0^T\int_\Rd(\chi_u{-}\chi_v)\big((V_{\vec{n}}{-}\bar{V}_{\vec{n}})\cdot\nabla\big)
\beta\Big(\frac{\sigdist(\cdot,I_v)}{r_c}\Big)\,\dx\,\dt, \nonumber
\end{align}
which holds true for almost every $T\in [0,\Tmax)$. 

The asserted relative entropy inequality
now follows from a combination of the bounds \eqref{velErrorFinal}, \eqref{tiltErrorFinal}
as well as \eqref{weightVolErrorFinal}. This concludes the proof.
\end{proof}

\section{Appendix}

\begin{theorem}[Boundedness of singular integral operators of convolution type in $L^p$]
\label{SingularIntegralOperator}
Let $d\geq 2$, $p\in (1,\infty)$, and let $K:\mathbb{S}^{d-1}\to \R$ be a function of class $C^{1}$ 
with vanishing average. Let $f\in L^p(\Rd)$ and define
\begin{align*}
\mathcal{K}f(x):=\int_\Rd \frac{K\big(\frac{x-\tilde x}{|x-\tilde x|}\big)}{|x-\tilde x|^d} f(\tilde x) \,d\tilde x,
\end{align*}
where the integral is understood in the Cauchy principal value sense.
Then there exists a constant $C>0$ depending only on $d$, $p$, and $K$ such that
\begin{align*}
\|\mathcal{K}f\|_{L^p(\Rd)}\leq C \|f\|_{L^p(\Rd)}.
\end{align*}
\end{theorem}

We also state a non-trivial result from geometric measure theory on
properties of one-dimensional sections of Caccioppoli sets.

\begin{theorem}[{\cite[Theorem G]{Chlebik2005}}]\label{TheoG}
Consider a set $G$ of finite perimeter in $\R^d$, denote by 
$\vec{\nu}^{\,G}=(\nu^G_{x_1},\ldots,\nu^G_{x_{d-1}},\nu^G_{y})\in\Rd$
the associated measure theoretic inner unit normal vector field of 
the reduced boundary $\partial^*G$, and let $\chi^*_G$ be the
precise representative of the bounded variation function $\chi_G$. 
Then for Lebesgue almost every $x\in\R^{d-1}$ the one-dimensional sections 
$G_{x}:=\{y\in\R\colon (x,y)\in G\}$ satisfy the following properties:
\begin{itemize}\itemsep3pt
	\item[i)] $G_{x}$ is a set of finite perimeter in $\R$, 
						$\chi_G(x,\cdot)=\chi_G^*(x,\cdot)$ Lebesgue almost everywhere in $G_x$,
	\item[ii)] $(\partial^*G)_{x}=\partial^*G_{x}$, 
	\item[iii)] $\nu_y^G(x,y)\neq 0$ for all $y\in\R$ such that $(x,y)\in\partial^*G$, and
	\item[iv)] $\lim_{y\to y_0^+}\chi^*_G(x,y)=1$ and $\lim_{y\to y_0^-}\chi^*_G(x,y)=0$
							whenever $\nu^G_y(x,y_0)>0$, and vice versa if $\nu^G_y(x,y_0)<0$.
\end{itemize}
In particular, for every Lebesgue measurable set $M\subset\R^{d-1}$ there exists a Borel
measurable subset $M_G\subset M$ such that $\mathcal{L}^{d-1}(M\setminus M_G) = 0$
and the four properties stated above are satisfied for all $y\in M_G$.
\end{theorem} 

To bound the $L^4$-norm of the interface error heights $h^\pm$ in the 
case of a two-dimensional interface, we employ the following optimal 
Orlicz--Sobolev embedding.

\begin{theorem}[Optimal Orlicz-Sobolev embedding, {\cite[Theorem 1]{Cianchi}}]
\label{TheoremOptimalOrliczSobolev}
For every $d\geq 2$, there exists a constant $K$ depending only on $d$ such that the following holds true:
Let $A:[0,\infty)\rightarrow [0,\infty)$ be a convex function with $A(0)=0$, 
$A(t)\rightarrow \infty$ for $t\rightarrow \infty$, and
\begin{align*}
\int_0^1 \bigg(\frac{t}{A(t)}\bigg)^{1/(d-1)} \,dt <\infty.
\end{align*}
Define
\begin{align*}
H(r) := \Bigg(\int_0^r \bigg(\frac{t}{A(t)}\bigg)^{1/(d-1)} \,dt\Bigg)^{(d-1)/d}
\end{align*}
and
\begin{align*}
B(s):=A(H^{-1}(s)).
\end{align*}
Then for any weakly differentiable function $u$ decaying to $0$ at infinity in the 
sense $\{|u(x)|>s\}<\infty$ for all $s>0$, the following estimate holds true:
\begin{align}
\label{OrliczSobolevEmbedding}
\int_\Rd B\Bigg(\frac{|u(x)|}{K \big(\int_\Rd A(|\nabla u(x)|) \,dx\big)^{1/d}}\Bigg) \,dx \leq \int_\Rd A(|\nabla u(x)|) \,dx.
\end{align}
\end{theorem}

The application of the optimal Orlicz-Sobolev embedding to our setting is stated and proved next.

\begin{proposition}\label{OrliczHeight}
Let $T>0$ and $(I(t))_{t\in[0,T]}$ be a family of smoothly evolving surfaces in $\R^3$
in the sense of Definition~\ref{definition:domains}.
Consider $u\in L^\infty([0,T];\BV(I(t)))$ such that $|u|\leq 1$. Let $e\colon [0,T]\to (0,\infty)$
be a measurable function. We define
\begin{align*}
A_{e(t)}(s) :=
\begin{cases}
e(t) s & \text{for }s\leq e(t),
\\
s^2 & \text{for }e(t)\leq s\leq 1,
\\
2s-1 & \text{for }s\geq 1.
\end{cases}
\end{align*}
We also set $A_{e(t)}(Du(t)) := \int_{I(t)} A_{e(t)}(|\nabla u(t)|)\,\mathrm{d}S + |D^su(t)|(I(t))$.
Then the following estimate holds true
\begin{align}\label{boundOrliczHeight}
&\int_{I(t)}|u(x,t)|^4\,\mathrm{d}S 
\leq \frac{C}{r_c^{12}}\Big(1{+}\log\frac{1}{e(t)}\Big)
\\&~~~~~\nonumber\times
\Big(e(t)^4 + \frac{1}{e(t)^2}\big(\|u(t)\|_{L^2(I(t))}^6 {+} A_{e(t)}^3(Du(t))\big)
+\|u(t)\|_{L^2(I(t))}^4 + A_{e(t)}^2(Du(t))\Big)
\end{align}
for almost every $t\in [0,T]$ and a constant $C>0$.
\end{proposition}

\begin{proof}
Let $U\subset\R^2$ be an open and bounded set and consider $u\in C_{cpt}^1(U)$ such that $\|u\|_{L^\infty}\leq 1$.
For the sake of brevity, let us suppress for the moment the dependence on the variable $t\in [0,T)$.
The idea is to apply the optimal Orlicz--Sobolev embedding provided by the preceding theorem
with respect to the convex function $A_{e}$. Observe first that $A_{e}$ indeed satisfies all
the assumptions. Moreover, since $d=2$ we compute
\begin{align*}
(H(r))^2 = \int_0^r \frac{s}{A_{e}(s)} \,\mathrm{d}s
=
\begin{cases}
\frac{r}{e} & \text{for }r\leq e,
\\
1+\log \frac{r}{e} & \text{for }e\leq r\leq 1,
\\
1+\log \frac{1}{e} + \frac{r-1}{2} + \frac{1}{4}\log (2r-1) & \text{for }r\geq 1.
\end{cases}
\end{align*}
As a consequence, we get
\begin{align*}
H^{-1}(y)
=
\begin{cases}
=e y^2 & \text{for }y\leq 1,
\\
=e \exp(y^2-1) & \text{for }1\leq y \leq \sqrt{1+\log \frac{1}{e}},
\\
\geq (y^2-1-\log \frac{1}{e}) +1 &\text{for }y\geq \sqrt{1+\log \frac{1}{e}},
\\
\leq 2(y^2-1-\log \frac{1}{e}) +1 &\text{for }y\geq \sqrt{1+\log \frac{1}{e}}.
\end{cases}
\end{align*}
This in turn entails
\begin{align}\label{FormulaB}
B(s) = A_{e}(H^{-1}(s)) =
\begin{cases}
=e^2 s^2 & \text{for }s\leq 1,
\\
=e^2 \exp(2s^2-2) & \text{for }1\leq s \leq \sqrt{1+\log \frac{1}{e}},
\\
\geq s^2-\log \frac{1}{e} &\text{for }s\geq \sqrt{1+\log \frac{1}{e}}.
\end{cases}
\end{align}
We then deduce from Theorem~\ref{TheoremOptimalOrliczSobolev}, $d=2$, $\|u\|_{L^\infty}\leq 1$, 
the bound $\exp(s^2)\geq \frac{1}{2} s^4$ for all $s\geq 0$ as well as the bound 
$s^2-\log \frac{1}{e}\geq \frac{s^2}{1+\log \frac{1}{e}}$ for all $s\geq \sqrt{1+\log \frac{1}{e}}$
\begin{align*}
&\int_U |u(x)|^4\,\dx \\ &= \int_{U\cap\big\{|u|\leq K \sqrt{A_{e}(Du)}\big\}} |u(x)|^4 \,\dx 
\\&~~~
+ \int_{U\cap\big\{K\sqrt{A_{e}(Du)}\leq |u| \leq K\sqrt{A_{e}(Du)}\sqrt{1+\log\frac{1}{e}}\big\}} |u(x)|^4 \,\dx 
\\&~~~
+ \int_{U\cap\big\{|u| \geq K\sqrt{A_{e}(Du)}\sqrt{1+\log\frac{1}{e}}\big\}} |u(x)|^4 \,\dx 
\\&
\leq K^4\frac{A_{e}^2(Du)}{e^2}\int_{U\cap\big\{|u|\leq K\sqrt{A_{e}(Du)}\big\}}
e^2\frac{|u(x)|^2}{K^2A_{e}(Du)}\,\dx 
\\&~~~
+ K^4\frac{A_{e}^2(Du)}{e^2}\int_{U\cap\big\{K\sqrt{A_{e}(Du)} 
\leq |u| \leq K\sqrt{A_{e}(Du)}\sqrt{1+\log\frac{1}{e}}\big\}}
e^2\frac{|u(x)|^4}{K^4A_{e}^2(Du)}\,\dx
\\&~~~
+ K^2\Big(1+\log\frac{1}{e}\Big)A_{e}(Du)\int_{U\cap\big\{|u| \geq K\sqrt{A_{e}(Du)}
\sqrt{1+\log\frac{1}{e}}\big\}}
\frac{|u(x)|^4}{K^2\big(1+\log\frac{1}{e}\big)A_{e}(Du)}\,\dx
\\&
\leq C\Big(1+\log\frac{1}{e}\Big)\Big(\frac{1}{e^2}A_{e}^3(Du) + A_{e}^2(Du)\Big),
\end{align*}
which is precisely what is claimed. Note that since $u$ is continuously differentiable, 
the singular part in the definition of $A_e(Du)$ vanishes.

In a next step, we want to extend to smooth functions $u$ on the manifold $I(t)$.
By assumption, we may cover $I(t)$ with a finite family of open sets of the form 
$U(x_i):=I(t)\cap B_{2r_c}(x_i)$, $x_i\in I(t)$, such that $U(x_i)$ can be represented
as the graph of a function $g\colon B_1(0)\subset\R^2\to\R$ with $|\nabla g|\leq 1$
and $|\nabla^2 g|\leq r_c^{-1}$. We fix a partition of unity $\{\varphi_i\}_i$ subordinate
to this finite cover of $I(t)$. Note that $|\nabla\varphi_i|\leq Cr_c^{-1}$.
Note also that the cardinality of the open cover is uniformly bounded in $t$.
Hence, we proceed with deriving the desired bound only for one $u\varphi$, where $\varphi=\varphi_i$
is supported in $U=U(x_i)$. Abbreviating $\tilde u = u\circ g$ and $\tilde\varphi = \varphi\circ g$,
we obtain from the previous step
\begin{align*}
\int_{U}|u\varphi|^4\,\dS &= \int_{B_1(0)} |(u\varphi)(g(x))|^4\sqrt{1+|\nabla g(x)|^2} \,\dx 
\\&
\leq \sqrt{2}C\Big(1+\log\frac{1}{e}\Big)\Big(\frac{1}{e^2}A_{e}^3\big(D(\tilde u\tilde\varphi)\big) 
+ A_{e}^2\big(D(\tilde u\tilde\varphi)\big)\Big).
\end{align*}
Using the bounds $A_e(t+\tilde{t})\leq CA_e(t)+ CA_e(\tilde{t})$ and 
$A_e(\lambda t) \leq C(\lambda+\lambda^2)A_e(t)$, which hold for all $\lambda >0$
and all $t,\tilde{t}\geq 0$, as well as the product and chain rule we compute
\begin{align*}
A_e(D(\tilde u\tilde \varphi)) \leq Cr_c^{-2}\int_{B_1(0)} A_e\big(|u|(g(x))\big)\,\dx
+C\int_{B_1(0)} A_e\big(|\nabla u|(g(x))\big)\,\dx.
\end{align*}
By definition of $A_e$ we can further estimate
\begin{align*}
\int_{B_1(0)} A_e\big(|u|(g(x))\big)\,\dx \leq Ce^2 + \int_{B_1(0)} |u|^2(g(x))\,\dx.
\end{align*}
Changing back to the local coordinates on the manifold $I(t)$ we deduce 
\begin{align}\label{step2}
&\int_{U}|u|^4\,\dS \leq \frac{C}{r_c^6}\Big(1{+}\log\frac{1}{e}\Big)
\\&\nonumber~~~~~~\times
\Big(e^4 + \frac{1}{e^2}\big(\|u\|_{L^2(I(t))}^6 {+} A_e^3(Du)\big)
+\|u\|_{L^2(I(t))}^4 + A_e^2(Du)\Big).
\end{align}
This yields the claim in the case of a smooth function $u\colon I(t)\to\R$.

In a last step, we extend this estimate by mollification to $u\in\BV(I(t))$
with $\|u\|_{L^\infty}\leq 1$. To this end, let $\theta\colon\mathbb{R}^+\rightarrow [0,1]$ 
be a smooth cutoff with $\theta(s)=1$ for $s\in [0,\frac{1}{4}]$ and 
$\theta(s)=0$ for $s\geq \frac{1}{2}$. We then define for each $n\in\N$
\begin{align*}
u_n(x,t):= \frac{\int_{I(t)} \theta(n|\tilde x-x|) u(\tilde x,t) 
\,\dS(\tilde x)}{\int_{I(t)} \theta(n|\tilde x-x|) \,\dS(\tilde x)}.
\end{align*}
Since the analogous bound to \eqref{EstimateMollifier} holds true,
we infer $\|u_n\|_{L^\infty}\leq 1$ as well as 
$\|u_n - u\|_{L^1(I(t))}\to 0$ as $n\to\infty$. In particular, we have
pointwise almost everywhere convergence at least for a subsequence.
This in turn implies by Lebesgue's dominated convergence theorem that
$\|u_n - u\|_{L^4(I(t))}\to 0$ as $n\to\infty$ at least for a subsequence.
Moreover, the exact same computation which led to \eqref{GGradEstimatehEplus} shows
\begin{align*}
A_{e(t)}(|\nabla u_n (x,t)|)
&\leq C \frac{\int_{I(t)} \theta (n|\tilde x - x|) \big(A_{e(t)}(|\nabla u(\tilde x,t)|) 
+ A_{e(t)}(r_c^{-1} |u(\tilde x,t)|)\big) \,\dS(\tilde x)}{\int_{I(t)} \theta (n|\tilde x - x|) \,\dS(\tilde x)}
\\&~~~
\nonumber
+C\frac{\int_{I(t)} \theta (n|\tilde x - x|) \,\mathrm{d}|D^s u|(\tilde x,t)}{\int_{I(t)} 
\theta (n|\tilde x - x|) \,\dS(\tilde x)}.
\end{align*}
Integrating this bound over the manifold and then using Fubini shows that
$$A_{e(t)}(Du_n(t))\leq Cr_c^{-2}A_{e(t)}(Du(t))$$ 
holds true uniformly over all $n\in\N$. By applying the bound \eqref{step2} from the
second step, we may conclude the proof.
\end{proof}

In the case where the interface $I_v$ is a curve in $\R^2$, a much more elementary argument
yields the following bound.

\begin{lemma}\label{OrliczHeight1D}
Let $T>0$ and let $(I(t))_{t\in[0,T]}$ be a family of smoothly evolving curves in $\R^2$
in the sense of Definition~\ref{definition:domains}.
Let $u\in L^\infty([0,T];\BV(I(t)))$ such that $|u|\leq 1$.
Consider the convex function
\begin{align*}
G(s) := 
\begin{cases}
s^2, & |s|\leq 1, \\
2s-1, & |s| > 1.
\end{cases}
\end{align*}
We also define $|Du(t)|_{G} := \int_{I(t)} G(|\nabla u(x,t)|)\,\mathrm{d}S + |D^su(t)|(\Gamma)$.
Then,
\begin{align}\label{boundOrliczHeight1D}
\int_{I(t)} |u(x,t)|^4\,\mathrm{d}S
\leq
\frac{C (1+\mathcal{H}^1(I(t)))^3}{r_c^{4}}
\big(|Du(t)|_{G}^2 +|Du(t)|_{G}^4 + ||u||_{L^2(I(t))}^4\big)
\end{align}
holds true for almost every $t\in [0,T]$ with some universal constant $C>0$.
\end{lemma}

\begin{proof}
Fix $t>0$. First, observe that $I(t)$ essentially consists of a finite number of nonintersecting curves. By approximation, we may assume $u(t) \in W^{1,1}(I(t))$.

Let $\eta_i$ be a partition of unity on $I(t)$ with $|\nabla^{\tan} \eta_i(x)|\leq C r_c^{-1}$ such that the support of each $\eta_i$ is isometrically equivalent to a bounded interval (note that the Definition~\ref{definition:domains} implies a lower bound of $c r_c$ for the length of any connected component of $I(t)$) and such that at any point $x\in I(t)$ there are at most two $i$ with $\eta_i(x)>0$.

Treating by abuse of notation the function $\eta_i u$ as if defined on a real interval $I=(a,b)$, we then write
\begin{align*}
\eta_i(x) u(x) &= \int_a^x \eta_i(y) u'(y) + \eta_i'(y) u(y) \,\mathrm{d}y
\\&
=
\int_a^x \eta_i'(y) u(y) \,\mathrm{d}y
+\int_a^x \eta_i(y) \big(\max\big\{\min\{ u'(y) , 1 \}, -1 \big\} \big) \,\mathrm{d}y
\\&~~~~
+ \int_a^x \eta_i(y) \big(\big(u'(y)-1\big)_+ - \big(u'(y)-(-1)\big)_- \big) \,\mathrm{d}y.
\end{align*}
Hence, we may estimate using Jensen's inequality
\begin{align*}
\eta_i(x) |u(x)|&
\leq |I(t)|^{1/2} \bigg(\int_{I(t)} \eta_i |\max\big\{\min\{ |\nabla^{\tan} u| , 1 \}, -1 \big\}|^2\,\mathrm{d}S\bigg)^{1/2}
\\&~~~~
+\int_{I(t)}\eta_i\, (|\nabla^{\tan} u|-1)_+ \,\mathrm{d}S
+C r_c^{-1} \int_{I(t)\cap \supp \eta_i} |u| \,\mathrm{d}S
\end{align*}
for any $x\in I(t)$. Taking the fourth power, integrating over $x$, and summing over $i$, we deduce
\begin{align*}
\int_{I(t)} |u(x)|^4 \,\mathrm{d}S
&
\leq C|I(t)|^3 \bigg(\int_{I(t)} |\max\big\{\min\{ |\nabla^{\tan} u|, 1 \}, -1 \big\}|^2\,\mathrm{d}y\bigg)^2
\\&~~~~
+C|I(t)| \bigg(\int_{I(t)}  (|\nabla^{\tan} u|-1)_+ \,\mathrm{d}S\bigg)^4
\\&~~~~
+C r_c^{-4} |I(t)|^3 \bigg(\int_{I(t)} |u|^2 \,\mathrm{d}y\bigg)^2.
\end{align*}
From this we infer the desired estimate by approximation.
\end{proof}

\bibliographystyle{abbrv}
\bibliography{multiphase}

\end{document}